\documentclass[11pt,reqno]{amsart}
\usepackage[colorlinks=true,allcolors=blue,backref=page]{hyperref}
\usepackage{color}
\usepackage{amsmath, amssymb, amsthm}

\usepackage{mathrsfs}
\usepackage{mathtools}
\usepackage[noabbrev,capitalize,nameinlink]{cleveref}
\usepackage{aliascnt}
\crefname{equation}{}{}
\usepackage{fullpage}
\usepackage[noadjust]{cite}
\usepackage{graphics}
\usepackage{pifont}
\usepackage{tikz}
\usepackage{bbm}
\usepackage[T1]{fontenc}

\usetikzlibrary{arrows.meta}

\usepackage{environ}
\usepackage{framed}
\usepackage{url}
\usepackage[linesnumbered,ruled,vlined]{algorithm2e}
\usepackage[noend]{algpseudocode}
\usepackage[labelfont=bf]{caption}
\usepackage{cite}
\usepackage{framed}
\usepackage[framemethod=tikz]{mdframed}
\usepackage{appendix}
\usepackage{graphicx}
\usepackage[textsize=tiny]{todonotes}
\usepackage{enumerate}
\usepackage[shortlabels]{enumitem}
\allowdisplaybreaks[1]

\apptocmd{\sloppy}{\hbadness 10000\relax}{}{} 

\crefname{algocf}{Algorithm}{Algorithms}

\crefname{equation}{}{} 
\crefname{conjecture}{Conjecture}{Conjectures} 
\AtBeginEnvironment{appendices}{\crefalias{section}{appendix}} 

\crefformat{enumi}{#2#1#3}
\Crefformat{enumi}{#2#1#3}
\crefrangeformat{enumi}{#3#1#4 to~#5#2#6}
\Crefrangeformat{enumi}{#3#1#4 to~#5#2#6}
\crefmultiformat{enumi}{#2#1#3}%
{ and~#2#1#3}{, #2#1#3}{ and~#2#1#3}
\Crefmultiformat{enumi}{#2#1#3}%
{ and~#2#1#3}{, #2#1#3}{ and~#2#1#3}

\usepackage[color,final]{showkeys} 

\colorlet{refkey}{orange!20}
\colorlet{labelkey}{blue!30}

\crefname{algocf}{Algorithm}{Algorithms}

\numberwithin{equation}{section}
\newtheorem{theorem}{Theorem}[section]
\newaliascnt{proposition}{theorem}
\newtheorem{proposition}[proposition]{Proposition}
\aliascntresetthe{proposition}
\newaliascnt{lemma}{theorem}
\newtheorem{lemma}[lemma]{Lemma}
\aliascntresetthe{lemma}
\newaliascnt{claim}{theorem}

\aliascntresetthe{claim}
\newtheorem*{claim*}{Claim}

\newaliascnt{corollary}{theorem}
\newtheorem{corollary}[corollary]{Corollary}
\aliascntresetthe{corollary}
\newaliascnt{conjecture}{theorem}

\newtheorem*{conjecture*}{Conjecture}
\aliascntresetthe{conjecture}
\newtheorem*{question*}{Question}
\newaliascnt{fact}{theorem}
\newtheorem{fact}[fact]{Fact}
\aliascntresetthe{fact}

\theoremstyle{definition}
\newaliascnt{definition}{theorem}
\newtheorem{definition}[definition]{Definition}
\aliascntresetthe{definition}
\newaliascnt{problem}{theorem}

\aliascntresetthe{problem}
\newaliascnt{question}{theorem}

\aliascntresetthe{question}
\newtheorem*{definition*}{Definition}
\newaliascnt{example}{theorem}

\aliascntresetthe{example}
\newaliascnt{setup}{theorem}

\aliascntresetthe{setup}

\crefname{theorem}{Theorem}{Theorems}
\crefname{proposition}{Proposition}{Propositions}
\crefname{lemma}{Lemma}{Lemmas}
\crefname{claim}{Claim}{Claims}
\crefname{corollary}{Corollary}{Corollaries}
\crefname{conjecture}{Conjecture}{Conjectures}
\crefname{fact}{Fact}{Facts}
\crefname{model}{Model}{Models}
\crefname{definition}{Definition}{Definitions}
\crefname{problem}{Problem}{Problems}
\crefname{question}{Question}{Questions}
\crefname{example}{Example}{Examples}
\crefname{setup}{Setup}{Setups}

\theoremstyle{remark}
\newtheorem*{remark}{Remark}

\newcommand{\norm}[1]{\left\lVert#1\right\rVert}

\newcommand{\one}{\mathbbm{1}}
\newcommand{\eps}{\varepsilon}

\newcommand{\mb}{\mathbb}
\newcommand{\mbf}{\mathbf}

\newcommand{\mc}{\mathcal}
\newcommand{\mf}{\mathfrak}
\newcommand{\mr}{\mathrm}
\newcommand{\msf}{\mathsf}
\newcommand{\ol}{\overline}
\newcommand{\on}{\operatorname}

\newcommand{\wt}{\widetilde}

\newcommand{\ve}{\varepsilon}

\newcommand{\Cst}[1]{C_{\ref{#1}}}
\newcommand{\Cstp}[1]{C'_{\ref{#1}}}
\newcommand{\Cof}[2]{C_{\ref{#1}}(#2)}
\newcommand{\Cofp}[2]{C'_{\ref{#1}}(#2)}

\allowdisplaybreaks

\title{Majority Dynamics on sparse random graphs}

\author[Goel]{Gopal Goel}
\author[Sah]{Ashwin Sah}
\email{\{gopal.krishna.goel,ashwinsahmath\}@gmail.com}

\thanks{Sah was supported by NSF Graduate Research Fellowship Program DGE-2141064.}

\begin{document}

\begin{abstract}
Consider a simple undirected graph $G$ on $N$ vertices, with each vertex holding an opinion from one of two options. Starting from the initial opinions on day $1$, we run \emph{majority dynamics}: on each subsequent day, all vertices simultaneously adopt the majority opinion among their neighbors, retaining their current opinion in case of a tie.
A well-studied conjecture of Benjamini, Chan, O'Donnell, Tamuz, and Tan states that if $G$ is drawn from the random binomial model $\mathbb{G}(N, p)$, and if the starting opinions are chosen uniformly at random, then as long as $pN\to\infty$ the network will converge to $99\%$ consensus with high probability.

We confirm this conjecture as long as $pN\ge N^\eps$ for any fixed $\eps > 0$, showing that full unanimity is reached with high probability in $O(1/\eps)$ days.
This result improves on a line of results by Fountoulakis, Kang, and Makai, who achieve $\eps = 1/2$; Chakraborti, Kim, Lee, and Tran, who achieve $\eps = 2/5$; and Jaffe, who achieves $\eps = 1/3$.

The proof technique involves iteratively revealing the ``opinion histories'' for each vertex through time, keeping track of the degrees between every vertex and each of $2^k$ opinion history classes on day $k$. Our analysis of this process requires intricate estimates for degree-constrained random graph models using graph enumeration tools from the work of McKay and Wormald as well as Canfield, Greenhill, and McKay, and their extensions by Liebenau and Wormald. In doing so, we connect the discrete dynamics to a deterministic idealized process, whose leading-order behavior is described by conditional Gaussian probabilities and expectations.
\end{abstract}

\maketitle
\tableofcontents

\section{Introduction}\label{sec:introduction}

The evolution of opinions held by a network of individuals has been studied extensively because of its central role in social systems. Even natural mathematical models of opinion dynamics exhibit complex behavior. One of the most basic is \emph{majority dynamics}, which has been studied in fields such as biophysics \cite{MP43} and psychology \cite{CH56} since the mid-twentieth century; see the survey of Mossel and Tamuz \cite{MT17} for broader context.

Explicitly, let $G$ be a simple graph on a finite vertex set $V$, and give every vertex an initial opinion in $\{\pm1\}$. We call the resulting pair $\mf g=(G,c)$ a \emph{politicized graph}, write $\on{gr}(\mf g)=G$, and write $c_1(\mf g)=c$ for its initial coloring. The opinions evolve in discrete time: on every day, all vertices simultaneously adopt the majority opinion among their neighbors, retaining their current opinion in case of a tie. Thus, for $k\ge1$,
\[
c_{k+1}(\mf g)[x]=
\begin{cases}
+1 & \text{if }\sum_{y\sim_{\on{gr}(\mf g)}x}c_k(\mf g)[y]>0,\\
-1 & \text{if }\sum_{y\sim_{\on{gr}(\mf g)}x}c_k(\mf g)[y]<0,\\
c_k(\mf g)[x] & \text{if }\sum_{y\sim_{\on{gr}(\mf g)}x}c_k(\mf g)[y]=0.
\end{cases}
\]
It follows from a theorem of Goles and Olivos \cite{GO80} that on every finite graph this process eventually reaches either a fixed point or a cycle of period two.

The geometry of the underlying network can affect the outcome in subtle ways, and a substantial literature studies this interaction. Tran and Vu \cite{TV20} discovered a striking \emph{power of few} phenomenon: in a dense random network, a constant-sized discrepancy in the initial opinion counts is likely to grow rapidly to consensus. Berkowitz and Devlin \cite{BD20} sharpened the size of the required discrepancy, and Sah and Sawhney \cite{SS24} established a \emph{power of one}.

The first systematic study of majority dynamics on the Erd\H{o}s--R\'enyi random graph was undertaken by Benjamini, Chan, O'Donnell, Tamuz, and Tan \cite{BCOTT16}. For a finite set $V$ and $p\in[0,1]$, let $\mb G(V,p)$ denote the random graph on $V$ in which every pair of vertices is an edge independently with probability $p$. Benjamini, Chan, O'Donnell, Tamuz, and Tan posed the following conjecture.

\begin{conjecture*}[{\cite[Conjecture~1.5]{BCOTT16}}]
Suppose that $pN\to\infty$, and fix any constant $\eta>0$. Run majority dynamics on $G\sim\mb G(V,p)$, where $|V|=N$, from a uniformly random initial coloring $c$. Then, with probability tending to $1$ as $N\to\infty$, we have
\[
\bigl|\sum_{x\in V}c_t((G,c))[x]\bigr|\ge(1-2\eta)N
\]
for all sufficiently large $t$.
\end{conjecture*}

They proved that for $p\ge CN^{-1/2}$, unanimity is reached by day four with probability at least $0.4$. Much of the intervening work has aimed either to prove this conjecture or to understand how a prescribed initial lead grows. A uniformly random initial coloring typically has a lead of order $\sqrt N$.

Fountoulakis, Kang, and Makai \cite{FKM20} improved the unanimity probability for $p\ge CN^{-1/2}$ to $1-\eps$, where $\eps=\eps(C)\to0$ as $C\to\infty$. Chakraborti, Kim, Lee, and Tran \cite{CKLT21} extended the range to $p\ge CN^{-3/5}\log N$, Jaffe \cite{Jaf25} reached $p\ge CN^{-2/3}$, and Kim and Tran \cite{KT25} obtained a related density threshold $p\gg N^{-2/3}\log^{2/3}N$. These arguments follow the random process for at most three days and then use pseudorandom properties of the graph to amplify the lead. As observed in \cite{KT25}, reaching lower densities requires following the random process for additional days.

There is also a parallel line of work on fixed initial leads. Write $2\Delta$ for the difference between the initial opinion counts. Extending their dense \emph{power of few} result, Tran and Vu \cite{TV23} showed that above the connectivity threshold, the conditions $p\Delta\ge10$ and $p\Delta^2\to\infty$ suffice for the initial majority to reach unanimity with high probability. Jaffe \cite{Jaf25} showed, in the same density range, that for suitable constants $A,B>0$ it is enough to have
\[
\Delta\ge\max\left\{p^{-1/2}\exp\left(A\sqrt{\log(1/p)}\right),\,Bp^{-3/2}N^{-1/2}\right\}.
\]
For $N^{-1}\log^2N\ll p\ll N^{-1/2}\log^{1/4}N$, Kim and Tran \cite{KT25} showed that $p^{3/2}\Delta\gg N^{-1/2}\log N$ suffices. Our theorem instead starts from the $\sqrt N$ scale corresponding to a uniformly random initial condition and follows its evolution through every fixed power-law density. Its linear-response description also suggests the possibility of treating smaller leads, but this would require control of the error terms at a correspondingly smaller scale.

\subsection{Main result}

For a finite set $V$, let $\mr{Polit}(V)$ denote the set of politicized graphs on $V$. We first state our result for a deterministic initial coloring.

\begin{theorem}\label{thm:succinct_main_result}
Fix a real parameter $\theta\in(1/2,1)$ and a constant $T>1$. For all sufficiently large $N\in\mb Z$, all $p\in(T^{-1}N^{-\theta},TN^{-\theta})$, and all $\tau\in[T^{-1},T]$, the following holds.

Let $V$ be a vertex set with $|V|=N$, and let $c:V\to\{\pm1\}$ be a coloring such that
\[
|c^{-1}(1)|=\lfloor N/2\rfloor+\lfloor\tau\sqrt N\rfloor.
\]
Sample $G\sim\mb G(V,p)$, and let $\mf g=(G,c)$. If
\[
k=2\left\lfloor\frac1{1-\theta}\right\rfloor+3,
\]
then with probability $1-o(1)$ we have $c_k(\mf g)[x]=1$ for every $x\in V$.
\end{theorem}

We also prove a uniform-density result, with one threshold for every $p$ from a power-law lower bound through $p=1$.

\begin{theorem}[Uniform density range]\label{thm:uniform-density}
Fix $\theta\in(1/2,1)$ and $T>1$. For every $\eta>0$ there exists $N_0=N_0(\theta,T,\eta)$ such that, for every integer $N\ge N_0$, every
\[
T^{-1}N^{-\theta}\le p\le1,
\]
every $\tau\in[T^{-1},T]$, and every deterministic coloring $c$ of a vertex set $V$ with $|V|=N$ and
\[
|c^{-1}(1)|=\lfloor N/2\rfloor+\lfloor\tau\sqrt N\rfloor,
\]
the graph $G\sim\mb G(V,p)$ satisfies
\[
\mb P\bigl[c_{k_{\rm u}}((G,c))[x]=1\text{ for every }x\in V\bigr]\ge1-\eta,
\qquad k_{\rm u}=2\left\lfloor\frac1{1-\theta}\right\rfloor+3.
\]
\end{theorem}

The density quantifier in \Cref{thm:uniform-density} is simultaneous: $N_0$ is chosen before $p$ and applies to every density in the displayed interval. \Cref{sec:uniform-proof} gives a proof sketch, including the response-dependent stopping argument and the use of the dense-range estimates of Fountoulakis, Kang, and Makai \cite{FKM20}; a complete Lean proof accompanies the paper, in the repository \url{https://github.com/gopalkgoel/sparse-majority-dynamics-lean}.

In both theorems the probability bound is uniform over the density, bias parameter, and initial coloring in the displayed ranges. The exponent $\theta$ is fixed before $N$ tends to infinity, and the threshold for $N$ may depend on $\theta$ and $T$. Once all vertices agree, they agree forever.

For uniformly random initial opinions, elementary binomial estimates place the absolute initial lead between two fixed positive multiples of $\sqrt N$ with probability arbitrarily close to one, by taking the lower multiple small and the upper multiple large. We may therefore apply \Cref{thm:succinct_main_result} uniformly in such a fixed window, reversing the two opinions if necessary, and then let the window expand.

\begin{corollary}[Uniformly random initial opinions]\label{cor:random-opinions}
Fix a real parameter $\theta\in(1/2,1)$ and a constant $T>1$, and let
\[
k=2\left\lfloor\frac1{1-\theta}\right\rfloor+3.
\]
For all sufficiently large $N\in\mb Z$ and all $p\in(T^{-1}N^{-\theta},TN^{-\theta})$, the following holds. Let $V$ be a vertex set with $|V|=N$, let $c:V\to\{\pm1\}$ be a uniformly random coloring, and let $G\sim\mb G(V,p)$ be independent of $c$. Then with probability $1-o(1)$ the coloring $c_t((G,c))$ is constant for every $t\ge k$.
\end{corollary}

Applying the same conditioning and color-flip argument to \Cref{thm:uniform-density} gives unanimity from uniformly random initial opinions, uniformly over $T^{-1}N^{-\theta}\le p\le1$, by day $k_{\rm u}$. In particular, for every fixed $\eps>0$, uniformly over $pN\ge N^\eps$, unanimity is reached with probability $1-o(1)$ in $O(1/\eps)$ days. The binomial window argument is given in \Cref{sec:uniform-proof}. This confirms the conjecture of \cite{BCOTT16} throughout the polynomial average-degree regime, with unanimity in place of near-unanimity. The regime in which $pN$ grows more slowly than every positive power of $N$ remains outside these results.

\subsection{Overview of key ideas}\label{sub:overview}

We now describe the main ideas of the proof. The argument views majority dynamics as an iteratively revealed degree-constrained random graph process. We transfer the relevant properties of this dependent process to a tilted independent model, and then compare its evolution with deterministic parameters governed by conditional expectations of multivariate Gaussians.

The heuristic is that a small lead grows by a factor of order $\sqrt{pN}$ per update. A typical neighborhood contains about $pN$ vertices, and sampling it from a population with a lead of order $\tau\sqrt N$ produces an excess probability of order $\tau\sqrt N\sqrt{pN}/N$ for the leading opinion. Summed across the vertices, this predicts a lead of order $\tau\sqrt N\sqrt{pN}$ on the next day. Since $\sqrt{pN}\asymp N^{(1-\theta)/2}$, it takes roughly $1/(1-\theta)$ updates for a lead of order $\sqrt N$ to become macroscopic. This explains both why earlier analyses of only the first few days encounter a density barrier and why a bounded but arbitrarily long iterative argument is needed here. Throughout the paper, day $1$ means the initial coloring.

\subsubsection{Iterative degree revelation}

On day $k$, we distinguish vertices according to their full opinion histories. A history is a string $s\in\{0,1\}^k$, where $0$ represents opinion $+1$ and $1$ represents opinion $-1$, and we write
\[
\Pi_k(\mf g)[s]
=
\bigl\{x\in V:c_r(\mf g)[x]=(-1)^{s[r-1]}\text{ for every }1\le r\le k\bigr\}.
\]
These $2^k$ classes form the day-$k$ partition of $V$. The transition to day $k+1$ depends only on each vertex's degrees into the day-$k$ classes: for $x\in\Pi_k(\mf g)[s]$, the next opinion is determined by the sign of
\[
\sum_{t\in\{0,1\}^k}(-1)^{t[k-1]}
\deg^{\on{gr}(\mf g)}_{\Pi_k(\mf g)[t]}x,
\]
together with the tie-breaking rule.

This observation lets us reveal only the randomness needed for the next update. The \emph{fine state} consists of the opinion-history partition and the array recording every vertex's degree into every class. Given this state, the unrevealed graph is uniform over graphs with the prescribed within-class and between-class degree sequences. Equivalently, its restrictions to the different class pairs are independent uniform graphs or bipartite graphs with prescribed degree sequences. This gives a universal one-step kernel $K_k$, independent of $p$ and of the initial coloring.

The fine state contains more information than the induction needs. We therefore pass to a \emph{coarse state}, which retains the history partition, the total number of edges between every pair of classes, and whether the degree array satisfies the required regularity bounds. Different fine states with the same coarse data can have different next-step laws, so this projection is not a lumping of the fine chain. Conditioning on the coarse state instead averages the universal fine-state kernel over all compatible fine states. Analyzing this fiber-averaged transition, which is the exact conditional one-day law, is a key structural step in the proof.

\subsubsection{Distributional transference through a solvability tilt}

The main obstacle is that conditioning on the current histories and edge totals introduces substantial dependence. We compare the resulting degree-array law with a simpler row model in which, for a vertex in class $s$, its degrees into the classes $t\in\{0,1\}^k$ are independent binomial variables. Their edge probabilities are allowed a small tilt from $p$, written $\lambda[s,t]$, and the row is conditioned to satisfy the inequalities encoding the vertex's opinion history.

The tilt is chosen to reproduce the prescribed edge totals in conditional expectation. The existence of such a tilt is the \emph{solvability} condition. Once it holds, concentration may first be proved in the independent row model and then transferred back to the true coarse fiber law, after imposing the exact edge totals.

The analytic input enabling this transfer is the asymptotic enumeration of graphs with prescribed degree sequences, beginning with McKay and Wormald \cite{MW90,MW97} and, in the bipartite setting, Canfield, Greenhill, and McKay \cite{CGM08}, and extended across the relevant densities by Liebenau and Wormald \cite{LW17,LW20}; see also the survey \cite{Wor18}. These enumeration formulae compare the probability masses of individual degree arrays satisfying suitable regularity bounds in the graph and binomial models; conditioning on the histories and exact totals then gives the required transference. A polynomial lower bound for the probability of the exact totals allows us to add that conditioning without losing the concentration estimates. The multidimensional Berry--Esseen theorem of Bentkus \cite{Ben05} enters the Fourier local limit estimate used to bound the probability of the exact totals.

The resulting local theorem says, roughly, that from any coarse state satisfying explicit admissibility conditions, the next class sizes and edge totals follow a deterministic local template with high probability, and the next degree array again satisfies the required regularity bounds. This is the probabilistic step that can be iterated.

\subsubsection{Asymptotic evolution of the history classes}

The majority decision for a vertex in a given history class is governed by an alternating sum of binomial degrees. After centering and scaling by $\sqrt{pN}$, the degree row converges to a multivariate Gaussian conditioned on the inequalities recording the previous majority decisions. Consequently, the local template has a universal limiting description in terms of conditional Gaussian probabilities and expectations.

Suppose that opinion $+1$ initially leads by $2\tau\sqrt N+O(1)$. We construct deterministic idealized class sizes $\wt n[s]$ and edge totals $\wt m[s,t]$, depending on $(N,p)$, and universal response coefficients $\varepsilon[s]$. The idealized quantities describe the evolution of the local template from balanced initial data, while $\varepsilon[s]$ describes its first-order response to the initial lead. For every fixed day in the perturbative range, the actual class sizes satisfy
\[
\bigl|\Pi_k(\mf g)[s]\bigr|
=
\wt n[s]
+
\tau\sqrt N\,(\sqrt{pN})^{k-1}\varepsilon[s]
+
\text{a smaller error}
\]
with high probability. The idealized sizes themselves have leading order $\wt n[s]\sim N\nu[s]$. The proportions $\nu[s]$ are defined together with edge-correction parameters $\mu[s,t]$ and Gaussian mean parameters $\gamma[s,t]$ by a universal recursion independent of $N$ and $p$. The response coefficients $\varepsilon[s]$ evolve through a companion linearized recursion built from the same Gaussian data. Thus, for each fixed day, all four families consist of universal real constants.

We emphasize two key points. First, there is no random \emph{idealized dynamics} that exactly realizes $\wt n$ and $\wt m$. They form a deterministic sequence obtained by iterating the local template from balanced initial data. Second, although $\wt n[s]\sim N\nu[s]$, the available bound
\[
\bigl|\wt n[s]-N\nu[s]\bigr|\le \frac{N}{\sqrt{pN}}\log^\ell N
\]
is much larger than the lead scale $\sqrt N(\sqrt{pN})^{k-1}$ on days $k<1/(1-\theta)$. Replacing $\wt n[s]$ by the simpler approximation $N\nu[s]$ would therefore obscure the signal we need to track. To extract the lead at its correct scale, we use the exact color-flip symmetry
\[
\wt n[\ol s]=\wt n[s],
\]
so the idealized contributions cancel when the history classes are summed according to their current opinions. The remaining first-order contribution satisfies
\[
\varepsilon[s0]>0>\varepsilon[s1]
\]
for every history $s$. Thus the initially leading opinion gains vertices coherently across all history classes, rather than suffering cancellation between them.

To turn this asymptotic description into an induction, we call a coarse state \emph{faithful} when its sizes and edge totals lie within prescribed tolerances of the idealized parameters plus their linear response to the initial lead. The Gaussian recursion provides the limiting parameters; inverse-function arguments produce the tilted binomial parameters realizing the required edge totals; and the local theorem propagates faithfulness from one day to the next.

When $k_0:=1/(1-\theta)$ is an integer, the perturbation reaches the boundary of the linear regime on day $k_0$. A separate one-sided argument then supplies the needed macroscopic lead on day $k_0+1$. In all cases, by day
\[
\left\lfloor\frac1{1-\theta}\right\rfloor+1
\]
the leading opinion has a lead of at least order $N\log N/\sqrt{pN}$.

\subsubsection{Contraction phase}

At that point we leave the degree-revelation analysis. An argument of Chakraborti, Kim, Lee, and Tran \cite{CKLT21} shows how the jumbledness and minimum-degree properties of $\mb G(V,p)$ force contraction toward unanimity once the lead is sufficiently large. Specifically, the lead becomes macroscopic after one further day, and thereafter the set of vertices holding the minority opinion shrinks by a factor $O((pN)^{-1})$ on each day. After at most $\lfloor1/(1-\theta)\rfloor+1$ additional days we reach unanimity, yielding the time bound in \Cref{thm:succinct_main_result}.

\subsection{Organization}

\Cref{sec:setup} encodes opinion histories as fine states, constructs the universal transition kernel, and passes to the coarse fiber law used by the induction.
\Cref{sec:local} proves the one-day local transition theorem by transferring concentration from a tilted independent row model and then controlling the resampled degree-constrained graph.
\Cref{sec:universal} develops the universal Gaussian recursion and proves the coherence of its linear response.
\Cref{sec:idealized} constructs the idealized parameters, defines faithful states, and proves that the actual coarse trajectory remains faithful until the lead leaves the perturbative scale.
\Cref{sec:main-proof} combines the resulting lead with the deterministic contraction phase to prove the fixed power-law theorem and its random-initial-opinions corollary, then sketches the uniform-density extension.
\Cref{app:tools} collects the comparison results of Liebenau and Wormald, together with the binomial, graphicality, concentration, and Gaussian-conditioning tools used throughout.
\Cref{app:local-transference} proves the comparison between the coarse fiber law and the conditioned row model.
\Cref{app:graph-enumeration-local} derives the edge-splitting and degree-regularity estimates needed by the transition kernel from single-edge estimates of Liebenau and Wormald.
\Cref{app:inverse} supplies the inverse-function results for the conditional-mean maps underlying the Gaussian and binomial tilts.
\Cref{app:idealized} proves the existence, perturbation, and critical-day statements used in the idealized evolution.

\subsection{Notation and conventions}\label{sec:conventions}

All asymptotic notation is interpreted as the relevant size parameter tends to infinity. In a statement which first fixes constants and then quantifies over all sufficiently large $N$ and over further choices depending on $N$, the implicit constants and $o(1)$-terms may depend on the initially fixed constants but are uniform over all subsequently quantified choices. Thus, for example, the conclusion of \Cref{thm:succinct_main_result} is uniform over $p$ and $\tau$ in the stated ranges.

A statement that asserts the existence of a constant names the constant after the statement and records the parameters on which it depends; for example, \Cref{thm:local-coarse-transition} produces $\Cst{thm:local-coarse-transition}=\Cst{thm:local-coarse-transition}(\theta,k,T,\phi)$. Within that statement and its proof the constant is written without arguments. Every later use writes it as the function it is, evaluated at the parameters in force at that point. For instance, $\Cof{thm:nice_deg}{\theta,T_1}$ is the constant of \Cref{thm:nice_deg} for the parameter value $T_1$ in place of $T$; and the constant of \Cref{lem:gamma_bound} with $T^{1+\theta}$ in place of $T$ and with input $K$ is written $\Cof{lem:gamma_bound}{\theta,T^{1+\theta},K}$. Constants named $C_0,C_1,\ldots$ or $c,c_0,\ldots$ inside a proof are internal to that proof.

Parentheses indicate map application, square brackets index structured data, and subscripts index families and parameters. For example, $\Pi_k(\mf g)[s]$ is the $s$-part of the partition $\Pi_k(\mf g)$; these objects are defined in \Cref{sec:setup}. We retain the usual probabilistic notation $\mb P[\cdot]$ and $\mb E[\cdot]$.

\subsection{Acknowledgments}

We thank Mehtaab Sawhney for many valuable ideas and contributions to this project, including the ideation and early stages.

\smallskip\noindent\textbf{Statement on AI use.} A complete manuscript containing the proof of \Cref{thm:succinct_main_result} was written by the authors without AI assistance by August 2025, with further human revision continuing into early 2026. From August 2025 through August 2026, the authors used GPT-5.5, GPT-5.6, and Claude Fable to compare versions, trace dependencies, reorganize arguments, and assist with rewriting. In September 2026, these systems helped restructure the proof according to its logical dependencies and audit the argument line by line; the audit resolved several minor issues in the original presentation. Between September 6 and 10, GPT-6 Astra produced a complete Lean formalization of the proof under extensive direction from the authors. Claude Fable separately formalized the required results of Liebenau and Wormald \cite{LW17,LW20} in a single session after being supplied the two papers. GPT-6 Astra and Claude Fable were subsequently used to prove and formalize the uniform-density strengthening, \Cref{thm:uniform-density}, including the dense-range argument based on Fountoulakis, Kang, and Makai \cite{FKM20}. The authors reviewed and substantially revised the resulting mathematics and exposition. The mathematical content, final text, and any errors are the responsibility of the authors.

\section{Setup}\label{sec:setup}

\subsection{Politicized graphs and opinion histories}

Throughout this section, fix a finite vertex set $V$.

We index strings in $\{0,1\}^k$ from $0$ to $k-1$, so if $s\in\{0,1\}^k$, then $s=s[0]\cdots s[k-1]$.

For $s\in\{0,1\}^2$, define
\[
\ge_s:=\begin{cases}
\ge & \text{if }s=00,\\
< & \text{if }s=01,\\
> & \text{if }s=10,\\
\le & \text{if }s=11.
\end{cases}
\]

For every $k\ge 1$, define
\[
\mr{Part}_k(V):=\left\{\Pi:\{0,1\}^k\to 2^V:V=\bigsqcup_{s\in\{0,1\}^k}\Pi[s]\right\}.
\]
Thus an element of $\mr{Part}_k(V)$ is a partition of $V$ whose parts are indexed by $\{0,1\}^k$.

For every $k\ge 1$, the day-$k$ partition of $\mf{g}$ is the element $\Pi_k(\mf{g})\in\mr{Part}_k(V)$ given by
\[
\Pi_k(\mf{g})[s]:=\left\{x\in V:c_r(\mf{g})[x]=(-1)^{s[r-1]}\text{ for all }1\le r\le k\right\}.
\]
In particular,
\[
V=\bigsqcup_{s\in\{0,1\}^k}\Pi_k(\mf{g})[s]
\]
for every $k\ge 1$, and $\Pi_k(\mf{g})[s]$ is exactly the set of vertices whose opinion history through day $k$ is the binary string $s$, with $0$ representing $+1$ and $1$ representing $-1$.

For any $u\in\mb{R}^{\{0,1\}^k}$ and $s\in\{0,1\}^k$, let $\mc{I}_s(u)$ denote the assertion that
\[
\sum_{t\in\{0,1\}^k}(-1)^{t[r-1]}u[t]\ge_{s[r-1]s[r]}0
\]
for all $1\le r\le k-1$. For any $u\in\mb{R}^{\{0,1\}^k}$ and $s\in\{0,1\}^{k+1}$, let $\mc{J}_s(u)$ denote the assertion that
\[
\sum_{t\in\{0,1\}^k}(-1)^{t[r-1]}u[t]\ge_{s[r-1]s[r]}0
\]
for all $1\le r\le k$.

For a graph $H$ on $V$, a vertex $x\in V$, and a subset $S\subseteq V$, write $\deg^H_Sx$ for the number of neighbors of $x$ in $S$ in the graph $H$. If $u[t]=\deg^{\on{gr}(\mf{g})}_{\Pi_k(\mf{g})[t]}x$, then the $r$th inequality above records the sign of the day-$r$ neighbor imbalance around $x$, with ties resolved according to the day-$r$ opinion of $x$. Consequently, for every $s\in\{0,1\}^k$ and $b\in\{0,1\}$,
\[
\Pi_{k+1}(\mf{g})[sb]=\left\{x\in \Pi_k(\mf{g})[s]:\mc{J}_{sb}\left(\left(\deg^{\on{gr}(\mf{g})}_{\Pi_k(\mf{g})[t]}x\right)_{t\in\{0,1\}^k}\right)\right\}.
\]

\subsection{State spaces and the iterative degree revelation process}\label{sec:setup-states}

For $\Pi\in\mr{Part}_k(V)$ and a graph $H$ on $V$, define the degree array $\delta(\Pi,H)\in\mb{Z}^{V\times\{0,1\}^k}$ by
\[
\delta(\Pi,H)[x,t]:=\deg^H_{\Pi[t]}x
\]
for $x\in V$ and $t\in\{0,1\}^k$. For $\mbf{d}\in\mb{Z}^{V\times\{0,1\}^k}$, write $\mbf{d}[x,\bullet]=(\mbf{d}[x,t])_{t\in\{0,1\}^k}$ and define
\[
m(\Pi,\mbf{d})[s,t]:=\sum_{x\in\Pi[s]}\mbf{d}[x,t].
\]

A finite sequence of nonnegative integers is \emph{graphical} if it is the degree sequence of a simple graph, and a pair of such sequences is \emph{bipartite graphical} if it is the pair of side-degree sequences of a simple bipartite graph. For $\Pi\in\mr{Part}_k(V)$, let $\mc{D}(\Pi)\subseteq\mb{Z}^{V\times\{0,1\}^k}$ be the set of graph-realizable degree arrays relative to $\Pi$:
\[
\mc{D}(\Pi):=\left\{\delta(\Pi,H):H\text{ is a simple graph on }V\right\}.
\]
We also call the elements of $\mc{D}(\Pi)$ \emph{graphical}.

We fix some terminology, which recurs throughout. The parts $\Pi[s]$ of a partition $\Pi\in\mr{Part}_k(V)$ are called its \emph{blocks}. A \emph{block pair} of $\Pi$ is an unordered pair $\{s,t\}$ of indices $s,t\in\{0,1\}^k$; it is \emph{internal} if $s=t$, in which case it refers to the edges inside the single block $\Pi[s]$, and \emph{bipartite} if $s\ne t$, in which case it refers to the edges between $\Pi[s]$ and $\Pi[t]$. Every simple graph on $V$ decomposes into its restrictions to the $2^k$ internal block pairs and the $\binom{2^k}{2}$ bipartite block pairs of $\Pi$, and these restrictions involve disjoint sets of vertex pairs. When several partitions are in play, we say ``block of $\Pi$'' and ``block pair of $\Pi$''. Graphicality of an array is a condition block pair by block pair:

\begin{fact}[Graphicality is a block-pair condition]\label{fact:graphical-blockwise}
Let $\Pi\in\mr{Part}_k(V)$ and $\mbf{d}\in\mb{Z}^{V\times\{0,1\}^k}$. Then $\mbf{d}\in\mc{D}(\Pi)$ if and only if, for each $s\in\{0,1\}^k$, the sequence $(\mbf{d}[x,s])_{x\in\Pi[s]}$ is graphical, and for each pair $s,t\in\{0,1\}^k$ with $s\ne t$, the pair of sequences $(\mbf{d}[x,t])_{x\in\Pi[s]}$ and $(\mbf{d}[y,s])_{y\in\Pi[t]}$ is bipartite graphical.
\end{fact}

\begin{proof}
If $\mbf{d}=\delta(\Pi,H)$, then the restriction of $H$ to $\Pi[s]$ realizes $(\mbf{d}[x,s])_{x\in\Pi[s]}$ and the restriction of $H$ to the pairs between $\Pi[s]$ and $\Pi[t]$ realizes the pair $(\mbf{d}[x,t])_{x\in\Pi[s]}$, $(\mbf{d}[y,s])_{y\in\Pi[t]}$. Conversely, given a graph on each $\Pi[s]$ and a bipartite graph between each $\Pi[s]$ and $\Pi[t]$ realizing these sequences, their union $H$ is a simple graph on $V$, because the block pairs involve disjoint sets of vertex pairs, and $\delta(\Pi,H)=\mbf{d}$ by construction.
\end{proof}

The history constraints select the subset
\[
\mc{D}^\mc{I}(\Pi):=\left\{\mbf{d}\in\mc{D}(\Pi):\mc{I}_s(\mbf{d}[x,\bullet])\text{ for all }s\in\{0,1\}^k\text{ and }x\in\Pi[s]\right\}.
\]

The day-$k$ state space is
\[
\mc{S}_k(V):=\left\{(\Pi,\mbf{d}):\Pi\in\mr{Part}_k(V)\text{ and }\mbf{d}\in\mc{D}^\mc{I}(\Pi)\right\}.
\]
For $\sigma=(\Pi,\mbf{d})\in\mc{S}_k(V)$, we write
\[
\on{part}(\sigma):=\Pi,\qquad \on{deg}(\sigma):=\mbf{d}.
\]
This is the convention we use throughout for structured objects: canonical components are accessed by named projection maps rather than repeated unpacking.

Every politicized graph $\mf{g}$ on $V$ determines a state
\[
\sigma_k(\mf{g}):=\left(\Pi_k(\mf{g}),\delta(\Pi_k(\mf{g}),\on{gr}(\mf{g}))\right)\in\mc{S}_k(V).
\]
This information determines the next partition.

\begin{fact}[Reconstruction from a fine state]\label{fact:state-reconstruction}
Let $k\ge1$, let $V$ be finite, and let $(\Pi,\mbf d)\in\mc S_k(V)$. For every simple graph $H$ on $V$ with $\delta(\Pi,H)=\mbf d$ and every coloring $c:V\to\{\pm1\}$ satisfying $c[x]=(-1)^{s[0]}$ for $x\in\Pi[s]$, the politicized graph $\mf g:=(H,c)$ satisfies
\[
c_r(\mf g)[x]=(-1)^{s[r-1]}\qquad(1\le r\le k,\ s\in\{0,1\}^k,\ x\in\Pi[s]).
\]
In particular, $\sigma_k(\mf g)=(\Pi,\mbf d)$, and every state in $\mc S_k(V)$ is attainable.
\end{fact}
\begin{proof}
The assertion for $r=1$ is the hypothesis on $c$. If it holds for some $r<k$, then for $x\in\Pi[s]$,
\[
\sum_{y\sim_H x}c_r(\mf g)[y]
=\sum_{t\in\{0,1\}^k}(-1)^{t[r-1]}\deg^H_{\Pi[t]}x
=\sum_{t\in\{0,1\}^k}(-1)^{t[r-1]}\mbf d[x,t]
\ge_{s[r-1]s[r]}0.
\]
The last relation is the $r$th constraint in $\mc I_s(\mbf d[x,\bullet])$. Together with $c_r(\mf g)[x]=(-1)^{s[r-1]}$, the tie-breaking rule gives $c_{r+1}(\mf g)[x]=(-1)^{s[r]}$. This proves the assertion by induction, hence $\Pi_k(\mf g)=\Pi$ and $\sigma_k(\mf g)=(\Pi,\mbf d)$. A realizing graph $H$ exists because $\mbf d\in\mc D(\Pi)$, which proves attainability.
\end{proof}

For $\sigma\in\mc{S}_k(V)$, define its refinement $\Phi_k(\sigma)\in\mr{Part}_{k+1}(V)$ by
\[
\Phi_k(\sigma)[sb]:=\left\{x\in\on{part}(\sigma)[s]:\mc{J}_{sb}(\on{deg}(\sigma)[x,\bullet])\right\}
\]
for every $s\in\{0,1\}^k$ and $b\in\{0,1\}$. Because for each fixed $x\in\on{part}(\sigma)[s]$ exactly one of $\mc{J}_{s0}(\on{deg}(\sigma)[x,\bullet])$ and $\mc{J}_{s1}(\on{deg}(\sigma)[x,\bullet])$ holds, this is indeed a partition of $V$. Moreover, if $\sigma=\sigma_k(\mf{g})$, then $\Phi_k(\sigma)=\Pi_{k+1}(\mf{g})$.

\begin{definition*}[The kernel $K_k$]
For every $\sigma=(\Pi,\mbf{d})\in\mc{S}_k(V)$, define $K_k(\sigma,\cdot)$ as follows. For each $s\in\{0,1\}^k$, sample a uniformly random simple graph $H_{s,s}$ on $\Pi[s]$ with degree sequence $(\mbf{d}[x,s])_{x\in\Pi[s]}$. For each unordered pair $\{s,t\}$ of distinct elements $s,t\in\{0,1\}^k$, sample a uniformly random simple bipartite graph $H_{s,t}=H_{t,s}$ between $\Pi[s]$ and $\Pi[t]$ with degree sequences $(\mbf{d}[x,t])_{x\in\Pi[s]}$ and $(\mbf{d}[y,s])_{y\in\Pi[t]}$. Take all these sampled graphs and bipartite graphs to be mutually independent.

Given these sampled objects, define $\mbf{D}\in\mb{Z}^{V\times\{0,1\}^{k+1}}$ by setting, for $x\in\Pi[s]$, $t\in\{0,1\}^k$, and $b\in\{0,1\}$,
\[
\mbf{D}[x,tb]=
\begin{cases}
\deg^{H_{s,s}}_{\Phi_k(\sigma)[tb]}x,& t=s,\\
\deg^{H_{s,t}}_{\Phi_k(\sigma)[tb]}x,& t\ne s.
\end{cases}
\]
Then, for every subset of states $A\subseteq\mc{S}_{k+1}(V)$, set
\[
K_k(\sigma,A):=\mb{P}\left[(\Phi_k(\sigma),\mbf{D})\in A\right].
\]
This is well-defined for every $\sigma\in\mc{S}_k(V)$: the sampled objects exist by \Cref{fact:graphical-blockwise}, and the pair $(\Phi_k(\sigma),\mbf{D})$ always lies in $\mc{S}_{k+1}(V)$. Indeed, let $H$ be the simple graph on $V$ whose edge set is the union of the edge sets of the graphs $H_{s,s}$ and the bipartite graphs $H_{s,t}$; these edge sets are disjoint, so this is well defined. Then $\mbf{D}=\delta(\Phi_k(\sigma),H)$ by construction, so $\mbf{D}\in\mc{D}(\Phi_k(\sigma))$. Moreover, if $x\in\Phi_k(\sigma)[sb]$, then for each $1\le r\le k$,
\[
\sum_{u\in\{0,1\}^{k+1}}(-1)^{u[r-1]}\mbf{D}[x,u]
=
\sum_{t\in\{0,1\}^k}(-1)^{t[r-1]}\mbf{d}[x,t],
\]
because $\Phi_k(\sigma)[t0]\sqcup\Phi_k(\sigma)[t1]=\Pi[t]$ for each $t\in\{0,1\}^k$, so that $\mbf{D}[x,t0]+\mbf{D}[x,t1]=\mbf{d}[x,t]$ and the two sums agree term by term. Since $x\in\Phi_k(\sigma)[sb]$, the condition $\mc{J}_{sb}(\mbf{d}[x,\bullet])$ holds, and by the displayed identity this is the same as the condition $\mc{I}_{sb}(\mbf{D}[x,\bullet])$. As this holds for every $x\in V$, we conclude that $\mbf{D}\in\mc{D}^{\mc I}(\Phi_k(\sigma))$, that is, $(\Phi_k(\sigma),\mbf{D})\in\mc{S}_{k+1}(V)$.
\end{definition*}

\begin{proposition}\label{prop:state-chain}
Fix $p\in(0,1)$, a finite set $V$, and a coloring $c:V\to\{\pm1\}$. Sample $G\sim\mb{G}(V,p)$, and let $\mf{g}=(G,c)$. Then $\left(\sigma_k(\mf{g})\right)_{k\ge 1}$ is a time-inhomogeneous Markov chain on the finite state spaces $\mc{S}_k(V)$ with transition kernels $(K_k)_{k\ge1}$. Equivalently, for every $\sigma\in\mc{S}_k(V)$ with $\mb{P}[\sigma_k(\mf{g})=\sigma]>0$ and every $A\subseteq\mc{S}_{k+1}(V)$,
\[
\mb{P}\left[\sigma_{k+1}(\mf{g})\in A\mid \sigma_k(\mf{g})=\sigma\right]=K_k(\sigma,A).
\]
\end{proposition}

\begin{proof}
We first explain why the displayed one-step statement, which conditions on the present state only, is equivalent to the Markov property, which conditions on the whole history $(\sigma_1(\mf{g}),\ldots,\sigma_k(\mf{g}))$. The reason is that $\sigma_k(\mf{g})$ determines the earlier states: for $r<k$, the block $\Pi_r(\mf{g})[u]$ is the union of the blocks $\Pi_k(\mf{g})[s]$ over the strings $s$ whose first $r$ bits form $u$, and $\delta(\Pi_r(\mf{g}),G)[x,u]$ is the sum of the entries $\delta(\Pi_k(\mf{g}),G)[x,s]$ over the same strings $s$. Hence conditioning on $\sigma_k(\mf{g})$ is the same as conditioning on the whole history, and it suffices to prove the displayed identity.

Once $\sigma_k(\mf{g})=(\Pi,\mbf{d})$ is known, the next partition is already determined: it is exactly $\Phi_k(\sigma)$. The only remaining randomness in $\sigma_{k+1}(\mf{g})$ is the way the degree-constrained edge sets of the block pairs of $\Pi$ split across the deterministic refinement $\Phi_k(\sigma)$.

Condition on $\sigma_k(\mf{g})=(\Pi,\mbf{d})$. Since this state has positive probability, $\Pi$ is compatible with the fixed initial coloring $c$, in the sense that $x\in\Pi[s]$ implies $c[x]=(-1)^{s[0]}$. Conversely, on the event $\delta(\Pi,G)=\mbf{d}$, \Cref{fact:state-reconstruction} gives $\sigma_k(\mf g)=\sigma$. Thus, for the present state $\sigma$, the event $\{\sigma_k(\mf{g})=\sigma\}$ is exactly the event $\{\delta(\Pi,G)=\mbf{d}\}$.

The random edge sets of the internal and bipartite block pairs of $\Pi$ are independent in $\mb{G}(V,p)$, and the constraints recorded by $\mbf{d}$ involve these disjoint edge sets separately. Therefore the conditional law is a product over the block pairs of $\Pi$. Inside $\Pi[s]$, every simple graph with degree sequence $(\mbf{d}[x,s])_{x\in\Pi[s]}$ has the same number of edges, and hence the same conditional probability; similarly, between $\Pi[s]$ and $\Pi[t]$, every simple bipartite graph with degree sequences $(\mbf{d}[x,t])_{x\in\Pi[s]}$ and $(\mbf{d}[y,s])_{y\in\Pi[t]}$ has the same number of edges. Thus the conditional distribution is exactly the product of uniformly random degree-constrained graphs and bipartite graphs used to define $K_k(\sigma,\cdot)$.

It follows that the conditional law of $\sigma_{k+1}(\mf{g})$ given $\sigma_k(\mf{g})=\sigma$ is $K_k(\sigma,\cdot)$, which depends only on the present state. This is the Markov property.
\end{proof}

\subsection{Coarsening the state space}\label{sec:coarsening}

We next record the coarser information that will be propagated in the inductive estimates. For $\Pi\in\mr{Part}_k(V)$, define the edge-count space
\[
\begin{aligned}
\mc{M}(\Pi):=\Bigl\{m\in\mb{Z}^{\{0,1\}^k\times\{0,1\}^k}:{}& m[s,t]=m[t,s],\quad m[s,s]\in 2\mb{Z},\\
&0\le m[s,t]\le |\Pi[s]|\left(|\Pi[t]|-\one_{s=t}\right)\text{ for all }s,t\in\{0,1\}^k\Bigr\}.
\end{aligned}
\]
Thus $m[s,t]$ records the number of edge incidences from $\Pi[s]$ to $\Pi[t]$; for $s\ne t$ this is the number of edges between the two parts, while $m[s,s]$ is twice the number of edges inside $\Pi[s]$. For every $\mbf{d}\in\mc{D}(\Pi)$, the array $m(\Pi,\mbf{d})$ defined above belongs to $\mc{M}(\Pi)$.

For $p\in(0,1)$ and $\mbf{d}\in\mb{Z}^{V\times\{0,1\}^k}$, let $\kappa_p(\Pi,\mbf{d})$ denote the assertion that
\[
\left|\mbf{d}[x,t]-p|\Pi[t]|\right|\le (p|V|)^{4/7}
\]
for every $x\in V$ and $t\in\{0,1\}^k$.

The day-$k$ coarse state space is the finite set
\[
\mc{Y}_k(V):=\left\{(\Pi,m,\eta):\Pi\in\mr{Part}_k(V),\; m\in\mc{M}(\Pi),\; \eta\in\{0,1\}\right\}.
\]
For $y=(\Pi,m,\eta)\in\mc{Y}_k(V)$, we write
\[
\on{part}(y):=\Pi,\qquad \on{edge}(y):=m,\qquad \on{reg}(y):=\eta.
\]
We use $\mc{Y}_k(V)$ as a common codomain; not every element of $\mc{Y}_k(V)$ need arise from a fine state. For each $p\in(0,1)$, define the coarsening map $\rho_k^{(p)}:\mc{S}_k(V)\to\mc{Y}_k(V)$ by
\[
\rho_k^{(p)}(\sigma):=\left(\on{part}(\sigma),m(\on{part}(\sigma),\on{deg}(\sigma)),\one_{\kappa_p(\on{part}(\sigma),\on{deg}(\sigma))}\right).
\]
Thus $\rho_k^{(p)}$ records the current partition, the edge counts between its parts, and whether the fine degree array is $p$-regular in the sense above.

Say that $y\in\mc{Y}_k(V)$ is $p$-attainable if there exists $\sigma\in\mc{S}_k(V)$ such that $\rho_k^{(p)}(\sigma)=y$. We next define the \emph{fiber law} $\Lambda_y^{(p)}$ of a $p$-attainable coarse state $y$: it is the law of a fine state sampled from the fiber $(\rho_k^{(p)})^{-1}(y)$, with the weights that the random graph $\mb{G}(V,p)$ induces on that fiber. Sample $G\sim\mb{G}(V,p)$ and set
\[
\widetilde\sigma_y(G):=\left(\on{part}(y),\delta(\on{part}(y),G)\right).
\]
Then $\Lambda_y^{(p)}$ is the law of $\widetilde\sigma_y(G)$ conditioned on the event that $\widetilde\sigma_y(G)$ lies in $\mc{S}_k(V)$ and, for this resulting fine state,
\[
\rho_k^{(p)}\left(\widetilde\sigma_y(G)\right)=y.
\]
Since $p\in(0,1)$ and $y$ is $p$-attainable, this conditioning event has positive probability.

Concretely, write $\Pi:=\on{part}(y)$ and $m:=\on{edge}(y)$. Choose independently a uniform graph with $m[s,s]/2$ edges inside each $\Pi[s]$ and a uniform bipartite graph with $m[s,t]$ edges between each unordered pair of distinct blocks $\Pi[s],\Pi[t]$. The fiber law $\Lambda_y^{(p)}$ is the law of $(\Pi,\mbf d)$ for the resulting degree array, conditioned on the vertexwise histories $\mc I_s(\mbf d[x,\bullet])$ and on $\kappa_p(\Pi,\mbf d)$ or its negation according to $\on{reg}(y)$. The local estimates use only $\on{reg}(y)=1$; \Cref{prop:fiber-law-auxiliary} gives the formal identification with the auxiliary degree laws.

For a $p$-attainable coarse state $y\in\mc{Y}_k(V)$ and every event $B\subseteq\mc{Y}_{k+1}(V)$, define
\[
\overline{K}_k^{(p)}(y,B):=\mb{E}_{\sigma\sim\Lambda_y^{(p)}}\left[K_k\left(\sigma,\left(\rho_{k+1}^{(p)}\right)^{-1}(B)\right)\right].
\]
Equivalently, to sample from $\overline{K}_k^{(p)}(y,\cdot)$, first sample $\sigma\sim\Lambda_y^{(p)}$, then sample $\sigma'\sim K_k(\sigma,\cdot)$, and output $\rho_{k+1}^{(p)}(\sigma')$.

\begin{remark}[The coarse one-step law is a fiber average]
The refinement $\Phi_k(\sigma)$ is decided vertex by vertex by the rows $\mbf{d}[x,\bullet]$, so fine states with the same coarse projection can have different next partitions. For instance, with $k=1$ and suitable parameters, rows $(a,b)$ with $a>b$ for every vertex of $\Pi[0]$ send all of that block to $00$, while replacing half the rows by $(a+c,b-c)$ and half by $(a-c,b+c)$, with $c>a-b$, preserves the edge counts but sends only half of the block to $00$. Thus the coarse process is not a lumping of the fine chain: no kernel $Q_k$ satisfies
\[
K_k\left(\sigma,(\rho_{k+1}^{(p)})^{-1}(B)\right)=Q_k\left(\rho_k^{(p)}(\sigma),B\right)
\qquad\text{for all }\sigma\in\mc S_k(V),\ B\subseteq\mc Y_{k+1}(V).
\]
The conditional one-step law instead averages over $\Lambda_y^{(p)}$, as proved in \Cref{prop:coarse-one-step}; this identity supports induction by the law of total probability. The history constraints can couple the block pairs in $\Lambda_y^{(p)}$. \Cref{sec:local} controls this law and then $K_k$ at a typical fine state.
\end{remark}

\begin{proposition}[One-step coarse conditioning]\label{prop:coarse-one-step}
Fix $p\in(0,1)$, a finite set $V$, and a coloring $c:V\to\{\pm1\}$. Sample $G\sim\mb{G}(V,p)$, and let $\mf{g}=(G,c)$. If $y\in\mc{Y}_k(V)$ satisfies
\[
\mb{P}\left[\rho_k^{(p)}(\sigma_k(\mf{g}))=y\right]>0,
\]
then $y$ is $p$-attainable, and for every $B\subseteq\mc{Y}_{k+1}(V)$,
\[
\mb{P}\left[\rho_{k+1}^{(p)}(\sigma_{k+1}(\mf{g}))\in B\mid \rho_k^{(p)}(\sigma_k(\mf{g}))=y\right]=\overline{K}_k^{(p)}(y,B).
\]
\end{proposition}

\begin{proof}
Let $\Pi=\on{part}(y)$ and write
\[
E_y:=\left\{\rho_k^{(p)}(\sigma_k(\mf{g}))=y\right\},\qquad \widetilde\sigma:=\left(\Pi,\delta(\Pi,G)\right).
\]
Since $\mb{P}[E_y]>0$, the partition $\Pi$ is compatible with the fixed coloring $c$: if $x\in\Pi[s]$, then $c[x]=(-1)^{s[0]}$.

We first identify the event $E_y$. On $E_y$ we have $\Pi_k(\mf{g})=\Pi$, hence $\sigma_k(\mf{g})=(\Pi,\delta(\Pi,G))=\widetilde\sigma$; in particular $\widetilde\sigma\in\mc{S}_k(V)$ and $\rho_k^{(p)}(\widetilde\sigma)=y$. Conversely, if $\widetilde\sigma\in\mc S_k(V)$ and $\rho_k^{(p)}(\widetilde\sigma)=y$, the compatibility of $\Pi$ with $c$ and \Cref{fact:state-reconstruction}, applied to the graph $G$, give $\sigma_k(\mf g)=\widetilde\sigma$. Thus $\rho_k^{(p)}(\sigma_k(\mf g))=y$, and therefore
\[
E_y=\left\{\widetilde\sigma\in\mc{S}_k(V),\; \rho_k^{(p)}(\widetilde\sigma)=y\right\}.
\]
In particular, $y$ is $p$-attainable. Since $\sigma_k(\mf{g})=\widetilde\sigma$ on $E_y$, conditional on $E_y$ the state $\sigma_k(\mf{g})$ has law $\Lambda_y^{(p)}$, which by definition is the law of $\widetilde\sigma$ conditioned on the event displayed above.

Let $A=(\rho_{k+1}^{(p)})^{-1}(B)$. Every state $\sigma$ in the support of $\Lambda_y^{(p)}$ satisfies $\mb{P}[\sigma_k(\mf{g})=\sigma]\ge\mb{P}[\sigma_k(\mf{g})=\sigma,\,E_y]>0$, so \Cref{prop:state-chain} applies to it and gives
\[
\mb{P}\left[\sigma_{k+1}(\mf{g})\in A\mid \sigma_k(\mf{g})=\sigma\right]=K_k(\sigma,A).
\]
For such $\sigma$ we also have $\{\sigma_k(\mf{g})=\sigma\}\subseteq E_y$, because $\rho_k^{(p)}(\sigma)=y$. Hence, by the law of total probability over the value of $\sigma_k(\mf{g})$, whose conditional law given $E_y$ is $\Lambda_y^{(p)}$,
\begin{align*}
\mb{P}\left[\rho_{k+1}^{(p)}(\sigma_{k+1}(\mf{g}))\in B\mid E_y\right]
&=\sum_{\sigma}\mb{P}\left[\sigma_k(\mf{g})=\sigma\mid E_y\right]\mb{P}\left[\sigma_{k+1}(\mf{g})\in A\mid\sigma_k(\mf{g})=\sigma\right]\\
&=\mb{E}_{\sigma\sim\Lambda_y^{(p)}}\left[K_k\left(\sigma,(\rho_{k+1}^{(p)})^{-1}(B)\right)\right],
\end{align*}
which is $\overline K_k^{(p)}(y,B)$ by definition.
\end{proof}

\section{The local coarse transition theorem}\label{sec:local}

We prove a one-day estimate for the fiber-averaged law $\overline K_k^{(p)}(y,\cdot)$: under the local admissibility conditions below, the next coarse state follows an explicit template up to the stated errors. The induction applies this estimate through \Cref{prop:coarse-one-step}.

\begin{definition}[Local admissibility]\label{def:local-admissibility}
Fix $k\ge1$, a finite set $V$, $p\in(0,1)$, and constants $T>1$ and $\phi\in(0,1/2)$. Let $y=(\Pi,m,\eta)\in\mc{Y}_k(V)$, write $N:=|V|$ and $n[s]:=|\Pi[s]|$, and let
\[
\lambda\in(0,1)^{\{0,1\}^k\times\{0,1\}^k}.
\]
We call $\lambda$ a \emph{tilt}. It prescribes, for each block $\Pi[s]$ and each block $\Pi[t]$, a binomial success probability $\lambda[s,t]$ with which the degrees from vertices of $\Pi[s]$ into $\Pi[t]$ will be modeled; the name refers to the fact that replacing $p$ by a nearby $\lambda[s,t]$ amounts to an exponential tilt of the binomial law $\mr{Bin}(\cdot,p)$, and the role of the tilt is to make an independent binomial model reproduce the prescribed edge counts $m$ after conditioning on the history constraints (condition \cref{LA:solvability} below). For each fixed $s\in\{0,1\}^k$, let $A_{s,\bullet}$ denote the $\{0,1\}^k$-indexed random vector with independent coordinates
\[
A_{s,t}\sim\mr{Bin}\left(n[t]-\one_{s=t},\lambda[s,t]\right),\qquad t\in\{0,1\}^k,
\]
which models the row $\mbf{d}[x,\bullet]$ of a vertex $x\in\Pi[s]$.
We say that $(y,\lambda)$ is \emph{$(T,\phi,p)$-locally admissible} if the following conditions hold:
\begin{enumerate}[label=\textup{(LA\arabic*)},ref=\textup{(LA\arabic*)},leftmargin=*]
\item\label{LA:reg} The regularity coordinate is true:
\[
\on{reg}(y)=1.
\]
\item\label{LA:sizes} The current part sizes are of order $N$, and the imbalances tested on earlier days are controlled:
\[
n[s]\ge T^{-1}N
\]
for every $s\in\{0,1\}^k$, and
\[
\left|\sum_{t\in\{0,1\}^k}(-1)^{t[r-1]}n[t]\right|\le T\frac{N}{\sqrt{pN}}
\]
for every $1\le r\le k-1$.
\item\label{LA:edge-scale} The current edge counts have the expected scale:
\[
\left|m[s,t]-pn[s]n[t]\right|\le T\frac{N^2p}{\sqrt{pN}}
\]
for all $s,t\in\{0,1\}^k$.
\item\label{LA:positive} The edge-count denominators are positive:
\[
m[s,t]>0
\]
for all $s,t\in\{0,1\}^k$.
\item\label{LA:separation} The current edge-count imbalances are separated from their tie thresholds:
\[
\sum_{t\in\{0,1\}^k}(-1)^{t[r-1]}m[s,t]
\ge_{s[r-1]s[r]}
(-1)^{s[r]}T^{-1}\frac{N^2p}{\sqrt{pN}}
\]
for every $s\in\{0,1\}^k$ and every $1\le r\le k-1$.
\item\label{LA:tilt} The tilt is controlled:
\[
\left|\lambda[s,t]-p\right|\le T\frac{p}{\sqrt{pN}}
\]
for all $s,t\in\{0,1\}^k$.
\item\label{LA:conditioning} The conditioning event is nondegenerate:
\[
\mb{P}\left[\mc{I}_s(A_{s,\bullet})\right]\ge\phi
\]
for every $s\in\{0,1\}^k$.
\item\label{LA:solvability} The tilt solves the edge-count constraints:
\[
n[s]\mb{E}\left[A_{s,t}\mid \mc{I}_s(A_{s,\bullet})\right]=m[s,t]
\]
for all $s,t\in\{0,1\}^k$.
\item\label{LA:split} The next split is nondegenerate:
\[
\phi\le
\mb{P}\left[\mc{J}_{sb}(A_{s,\bullet})\mid \mc{I}_s(A_{s,\bullet})\right]
\le1-\phi
\]
for every $s\in\{0,1\}^k$ and $b\in\{0,1\}$.
\end{enumerate}
We refer to these conditions by their labels \cref{LA:reg}--\cref{LA:split} throughout.
\end{definition}

\begin{definition}[The local template]\label{def:local-template}
Let $k\ge1$, let $n\in\mb{Z}_{\ge1}^{\{0,1\}^k}$, and let $\lambda\in(0,1)^{\{0,1\}^k\times\{0,1\}^k}$. For each $s\in\{0,1\}^k$ let $A_{s,\bullet}=A_{s,\bullet}(n,\lambda)$ be the $\{0,1\}^k$-indexed random vector with independent coordinates $A_{s,t}\sim\mr{Bin}(n[t]-\one_{s=t},\lambda[s,t])$; when $n[t]=|\Pi[t]|$ for a coarse state $y=(\Pi,m,\eta)$, this is the vector of \Cref{def:local-admissibility}. Suppose that $\mb{P}[\mc{I}_s(A_{s,\bullet})]>0$ for every $s\in\{0,1\}^k$, and let $m\in\mb{R}^{\{0,1\}^k\times\{0,1\}^k}$ satisfy $m[s,t]\ne0$ for all $s,t$. The \emph{local template} of $(n,m,\lambda)$ is the triple $(n^{\mr{loc}},\ell^{\mr{loc}},m^{\mr{loc}})$ defined by
\[
n^{\mr{loc}}[sb]:=n[s]\mb{P}\left[\mc{J}_{sb}(A_{s,\bullet})\mid \mc{I}_s(A_{s,\bullet})\right]
\]
for $s\in\{0,1\}^k$ and $b\in\{0,1\}$,
\[
\ell^{\mr{loc}}[sb,t]
:=n[s]\mb{E}\left[A_{s,t}\one_{\mc{J}_{sb}(A_{s,\bullet})}\mid \mc{I}_s(A_{s,\bullet})\right]
\]
for $s,t\in\{0,1\}^k$ and $b\in\{0,1\}$, and
\[
m^{\mr{loc}}[sb,tc]
:=\frac{\ell^{\mr{loc}}[sb,t] \ell^{\mr{loc}}[tc,s]}{m[s,t]}
\]
for $s,t\in\{0,1\}^k$ and $b,c\in\{0,1\}$. For a coarse state $y=(\Pi,m,\eta)$ and a tilt $\lambda$ such that $(y,\lambda)$ is $(T,\phi,p)$-locally admissible, the local template of $(y,\lambda)$ is the local template of $((|\Pi[s]|)_{s\in\{0,1\}^k},m,\lambda)$; conditions \cref{LA:conditioning} and \cref{LA:positive} supply the two hypotheses. The template depends on $y$ only through the sizes $|\Pi[s]|$ and the edge counts $m$; stating it for arbitrary size vectors $n$ allows the same formulas to drive the idealized process of \Cref{sec:idealized}, whose sizes are not the block sizes of any partition.
\end{definition}

\begin{theorem}[Local coarse transition]\label{thm:local-coarse-transition}
Fix $\theta\in(1/2,1)$, an integer $k\ge1$, and constants $T>1$ and $\phi\in(0,1/2)$. There exists a constant $\Cst{thm:local-coarse-transition}=\Cst{thm:local-coarse-transition}(\theta,k,T,\phi)>0$ such that the following holds for all sufficiently large $N$. Let $V$ be a finite set with $|V|=N$, let
\[
p\in\left(T^{-1}N^{-\theta},TN^{-\theta}\right),
\]
and let $y=(\Pi,m,\eta)\in\mc{Y}_k(V)$ be $p$-attainable. Suppose that $(y,\lambda)$ is $(T,\phi,p)$-locally admissible for some
\[
\lambda\in(0,1)^{\{0,1\}^k\times\{0,1\}^k}.
\]
Let $n^{\mr{loc}}$, $\ell^{\mr{loc}}$, and $m^{\mr{loc}}$ be the local template from \Cref{def:local-template}. Let $\mc{L}_{k+1}^{(p)}(y,\lambda)\subseteq\mc{Y}_{k+1}(V)$ be the set of all $y'\in\mc{Y}_{k+1}(V)$ such that, writing $\Pi':=\on{part}(y')$ and $m':=\on{edge}(y')$, the following conditions
hold:
\begin{enumerate}[label=\textup{(L\arabic*)},ref=\textup{(L\arabic*)},leftmargin=*]
\item\label{L:refine} $\Pi'$ refines $\Pi$ one step, in the sense that
\[
\Pi'[s0]\sqcup \Pi'[s1] = \Pi[s]
\]
for every $s\in\{0,1\}^k$.
\item\label{L:reg} The next regularity coordinate is true:
\[
\on{reg}(y')=1.
\]
\item\label{L:sizes} The next part sizes follow the local template:
\[
\left||\Pi'[sb]|-n^{\mr{loc}}[sb]\right|\le \Cst{thm:local-coarse-transition}\sqrt{N}\log N
\]
for every $s\in\{0,1\}^k$ and $b\in\{0,1\}$.
\item\label{L:edges} The next edge counts follow the local template:
\[
\left|m'[sb,tc]-m^{\mr{loc}}[sb,tc]\right|\le \Cst{thm:local-coarse-transition} N^2p\sqrt{\frac{(pN)^{1/7}}{N}}\log N
\]
for every $s,t\in\{0,1\}^k$ and $b,c\in\{0,1\}$.
\end{enumerate}
Then
\[
\overline K_k^{(p)}\left(y,\mc{L}_{k+1}^{(p)}(y,\lambda)\right)=1-o(1).
\]
\end{theorem}

\Cref{thm:local-coarse-transition} is the only place in the argument where the degree-constrained random graph models have to be analyzed: its proof is where the transference machinery of \Cref{app:local-transference}, which controls the fiber law $\Lambda_y^{(p)}$, and the degree-constrained graph estimates of \Cref{app:graph-enumeration-local}, which control the kernel $K_k$, enter; both rest on the imported enumeration theorems and elementary estimates collected in \Cref{app:tools}. Once the theorem is available, the remainder of the analysis is independent of these tools: it consists of deterministic estimates showing that the local template of \Cref{def:local-template} propagates the quantities we track, together with an induction over days through \Cref{prop:coarse-one-step}. The rest of this section is devoted to the proof of \Cref{thm:local-coarse-transition}.

\subsection{Ambient degree arrays and auxiliary laws}

We next introduce the auxiliary objects used in the proof of
\Cref{thm:local-coarse-transition}. These objects are not part of the state
space $\mc{S}_k(V)$.
They separate an ambient model of independent binomial rows, the row model $\msf{R}_{\Pi,\lambda}$ of \Cref{def:aux-degree-laws} below, from the graphical
degree-array model appearing in the fiber law $\Lambda_y^{(p)}$.

\begin{definition}[Ambient degree arrays]\label{def:ambient-degree-arrays}
For $\Pi\in\mr{Part}_k(V)$, define
\[
\mc{A}(\Pi):=\left\{\mbf{d}\in\mb{Z}^{V\times\{0,1\}^k}:0\le \mbf{d}[x,t]\le |\Pi[t]|-\one_{x\in\Pi[t]}\text{ for all }x\in V,\ t\in\{0,1\}^k\right\}.
\]
For $\mbf{d}\in\mc{A}(\Pi)$, let $\mc{I}_\Pi(\mbf{d})$ denote the assertion that
\[
\mc{I}_s(\mbf{d}[x,\bullet])
\]
holds for every $s\in\{0,1\}^k$ and every $x\in\Pi[s]$.
For $C>0$, $p\in(0,1)$, and $m\in\mc{M}(\Pi)$, let $\Gamma_{C,p}(\Pi,m,\mbf{d})$ denote the assertion that, writing $N:=|V|$,
\[
\frac{1}{pN^2}\sum_{x\in\Pi[s]}\left(\mbf{d}[x,t]-\frac{m[s,t]}{|\Pi[s]|}\right)^2\le C
\]
for every $s,t\in\{0,1\}^k$.
\end{definition}

The ambient space $\mc{A}(\Pi)$ should be distinguished from the graphical space $\mc{D}(\Pi)$. The independent binomial model below lives on $\mc{A}(\Pi)$; graphicality enters through comparison with the graphical law and through a support-compatibility lemma, \Cref{lem:auto-graphicality}.

\begin{definition}[Auxiliary degree-array laws]\label{def:aux-degree-laws}
Fix $\Pi\in\mr{Part}_k(V)$ and $m\in\mc{M}(\Pi)$. The event $m(\Pi,\delta(\Pi,G))=m$ has positive probability for $G\sim\mb{G}(V,p)$, because the constraints defining $\mc{M}(\Pi)$ are exactly those that make the edge counts $m[s,s]/2$ inside $\Pi[s]$ and $m[s,t]$ between $\Pi[s]$ and $\Pi[t]$ realizable by a simple graph. Let $\msf{G}_{\Pi,m}$ be the law of $\delta(\Pi,G)$ conditioned on
\[
m(\Pi,\delta(\Pi,G))=m.
\]
We regard $\msf{G}_{\Pi,m}$ as a probability measure on $\mc{A}(\Pi)$, supported on $\mc{D}(\Pi)\cap\{\mbf{d}:m(\Pi,\mbf{d})=m\}$. This law does not depend on the value of $p\in(0,1)$.
Equivalently, $\msf{G}_{\Pi,m}$ is the product, over the block pairs of $\Pi$, of the degree-array laws obtained from a uniformly random simple graph on $\Pi[s]$ with $m[s,s]/2$ edges for each internal block pair $\{s,s\}$, and from a uniformly random simple bipartite graph between $\Pi[s]$ and $\Pi[t]$ with $m[s,t]$ edges for each bipartite block pair $\{s,t\}$.

For $\lambda\in(0,1)^{\{0,1\}^k\times\{0,1\}^k}$, let $\msf{R}_{\Pi,\lambda}$ be the product law on $\mc{A}(\Pi)$ under which, independently for each $x\in\Pi[s]$ and each $t\in\{0,1\}^k$,
\[
\mbf{D}[x,t]\sim\mr{Bin}\left(|\Pi[t]|-\one_{s=t},\lambda[s,t]\right).
\]
We call $\msf{R}_{\Pi,\lambda}$ the \emph{row model}: under it the rows $\mbf{D}[x,\bullet]$ are independent, and the law of the row of $x\in\Pi[s]$ depends only on the block index $s$.
\end{definition}

The following fact, used in \Cref{sec:local} and in \Cref{app:local-transference}, says that conditioning the row model on the history constraints preserves its product structure. Here $A_{s,\bullet}$ is the random row vector of \Cref{def:local-admissibility}, whose law under $\msf{R}_{\Pi,\lambda}$ is that of $\mbf{D}[x,\bullet]$ for any $x\in\Pi[s]$.

\begin{fact}[Rows of the conditioned row model]\label{fact:conditioned-rows}
Let $\Pi\in\mr{Part}_k(V)$ and $\lambda\in(0,1)^{\{0,1\}^k\times\{0,1\}^k}$, and suppose $\mb{P}[\mc{I}_s(A_{s,\bullet})]>0$ for every $s\in\{0,1\}^k$. Under $\msf{R}_{\Pi,\lambda}(\cdot\mid\mc{I}_\Pi)$, the rows $\mbf{D}[x,\bullet]$, $x\in V$, are independent, and for $x\in\Pi[s]$ the row $\mbf{D}[x,\bullet]$ has the law of $A_{s,\bullet}$ conditioned on $\mc{I}_s(A_{s,\bullet})$. In particular, for each $s$ the rows $(\mbf{D}[x,\bullet])_{x\in\Pi[s]}$ are $n[s]$ independent copies of $A_{s,\bullet}\mid\mc{I}_s(A_{s,\bullet})$.
\end{fact}

\begin{proof}
The event $\mc{I}_\Pi=\bigcap_{x\in V}\mc{I}_{s(x)}(\mbf{D}[x,\bullet])$, where $s(x)$ denotes the block index of $x$, is an intersection of one event per row, each depending on that row alone, and it has positive probability by the hypothesis. Since $\msf{R}_{\Pi,\lambda}$ is a product over rows, for arrays $\mbf{d}\in\mc{A}(\Pi)$ satisfying $\mc{I}_\Pi(\mbf{d})$,
\[
\msf{R}_{\Pi,\lambda}\left(\mbf{d}\mid\mc{I}_\Pi\right)
=\frac{\prod_{x\in V}\mb{P}\left[A_{s(x),\bullet}=\mbf{d}[x,\bullet]\right]}{\prod_{x\in V}\mb{P}\left[\mc{I}_{s(x)}(A_{s(x),\bullet})\right]}
=\prod_{x\in V}\mb{P}\left[A_{s(x),\bullet}=\mbf{d}[x,\bullet]\ \middle|\ \mc{I}_{s(x)}(A_{s(x),\bullet})\right],
\]
and $\msf{R}_{\Pi,\lambda}(\mbf{d}\mid\mc{I}_\Pi)=0$ when $\mc{I}_\Pi(\mbf{d})$ fails, which is also the value of the right-hand side. The right-hand side is a product law with the stated factors.
\end{proof}

\begin{proposition}[Fiber law in auxiliary notation]\label{prop:fiber-law-auxiliary}
Let $p\in(0,1)$ and let $y=(\Pi,m,1)\in\mc{Y}_k(V)$ be $p$-attainable.
If $\sigma\sim\Lambda_y^{(p)}$ and $\mbf{D}:=\on{deg}(\sigma)$, then $\mbf{D}$, viewed as an $\mc{A}(\Pi)$-valued random array, has law
\[
\msf{G}_{\Pi,m}\left(\cdot\mid \mc{I}_\Pi,\kappa_p(\Pi,\cdot)\right).
\]
Equivalently, for every event $E\subseteq\mc{A}(\Pi)$,
\[
\Lambda_y^{(p)}\left(\left\{\sigma:\on{deg}(\sigma)\in E\right\}\right)
=
\msf{G}_{\Pi,m}\left(E\mid \mc{I}_\Pi,\kappa_p(\Pi,\cdot)\right).
\]
\end{proposition}

\begin{proof}
By definition, $\Lambda_y^{(p)}$ is obtained by sampling $G\sim\mb{G}(V,p)$, forming
\[
\widetilde\sigma_y(G)=(\Pi,\delta(\Pi,G)),
\]
and conditioning on
\[
\widetilde\sigma_y(G)\in\mc{S}_k(V),\qquad \rho_k^{(p)}(\widetilde\sigma_y(G))=y.
\]
Writing $\mbf{D}=\delta(\Pi,G)$, the first condition says exactly that $\mbf{D}\in\mc{D}(\Pi)$ and $\mc{I}_\Pi(\mbf{D})$ holds, and the condition $\mbf{D}\in\mc{D}(\Pi)$ is automatic because $\mbf{D}$ is the degree array of the graph $G$. Since $y=(\Pi,m,1)$, the second condition imposes
\[
m(\Pi,\mbf{D})=m,\qquad \kappa_p(\Pi,\mbf{D}).
\]
Thus $\Lambda_y^{(p)}$ is the law of $(\Pi,\mbf{D})$ conditioned on the single event $\{m(\Pi,\mbf{D})=m\}\cap\mc{I}_\Pi\cap\kappa_p(\Pi,\cdot)$. Writing $\mc{B}:=\mc{I}_\Pi\cap\kappa_p(\Pi,\cdot)$ and $\{m\}:=\{\mbf{d}\in\mc{A}(\Pi):m(\Pi,\mbf{d})=m\}$, we therefore have, for every $E\subseteq\mc{A}(\Pi)$,
\begin{align*}
\Lambda_y^{(p)}\left(\left\{\sigma:\on{deg}(\sigma)\in E\right\}\right)
&=\frac{\mb{P}\left[\mbf{D}\in E\cap\{m\}\cap\mc{B}\right]}{\mb{P}\left[\mbf{D}\in\{m\}\cap\mc{B}\right]}
=\frac{\mb{P}\left[\mbf{D}\in E\cap\mc{B}\mid \mbf{D}\in\{m\}\right]}{\mb{P}\left[\mbf{D}\in\mc{B}\mid \mbf{D}\in\{m\}\right]}\\
&=\frac{\msf{G}_{\Pi,m}(E\cap\mc{B})}{\msf{G}_{\Pi,m}(\mc{B})}
=\msf{G}_{\Pi,m}\left(E\mid\mc{B}\right),
\end{align*}
where the second equality divides numerator and denominator by $\mb{P}[\mbf{D}\in\{m\}]>0$, and the third is the definition of $\msf{G}_{\Pi,m}$ as the law of $\mbf{D}=\delta(\Pi,G)$ conditioned on $\{m\}$ (\Cref{def:aux-degree-laws}).
\end{proof}

\subsection{Local transference for a coarse fiber}

The purpose of this subsection is the following transference principle: an event whose failure probability under the independent row model $\msf{R}_{\Pi,\lambda}$, conditioned on the history constraints and on the exact block-pair edge counts, is small also has small failure probability under the fiber law $\Lambda_y^{(p)}$, up to a constant factor and an additive $e^{-N}$. This is what allows all subsequent probabilistic computations to be carried out in the row model, where the rows are independent, rather than in the graphical model. The proof relies on the enumeration comparison and the second-moment regularity results of \Cref{app:local-transference}.

\begin{proposition}[Local degree-array transference]\label{prop:local-degree-transference}
Fix $\theta\in(1/2,1)$, an integer $k\ge1$, and constants $T>1$ and $\phi\in(0,1/2)$. There is a constant $\Cst{prop:local-degree-transference}=\Cst{prop:local-degree-transference}(\theta,k,T,\phi)>0$ such that for all sufficiently large $N$ the following holds. Let $V$ be a finite set with $|V|=N$, let
\[
p\in(T^{-1}N^{-\theta},TN^{-\theta}),
\]
and let $y=(\Pi,m,1)\in\mc{Y}_k(V)$ be $p$-attainable. Suppose that $(y,\lambda)$ is $(T,\phi,p)$-locally admissible. Then for every event $E\subseteq\mc{A}(\Pi)$,
\[
\Lambda_y^{(p)}\left(\sigma:\on{deg}(\sigma)\notin E\right)\le e^{-N}+\Cst{prop:local-degree-transference}\,\msf{R}_{\Pi,\lambda}\left(E^c\mid \mc{I}_\Pi,\ m(\Pi,\mbf{D})=m\right).
\]
In particular, if $\msf{R}_{\Pi,\lambda}(E\mid \mc{I}_\Pi,\ m(\Pi,\mbf{D})=m)=1-o(1)$, then $\Lambda_y^{(p)}(\sigma:\on{deg}(\sigma)\in E)=1-o(1)$.
\end{proposition}

\begin{proof}
Throughout the proof we abbreviate $\Gamma:=\Gamma_{\Cof{lem:local-gamma-concentration}{\theta,k,T,\phi},p}(\Pi,m,\cdot)$, $\kappa_p:=\kappa_p(\Pi,\cdot)$, and $\{m\}:=\{\mbf{D}\in\mc{A}(\Pi):m(\Pi,\mbf{D})=m\}$. By \Cref{prop:fiber-law-auxiliary}, it suffices to prove
\begin{equation}\label{eq:transf-goal}
\msf{G}_{\Pi,m}\left(E^c\mid \mc{I}_\Pi\cap\kappa_p\right)\le e^{-N}+\Cst{prop:local-degree-transference}\,\msf{R}_{\Pi,\lambda}\left(E^c\mid\mc{I}_\Pi,\{m\}\right)
\end{equation}
for a suitable constant $\Cst{prop:local-degree-transference}$.

\emph{Step 1: restriction to $\Gamma$.} By \Cref{lem:local-gamma-concentration},
\begin{equation}\label{eq:transf-gamma}
\msf{G}_{\Pi,m}\left(\Gamma^c\mid \mc{I}_\Pi\cap\kappa_p\right)\le e^{-N}.
\end{equation}
Since $E^c\subseteq\Gamma^c\cup(E^c\cap\Gamma)$, and since $\msf{G}_{\Pi,m}(\mc{I}_\Pi\cap\kappa_p)\ge\msf{G}_{\Pi,m}(\Gamma\cap\mc{I}_\Pi\cap\kappa_p)$, Bayes' rule and \eqref{eq:transf-gamma} give
\begin{equation}\label{eq:transf-split}
\msf{G}_{\Pi,m}\left(E^c\mid \mc{I}_\Pi\cap\kappa_p\right)
\le e^{-N}+\frac{\msf{G}_{\Pi,m}\left(E^c\cap\Gamma\cap\mc{I}_\Pi\cap\kappa_p\right)}{\msf{G}_{\Pi,m}\left(\mc{I}_\Pi\cap\kappa_p\right)}
\le e^{-N}+\frac{\msf{G}_{\Pi,m}\left(E^c\cap\Gamma\cap\mc{I}_\Pi\cap\kappa_p\right)}{\msf{G}_{\Pi,m}\left(\Gamma\cap\mc{I}_\Pi\cap\kappa_p\right)}.
\end{equation}

\emph{Step 2: comparison with the row model.} The measure $\msf{G}_{\Pi,m}$ is supported on $\{m\}$, so the numerator and denominator in \eqref{eq:transf-split} are sums of $\msf{G}_{\Pi,m}(\mbf{d})$ over arrays $\mbf{d}\in\{m\}\cap\kappa_p\cap\Gamma$. By \cref{LA:sizes}, \cref{LA:edge-scale}, and \Cref{lem:auto-graphicality}, every such array is graphical, and the restriction to $\Gamma$ in Step 1 is exactly what makes the bounded-ratio comparison \Cref{lem:local-enumeration-comparison}\ref{lem:local-enumeration-comparison:strong} available for every such array (on $\kappa_p$ alone only the weak comparison \ref{lem:local-enumeration-comparison:weak} with the factor $\exp\{O((pN)^{2/7})\}$ holds). Write $C_1:=\Cofp{lem:local-enumeration-comparison}{\theta,k,T,\Cof{lem:local-gamma-concentration}{\theta,k,T,\phi}}$ for the constant of that comparison. Summing the upper bound $\msf{G}_{\Pi,m}(\mbf{d})\le C_1\msf{R}_{\Pi,\lambda}(\mbf{d}\mid\{m\})$ over the arrays in the numerator and the lower bound $\msf{G}_{\Pi,m}(\mbf{d})\ge C_1^{-1}\msf{R}_{\Pi,\lambda}(\mbf{d}\mid\{m\})$ over the arrays in the denominator,
\begin{equation}\label{eq:transf-compare}
\frac{\msf{G}_{\Pi,m}\left(E^c\cap\Gamma\cap\mc{I}_\Pi\cap\kappa_p\right)}{\msf{G}_{\Pi,m}\left(\Gamma\cap\mc{I}_\Pi\cap\kappa_p\right)}
\le C_1^2\,
\frac{\msf{R}_{\Pi,\lambda}\left(E^c\cap\Gamma\cap\mc{I}_\Pi\cap\kappa_p\mid\{m\}\right)}{\msf{R}_{\Pi,\lambda}\left(\Gamma\cap\mc{I}_\Pi\cap\kappa_p\mid\{m\}\right)}.
\end{equation}

\emph{Step 3: Bayes' rule in the row model.} The event $\mc{I}_\Pi\cap\{m\}$ has positive $\msf{R}_{\Pi,\lambda}$-probability by \Cref{lem:local-exact-total-lower-bound}, so we may divide the numerator and the denominator of the right-hand side of \eqref{eq:transf-compare} by $\msf{R}_{\Pi,\lambda}(\mc{I}_\Pi\mid\{m\})$, which converts both into probabilities conditional on $\mc{I}_\Pi\cap\{m\}$:
\begin{equation}\label{eq:transf-bayes}
\frac{\msf{R}_{\Pi,\lambda}\left(E^c\cap\Gamma\cap\mc{I}_\Pi\cap\kappa_p\mid\{m\}\right)}{\msf{R}_{\Pi,\lambda}\left(\Gamma\cap\mc{I}_\Pi\cap\kappa_p\mid\{m\}\right)}
=
\frac{\msf{R}_{\Pi,\lambda}\left(E^c\cap\Gamma\cap\kappa_p\mid\mc{I}_\Pi,\{m\}\right)}{\msf{R}_{\Pi,\lambda}\left(\Gamma\cap\kappa_p\mid\mc{I}_\Pi,\{m\}\right)}
\le
\frac{\msf{R}_{\Pi,\lambda}\left(E^c\mid\mc{I}_\Pi,\{m\}\right)}{\msf{R}_{\Pi,\lambda}\left(\Gamma\cap\kappa_p\mid\mc{I}_\Pi,\{m\}\right)}.
\end{equation}
For the denominator, the chain rule together with \Cref{lem:local-binomial-kappa} and \Cref{lem:local-row-gamma} gives
\begin{equation}\label{eq:transf-denominator}
\msf{R}_{\Pi,\lambda}\left(\Gamma\cap\kappa_p\mid\mc{I}_\Pi,\{m\}\right)
=\msf{R}_{\Pi,\lambda}\left(\kappa_p\mid\mc{I}_\Pi,\{m\}\right)\,
\msf{R}_{\Pi,\lambda}\left(\Gamma\mid\mc{I}_\Pi,\kappa_p,\{m\}\right)
=(1-o(1))(1-o(1))\ge\frac12
\end{equation}
for all sufficiently large $N$. Combining \eqref{eq:transf-split}, \eqref{eq:transf-compare}, \eqref{eq:transf-bayes}, and \eqref{eq:transf-denominator},
\[
\msf{G}_{\Pi,m}\left(E^c\mid \mc{I}_\Pi\cap\kappa_p\right)\le e^{-N}+2C_1^2\,\msf{R}_{\Pi,\lambda}\left(E^c\mid\mc{I}_\Pi,\{m\}\right),
\]
which is \eqref{eq:transf-goal} with $\Cst{prop:local-degree-transference}:=2C_1^2=2\,\Cofp{lem:local-enumeration-comparison}{\theta,k,T,\Cof{lem:local-gamma-concentration}{\theta,k,T,\phi}}^2$; this constant depends only on $\theta,k,T,\phi$. The final sentence of the statement follows because the bound holds uniformly in $E$.
\end{proof}

\subsection{Degree-array concentration in a coarse fiber}

In this subsection we prove that a degree array sampled from the fiber law $\Lambda_y^{(p)}$ is typical in the three respects that the kernel splitting step will need: individual degrees are concentrated, the sizes of the parts of the next partition are close to the local template $n^{\mr{loc}}$, and the degree masses carried by these next parts are close to the local template $\ell^{\mr{loc}}$. The proof works in the row model and transfers to the fiber by \Cref{prop:local-degree-transference}. We first name the two statistics of a degree array that determine the next partition and its degree masses.

For $\Pi\in\mr{Part}_k(V)$, $\mbf{d}\in\mc{A}(\Pi)$, $s\in\{0,1\}^k$, and $b\in\{0,1\}$, define
\[
W_{\mbf{d}}[sb]:=\left\{x\in\Pi[s]:\mc{J}_{sb}(\mbf{d}[x,\bullet])\right\}.
\]
For $s,t\in\{0,1\}^k$ and $b\in\{0,1\}$, define
\[
L_{\mbf{d}}[sb,t]:=\sum_{x\in W_{\mbf{d}}[sb]}\mbf{d}[x,t].
\]
For a fine state $(\Pi,\mbf d)$, $W_{\mbf d}[sb]=\Phi_k((\Pi,\mbf d))[sb]$ is the next part, and $L_{\mbf d}[sb,t]$ is its degree mass into $\Pi[t]$, which $K_k$ splits between the next parts $t0$ and $t1$. Both quantities are defined on all of $\mc A(\Pi)$, including arrays that do not arise from a graph, because the row model lives on this ambient space.

\begin{proposition}[Fiber degree-array estimates]\label{prop:fiber-degree-array-estimates}
Fix $\theta\in(1/2,1)$, an integer $k\ge1$, and constants $T>1$ and $\phi\in(0,1/2)$. There exists a constant $\Cst{prop:fiber-degree-array-estimates}=\Cst{prop:fiber-degree-array-estimates}(\theta,k,T,\phi)>0$ such that the following holds for all sufficiently large $N$.

Let $V$ be a finite set with $|V|=N$, let
\[
p\in(T^{-1}N^{-\theta},TN^{-\theta}),
\]
and let $y=(\Pi,m,1)\in\mc{Y}_k(V)$ be $p$-attainable. Suppose that $(y,\lambda)$ is $(T,\phi,p)$-locally admissible, and let $n^{\mr{loc}}$ and $\ell^{\mr{loc}}$ be the local template from \Cref{def:local-template}. Sample
\[
(\Pi,\mbf{d})\sim\Lambda_y^{(p)}.
\]
Then, with probability $1-o(1)$, the following estimates hold simultaneously:
\begin{enumerate}[label=\textup{(R\arabic*)},ref=\textup{(R\arabic*)},leftmargin=*]
\item\label{R:degrees} Individual degrees are concentrated:
\[
\left|\mbf{d}[x,t]-p|\Pi[t]|\right|
\le
\sqrt{pN}\,(\log N)^{2/3}
\]
for every $x\in V$ and $t\in\{0,1\}^k$.
\item\label{R:sizes} The next part sizes follow the local template:
\[
\left||W_{\mbf{d}}[sb]|-n^{\mr{loc}}[sb]\right|
\le
\Cst{prop:fiber-degree-array-estimates}\sqrt N\log N
\]
for every $s\in\{0,1\}^k$ and $b\in\{0,1\}$.
\item\label{R:masses} The degree masses follow the local template:
\[
\left|L_{\mbf{d}}[sb,t]-\ell^{\mr{loc}}[sb,t]\right|
\le
\Cst{prop:fiber-degree-array-estimates}\frac{N^2p}{\sqrt N}\log N
\]
for every $s,t\in\{0,1\}^k$ and $b\in\{0,1\}$.
\end{enumerate}
\end{proposition}

\begin{proof}
We prove the proposition with $\Cst{prop:fiber-degree-array-estimates}:=1$. Let $\mc{E}$ denote the event that \cref{R:degrees}, \cref{R:sizes}, and \cref{R:masses} hold. We first prove
\[
\msf{R}_{\Pi,\lambda}\left(\mc{E}\mid \mc{I}_\Pi,\ m(\Pi,\mbf{D})=m\right)=1-o(1).
\]
By \cref{LA:conditioning} the hypothesis of \Cref{fact:conditioned-rows} holds, so under $\msf{R}_{\Pi,\lambda}(\cdot\mid\mc{I}_\Pi)$ the rows $\mbf{D}[x,\bullet]$ are independent, and if $x\in\Pi[s]$ then the row has law
\[
A_{s,\bullet}\mid \mc{I}_s(A_{s,\bullet}).
\]
Conditions \cref{LA:conditioning} and \cref{LA:tilt} give $\mb{P}[\mc{I}_s(A_{s,\bullet})]\ge\phi$ and $|\lambda[s,t]-p|\le T p/\sqrt{pN}$.
For fixed $x\in\Pi[s]$ and $t$, the mean of $A_{s,t}$ satisfies
\begin{align*}
\left|\mb{E}A_{s,t}-p|\Pi[t]|\right|
&=\left|\left(n[t]-\one_{s=t}\right)\lambda[s,t]-p\,n[t]\right|
\le n[t]\left|\lambda[s,t]-p\right|+\lambda[s,t]\\
&\le T\sqrt{pN}+2p
\le\tfrac12\sqrt{pN}(\log N)^{2/3}
\end{align*}
for large $N$, so Chernoff's inequality gives
\[
\mb{P}\left[\left|A_{s,t}-p|\Pi[t]|\right|>
\sqrt{pN}\,(\log N)^{2/3}\right]\le2\exp\left(-\frac{(\log N)^{4/3}}{24}\right)=N^{-\omega(1)}.
\]
After division by $\mb{P}[\mc{I}_s(A_{s,\bullet})]\ge\phi$ and a union bound over the $2^kN$ pairs $(x,t)$, estimate \cref{R:degrees} fails with probability $N^{-\omega(1)}$ under $\msf{R}_{\Pi,\lambda}(\cdot\mid\mc{I}_\Pi)$.

For fixed $s\in\{0,1\}^k$ and $b\in\{0,1\}$, the random variables in the family $\left(\one_{\mc{J}_{sb}(\mbf{D}[x,\bullet])}\right)_{x\in\Pi[s]}$ are mutually independent under $\msf{R}_{\Pi,\lambda}(\cdot\mid\mc{I}_\Pi)$, with common mean
\[
\mb{P}[\mc{J}_{sb}(A_{s,\bullet})\mid\mc{I}_s(A_{s,\bullet})].
\]
Therefore $|W_{\mbf{D}}[sb]|$ is binomial with $n[s]\le N$ trials and mean $n^{\mr{loc}}[sb]$ by \Cref{def:local-template}, and Chernoff's inequality gives
\[
\mb{P}\left[\left||W_{\mbf{D}}[sb]|-n^{\mr{loc}}[sb]\right|>\sqrt N\log N\right]\le2\exp\left(-\frac{N\log^2N}{2(N/4+\sqrt N\log N/3)}\right)=N^{-\omega(1)},
\]
which is estimate \cref{R:sizes} with $\Cst{prop:fiber-degree-array-estimates}=1$.

For fixed $s,t\in\{0,1\}^k$ and $b\in\{0,1\}$, put
\[
X_x:=\mbf{D}[x,t]\one_{\mc{J}_{sb}(\mbf{D}[x,\bullet])},\qquad x\in\Pi[s].
\]
The random variables in the family $(X_x)_{x\in\Pi[s]}$ are mutually independent under $\msf{R}_{\Pi,\lambda}(\cdot\mid\mc{I}_\Pi)$ and have mean
\[
\mb{E}\left[A_{s,t}\one_{\mc{J}_{sb}(A_{s,\bullet})}\mid\mc{I}_s(A_{s,\bullet})\right]=\frac{\ell^{\mr{loc}}[sb,t]}{n[s]},
\]
where the equality is the definition of $\ell^{\mr{loc}}$ in \Cref{def:local-template}.
On the high-probability event \cref{R:degrees} already proved, every $X_x$ is at most $\mbf{D}[x,t]\le p|\Pi[t]|+\sqrt{pN}(\log N)^{2/3}\le2pN$, using $|\Pi[t]|\le N$ and taking $N$ sufficiently large. Therefore, replacing $X_x$ by the truncated variable $X_x':=X_x\one_{X_x\le2pN}$ changes the sum $\sum_{x\in\Pi[s]}X_x$ with probability $N^{-\omega(1)}$ only. The truncation also changes the means, but negligibly: since $0\le X_x\le N$,
\[
0\le\mb{E}[X_x]-\mb{E}[X_x']=\mb{E}\left[X_x\one_{X_x>2pN}\right]\le N\,\mb{P}[X_x>2pN]=N^{-\omega(1)},
\]
so $\left|\sum_{x\in\Pi[s]}\mb{E}[X_x']-\ell^{\mr{loc}}[sb,t]\right|\le n[s]N^{-\omega(1)}=N^{-\omega(1)}$, which is negligible against the error term below. Hoeffding's inequality applied to the $n[s]\le N$ independent variables $X_x'\in[0,2pN]$ gives
\begin{align*}
\mb{P}\left[\left|\sum_{x\in\Pi[s]}X_x'-\sum_{x\in\Pi[s]}\mb{E}[X_x']\right|>\frac12\frac{N^2p}{\sqrt N}\log N\right]
&\le2\exp\left(-\frac{2\cdot\frac14N^3p^2\log^2N}{N(2pN)^2}\right)\\
&=2\exp\left(-\frac{\log^2N}{8}\right)=N^{-\omega(1)},
\end{align*}
and combining the three displays,
\[
\left|\sum_{x\in\Pi[s]}X_x-\ell^{\mr{loc}}[sb,t]\right|
\le\frac12\frac{N^2p}{\sqrt N}\log N+N^{-\omega(1)}\le\frac{N^2p}{\sqrt N}\log N
\]
with failure probability $N^{-\omega(1)}$; this is estimate \cref{R:masses} with $\Cst{prop:fiber-degree-array-estimates}=1$. A union bound over the at most $2^{2k+1}$ choices of $s,t,b$ gives
\begin{equation}\label{eq:E-unconditioned}
\msf{R}_{\Pi,\lambda}(\mc{E}^c\mid\mc{I}_\Pi)=N^{-\omega(1)}.
\end{equation}
By \Cref{lem:local-exact-total-lower-bound},
\begin{equation}\label{eq:exact-total-poly}
\msf{R}_{\Pi,\lambda}\left(m(\Pi,\mbf{D})=m\mid\mc{I}_\Pi\right)\ge (N^2p)^{-\Cof{lem:local-exact-total-lower-bound}{\theta,k,T,\phi}}.
\end{equation}
By Bayes' rule, \eqref{eq:exact-total-poly}, and \eqref{eq:E-unconditioned},
\begin{align*}
\msf{R}_{\Pi,\lambda}\left(\mc{E}^c\mid\mc{I}_\Pi,m(\Pi,\mbf{D})=m\right)
&=\frac{\msf{R}_{\Pi,\lambda}\left(\mc{E}^c\cap\{m(\Pi,\mbf{D})=m\}\mid\mc{I}_\Pi\right)}{\msf{R}_{\Pi,\lambda}\left(m(\Pi,\mbf{D})=m\mid\mc{I}_\Pi\right)}\\
&\le\frac{\msf{R}_{\Pi,\lambda}\left(\mc{E}^c\mid\mc{I}_\Pi\right)}{(N^2p)^{-\Cof{lem:local-exact-total-lower-bound}{\theta,k,T,\phi}}}
\le(N^2p)^{\Cof{lem:local-exact-total-lower-bound}{\theta,k,T,\phi}}N^{-\omega(1)}=N^{-\omega(1)}.
\end{align*}
Hence \Cref{prop:local-degree-transference}, applied to the event $\mc{E}$, gives
\[
\Lambda_y^{(p)}\left(\sigma:\on{deg}(\sigma)\notin\mc{E}\right)\le e^{-N}+\Cof{prop:local-degree-transference}{\theta,k,T,\phi}N^{-\omega(1)}=o(1),
\]
which is the claim.
\end{proof}

\subsection{Splitting the block-pair graphs under the kernel}

We now fix a typical degree array $\mbf{d}$ from the fiber, in the sense of \Cref{prop:fiber-degree-array-estimates}, and study the second stage of the sampling procedure for $\overline{K}_k^{(p)}$: the kernel $K_k((\Pi,\mbf{d}),\cdot)$ samples a uniformly random degree-constrained graph or bipartite graph on each block pair of $\Pi$ and restricts it to the next partition. The next partition is deterministic given $\mbf{d}$, so what has to be shown is that the edges of each block-pair graph split between the next parts in proportion to the degree masses $L_{\mbf{d}}$, and that the new degree array is again regular. The first statement is a second-moment computation based on the edge probabilities in a random graph with prescribed degrees (\Cref{lem:edge_enum}), and the second is a large-deviation estimate for a single degree into a subset (\Cref{thm:nice_deg}); both are proved in \Cref{app:graph-enumeration-local}.

\begin{proposition}[Kernel splitting estimates]\label{prop:kernel-splitting-estimates}
Fix $\theta\in(1/2,1)$, an integer $k\ge1$, and constants $T>1$ and $\phi\in(0,1/2)$. There exists a constant $\Cst{prop:kernel-splitting-estimates}=\Cst{prop:kernel-splitting-estimates}(\theta,k,T,\phi)>0$ such that the following holds for all sufficiently large $N$.

Let $V$ be a finite set with $|V|=N$, let
\[
p\in(T^{-1}N^{-\theta},TN^{-\theta}),
\]
and let $(y,\lambda)$ be $(T,\phi,p)$-locally admissible with $y=(\Pi,m,1)$. Let $\mbf{d}\in\mc{D}^{\mc I}(\Pi)$ satisfy
\[
m(\Pi,\mbf{d})=m
\]
together with the estimates \cref{R:degrees} and \cref{R:sizes} of \Cref{prop:fiber-degree-array-estimates}. Sample
\[
\sigma'=(\Pi',\mbf{D}')\sim K_k((\Pi,\mbf{d}),\cdot).
\]
Then, with probability $1-o(1)$, the following hold simultaneously:
\begin{enumerate}[label=\textup{(S\arabic*)},ref=\textup{(S\arabic*)},leftmargin=*]
\item\label{S:parts} The next partition is the deterministic split:
\[
\Pi'[sb]=W_{\mbf{d}}[sb]
\]
for every $s\in\{0,1\}^k$ and $b\in\{0,1\}$.
\item\label{S:reg} The next degree array is regular: $\kappa_p(\Pi',\mbf{D}')$ holds.
\item\label{S:edges} The next edge counts split proportionally to the degree masses:
\[
\left|
m(\Pi',\mbf{D}')[sb,tc]
-
\frac{L_{\mbf{d}}[sb,t]L_{\mbf{d}}[tc,s]}{m[s,t]}
\right|
\le
\Cst{prop:kernel-splitting-estimates} N^2p\sqrt{\frac{(pN)^{1/7}}{N}}\log N
\]
for every $s,t\in\{0,1\}^k$ and $b,c\in\{0,1\}$.
\end{enumerate}
\end{proposition}

\begin{proof}
Assertion \cref{S:parts} is deterministic: by the definition of $K_k$ we have $\Pi'=\Phi_k((\Pi,\mbf{d}))$, and $\Phi_k((\Pi,\mbf{d}))[sb]=W_{\mbf{d}}[sb]$ by the definitions of $\Phi_k$ and $W_{\mbf{d}}$.

We record two consequences of the hypotheses that are used repeatedly below. First, by \cref{LA:sizes}, \cref{LA:split}, and \cref{R:sizes}, for all sufficiently large $N$,
\begin{equation}\label{eq:W-sizes}
\frac{\phi}{2T}N\le |W_{\mbf{d}}[sb]|\le N\qquad\text{for all }s\in\{0,1\}^k\text{ and }b\in\{0,1\},
\end{equation}
since $|W_{\mbf{d}}[sb]|\ge n^{\mr{loc}}[sb]-\Cof{prop:fiber-degree-array-estimates}{\theta,k,T,\phi}\sqrt N\log N\ge\phi\,n[s]-\Cof{prop:fiber-degree-array-estimates}{\theta,k,T,\phi}\sqrt N\log N$ and $n[s]\ge T^{-1}N$. Second, by \cref{R:degrees} and \cref{LA:sizes}, for all sufficiently large $N$,
\begin{equation}\label{eq:d-window}
\left|\mbf{d}[x,t]-p\,n[t]\right|\le\sqrt{pN}\,(\log N)^{2/3}\le\min_{u\in\{0,1\}^k}\left(p\,n[u]\right)^{4/7}\qquad\text{for all }x\in V\text{ and }t\in\{0,1\}^k,
\end{equation}
where the minimum over $u$ is what allows the degrees on either side of a bipartite block pair $\{s,t\}$ to be measured against the size of the other block. Finally, put
\begin{equation}\label{eq:T1}
T_1:=\max\left\{T^{5/2},\,\frac{2T}{\phi}\right\},
\end{equation}
which is the value of the constant $T$ with which \Cref{lem:edge_enum,thm:nice_deg} are applied below: since every block size $n[u]$ lies in $[T^{-1}N,N]$, their hypothesis $p\in(T_1^{-1}n^{-\theta},T_1n^{-\theta})$ holds because $T_1\ge T^{1+\theta}$, the hypothesis $\ell\in[T_1^{-1}n,T_1n]$ holds because $T_1\ge T$, the edge-count hypothesis $|m-p\ell n|\le T_1n^2p/\sqrt{pn}$ (respectively $|m-pn(n-1)/2|\le T_1n^2p/\sqrt{pn}$) follows from \cref{LA:edge-scale} because $TN^2p/\sqrt{pN}\le T^{5/2}n^2p/\sqrt{pn}$ for $n\ge T^{-1}N$ (in the internal case, where $m=m[s,s]/2$, the additional term $\tfrac12p\,n[s]$ from $pn(n-1)/2$ versus $pn^2/2$ is absorbed, since $pn\le n^{3/2}\sqrt p$), and the subset-size hypotheses hold by \eqref{eq:W-sizes} because $T_1\ge2T/\phi$.

We prove \cref{S:edges}. Suppose first that $s\ne t$. Under $K_k((\Pi,\mbf{d}),\cdot)$, the bipartite graph between $\Pi[s]$ and $\Pi[t]$ is uniformly random among all bipartite graphs with side-degree sequences $(\mbf{d}[x,t])_{x\in\Pi[s]}$ and $(\mbf{d}[y,s])_{y\in\Pi[t]}$; it has exactly $\sum_{x\in\Pi[s]}\mbf{d}[x,t]=m(\Pi,\mbf{d})[s,t]=m[s,t]$ edges. We apply \Cref{lem:edge_enum}\cref{lem:edge_enum:bipartite} to this bipartite graph with $\ell=n[s]$, $n=n[t]$, $U=W_{\mbf{d}}[sb]$, $U'=W_{\mbf{d}}[tc]$, and with the constant $T$ of that lemma equal to $T_1$: the hypotheses on $p$, on $\ell$, on the edge count $m[s,t]$, and on the subset sizes hold by the discussion following \eqref{eq:T1}, and the degree hypotheses hold by \eqref{eq:d-window} with $u=t$ (the lemma measures the deviations of both degree sequences, of $(\mbf{d}[x,t])_{x\in\Pi[s]}$ from $p\,n[t]$ and of $(\mbf{d}[y,s])_{y\in\Pi[t]}$ from $p\,n[s]$, in units of $(pn)^{4/7}=(p\,n[t])^{4/7}$). Since $\sum_{x\in W_{\mbf{d}}[sb]}\mbf{d}[x,t]=L_{\mbf{d}}[sb,t]$, $\sum_{y\in W_{\mbf{d}}[tc]}\mbf{d}[y,s]=L_{\mbf{d}}[tc,s]$, and $m(\Pi',\mbf{D}')[sb,tc]$ is the number of edges between $W_{\mbf{d}}[sb]$ and $W_{\mbf{d}}[tc]$, the lemma gives
\[
\left|m(\Pi',\mbf{D}')[sb,tc]-\frac{L_{\mbf{d}}[sb,t]L_{\mbf{d}}[tc,s]}{m[s,t]}\right|
\le
C_1 N^2p\sqrt{\frac{(pN)^{1/7}}{N}}\log N
\]
with probability $1-o(1)$, where $C_1:=2T$: the lemma is stated with $n=n[t]\in[T^{-1}N,N]$, and its error term $n^2p\sqrt{(pn)^{1/7}/n}\log n$ is at most $T^{3/7}N^2p\sqrt{(pN)^{1/7}/N}\log N$ because $n^{-6/7}\le T^{6/7}N^{-6/7}$ (the factor $2$ in $C_1$ is not needed here; it is needed in the internal case with $b=c$ below).

Now suppose $s=t$. Under $K_k((\Pi,\mbf{d}),\cdot)$, the graph on $\Pi[s]$ is uniformly random among all graphs with degree sequence $(\mbf{d}[x,s])_{x\in\Pi[s]}$; it has exactly $m[s,s]/2$ edges. The hypotheses of \Cref{lem:edge_enum}\cref{lem:edge_enum:graph} and \cref{lem:edge_enum:cut} are verified exactly as above, with $n=n[s]$, $U=W_{\mbf{d}}[sb]$, and the constant $T$ of the lemma equal to $T_1$. If $b=c$, the quantity $m(\Pi',\mbf{D}')[sb,sb]$ is twice the number $X$ of edges inside $W_{\mbf{d}}[sb]$, and $L_{\mbf{d}}[sb,s]^2/m[s,s]$ is twice the centering $(\sum_{x\in W_{\mbf{d}}[sb]}\mbf{d}[x,s])^2/(4\cdot m[s,s]/2)$ of \Cref{lem:edge_enum}\cref{lem:edge_enum:graph}, so that lemma gives
\[
\left|m(\Pi',\mbf{D}')[sb,sb]-\frac{L_{\mbf{d}}[sb,s]^2}{m[s,s]}\right|
=2\left|X-\frac{\left(\sum_{x\in W_{\mbf{d}}[sb]}\mbf{d}[x,s]\right)^2}{4\cdot m[s,s]/2}\right|
\le
C_1 N^2p\sqrt{\frac{(pN)^{1/7}}{N}}\log N,
\]
where the factor $2$ is absorbed by $C_1=2T$.
If $b\ne c$, then $W_{\mbf{d}}[sc]=\Pi[s]\setminus W_{\mbf{d}}[sb]$ and $m(\Pi',\mbf{D}')[sb,sc]$ is the number of edges between $W_{\mbf{d}}[sb]$ and $W_{\mbf{d}}[sc]$, so \Cref{lem:edge_enum}\cref{lem:edge_enum:cut} gives
\[
\left|m(\Pi',\mbf{D}')[sb,sc]-\frac{L_{\mbf{d}}[sb,s]L_{\mbf{d}}[sc,s]}{m[s,s]}\right|
\le
C_1 N^2p\sqrt{\frac{(pN)^{1/7}}{N}}\log N.
\]
A union bound over the at most $2^{2k+2}$ choices of $s,t,b,c$ proves \cref{S:edges} with $\Cst{prop:kernel-splitting-estimates}:=C_1$.

It remains to prove \cref{S:reg}. Fix a vertex $x\in\Pi[s]$ and a new part $\Pi'[tc]=W_{\mbf{d}}[tc]$. If $t=s$, then $\mbf{D}'[x,tc]$ is the degree of $x$ into the subset $W_{\mbf{d}}[sc]$ of $\Pi[s]$ in the uniformly random graph on $\Pi[s]$ with degree sequence $(\mbf{d}[z,s])_{z\in\Pi[s]}$; if $t\ne s$, it is the degree of $x$ into the subset $W_{\mbf{d}}[tc]$ of $\Pi[t]$ in the uniformly random bipartite graph between $\Pi[s]$ and $\Pi[t]$ with the side-degree sequences above. We apply \Cref{thm:nice_deg}\cref{thm:nice_deg:graph} in the first case and \Cref{thm:nice_deg}\cref{thm:nice_deg:bipartite} in the second, with $n=n[t]$ (and $\ell=n[s]$ in the bipartite case), $v=x$, $S=W_{\mbf{d}}[tc]$, and the constant $T$ of that theorem equal to $T_1$: the hypotheses on $p$, on $\ell$, on the edge count, and on the size of $S$ hold by the discussion following \eqref{eq:T1}, and the degree hypotheses $|\alpha_i|,|\beta_j|\le\log n$ from \cref{R:degrees}, because $n[u]\ge T^{-1}N$ gives $\sqrt{pN}\,(\log N)^{2/3}\le\sqrt{T}\sqrt{p\,n[u]}\,(\log N)^{2/3}\le\sqrt{p\,n[u]}\log n[t]$ for $u\in\{s,t\}$ and all sufficiently large $N$. Now the event that $\kappa_p(\Pi',\mbf{D}')$ fails at the entry $(x,tc)$ is the event that $\mbf{D}'[x,tc]=t'$ for some integer $t'$ with $|t'-p|W_{\mbf{d}}[tc]||>(pN)^{4/7}$. Writing such a $t'$ as $t'=p|W_{\mbf{d}}[tc]|+\tau\sqrt{p|W_{\mbf{d}}[tc]|}$, we have $|\tau|\ge(pN)^{4/7}/\sqrt{pN}=(pN)^{1/14}\ge\log^{100}N$ for all sufficiently large $N$, so \Cref{thm:nice_deg}, whose constant is $\Cof{thm:nice_deg}{\theta,T_1}$ here, bounds the probability of this event by
\[
\sum_{t'}\exp\left(-\Cof{thm:nice_deg}{\theta,T_1}\tau^2\right)\le O(pN)\exp\left(-\Cof{thm:nice_deg}{\theta,T_1}(pN)^{1/7}\right)=N^{-\omega(1)}
\]
for each fixed $x$ and $tc$. A union bound over all $x\in V$ and the $2^{k+1}$ choices of $tc$ gives \cref{S:reg} with probability $1-o(1)$.
\end{proof}

\subsection{Proof of the local coarse transition theorem}

The proof assembles the two preceding propositions along the two-stage sampling procedure that defines $\overline{K}_k^{(p)}(y,\cdot)$: a fiber sample $\sigma\sim\Lambda_y^{(p)}$ is typical by \Cref{prop:fiber-degree-array-estimates}, a kernel sample from a typical $\sigma$ is well behaved by \Cref{prop:kernel-splitting-estimates}, and the conclusions of the two propositions combine into the four defining conditions of $\mc{L}_{k+1}^{(p)}(y,\lambda)$. The only computation left is to replace the degree masses $L_{\mbf{d}}$ appearing in \cref{S:edges} by their local templates $\ell^{\mr{loc}}$.

\begin{proof}[Proof of \Cref{thm:local-coarse-transition}]
We prove the theorem with
\[
\Cst{thm:local-coarse-transition}:=\Cof{prop:fiber-degree-array-estimates}{\theta,k,T,\phi}+\Cof{prop:kernel-splitting-estimates}{\theta,k,T,\phi}.
\]
By the definition of $\overline K_k^{(p)}(y,\cdot)$, we may sample from it by first sampling
\[
\sigma=(\Pi,\mbf{d})\sim\Lambda_y^{(p)},
\]
then sampling
\[
\sigma'=(\Pi',\mbf{D}')\sim K_k(\sigma,\cdot),
\]
and finally setting $y':=\rho_{k+1}^{(p)}(\sigma')$. We must show that $y'$ satisfies \cref{L:refine}, \cref{L:reg}, \cref{L:sizes}, and \cref{L:edges} with probability $1-o(1)$.

\emph{Step 1: the fiber sample.} Every $\sigma$ in the support of $\Lambda_y^{(p)}$ lies in $\mc{S}_k(V)$ and satisfies $\rho_k^{(p)}(\sigma)=y$; hence
\begin{equation}\label{eq:fiber-support}
\mbf{d}\in\mc{D}^{\mc I}(\Pi)\qquad\text{and}\qquad m(\Pi,\mbf{d})=m.
\end{equation}
By \Cref{prop:fiber-degree-array-estimates}, with probability $1-o(1)$ over $\sigma\sim\Lambda_y^{(p)}$, the array $\mbf{d}$ also satisfies \cref{R:degrees}, \cref{R:sizes}, and \cref{R:masses}. Fix such a $\mbf{d}$ for the rest of the proof. Then \eqref{eq:fiber-support}, \cref{R:degrees}, and \cref{R:sizes} are exactly the hypotheses on $\mbf{d}$ in \Cref{prop:kernel-splitting-estimates}.

\emph{Step 2: the kernel sample.} By \Cref{prop:kernel-splitting-estimates}, with probability $1-o(1)$ over $\sigma'\sim K_k((\Pi,\mbf{d}),\cdot)$, uniformly in the array $\mbf{d}$ fixed in Step 1, the conclusions \cref{S:parts}, \cref{S:reg}, and \cref{S:edges} hold. Fix such a $\sigma'$.

\emph{Step 3: verification of \cref{L:refine}--\cref{L:edges}.} By \cref{S:parts}, $\Pi'[sb]=W_{\mbf{d}}[sb]$ for all $s\in\{0,1\}^k$ and $b\in\{0,1\}$. Since for each $x\in\Pi[s]$ exactly one of $\mc{J}_{s0}(\mbf{d}[x,\bullet])$ and $\mc{J}_{s1}(\mbf{d}[x,\bullet])$ holds, we have $W_{\mbf{d}}[s0]\sqcup W_{\mbf{d}}[s1]=\Pi[s]$, which is \cref{L:refine}. By \cref{S:reg} and the definition of $\rho_{k+1}^{(p)}$, we have $\on{reg}(y')=\one_{\kappa_p(\Pi',\mbf{D}')}=1$, which is \cref{L:reg}. By \cref{S:parts} and \cref{R:sizes},
\[
\left||\Pi'[sb]|-n^{\mr{loc}}[sb]\right|
=
\left||W_{\mbf{d}}[sb]|-n^{\mr{loc}}[sb]\right|
\le \Cof{prop:fiber-degree-array-estimates}{\theta,k,T,\phi}\sqrt N\log N\le \Cst{thm:local-coarse-transition}\sqrt N\log N,
\]
which is \cref{L:sizes}.

It remains to prove \cref{L:edges}. Fix $s,t\in\{0,1\}^k$ and $b,c\in\{0,1\}$, and write $m':=\on{edge}(y')=m(\Pi',\mbf{D}')$. Recall that $m[s,t]>0$ by \cref{LA:positive}. By the definition of $m^{\mr{loc}}$, the triangle inequality, and \cref{S:edges},
\begin{equation}\label{eq:L4-split}
\begin{aligned}
\left|m'[sb,tc]-m^{\mr{loc}}[sb,tc]\right|
&\le
\Cof{prop:kernel-splitting-estimates}{\theta,k,T,\phi}N^2p\sqrt{\frac{(pN)^{1/7}}{N}}\log N\\
&\quad+
\frac{\left|L_{\mbf{d}}[sb,t]L_{\mbf{d}}[tc,s]-\ell^{\mr{loc}}[sb,t]\ell^{\mr{loc}}[tc,s]\right|}{m[s,t]}.
\end{aligned}
\end{equation}
We bound the last term. Since $W_{\mbf{d}}[tc]\subseteq\Pi[t]$ and $m(\Pi,\mbf{d})=m$ by \eqref{eq:fiber-support}, we have
\[
0\le L_{\mbf{d}}[tc,s]\le\sum_{x\in\Pi[t]}\mbf{d}[x,s]=m[t,s]=m[s,t],
\]
and since $\one_{\mc{J}_{sb}(A_{s,\bullet})}\le1$, condition \cref{LA:solvability} gives
\[
0\le\ell^{\mr{loc}}[sb,t]\le n[s]\mb{E}\left[A_{s,t}\mid\mc{I}_s(A_{s,\bullet})\right]=m[s,t].
\]
Writing
\begin{align*}
&L_{\mbf{d}}[sb,t]L_{\mbf{d}}[tc,s]-\ell^{\mr{loc}}[sb,t]\ell^{\mr{loc}}[tc,s]\\
&\qquad=
\left(L_{\mbf{d}}[sb,t]-\ell^{\mr{loc}}[sb,t]\right)L_{\mbf{d}}[tc,s]
+
\ell^{\mr{loc}}[sb,t]\left(L_{\mbf{d}}[tc,s]-\ell^{\mr{loc}}[tc,s]\right)
\end{align*}
and applying \cref{R:masses} to both differences, we obtain
\begin{align*}
\frac{\left|L_{\mbf{d}}[sb,t]L_{\mbf{d}}[tc,s]-\ell^{\mr{loc}}[sb,t]\ell^{\mr{loc}}[tc,s]\right|}{m[s,t]}
&\le
\left|L_{\mbf{d}}[sb,t]-\ell^{\mr{loc}}[sb,t]\right|+\left|L_{\mbf{d}}[tc,s]-\ell^{\mr{loc}}[tc,s]\right|\\
&\le
2\Cof{prop:fiber-degree-array-estimates}{\theta,k,T,\phi}\frac{N^2p}{\sqrt N}\log N.
\end{align*}
Finally,
\[
N^2p\sqrt{\frac{(pN)^{1/7}}{N}}\log N=(pN)^{1/14}\cdot\frac{N^2p}{\sqrt N}\log N,
\]
and $(pN)^{1/14}\ge2$ for all sufficiently large $N$, so the bound just obtained is at most
\[
\Cof{prop:fiber-degree-array-estimates}{\theta,k,T,\phi}N^2p\sqrt{\frac{(pN)^{1/7}}{N}}\log N.
\]
Inserting this into \eqref{eq:L4-split} gives \cref{L:edges} with $\Cst{thm:local-coarse-transition}=\Cof{prop:fiber-degree-array-estimates}{\theta,k,T,\phi}+\Cof{prop:kernel-splitting-estimates}{\theta,k,T,\phi}$.

Steps 1 and 2 each fail with probability $o(1)$, so $y'\in\mc{L}_{k+1}^{(p)}(y,\lambda)$ with probability $1-o(1)$ under the sampling procedure defining $\overline K_k^{(p)}(y,\cdot)$. This proves the theorem.
\end{proof}

\section{The universal recursion}\label{sec:universal}

The local coarse transition theorem describes one day of the coarse process from an admissible state in terms of its local template, which is a functional of the current sizes, edge counts, and tilt. To run the theorem for a bounded number of days we need to know, ahead of time, states around which the template can be controlled uniformly: these are the faithful states of \Cref{sec:idealized}, defined by comparison with an idealized process. The idealized process is, to first order, governed by universal numbers that do not depend on $N$ or $p$: proportions $\nu[s]$ of vertices in each history class, first-order corrections $\mu[s,t]$ to the edge densities between classes, and linear responses $\varepsilon[s]$ of the class sizes to the initial imbalance. This section defines these numbers through a recursion driven by conditioned Gaussian vectors, and proves the sign facts about them on which the whole argument rests: the responses are coherent, $\varepsilon[s0]>0>\varepsilon[s1]$, so that the imbalance is amplified rather than washed out.

Nothing in this section depends on $N$, $p$, or the setup of \Cref{sec:setup,sec:local}, except through the assertions $\mc{I}_s$ and $\mc{J}_{sb}$ of \Cref{sec:setup}, which are statements about real vectors. The analytic input is \Cref{app:inverse}. For a string $s\in\{0,1\}^k$ we write $\ol{s}$ for its bitwise complement.

\subsection{Gaussian row models and the conditional-mean map}

\begin{definition}[Gaussian row model]\label{def:gaussian-row}
Fix $k\ge1$. For $\nu\in\mb{R}_{>0}^{\{0,1\}^k}$ and $\gamma\in\mb{R}^{\{0,1\}^k\times\{0,1\}^k}$, the \emph{Gaussian row model} with means $\gamma$ and variances $\nu$ is the family of random vectors $W_{s,\bullet}=(W_{s,t})_{t\in\{0,1\}^k}$, $s\in\{0,1\}^k$, with all coordinates independent and
\[
W_{s,t}\sim\mc{N}\left(\gamma[s,t],\nu[t]\right).
\]
For $s\in\{0,1\}^k$, $x\in\mb{R}^{\{0,1\}^k}$, and $1\le r\le k$, define the linear forms
\[
Z_{s,r}(x):=(-1)^{s[r]}\sum_{t\in\{0,1\}^k}(-1)^{t[r-1]}x[t]\quad(1\le r\le k-1),\qquad
Z_{s,k}(x):=\sum_{t\in\{0,1\}^k}(-1)^{t[k-1]}x[t],
\]
and the open cone
\[
\mc{C}_s:=\left\{x\in\mb{R}^{\{0,1\}^k}:Z_{s,r}(x)>0\text{ for all }1\le r\le k-1\right\},
\]
with $\mc{C}_s=\mb{R}^{\{0,1\}^k}$ when $k=1$. Thus $\mc{C}_s$ is the set of vectors satisfying every inequality of $\mc{I}_s$ strictly, and for $b\in\{0,1\}$ the set $\mc{C}_s\cap\{(-1)^bZ_{s,k}>0\}$ is the set of vectors satisfying every inequality of $\mc{J}_{sb}$ strictly.
\end{definition}

Since a Gaussian vector with positive variances lies on a given hyperplane with probability $0$, we have $\mb{P}[\mc{I}_s(W_{s,\bullet})]=\mb{P}[W_{s,\bullet}\in\mc{C}_s]>0$ and $\mb{P}[\mc{J}_{sb}(W_{s,\bullet})]=\mb{P}[W_{s,\bullet}\in\mc{C}_s,\ (-1)^bZ_{s,k}(W_{s,\bullet})>0]>0$, so all conditional expectations below are well defined, and conditioning on $\mc{I}_s(W_{s,\bullet})$ or on $\mc{J}_{sb}(W_{s,\bullet})$ is the same as conditioning on the corresponding open set. The $k$ linear forms $Z_{s,1},\ldots,Z_{s,k}$ are linearly independent (their coefficient vectors are pairwise orthogonal).

\begin{proposition}[The conditional-mean map is a strong bijection]\label{prop:cond-mean-bijection}
Fix $k\ge1$, $s\in\{0,1\}^k$, and $\nu\in\mb{R}_{>0}^{\{0,1\}^k}$. For $\gamma\in\mb{R}^{\{0,1\}^k}$ let $W^{\gamma}$ be the random vector with independent coordinates $W^\gamma[t]\sim\mc{N}(\gamma[t],\nu[t])$. Then
\[
F_{s,\nu}\colon\mb{R}^{\{0,1\}^k}\to\mc{C}_s,\qquad F_{s,\nu}(\gamma):=\mb{E}\left[W^\gamma\mid W^\gamma\in\mc{C}_s\right],
\]
is well defined, that is, its values lie in $\mc{C}_s$, and it is a strong bijection in the sense of \Cref{def:strong-bijection}.
\end{proposition}

\begin{proof}
For $k=1$ the cone $\mc{C}_s$ is all of $\mb{R}^{\{0,1\}}$, the conditioning is on a sure event, and $F_{s,\nu}$ is the identity, which is a strong bijection. Let $k\ge2$, and let $M_s\in\mb{R}^{(k-1)\times2^k}$ be the matrix whose $r$th row is the coefficient vector $\mu_r$ of $Z_{s,r}$, that is, $\mu_r[t]=(-1)^{s[r]}(-1)^{t[r-1]}$. Then $\mc{C}_s=\{x:M_sx>0\}$. The rows are nonzero, and they are pairwise orthogonal: for $r\ne r'$, flipping bit $r-1$ of $t$ is a bijection of $\{0,1\}^k$ that changes the sign of $\mu_r[t]\mu_{r'}[t]$, so $\sum_t\mu_r[t]\mu_{r'}[t]=0$. Hence $M_s$ has rank $k-1$, and \Cref{cor:cone-bijection} with $d=2^k$, $M=M_s$, and $\Sigma=\operatorname{diag}(\nu)$, which is positive definite, shows that $F_{s,\nu}=F_{\mc{C}_s}$ is a strong bijection onto $\mc{C}_s$.
\end{proof}

\subsection{Proportions and edge-density corrections}

\begin{definition}[The universal sequences $\nu$, $\gamma$, $\mu$]\label{def:universal-nu-mu}
We define, for every $k\ge1$, numbers $\nu[s]>0$ for $s\in\{0,1\}^k$ and $\gamma[s,t],\mu[s,t]\in\mb{R}$ for $s,t\in\{0,1\}^k$, by induction on $k$.

For $k=1$: $\nu[0]=\nu[1]:=1/2$, and $\gamma[s,t]=\mu[s,t]:=0$ for all $s,t\in\{0,1\}$.

Step $k\to k+1$: suppose $\nu$ and $\mu$ are defined on $\{0,1\}^k$ and satisfy
\begin{equation}\label{eq:cone-membership}
\left(\nu[t]\,\mu[s,t]\right)_{t\in\{0,1\}^k}\in\mc{C}_s\qquad\text{for every }s\in\{0,1\}^k
\end{equation}
(\Cref{lem:universal-well-defined} shows that this holds). Let
\[
\gamma[s,\bullet]:=F_{s,\nu}^{-1}\left(\left(\nu[t]\,\mu[s,t]\right)_{t\in\{0,1\}^k}\right)\qquad(s\in\{0,1\}^k),
\]
with $F_{s,\nu}$ the map of \Cref{prop:cond-mean-bijection}, and let $W_{s,\bullet}$, $s\in\{0,1\}^k$, be the Gaussian row model with means $\gamma$ and variances $\nu$, so that
\begin{equation}\label{eq:gamma-defining}
\mb{E}\left[W_{s,t}\mid\mc{I}_s(W_{s,\bullet})\right]=\nu[t]\,\mu[s,t]\qquad\text{for all }s,t\in\{0,1\}^k.
\end{equation}
For $s,t\in\{0,1\}^k$ and $b,c\in\{0,1\}$ define
\begin{align}
\nu[sb]&:=\nu[s]\,\mb{P}\left[\mc{J}_{sb}(W_{s,\bullet})\mid\mc{I}_s(W_{s,\bullet})\right],\label{eq:nu-recursion}\\
\mu[sb,tc]&:=\frac{1}{\nu[t]}\mb{E}\left[W_{s,t}\mid\mc{J}_{sb}(W_{s,\bullet})\right]+\frac{1}{\nu[s]}\mb{E}\left[W_{t,s}\mid\mc{J}_{tc}(W_{t,\bullet})\right]-\mu[s,t].\label{eq:mu-recursion}
\end{align}
\end{definition}

The Gaussian row model $W_{s,\bullet}$ attached to day $k$ in this definition is used throughout the section; we refer to it as \emph{the day-$k$ Gaussian row model}. The meaning of the three sequences is the following. In the idealized process of \Cref{sec:idealized}, a vertex of the history class $s$ has degree vector whose fluctuations, in units of $\sqrt{pN}$, are governed by $W_{s,\bullet}$: the class $s$ has proportion $\nu[s]$ of the vertices, its members have on average $\sqrt{pN}\,\nu[t]\mu[s,t]$ more neighbors in class $t$ than the trivial $pN\nu[t]$, a relative excess of $\mu[s,t]/\sqrt{pN}$, and $\gamma[s,\bullet]$ is the mean that makes a Gaussian row conditioned on the history $\mc{I}_s$ reproduce these excesses.

\begin{lemma}[The recursion is well defined]\label{lem:universal-well-defined}
For every $k\ge1$ the quantities of \Cref{def:universal-nu-mu} are defined on $\{0,1\}^k$, and:
\begin{enumerate}[label=\textup{(\roman*)},ref=\textup{(\roman*)}]
\item\label{univ:positive} $\nu[s]>0$ for every $s\in\{0,1\}^k$, and $\sum_{s\in\{0,1\}^k}\nu[s]=1$;
\item\label{univ:symmetric} $\mu[s,t]=\mu[t,s]$ for all $s,t\in\{0,1\}^k$;
\item\label{univ:cone} \eqref{eq:cone-membership} holds, that is, $Z_{s,r}\big((\nu[t]\mu[s,t])_t\big)>0$ for every $s\in\{0,1\}^k$ and $1\le r\le k-1$.
\end{enumerate}
\end{lemma}

\begin{proof}
We prove by induction on $k$ that the day-$k$ quantities are defined and satisfy \ref{univ:positive}--\ref{univ:cone}. For $k=1$, $\nu[0]=\nu[1]=1/2$ are positive and sum to $1$, $\mu\equiv0$ is symmetric, and \ref{univ:cone} is vacuous because $\mc{C}_s=\mb{R}^{\{0,1\}}$; also $\gamma=0=F_{s,\nu}^{-1}(0)$, since $F_{s,\nu}$ is the identity for $k=1$.

Suppose the claims hold for some $k\ge1$. By \ref{univ:cone}, condition \eqref{eq:cone-membership} holds, so $\gamma[s,\bullet]$ is defined through \Cref{prop:cond-mean-bijection}, the day-$k$ Gaussian row model is defined, and \eqref{eq:gamma-defining} holds. All conditional expectations in \eqref{eq:nu-recursion} and \eqref{eq:mu-recursion} are defined because $\mb{P}[\mc{J}_{sb}(W_{s,\bullet})]>0$, as noted after \Cref{def:gaussian-row}. We verify the three claims on $\{0,1\}^{k+1}$.

\ref{univ:positive}: $\nu[sb]>0$ because $\mb{P}[\mc{J}_{sb}(W_{s,\bullet})\mid\mc{I}_s(W_{s,\bullet})]>0$. For a vector $x$ satisfying $\mc{I}_s(x)$, exactly one of $\mc{J}_{s0}(x)$ and $\mc{J}_{s1}(x)$ holds, since the two differ only in the $k$th inequality, which is $\ge_{s[k-1]0}$ in one and $\ge_{s[k-1]1}$ in the other, and these two relations are complementary. Hence
\[
\nu[s0]+\nu[s1]=\nu[s]\left(\mb{P}[\mc{J}_{s0}(W_{s,\bullet})\mid\mc{I}_s(W_{s,\bullet})]+\mb{P}[\mc{J}_{s1}(W_{s,\bullet})\mid\mc{I}_s(W_{s,\bullet})]\right)=\nu[s],
\]
and summing over $s$ gives $\sum_{u\in\{0,1\}^{k+1}}\nu[u]=\sum_{s\in\{0,1\}^k}\nu[s]=1$.

\ref{univ:symmetric}: the right-hand side of \eqref{eq:mu-recursion} is invariant under exchanging $(s,b)$ with $(t,c)$, because $\mu[s,t]=\mu[t,s]$ by the induction hypothesis.

\ref{univ:cone}: fix $s\in\{0,1\}^k$ and $b\in\{0,1\}$, and let $x\in\mb{R}^{\{0,1\}^{k+1}}$ be the vector $x[u]:=\nu[u]\mu[sb,u]$. For $t\in\{0,1\}^k$, \eqref{eq:mu-recursion} gives
\begin{align*}
x[t0]+x[t1]&=\frac{\nu[t0]+\nu[t1]}{\nu[t]}\,\mb{E}\left[W_{s,t}\mid\mc{J}_{sb}(W_{s,\bullet})\right]\\
&\quad+\frac{1}{\nu[s]}\sum_{c\in\{0,1\}}\nu[tc]\,\mb{E}\left[W_{t,s}\mid\mc{J}_{tc}(W_{t,\bullet})\right]-\left(\nu[t0]+\nu[t1]\right)\mu[s,t].
\end{align*}
The first coefficient is $1$ by \ref{univ:positive}. In the second term, $\nu[tc]=\nu[t]\,\mb{P}[\mc{J}_{tc}(W_{t,\bullet})\mid\mc{I}_t(W_{t,\bullet})]$ by \eqref{eq:nu-recursion}, and since $\mc{J}_{t0}$ and $\mc{J}_{t1}$ partition $\mc{I}_t$, the law of total expectation gives $\sum_c\mb{P}[\mc{J}_{tc}\mid\mc{I}_t]\,\mb{E}[W_{t,s}\mid\mc{J}_{tc}]=\mb{E}[W_{t,s}\mid\mc{I}_t(W_{t,\bullet})]=\nu[s]\mu[t,s]$ by \eqref{eq:gamma-defining}; so the second term equals $\nu[t]\mu[t,s]$. The third term is $-\nu[t]\mu[s,t]$. By \ref{univ:symmetric} the last two cancel, and
\begin{equation}\label{eq:children-sum}
x[t0]+x[t1]=\mb{E}\left[W_{s,t}\mid\mc{J}_{sb}(W_{s,\bullet})\right]\qquad(t\in\{0,1\}^k).
\end{equation}
Now let $1\le r\le k$. For $u=tc\in\{0,1\}^{k+1}$ we have $u[r-1]=t[r-1]$, and $(sb)[r]=s[r]$ for $r\le k-1$ while $(sb)[k]=b$. Hence, by \eqref{eq:children-sum},
\[
Z_{sb,r}(x)=(-1)^{(sb)[r]}\sum_{t}(-1)^{t[r-1]}\left(x[t0]+x[t1]\right)
=\begin{cases}
\mb{E}\left[Z_{s,r}(W_{s,\bullet})\mid\mc{J}_{sb}\right]&\text{if }1\le r\le k-1,\\[2pt]
(-1)^b\,\mb{E}\left[Z_{s,k}(W_{s,\bullet})\mid\mc{J}_{sb}\right]&\text{if }r=k,
\end{cases}
\]
where $\mc{J}_{sb}$ abbreviates $\mc{J}_{sb}(W_{s,\bullet})$.
Conditioning on $\mc{J}_{sb}(W_{s,\bullet})$ is conditioning on the nonempty open convex set $\mc{C}_s\cap\{(-1)^bZ_{s,k}>0\}$, on which the linear forms $Z_{s,r}$ ($r\le k-1$) and $(-1)^bZ_{s,k}$ are positive; by \Cref{lem:open-conditioning} the conditional mean of $W_{s,\bullet}$ lies in this set, so all the displayed quantities are positive. Thus $x\in\mc{C}_{sb}$, which is \eqref{eq:cone-membership} on $\{0,1\}^{k+1}$.
\end{proof}

\begin{lemma}[Bit-flip symmetry]\label{lem:universal-symmetry}
For every $k\ge1$ and $s,t\in\{0,1\}^k$,
\[
\nu[\ol{s}]=\nu[s],\qquad\gamma[\ol{s},\ol{t}]=\gamma[s,t],\qquad\mu[\ol{s},\ol{t}]=\mu[s,t].
\]
Consequently $\sum_{t\in\{0,1\}^k}(-1)^{t[r-1]}\nu[t]=0$ for every $1\le r\le k$.
\end{lemma}

\begin{proof}
Let $R\colon\mb{R}^{\{0,1\}^k}\to\mb{R}^{\{0,1\}^k}$ be the reflection $(Rx)[t]:=x[\ol{t}]$, an involution. Substituting $u=\ol{t}$ and using $(-1)^{\ol{t}[i]}=-(-1)^{t[i]}$,
\begin{align*}
Z_{\ol{s},r}(Rx)&=(-1)^{1-s[r]}\sum_{u}(-1)^{1-u[r-1]}x[u]=Z_{s,r}(x)\qquad(1\le r\le k-1),\\
Z_{\ol{s},k}(Rx)&=\sum_u(-1)^{1-u[k-1]}x[u]=-Z_{s,k}(x).
\end{align*}
Consequently $R\mc{C}_s=\mc{C}_{\ol{s}}$, and $R$ maps $\mc{C}_s\cap\{(-1)^bZ_{s,k}>0\}$ onto $\mc{C}_{\ol{s}}\cap\{(-1)^{\ol{b}}Z_{\ol{s},k}>0\}$; in words, $R$ carries the strict forms of $\mc{I}_s$ and $\mc{J}_{sb}$ to those of $\mc{I}_{\ol{s}}$ and $\mc{J}_{\ol{s}\,\ol{b}}$.

We prove the three identities by induction on $k$; for $k=1$ they hold because $\nu[0]=\nu[1]$ and $\gamma=\mu=0$. Suppose they hold on $\{0,1\}^k$. If $W$ has independent coordinates $W[t]\sim\mc{N}(\gamma[t],\nu[t])$ and $\nu[\ol{t}]=\nu[t]$, then $RW$ has independent coordinates $(RW)[t]=W[\ol{t}]\sim\mc{N}((R\gamma)[t],\nu[t])$; therefore
\[
F_{\ol{s},\nu}(R\gamma)=\mb{E}\left[RW\mid RW\in\mc{C}_{\ol{s}}\right]=R\,\mb{E}\left[W\mid W\in\mc{C}_s\right]=R\,F_{s,\nu}(\gamma),
\]
that is, $F_{\ol{s},\nu}\circ R=R\circ F_{s,\nu}$, and hence $F_{\ol{s},\nu}^{-1}\circ R=R\circ F_{s,\nu}^{-1}$. The vector $v_s:=(\nu[t]\mu[s,t])_t$ satisfies $v_{\ol{s}}[t]=\nu[t]\mu[\ol{s},t]=\nu[\ol{t}]\mu[s,\ol{t}]=v_s[\ol{t}]$ by the induction hypothesis (applied to the pair $(\ol{s},t)$), so $v_{\ol{s}}=Rv_s$, and therefore $\gamma[\ol{s},\bullet]=F^{-1}_{\ol{s},\nu}(Rv_s)=R\,\gamma[s,\bullet]$, which is $\gamma[\ol{s},\ol{t}]=\gamma[s,t]$. It follows that the day-$k$ Gaussian row model satisfies $W_{\ol{s},\bullet}\overset{d}{=}RW_{s,\bullet}$, and by the first paragraph the events $\mc{I}_{\ol{s}}(W_{\ol{s},\bullet})$ and $\mc{J}_{\ol{s}\,\ol{b}}(W_{\ol{s},\bullet})$ correspond, up to null sets, to $\mc{I}_s(W_{s,\bullet})$ and $\mc{J}_{sb}(W_{s,\bullet})$ under this identification. Hence
\[
\mb{P}\left[\mc{J}_{\ol{s}\,\ol{b}}(W_{\ol{s},\bullet})\mid\mc{I}_{\ol{s}}(W_{\ol{s},\bullet})\right]=\mb{P}\left[\mc{J}_{sb}(W_{s,\bullet})\mid\mc{I}_s(W_{s,\bullet})\right],\qquad
\mb{E}\left[W_{\ol{s},\ol{t}}\mid\mc{J}_{\ol{s}\,\ol{b}}(W_{\ol{s},\bullet})\right]=\mb{E}\left[W_{s,t}\mid\mc{J}_{sb}(W_{s,\bullet})\right],
\]
and \eqref{eq:nu-recursion}, \eqref{eq:mu-recursion} give $\nu[\ol{s}\,\ol{b}]=\nu[sb]$ and $\mu[\ol{s}\,\ol{b},\ol{t}\,\ol{c}]=\mu[sb,tc]$, completing the induction. Finally, for $1\le r\le k$, pairing $t$ with $\ol{t}$ in $\sum_t(-1)^{t[r-1]}\nu[t]$ shows that the sum equals its own negative, hence vanishes.
\end{proof}

\subsection{The linear response}

\begin{lemma}[Conditional covariance]\label{lem:cond-cov-pd}
For every $k\ge1$ and $s\in\{0,1\}^k$, the matrix
\[
\Sigma_s:=\mr{Cov}\left(W_{s,\bullet}\mid\mc{I}_s(W_{s,\bullet})\right)\in\mb{R}^{\{0,1\}^k\times\{0,1\}^k},
\]
with entries $\Sigma_s[t,t']=\mb{E}[W_{s,t}W_{s,t'}\mid\mc{I}_s(W_{s,\bullet})]-\mb{E}[W_{s,t}\mid\mc{I}_s(W_{s,\bullet})]\,\mb{E}[W_{s,t'}\mid\mc{I}_s(W_{s,\bullet})]$,
is positive definite, where $W_{s,\bullet}$ is the day-$k$ Gaussian row model.
\end{lemma}

\begin{proof}
Conditioning on $\mc{I}_s(W_{s,\bullet})$ is conditioning on the nonempty open convex set $\mc{C}_s$, and $W_{s,\bullet}$ is Gaussian with the positive definite covariance $\operatorname{diag}(\nu)$. \Cref{lem:open-conditioning} applies.
\end{proof}

\begin{definition}[The universal sequences $\varepsilon$, $\beta$]\label{def:universal-epsilon}
We define, for every $k\ge1$, numbers $\varepsilon[s]$ for $s\in\{0,1\}^k$ and $\beta[s,t]$ for $s,t\in\{0,1\}^k$, by induction on $k$.

For $k=1$: $\varepsilon[0]:=1$ and $\varepsilon[1]:=-1$.

Step $k\to k+1$: suppose $\varepsilon$ is defined on $\{0,1\}^k$, and let $W_{s,\bullet}$ be the day-$k$ Gaussian row model. For $s\in\{0,1\}^k$ let $\beta[s,\bullet]\in\mb{R}^{\{0,1\}^k}$ be the unique solution of the linear system
\begin{equation}\label{eq:beta-defining}
\sum_{t'\in\{0,1\}^k}\beta[s,t']\,\Sigma_s[t',t]=\varepsilon[t]\qquad(t\in\{0,1\}^k),
\end{equation}
which exists by \Cref{lem:cond-cov-pd}; equivalently, the random variable $B_s:=\sum_{t'\in\{0,1\}^k}\beta[s,t']W_{s,t'}$ satisfies $\mr{Cov}(B_s,W_{s,t}\mid\mc{I}_s(W_{s,\bullet}))=\varepsilon[t]$ for every $t$. For $s\in\{0,1\}^k$ and $b\in\{0,1\}$ define
\begin{equation}\label{eq:epsilon-recursion}
\varepsilon[sb]:=\nu[sb]\left(\mb{E}\left[B_s\mid\mc{J}_{sb}(W_{s,\bullet})\right]-\mb{E}\left[B_s\mid\mc{I}_s(W_{s,\bullet})\right]\right).
\end{equation}
\end{definition}

The role of these quantities will be made precise in \Cref{sec:idealized}. There, an initial imbalance of $\tau\sqrt N$ vertices will be represented, relative to the balanced idealized process, by the first-order class-size displacement $\tau\sqrt N\sqrt{pN}^{k-1}\varepsilon[s]$ on day $k$, together with an effective-tilt displacement $\tau\frac{\sqrt{pN}^{k-1}}{\sqrt N}\beta[s,t]$ in direction $t$. Equation \eqref{eq:beta-defining} is the compatibility relation that will make this tilt reproduce the neighboring class-size responses $\varepsilon[t]$, while \eqref{eq:epsilon-recursion} is the corresponding linearized update through the next majority decision.

\begin{lemma}[Symmetries of $\varepsilon$]\label{lem:epsilon-symmetry}
For every $k\ge1$ and $s,t\in\{0,1\}^k$:
\begin{enumerate}[label=\textup{(\roman*)},ref=\textup{(\roman*)}]
\item\label{eps:children} $\varepsilon[s0]+\varepsilon[s1]=0$;
\item\label{eps:flip} $\varepsilon[\ol{s}]=-\varepsilon[s]$ and $\beta[\ol{s},\ol{t}]=-\beta[s,t]$.
\end{enumerate}
\end{lemma}

\begin{proof}
\ref{eps:children}: by \eqref{eq:epsilon-recursion}, \eqref{eq:nu-recursion}, and \Cref{lem:universal-well-defined}\ref{univ:positive},
\[
\varepsilon[s0]+\varepsilon[s1]=\nu[s]\sum_{b\in\{0,1\}}\mb{P}\left[\mc{J}_{sb}(W_{s,\bullet})\mid\mc{I}_s(W_{s,\bullet})\right]\mb{E}\left[B_s\mid\mc{J}_{sb}(W_{s,\bullet})\right]-\nu[s]\,\mb{E}\left[B_s\mid\mc{I}_s(W_{s,\bullet})\right]=0,
\]
because $\mc{J}_{s0}$ and $\mc{J}_{s1}$ partition $\mc{I}_s$ and the law of total expectation applies.

\ref{eps:flip}: we prove $\varepsilon[\ol{s}]=-\varepsilon[s]$ by induction on $k$, deriving the identity for $\beta$ along the way. For $k=1$ it is $\varepsilon[1]=-\varepsilon[0]$. Suppose $\varepsilon[\ol{t}]=-\varepsilon[t]$ for all $t\in\{0,1\}^k$. With $R$ the reflection of the proof of \Cref{lem:universal-symmetry}, that proof gives $W_{\ol{s},\bullet}\overset{d}{=}RW_{s,\bullet}$ with corresponding conditioning events, hence
\[
\Sigma_{\ol{s}}[\ol{t},\ol{t}']=\mr{Cov}\left(W_{\ol{s},\ol{t}},W_{\ol{s},\ol{t}'}\mid\mc{I}_{\ol{s}}\right)=\mr{Cov}\left(W_{s,t},W_{s,t'}\mid\mc{I}_s\right)=\Sigma_s[t,t'].
\]
Define $\beta'[t]:=-\beta[s,\ol{t}]$. Then for every $t$,
\[
\sum_{t'}\beta'[t']\,\Sigma_{\ol{s}}[t',t]=-\sum_{t'}\beta[s,\ol{t}']\,\Sigma_s[\ol{t}',\ol{t}]=-\sum_{u}\beta[s,u]\,\Sigma_s[u,\ol{t}]=-\varepsilon[\ol{t}]=\varepsilon[t],
\]
by \eqref{eq:beta-defining} for $s$ and the induction hypothesis. So $\beta'$ solves \eqref{eq:beta-defining} for $\ol{s}$, and by uniqueness $\beta[\ol{s},t]=-\beta[s,\ol{t}]$ for all $t$, that is, $\beta[\ol{s},\ol{t}]=-\beta[s,t]$. Consequently, under the identification $W_{\ol{s},\bullet}=RW_{s,\bullet}$,
\[
B_{\ol{s}}=\sum_t\beta[\ol{s},t]\,W_{s,\ol{t}}=\sum_u\beta[\ol{s},\ol{u}]\,W_{s,u}=-B_s,
\]
and since the conditioning events correspond, \eqref{eq:epsilon-recursion} together with $\nu[\ol{s}\,\ol{b}]=\nu[sb]$ gives $\varepsilon[\ol{s}\,\ol{b}]=-\varepsilon[sb]$.
\end{proof}

\begin{theorem}[Coherence of the linear response]\label{thm:coherence}
For every $k\ge1$ and $s\in\{0,1\}^k$,
\[
\varepsilon[s0]>0\qquad\text{and}\qquad\varepsilon[s1]<0.
\]
\end{theorem}

\begin{proof}
By \Cref{lem:epsilon-symmetry}\ref{eps:children} it suffices to prove $\varepsilon[s0]>0$. We proceed by induction on $k$, using at day $k$ the following consequence of the statement for smaller $k$:
\begin{equation}\label{eq:lead-sum-positive}
S_k:=\sum_{t\in\{0,1\}^k}(-1)^{t[k-1]}\varepsilon[t]>0 .
\end{equation}
For $k=1$, $S_1=\varepsilon[0]-\varepsilon[1]=2$ by definition; for $k\ge2$, grouping $t=t'b$ with $t'\in\{0,1\}^{k-1}$ gives $S_k=\sum_{t'}(\varepsilon[t'0]-\varepsilon[t'1])=2\sum_{t'}\varepsilon[t'0]$ by \Cref{lem:epsilon-symmetry}\ref{eps:children}, which is positive by the statement for $k-1$. We also use
\begin{equation}\label{eq:history-sum-zero}
\sum_{t\in\{0,1\}^k}(-1)^{t[r-1]}\varepsilon[t]=0\qquad(1\le r\le k-1),
\end{equation}
which holds because for $r\le k-1$ the bit $t[r-1]$ depends only on $t'$ when $t=t'b$, so the sum is $\sum_{t'}(-1)^{t'[r-1]}(\varepsilon[t'0]+\varepsilon[t'1])=0$.

Fix $k\ge1$ and $s\in\{0,1\}^k$, let $W:=W_{s,\bullet}$ be the day-$k$ Gaussian row model, and write $Z_r:=Z_{s,r}(W)$ for $1\le r\le k$. Up to null sets, $\mc{I}_s(W)$ is the event $E:=\{Z_1>0,\ldots,Z_{k-1}>0\}$ and $\mc{J}_{s0}(W)$ is $E\cap\{Z_k>0\}$. Let $\mc{G}$ be the linear span of the centered variables $W_t-\mb{E}W_t$, $t\in\{0,1\}^k$, a finite-dimensional space of centered jointly Gaussian random variables, with the inner product $\langle X,Y\rangle:=\mb{E}[XY]=\mr{Cov}(X,Y)$; in $\mc{G}$, orthogonal families are independent. The centered variables $\ol{Z}_r:=Z_r-\mb{E}Z_r$ and $\ol{B}:=B_s-\mb{E}B_s$ lie in $\mc{G}$. The $\ol{Z}_r$, $1\le r\le k$, are linearly independent: their coefficient vectors are nonzero and pairwise orthogonal in $\mb{R}^{\{0,1\}^k}$ (see the proof of \Cref{prop:cond-mean-bijection}), and $a\mapsto\sum_ta[t](W_t-\mb{E}W_t)$ is injective because its image has variance $\sum_ta[t]^2\nu[t]>0$ for $a\ne0$.

\emph{Step 1: the covariances of $B_s$ with the constraints.} By \eqref{eq:beta-defining}, $\mr{Cov}(B_s,W_t\mid E)=\varepsilon[t]$ for every $t$, so by linearity
\begin{align*}
\mr{Cov}(B_s,Z_r\mid E)&=(-1)^{s[r]}\sum_t(-1)^{t[r-1]}\varepsilon[t]=0\qquad(1\le r\le k-1),\\
\mr{Cov}(B_s,Z_k\mid E)&=\sum_t(-1)^{t[k-1]}\varepsilon[t]=S_k>0,
\end{align*}
by \eqref{eq:history-sum-zero} and \eqref{eq:lead-sum-positive}.

\emph{Step 2: $B_s$ is independent of the history constraints.} Decompose $\ol{B}=\sum_{r<k}\alpha_r\ol{Z}_r+X$ with $X\in\mc{G}$ orthogonal to $\ol{Z}_1,\ldots,\ol{Z}_{k-1}$ (orthogonal projection). Then $X$ is independent of $(Z_r)_{r<k}$, and since $E$ is an event of $(Z_r)_{r<k}$, under the law conditioned on $E$ the variable $X$ is still independent of $(Z_r)_{r<k}$; hence $\mr{Cov}(X,Z_r\mid E)=0$ for $r<k$, and Step 1 gives
\[
0=\mr{Cov}(B_s,Z_r\mid E)=\sum_{r'<k}\alpha_{r'}\mr{Cov}(Z_{r'},Z_r\mid E)\qquad(1\le r\le k-1).
\]
The matrix $(\mr{Cov}(Z_{r'},Z_r\mid E))_{r',r<k}$ is the covariance of the nondegenerate Gaussian vector $(Z_r)_{r<k}\in\mb{R}^{k-1}$ conditioned on the open orthant $E$, hence positive definite by \Cref{lem:open-conditioning}. Therefore $\alpha_r=0$ for all $r$, and $\ol{B}=X$: the centered variable $B_s-\mb{E}B_s$ is independent of $(Z_r)_{r<k}$. (For $k=1$ there is nothing to prove in this step.)

\emph{Step 3: the component of $Z_k$ orthogonal to the constraints.} Decompose $\ol{Z}_k=\sum_{r<k}\delta_r\ol{Z}_r+Y$ with $Y\in\mc{G}$ orthogonal to $\ol{Z}_1,\ldots,\ol{Z}_{k-1}$. By linear independence $Y\ne0$, so $\mr{Var}(Y)>0$. The pair $(X,Y)$ is independent of $(Z_r)_{r<k}$, so conditioning on $E$ does not change its joint law, and using Step 2,
\[
S_k=\mr{Cov}(B_s,Z_k\mid E)=\mr{Cov}\Big(X,\sum_{r<k}\delta_rZ_r+Y\ \Big|\ E\Big)=\mr{Cov}(X,Y).
\]
Write $X=\alpha Y+X'$ with $\alpha:=\mr{Cov}(X,Y)/\mr{Var}(Y)=S_k/\mr{Var}(Y)>0$ and $X'\in\mc{G}$ orthogonal to $Y$; then $X'$ is orthogonal to $Y$ and to all $\ol{Z}_r$, $r<k$, hence independent of $(Y,(Z_r)_{r<k})$.

\emph{Step 4: the conditional means.} Since $B_s=\mb{E}B_s+X$ and $X$ is independent of $E$ with mean $0$, we have $\mb{E}[B_s\mid E]=\mb{E}B_s$, and \eqref{eq:epsilon-recursion} reads
\[
\varepsilon[s0]=\nu[s0]\,\mb{E}\left[X\mid E\cap\{Z_k>0\}\right].
\]
Put $U:=-\mb{E}Z_k-\sum_{r<k}\delta_r\ol{Z}_r$, a function of $(Z_r)_{r<k}$, so that $\{Z_k>0\}=\{Y>U\}$. The variable $X'$ is centered and independent of $(Y,U,\mathbf{1}_E)$, so $\mb{E}[X'\mid E\cap\{Y>U\}]=0$ and
\[
\mb{E}\left[X\mid E\cap\{Z_k>0\}\right]=\alpha\,\mb{E}\left[Y\mid E\cap\{Y>U\}\right]=\alpha\,\frac{\mb{E}\left[\mathbf{1}_E\,\phi(U)\right]}{\mb{P}\left[E\cap\{Y>U\}\right]},\qquad\phi(u):=\mb{E}\left[Y\mathbf{1}_{Y>u}\right],
\]
where we conditioned on $(Z_r)_{r<k}$ and used that $Y$ is independent of $(U,\mathbf{1}_E)$. For a centered Gaussian $Y$ with positive variance, $\phi(u)>0$ for every $u\in\mb{R}$: if $u\ge0$ then $Y\mathbf{1}_{Y>u}\ge0$ with positive probability of being positive, and if $u<0$ then $\phi(u)=\mb{E}Y-\mb{E}[Y\mathbf{1}_{Y\le u}]=-\mb{E}[Y\mathbf{1}_{Y\le u}]>0$. Since $\mb{P}[E]>0$, we get $\mb{E}[\mathbf{1}_E\phi(U)]>0$, hence $\mb{E}[X\mid E\cap\{Z_k>0\}]>0$, and $\varepsilon[s0]>0$ because $\nu[s0]>0$ by \Cref{lem:universal-well-defined}\ref{univ:positive}.
\end{proof}

\begin{corollary}[Sign identities]\label{cor:sign-identities}
For every $k\ge1$:
\begin{enumerate}[label=\textup{(\roman*)},ref=\textup{(\roman*)}]
\item\label{sign:history} $\sum_{t\in\{0,1\}^k}(-1)^{t[r-1]}\varepsilon[t]=0$ for every $1\le r\le k-1$, and $\sum_{t\in\{0,1\}^k}\varepsilon[t]=0$;
\item\label{sign:lead} $\sum_{t\in\{0,1\}^k}(-1)^{t[k-1]}\varepsilon[t]=2\sum_{t'\in\{0,1\}^{k-1}}\varepsilon[t'0]>0$, where for $k=1$ the right-hand side is read as $2\varepsilon[0]=2$;
\item\label{sign:nu} $\sum_{t\in\{0,1\}^k}(-1)^{t[r-1]}\nu[t]=0$ for every $1\le r\le k$;
\item\label{sign:separation} $Z_{s,r}\big((\nu[t]\mu[s,t])_t\big)>0$ for every $s\in\{0,1\}^k$ and $1\le r\le k-1$.
\end{enumerate}
\end{corollary}

\begin{proof}
\ref{sign:history}: for $k\ge2$ and $1\le r\le k-1$, write $t=t'b$ with $t'\in\{0,1\}^{k-1}$; the bit $t[r-1]=t'[r-1]$ does not depend on $b$, so the sum is $\sum_{t'}(-1)^{t'[r-1]}(\varepsilon[t'0]+\varepsilon[t'1])=0$ by \Cref{lem:epsilon-symmetry}\ref{eps:children}. The same grouping without the sign gives $\sum_t\varepsilon[t]=0$ for $k\ge2$, and for $k=1$ it is $\varepsilon[0]+\varepsilon[1]=0$.
\ref{sign:lead}: for $k\ge2$, $(-1)^{t[k-1]}=(-1)^b$ when $t=t'b$, so the sum is $\sum_{t'}(\varepsilon[t'0]-\varepsilon[t'1])=2\sum_{t'}\varepsilon[t'0]$ by \Cref{lem:epsilon-symmetry}\ref{eps:children}, which is positive by \Cref{thm:coherence}; for $k=1$ the sum is $\varepsilon[0]-\varepsilon[1]=2$.
\ref{sign:nu} is the last assertion of \Cref{lem:universal-symmetry}, and \ref{sign:separation} is \Cref{lem:universal-well-defined}\ref{univ:cone}.
\end{proof}

\begin{lemma}[Universal nondegeneracy constants]\label{lem:universal-nondegeneracy}
For every $k\ge1$ there are constants $\phi^\ast_k\in(0,1/2)$ and $\varsigma^\ast_k>0$, depending only on $k$, such that for the day-$k$ Gaussian row model and all $s\in\{0,1\}^k$, $b\in\{0,1\}$, $1\le r\le k-1$,
\[
\mb{P}\left[\mc{I}_s(W_{s,\bullet})\right]\ge\phi^\ast_k,\qquad
\phi^\ast_k\le\mb{P}\left[\mc{J}_{sb}(W_{s,\bullet})\mid\mc{I}_s(W_{s,\bullet})\right]\le1-\phi^\ast_k,\qquad
Z_{s,r}\big((\nu[t]\mu[s,t])_t\big)\ge\varsigma^\ast_k .
\]
\end{lemma}

\begin{proof}
There are finitely many quantities involved: $2^k$ probabilities $\mb{P}[\mc{I}_s(W_{s,\bullet})]$, $2^{k+1}$ conditional probabilities $\mb{P}[\mc{J}_{sb}(W_{s,\bullet})\mid\mc{I}_s(W_{s,\bullet})]$, and $2^k(k-1)$ numbers $Z_{s,r}((\nu[t]\mu[s,t])_t)$. Each $\mb{P}[\mc{I}_s(W_{s,\bullet})]$ is positive, as noted after \Cref{def:gaussian-row}. Each $\mb{P}[\mc{J}_{sb}\mid\mc{I}_s]$ is positive for the same reason, and is less than $1$ because $\mb{P}[\mc{J}_{s0}\mid\mc{I}_s]+\mb{P}[\mc{J}_{s1}\mid\mc{I}_s]=1$ (proof of \Cref{lem:universal-well-defined}) with both terms positive. Each $Z_{s,r}((\nu[t]\mu[s,t])_t)$ is positive by \Cref{cor:sign-identities}\ref{sign:separation}. Let $\phi^\ast_k$ be one half of the minimum of all the probabilities $\mb{P}[\mc{I}_s]$, $\mb{P}[\mc{J}_{sb}\mid\mc{I}_s]$, and $1-\mb{P}[\mc{J}_{sb}\mid\mc{I}_s]$, so that $\phi^\ast_k\in(0,1/4]$ and the first two displayed inequalities hold; and let $\varsigma^\ast_k$ be the minimum of the numbers $Z_{s,r}((\nu[t]\mu[s,t])_t)$ when $k\ge2$, and $\varsigma^\ast_1:=1$.
\end{proof}

These constants are the universal versions of the admissibility conditions \cref{LA:separation}, \cref{LA:conditioning}, and \cref{LA:split}: in \Cref{sec:idealized} the binomial row models attached to faithful states are shown to be close to the day-$k$ Gaussian row model, and $\phi^\ast_k$, $\varsigma^\ast_k$ supply the constants $\phi$ and $T^{-1}$ there.

\section{The idealized process and faithful states}\label{sec:idealized}

The local coarse transition theorem is applied one day at a time, starting from the day-$1$ coarse state, for a bounded number $K_\theta$ of days. To do so we need a set of coarse states on which its hypotheses hold uniformly and which is mapped into itself, up to the $1-o(1)$ probability, by the fiber-averaged transition. This section defines that set, the \emph{faithful} states, as the coarse states whose sizes and edge counts follow an \emph{idealized process} attached to $(N,p)$, perturbed linearly in the direction given by the universal responses $\varepsilon$ of \Cref{sec:universal}, up to a polynomially small error. The idealized process is deterministic; it is the local template of \Cref{def:local-template} iterated on itself, with sizes rounded to integers, and its tilts exist by the inverse-function theorems of \Cref{app:inverse}. The section states the results and the definitions; the proofs are in \Cref{app:idealized}.

Throughout, $\theta\in(1/2,1)$ is fixed and we write
\[
D_\theta:=\left\lfloor\frac{1}{1-\theta}\right\rfloor+1,
\]
the number of days for which the idealized process is needed. The integer $\frac{1}{1-\theta}$, when $\frac1{1-\theta}\in\mb{Z}$, plays a special role (\Cref{thm:edge-day}); we then write $k_0:=\frac{1}{1-\theta}$.

\subsection{The idealized process}

\begin{definition}[Logit tilt]\label{def:logit-tilt}
For $p\in(0,1)$, $N\ge1$, $g\in\mb{R}$, and $v>0$, let
\[
\Lambda_{N,p}(g,v):=\frac{\frac{p}{1-p}\,e^{g/(v\sqrt{pN})}}{1+\frac{p}{1-p}\,e^{g/(v\sqrt{pN})}}\in(0,1),
\]
the success probability whose log-odds exceed those of $p$ by $g/(v\sqrt{pN})$. By Taylor expansion, $\Lambda_{N,p}(g,v)=p+\frac{g}{v}\frac{p(1-p)}{\sqrt{pN}}+O\big(\frac{g^2}{v^2}\frac{p}{pN}\big)$ uniformly for $|g|/v$ bounded.
\end{definition}

\begin{theorem}[The idealized process]\label{thm:idealized-process}
Fix $\theta\in(1/2,1)$, an integer $D\ge1$, and $T>1$. There is a constant $\ell=\ell(\theta,D,T)\ge1$ such that for all sufficiently large $N$ and all $p\in(T^{-1}N^{-\theta},TN^{-\theta})$ there exist, for $1\le k\le D$, vectors
\[
\wt{n}_k\in\mb{Z}^{\{0,1\}^k},\qquad\wt{m}_k\in\mb{R}^{\{0,1\}^k\times\{0,1\}^k},\qquad\text{and, for }1\le k\le D-1,\qquad\wt{\lambda}_k\in(0,1)^{\{0,1\}^k\times\{0,1\}^k},
\]
written $\wt{n}[s]$, $\wt{m}[s,t]$, $\wt{\lambda}[s,t]$ for $s,t\in\{0,1\}^k$, with the following properties.
\begin{enumerate}[label=\textup{(\roman*)},ref=\textup{(\roman*)}]
\item\label{ideal:base} (Base case.) $\wt{n}[0]=\wt{n}[1]=\lfloor N/2\rfloor$, and $\wt{m}[s,t]=p\,\wt{n}[s]\left(\wt{n}[t]-\one_{s=t}\right)$ for $s,t\in\{0,1\}$.
\item\label{ideal:solvable} (Solvability.) For $1\le k\le D-1$ and all $s,t\in\{0,1\}^k$: $\wt{n}[s]\ge1$, $\wt{m}[s,t]>0$, and, writing $\wt{A}_{s,\bullet}:=A_{s,\bullet}(\wt{n}_k,\wt{\lambda}_k)$ for the row vector of \Cref{def:local-template},
\[
\mb{P}\left[\mc{I}_s(\wt{A}_{s,\bullet})\right]>0\qquad\text{and}\qquad\wt{n}[s]\,\mb{E}\left[\wt{A}_{s,t}\mid\mc{I}_s(\wt{A}_{s,\bullet})\right]=\wt{m}[s,t].
\]
Moreover $\wt{\lambda}_k$ is the unique element of $(0,1)^{\{0,1\}^k\times\{0,1\}^k}$ with this property.
\item\label{ideal:evolution} (Evolution.) For $1\le k\le D-1$, if $(\wt{n}^{\mr{loc}},\wt{\ell}^{\mr{loc}},\wt{m}^{\mr{loc}})$ denotes the local template of $(\wt{n}_k,\wt{m}_k,\wt{\lambda}_k)$, then for all $s,t\in\{0,1\}^k$ and $b,c\in\{0,1\}$,
\[
\wt{n}[sb]=\left\lfloor\wt{n}^{\mr{loc}}[sb]\right\rfloor\qquad\text{and}\qquad\wt{m}[sb,tc]=\wt{m}^{\mr{loc}}[sb,tc].
\]
\item\label{ideal:symmetry} (Symmetry.) $\wt{n}[\ol{s}]=\wt{n}[s]$, $\wt{m}[\ol{s},\ol{t}]=\wt{m}[s,t]$, and $\wt{\lambda}[\ol{s},\ol{t}]=\wt{\lambda}[s,t]$ for all $s,t$.
\item\label{ideal:asymptotics} (Asymptotics.) For $1\le k\le D$ and all $s,t\in\{0,1\}^k$,
\[
\left|\wt{n}[s]-N\nu[s]\right|\le\frac{N}{\sqrt{pN}}\log^\ell N,\qquad
\left|\wt{m}[s,t]-p\,\wt{n}[s]\,\wt{n}[t]\left(1+\frac{\mu[s,t]}{\sqrt{pN}}\right)\right|\le p\,\wt{n}[s]\,\wt{n}[t]\,\frac{\log^\ell N}{pN},
\]
and for $1\le k\le D-1$,
\[
\left|\wt{\lambda}[s,t]-\Lambda_{N,p}\left(\gamma[s,t],\nu[t]\right)\right|\le\frac{\log^\ell N}{N}.
\]
\end{enumerate}
By \ref{ideal:base}, \ref{ideal:solvable}, and \ref{ideal:evolution}, the family is determined by $(N,p)$; we call it the \emph{idealized process} and suppress its dependence on $N$ and $p$.
\end{theorem}

The bound on $|\wt n[s]-N\nu[s]|$ in \Cref{thm:idealized-process}\ref{ideal:asymptotics} is too large to resolve the imbalance $\tau\sqrt N\sqrt{pN}^{k-1}$ on days $k<\frac1{1-\theta}$. Faithfulness therefore tracks the actual sizes around $\wt n[s]$, with error polynomially smaller than the imbalance.

\subsection{Faithful coarse states}

\begin{definition}[Faithful coarse states]\label{def:faithful}
Fix $\theta\in(1/2,1)$ and $T>1$, and let $N$ be large enough that \Cref{thm:idealized-process} applies with $D=D_\theta$ and this $T$; let $p\in(T^{-1}N^{-\theta},TN^{-\theta})$, let $V$ be a set with $|V|=N$, and let $1\le k\le D_\theta$. For $T'\ge1$, $\delta>0$, and $\tau>0$, the set $\mc{G}^{(p)}_k(T',\delta,\tau)\subseteq\mc{Y}_k(V)$ of \emph{$(T',\delta,\tau)$-faithful} coarse states consists of the $y=(\Pi,m,1)\in\mc{Y}_k(V)$ such that, writing $n[s]:=|\Pi[s]|$,
\begin{enumerate}[label=\textup{(F\arabic*)},ref=\textup{(F\arabic*)},leftmargin=*]
\item\label{F:sizes} $\displaystyle\left|n[s]-\wt{n}[s]-\tau\sqrt N\sqrt{pN}^{\,k-1}\varepsilon[s]\right|\le T'\sqrt N\sqrt{pN}^{\,k-1}N^{-\delta}$ for every $s\in\{0,1\}^k$;
\item\label{F:edges} $\displaystyle\left|m[s,t]-\wt{m}[s,t]\left(1+\tau\frac{\sqrt{pN}^{\,k-1}}{\sqrt N}\left(\frac{\varepsilon[s]}{\nu[s]}+\frac{\varepsilon[t]}{\nu[t]}\right)\right)\right|\le T'\frac{\sqrt{pN}^{\,k-1}}{\sqrt N}N^{-\delta}N^2p$ for all $s,t\in\{0,1\}^k$.
\end{enumerate}
In particular $\on{reg}(y)=1$ for every faithful state.
\end{definition}

The two scales in \Cref{def:faithful} are those of the imbalance: on day $k$ the leading side is ahead by order $\sqrt N\sqrt{pN}^{k-1}$ vertices, which is $o(N)$ exactly when $k<\frac{1}{1-\theta}+1$, and the edge counts respond at relative order $\sqrt{pN}^{k-1}/\sqrt N$. The next two theorems are the deterministic core of the induction: a faithful state is admissible with an explicit tilt, and its local template is faithful on the next day. Their proofs are in \Cref{app:idealized}.

\begin{theorem}[Perturbed evolution]\label{thm:perturbed-evolution}
Fix $\theta\in(1/2,1)$, an integer $k$ with $1\le k<\frac{1}{1-\theta}$, and constants $T>1$ and $\delta>0$. There exist constants $T_1=T_1(\theta,k,T,\delta)\ge T$, $\delta_1=\delta_1(\theta,k,T,\delta)>0$, and $\phi_1=\phi_1(\theta,k,T,\delta)\in(0,1/2)$ such that for all sufficiently large $N$, all $p\in(T^{-1}N^{-\theta},TN^{-\theta})$, all $\tau\in[T^{-1},T]$, and every $y=(\Pi,m,1)\in\mc{G}^{(p)}_k(T,\delta,\tau)$ with $\Pi\in\mr{Part}_k(V)$, $|V|=N$, the following hold, where $n[s]:=|\Pi[s]|$. There is a unique tilt $\lambda_y\in(0,1)^{\{0,1\}^k\times\{0,1\}^k}$ such that, with $A_{s,\bullet}:=A_{s,\bullet}(n,\lambda_y)$,
\[
\mb{P}\left[\mc{I}_s(A_{s,\bullet})\right]>0\qquad\text{and}\qquad n[s]\,\mb{E}\left[A_{s,t}\mid\mc{I}_s(A_{s,\bullet})\right]=m[s,t]\qquad\text{for all }s,t\in\{0,1\}^k,
\]
and it satisfies:
\begin{enumerate}[label=\textup{(\alph*)},ref=\textup{(\alph*)}]
\item\label{pe:admissible} $(y,\lambda_y)$ is $(T_1,\phi_1,p)$-locally admissible.
\item\label{pe:tilt} For all $s,t\in\{0,1\}^k$, writing
\[
\vartheta_y[s,t]:=\log\frac{\lambda_y[s,t]\left(n[t]-\one_{s=t}-p\,\wt{n}[t]\right)}{\left(1-\lambda_y[s,t]\right)p\,\wt{n}[t]}
\qquad\text{and}\qquad
\wt\vartheta[s,t]:=\log\frac{\wt{\lambda}[s,t]\left(\wt{n}[t]-\one_{s=t}-p\,\wt{n}[t]\right)}{\left(1-\wt{\lambda}[s,t]\right)p\,\wt{n}[t]},
\]
we have
\[
\left|\vartheta_y[s,t]-\wt\vartheta[s,t]-\tau\frac{\sqrt{pN}^{\,k-1}}{\sqrt N}\beta[s,t]\right|
\le T_1\frac{\sqrt{pN}^{\,k-1}}{\sqrt N}N^{-\delta_1}.
\]
\item\label{pe:template} The local template $(n^{\mr{loc}},\ell^{\mr{loc}},m^{\mr{loc}})$ of $(y,\lambda_y)$ satisfies, for all $s,t\in\{0,1\}^k$ and $b,c\in\{0,1\}$,
\begin{align*}
\left|n^{\mr{loc}}[sb]-\wt{n}[sb]-\tau\sqrt N\sqrt{pN}^{\,k}\varepsilon[sb]\right|&\le T_1\sqrt N\sqrt{pN}^{\,k}N^{-\delta_1},\\
\left|m^{\mr{loc}}[sb,tc]-\wt{m}[sb,tc]\left(1+\tau\frac{\sqrt{pN}^{\,k}}{\sqrt N}\left(\frac{\varepsilon[sb]}{\nu[sb]}+\frac{\varepsilon[tc]}{\nu[tc]}\right)\right)\right|&\le T_1\frac{\sqrt{pN}^{\,k}}{\sqrt N}N^{-\delta_1}N^2p.
\end{align*}
\end{enumerate}
\end{theorem}

The quantity $\vartheta_y[s,t]$ in \ref{pe:tilt} is the effective tilt of the law $\mr{Bin}(n[t]-\one_{s=t},\lambda_y[s,t])$ around the center $p\,\wt{n}[t]$ in the sense of the first formula of \Cref{lem:binomial_estimate}, and $\wt\vartheta[s,t]$ is that of the idealized law; the statement says that the tilt of a faithful state is the idealized tilt shifted by $\tau\frac{\sqrt{pN}^{k-1}}{\sqrt N}\beta[s,t]$, which is the tilt that reproduces the size shifts $\tau\sqrt N\sqrt{pN}^{k-1}\varepsilon[t]$ of the neighboring blocks (\Cref{def:universal-epsilon}).

When $\frac{1}{1-\theta}=k_0$ is an integer, the normalized perturbation $\sqrt{pN}^{\,k_0}/\sqrt N$ is of constant order, so the error in the linearization behind \Cref{thm:perturbed-evolution} need not be smaller than its main term. More explicitly, if
\[
\beta_0:=\frac{\sqrt{pN}^{\,k_0-1}}{\sqrt N},
\]
then the current part-size displacement $\beta_0N=\sqrt N\sqrt{pN}^{\,k_0-1}$ is still $o(N)$, but the resulting displacement of the next majority statistic, measured in its natural $\sqrt{pN}$ standard-deviation scale, is
\[
\beta_0\sqrt{pN}=\frac{\sqrt{pN}^{\,k_0}}{\sqrt N}=\Theta(1).
\]
Thus the next conditional split probabilities differ from their balanced values by a fixed amount, and multiplying by part sizes of order $N$ produces an order-$N$ displacement. At this terminal step $\wt n[s]=N\nu[s]+o(N)$, so it is enough to record a macroscopic signed gap from $N\nu[s]$, rather than a finer linear-response expansion around $\wt n[s]$. The following theorem makes this argument uniform and gives the macroscopic lead on the next day.

\begin{theorem}[The edge day]\label{thm:edge-day}
Fix $\theta\in(1/2,1)$ with $k_0:=\frac{1}{1-\theta}\in\mb{Z}$, and constants $T>1$ and $\delta>0$. There exist constants $T_1=T_1(\theta,T,\delta)\ge T$, $\delta_1=\delta_1(\theta,T,\delta)>0$, $\phi_1=\phi_1(\theta,T,\delta)\in(0,1/2)$, and $\zeta=\zeta(\theta,T)>0$ such that for all sufficiently large $N$, all $p\in(T^{-1}N^{-\theta},TN^{-\theta})$, all $\tau\in[T^{-1},T]$, and every $y=(\Pi,m,1)\in\mc{G}^{(p)}_{k_0}(T,\delta,\tau)$ with $\Pi\in\mr{Part}_{k_0}(V)$, $|V|=N$, the following hold, where $n[s]:=|\Pi[s]|$. There is a unique tilt $\lambda_y$ with $\mb{P}[\mc{I}_s(A_{s,\bullet})]>0$ and $n[s]\,\mb{E}[A_{s,t}\mid\mc{I}_s(A_{s,\bullet})]=m[s,t]$ for all $s,t\in\{0,1\}^{k_0}$, where $A_{s,\bullet}:=A_{s,\bullet}(n,\lambda_y)$, and it satisfies:
\begin{enumerate}[label=\textup{(\alph*)},ref=\textup{(\alph*)}]
\item\label{ed:admissible} $(y,\lambda_y)$ is $(T_1,\phi_1,p)$-locally admissible.
\item\label{ed:tilt} $\left|\lambda_y[s,t]-\Lambda_{N,p}(\gamma[s,t],\nu[t])\right|\le\frac{p}{\sqrt{pN}}N^{-\delta_1}$ for all $s,t\in\{0,1\}^{k_0}$.
\item\label{ed:sizes} The local template of $(y,\lambda_y)$ satisfies $(-1)^b\left(n^{\mr{loc}}[sb]-N\nu[sb]\right)\ge\zeta N$ for all $s\in\{0,1\}^{k_0}$ and $b\in\{0,1\}$.
\end{enumerate}
\end{theorem}

\subsection{The faithful trajectory}

The next proposition compares the errors of \Cref{thm:local-coarse-transition} with the tolerances of \Cref{def:faithful}; it is where the constant $C_{\ref{thm:local-coarse-transition}}$ is evaluated at the admissibility parameters produced by \Cref{thm:perturbed-evolution}.

\begin{proposition}[The local theorem preserves faithfulness]\label{prop:faithful-inclusion}
Fix $\theta\in(1/2,1)$, an integer $k$ with $1\le k<\frac{1}{1-\theta}$, and constants $T>1$, $\delta>0$; let $T_1,\delta_1,\phi_1$ be the constants of \Cref{thm:perturbed-evolution}, and put
\[
T_2:=2T_1,\qquad\delta_2:=\min\left\{\delta_1,\tfrac{1-\theta}{8}\right\}.
\]
For all sufficiently large $N$, all $p\in(T^{-1}N^{-\theta},TN^{-\theta})$, all $\tau\in[T^{-1},T]$, and every $y\in\mc{G}^{(p)}_k(T,\delta,\tau)$, with $\lambda_y$ the tilt of \Cref{thm:perturbed-evolution} and $\mc{L}^{(p)}_{k+1}(y,\lambda_y)$ the set of \Cref{thm:local-coarse-transition} formed with the constant $\Cof{thm:local-coarse-transition}{\theta,k,T_1,\phi_1}$,
\[
\mc{L}^{(p)}_{k+1}(y,\lambda_y)\subseteq\mc{G}^{(p)}_{k+1}(T_2,\delta_2,\tau).
\]
\end{proposition}
\begin{proof}
Let $y'\in\mc{L}^{(p)}_{k+1}(y,\lambda_y)$, and write $\Pi':=\on{part}(y')$, $m':=\on{edge}(y')$; by \cref{L:reg}, $y'=(\Pi',m',1)$. Write $C:=\Cof{thm:local-coarse-transition}{\theta,k,T_1,\phi_1}$ and let $(n^{\mr{loc}},\ell^{\mr{loc}},m^{\mr{loc}})$ be the local template of $(y,\lambda_y)$. For \ref{F:sizes} at day $k+1$, the triangle inequality, \cref{L:sizes}, and \Cref{thm:perturbed-evolution}\ref{pe:template} give, for all $s\in\{0,1\}^k$ and $b\in\{0,1\}$,
\[
\left||\Pi'[sb]|-\wt n[sb]-\tau\sqrt N\sqrt{pN}^{\,k}\varepsilon[sb]\right|\le C\sqrt N\log N+T_1\sqrt N\sqrt{pN}^{\,k}N^{-\delta_1}.
\]
Since $\delta_2\le\delta_1$, the second term is at most $T_1\sqrt N\sqrt{pN}^kN^{-\delta_2}$. For the first, because $k\ge1$ and $\delta_2\le(1-\theta)/8$,
\[
\sqrt{pN}^{\,k}N^{-\delta_2}\ge(pN)^{1/2}N^{-\delta_2}\ge T^{-1/2}N^{(1-\theta)/2-\delta_2}\ge T^{-1/2}N^{3(1-\theta)/8},
\]
so $C\log N\le T_1\sqrt{pN}^kN^{-\delta_2}$ for large $N$. Hence the left-hand side is at most $2T_1\sqrt N\sqrt{pN}^kN^{-\delta_2}=T_2\sqrt N\sqrt{pN}^kN^{-\delta_2}$. For \ref{F:edges} at day $k+1$, similarly, by \cref{L:edges} and \Cref{thm:perturbed-evolution}\ref{pe:template},
\begin{multline*}
\left|m'[sb,tc]-\wt m[sb,tc]\left(1+\tau\frac{\sqrt{pN}^{\,k}}{\sqrt N}\left(\frac{\varepsilon[sb]}{\nu[sb]}+\frac{\varepsilon[tc]}{\nu[tc]}\right)\right)\right|\\
\le CN^2p\sqrt{\frac{(pN)^{1/7}}{N}}\log N+T_1\frac{\sqrt{pN}^{\,k}}{\sqrt N}N^{-\delta_1}N^2p .
\end{multline*}
The first term is at most $T_1\frac{\sqrt{pN}^k}{\sqrt N}N^{-\delta_2}N^2p$ as soon as $C(pN)^{1/14}\log N\le T_1\sqrt{pN}^kN^{-\delta_2}$, and $\sqrt{pN}^k/(pN)^{1/14}\ge(pN)^{3/7}\ge T^{-3/7}N^{3(1-\theta)/7}$, while $\delta_2\le(1-\theta)/8<3(1-\theta)/7$; so this holds for large $N$. Both conditions of \Cref{def:faithful} therefore hold with the constants $(T_2,\delta_2)$, and $y'\in\mc{G}^{(p)}_{k+1}(T_2,\delta_2,\tau)$.
\end{proof}

\begin{proposition}[Day one]\label{prop:day-one}
Fix $\theta\in(1/2,1)$ and $T>1$, and put $\delta_0:=\frac{1-\theta}{4}$. For all sufficiently large $N$, all $p\in(T^{-1}N^{-\theta},TN^{-\theta})$, and all $\tau\in[T^{-1},T]$, the following holds. Let $V$ be a set with $|V|=N$, let $c\colon V\to\{\pm1\}$ satisfy $|c^{-1}(1)|=\lfloor N/2\rfloor+\lfloor\tau\sqrt N\rfloor$, sample $G\sim\mb{G}(V,p)$, and let $\mf{g}=(G,c)$. Then
\[
\mb{P}\left[\rho_1^{(p)}\left(\sigma_1(\mf{g})\right)\in\mc{G}^{(p)}_1(1,\delta_0,\tau)\right]=1-o(1).
\]
\end{proposition}
\begin{proof}
Write $\Pi:=\Pi_1(\mf g)$, $\mbf d:=\delta(\Pi,G)$, $n[s]:=|\Pi[s]|$, and $y:=\rho_1^{(p)}(\sigma_1(\mf g))=(\Pi,m,\one_{\kappa_p(\Pi,\mbf d)})$ with $m[s,t]=\sum_{x\in\Pi[s]}\mbf d[x,t]$. By definition, $\Pi[0]=c^{-1}(1)$ and $\Pi[1]=c^{-1}(-1)$, so $n[0]=a:=\lfloor N/2\rfloor+\lfloor\tau\sqrt N\rfloor$ and $n[1]=N-a$. At day $1$ we have $\wt n[0]=\wt n[1]=\lfloor N/2\rfloor$, $\wt m[s,t]=p\wt n[s](\wt n[t]-\one_{s=t})$, $\nu[0]=\nu[1]=\frac12$, and $\varepsilon[0]=1=-\varepsilon[1]$, so $\varepsilon[s]/\nu[s]=2(-1)^s$ for $s\in\{0,1\}$.

\emph{Sizes.} Deterministically, $|n[0]-\wt n[0]-\tau\sqrt N|=|\lfloor\tau\sqrt N\rfloor-\tau\sqrt N|<1$ and $|n[1]-\wt n[1]+\tau\sqrt N|=|N-2\lfloor N/2\rfloor-\lfloor\tau\sqrt N\rfloor+\tau\sqrt N|<2$, while the tolerance in \ref{F:sizes} is $\sqrt NN^{-\delta_0}\ge2$ for large $N$.

\emph{Edge counts.} For $s\ne t$, $m[s,t]$ is the number of edges of $G$ between $\Pi[s]$ and $\Pi[t]$, a $\mr{Bin}(n[s]n[t],p)$ variable; and $m[s,s]$ is twice the number of edges inside $\Pi[s]$, twice a $\mr{Bin}(\binom{n[s]}{2},p)$ variable. In both cases $\mb E\,m[s,t]=pn[s](n[t]-\one_{s=t})$ and $\mr{Var}\,m[s,t]\le2N^2p$, so by Chernoff's inequality, with probability $1-o(1)$,
\begin{equation}\label{eq:day-one-chernoff}
\left|m[s,t]-pn[s]\left(n[t]-\one_{s=t}\right)\right|\le2N\sqrt p\log N\qquad\text{for all }s,t\in\{0,1\}.
\end{equation}
Next, $n[s]=\wt n[0]+(-1)^s\lfloor\tau\sqrt N\rfloor+\one_{s=1}(N-2\lfloor N/2\rfloor)$, and since $\lfloor\tau\sqrt N\rfloor/\wt n[0]=2\tau/\sqrt N+O(1/N)$ (recall $\tau\le T$), both $n[s]$ and $n[s]-1$ equal $(\wt n[0]-\one)\big(1+(-1)^s\frac{2\tau}{\sqrt N}+O(\frac1N)\big)$ with the corresponding $\one\in\{0,1\}$. Hence
\begin{align*}
pn[s]\left(n[t]-\one_{s=t}\right)&=\wt m[s,t]\left(1+\frac{2\tau}{\sqrt N}\left((-1)^s+(-1)^t\right)+O\left(\frac1N\right)\right)\\
&=\wt m[s,t]\left(1+\frac{\tau}{\sqrt N}\left(\frac{\varepsilon[s]}{\nu[s]}+\frac{\varepsilon[t]}{\nu[t]}\right)\right)+O(Np),
\end{align*}
the cross term $4\tau^2/N$ being part of the $O(1/N)$. Combining with \eqref{eq:day-one-chernoff}, the left-hand side of \ref{F:edges} is at most $2N\sqrt p\log N+O(Np)\le3N\sqrt p\log N$, while the tolerance is $N^{-\delta_0}N^{3/2}p$. Since $\sqrt p\ge T^{-1/2}N^{-\theta/2}$, the ratio of the tolerance to $N\sqrt p\log N$ is at least $T^{-1/2}N^{(1-\theta)/2-\delta_0}/\log N=T^{-1/2}N^{(1-\theta)/4}/\log N$, which tends to infinity; so \ref{F:edges} holds for large $N$ on the event \eqref{eq:day-one-chernoff}.

\emph{Regularity.} For $x\in V$ and $t\in\{0,1\}$, $\mbf d[x,t]=\deg^G_{\Pi[t]}x\sim\mr{Bin}(n[t]-\one_{x\in\Pi[t]},p)$, whose mean differs from $pn[t]$ by at most $p$. By Chernoff's inequality, $\mb P[|\mbf d[x,t]-pn[t]|>(pN)^{4/7}]\le2\exp(-c(pN)^{1/7})$ for an absolute constant $c>0$ and large $N$, and a union bound over the $2N$ pairs $(x,t)$ gives $\mb P[\neg\kappa_p(\Pi,\mbf d)]\le4N\exp(-c(pN)^{1/7})=o(1)$, since $(pN)^{1/7}\ge T^{-1/7}N^{(1-\theta)/7}$. On the intersection of the two events, $y=(\Pi,m,1)\in\mc G^{(p)}_1(1,\delta_0,\tau)$.
\end{proof}

\begin{theorem}[The faithful trajectory]\label{thm:faithful-trajectory}
Fix $\theta\in(1/2,1)$ and $T>1$, and let $K_\theta:=\lceil\frac{1}{1-\theta}\rceil$, so that $K_\theta=D_\theta$ if $\frac1{1-\theta}\notin\mb{Z}$ and $K_\theta=k_0=D_\theta-1$ otherwise. There exist constants $T^\ast=T^\ast(\theta,T)>1$ and $\delta^\ast=\delta^\ast(\theta,T)>0$ such that for all sufficiently large $N$, all $p\in(T^{-1}N^{-\theta},TN^{-\theta})$, and all $\tau\in[T^{-1},T]$, the following holds. Let $V$, $c$, $G$, and $\mf{g}=(G,c)$ be as in \Cref{prop:day-one}. Then with probability $1-o(1)$:
\begin{enumerate}[label=\textup{(\roman*)},ref=\textup{(\roman*)}]
\item\label{traj:faithful} $\rho_k^{(p)}(\sigma_k(\mf{g}))\in\mc{G}^{(p)}_k(T^\ast,\delta^\ast,\tau)$ for every $1\le k\le K_\theta$;
\item\label{traj:edge} if $\frac{1}{1-\theta}=k_0\in\mb{Z}$, then in addition $(-1)^{s[k_0]}\left(|\Pi_{k_0+1}(\mf{g})[s]|-N\nu[s]\right)\ge\frac{\zeta}{2}N$ for every $s\in\{0,1\}^{k_0+1}$, where $\zeta=\zeta(\theta,T^\ast)$ is the constant of \Cref{thm:edge-day} for the parameters $(\theta,T^\ast,\delta^\ast)$; we arrange $T^\ast\ge T$, so that these parameters cover the ranges of $p$ and $\tau$.
\end{enumerate}
\end{theorem}
\begin{proof}
Write $y_j:=\rho_j^{(p)}(\sigma_j(\mf g))$ for the day-$j$ coarse state of $\mf g$, and $\mc G_j(T',\delta'):=\mc G^{(p)}_j(T',\delta',\tau)$. Note that $\mc G_j(T',\delta')\subseteq\mc G_j(T'',\delta'')$ whenever $T'\le T''$ and $\delta'\ge\delta''$, since the tolerances in \Cref{def:faithful} increase.

\emph{Constants.} Define $(T^\circ_j,\delta^\circ_j)$ for $1\le j\le K_\theta$ recursively: $T^\circ_1:=T$ and $\delta^\circ_1:=\frac{1-\theta}{4}$; and for $1\le j<K_\theta$ (so that $j<\frac1{1-\theta}$), let $(T_2,\delta_2)$ be the constants of \Cref{prop:faithful-inclusion} for the parameters $(\theta,j,T^\circ_j,\delta^\circ_j)$, and put $T^\circ_{j+1}:=\max\{T_2,T^\circ_j\}$ and $\delta^\circ_{j+1}:=\delta_2$. These depend only on $\theta$ and $T$; $T^\circ_j$ is nondecreasing in $j$, and $\delta^\circ_j$ is nonincreasing because $\delta_2\le\delta_1<\delta$ in \Cref{prop:faithful-inclusion} and \Cref{thm:perturbed-evolution}. Set $T^\ast:=T^\circ_{K_\theta}\ge T$ and $\delta^\ast:=\delta^\circ_{K_\theta}$; then $\mc G_j(T^\circ_j,\delta^\circ_j)\subseteq\mc G_j(T^\ast,\delta^\ast)$ for every $j\le K_\theta$.

\emph{Induction.} We show by induction on $j$ that $\mb P[y_j\in\mc G_j(T^\circ_j,\delta^\circ_j)]=1-o(1)$ for $1\le j\le K_\theta$. For $j=1$ this is \Cref{prop:day-one}, since $1\le T^\circ_1$. Let $1\le j<K_\theta$ and suppose the claim holds for $j$. Fix $y\in\mc G_j(T^\circ_j,\delta^\circ_j)$ with $\mb P[y_j=y]>0$. By \Cref{prop:coarse-one-step}, $y$ is $p$-attainable and $\mb P[y_{j+1}\in B\mid y_j=y]=\overline K^{(p)}_j(y,B)$ for every $B\subseteq\mc Y_{j+1}(V)$. Since $T^\circ_j\ge T$, the ranges $p\in(T^{-1}N^{-\theta},TN^{-\theta})$ and $\tau\in[T^{-1},T]$ are contained in those of \Cref{thm:perturbed-evolution} with the parameters $(\theta,j,T^\circ_j,\delta^\circ_j)$, and that theorem provides the tilt $\lambda_y$ and constants $T_1,\phi_1$, depending only on $\theta,T$, such that $(y,\lambda_y)$ is $(T_1,\phi_1,p)$-locally admissible. Since $T_1\ge T^\circ_j\ge T$, the range of $p$ is also contained in that of \Cref{thm:local-coarse-transition} with the parameters $(\theta,j,T_1,\phi_1)$, which gives $\overline K^{(p)}_j(y,\mc L^{(p)}_{j+1}(y,\lambda_y))=1-o(1)$, where, by the convention of \Cref{sec:conventions}, the $o(1)$ is a function of $N$ alone, uniform over all $p$ in the range and all admissible $(y,\lambda)$. By \Cref{prop:faithful-inclusion} with the parameters $(\theta,j,T^\circ_j,\delta^\circ_j)$, $\mc L^{(p)}_{j+1}(y,\lambda_y)\subseteq\mc G_{j+1}(T_2,\delta_2)\subseteq\mc G_{j+1}(T^\circ_{j+1},\delta^\circ_{j+1})$. Therefore
\begin{align*}
\mb P\left[y_{j+1}\in\mc G_{j+1}(T^\circ_{j+1},\delta^\circ_{j+1})\right]&\ge\sum_{\substack{y\in\mc G_j(T^\circ_j,\delta^\circ_j)\\ \mb P[y_j=y]>0}}\mb P[y_j=y]\,\overline K^{(p)}_j\left(y,\mc L^{(p)}_{j+1}(y,\lambda_y)\right)\\
&\ge(1-o(1))\,\mb P\left[y_j\in\mc G_j(T^\circ_j,\delta^\circ_j)\right]=1-o(1),
\end{align*}
which completes the induction. Since there are finitely many days, a union bound gives \ref{traj:faithful}: with probability $1-o(1)$, $y_j\in\mc G_j(T^\circ_j,\delta^\circ_j)\subseteq\mc G_j(T^\ast,\delta^\ast)$ for all $1\le j\le K_\theta$.

\emph{The integer case.} Suppose $\frac1{1-\theta}=k_0\in\mb Z$, so $K_\theta=k_0$. Fix $y\in\mc G_{k_0}(T^\ast,\delta^\ast)$ with $\mb P[y_{k_0}=y]>0$. \Cref{thm:edge-day} with the parameters $(\theta,T^\ast,\delta^\ast)$, whose ranges contain ours since $T^\ast\ge T$, provides $\lambda_y$ and constants $T_1',\phi_1',\zeta$ such that $(y,\lambda_y)$ is $(T_1',\phi_1',p)$-locally admissible and its local template satisfies $(-1)^b(n^{\mr{loc}}[sb]-N\nu[sb])\ge\zeta N$ for all $s\in\{0,1\}^{k_0}$, $b\in\{0,1\}$. Since $T_1'\ge T^\ast\ge T$, the density range is contained in that of \Cref{thm:local-coarse-transition} with the parameters $(\theta,k_0,T_1',\phi_1')$. By that theorem and \Cref{prop:coarse-one-step}, conditionally on $y_{k_0}=y$, with probability $1-o(1)$ the state $y_{k_0+1}$ lies in $\mc L^{(p)}_{k_0+1}(y,\lambda_y)$, and then \cref{L:sizes} gives, for every $s\in\{0,1\}^{k_0}$ and $b\in\{0,1\}$,
\[
(-1)^b\left(|\Pi_{k_0+1}(\mf g)[sb]|-N\nu[sb]\right)\ge\zeta N-\Cof{thm:local-coarse-transition}{\theta,k_0,T_1',\phi_1'}\sqrt N\log N\ge\tfrac{\zeta}{2}N
\]
for large $N$, where we used $\on{part}(y_{k_0+1})=\Pi_{k_0+1}(\mf g)$. Summing over $y$ as above and combining with \ref{traj:faithful} at day $k_0$ gives \ref{traj:edge}, since for $s\in\{0,1\}^{k_0+1}$ the sign $(-1)^{s[k_0]}$ is $(-1)^b$ when $s$ is written as $s'b$.
\end{proof}

\subsection{The lead}

For a politicized graph $\mf{g}$ on $V$ and $k\ge1$, define the \emph{lead} on day $k$ by
\begin{equation}\label{eq:lead}
\Delta_k(\mf{g}):=\sum_{s\in\{0,1\}^k}(-1)^{s[k-1]}\left|\Pi_k(\mf{g})[s]\right|=\left|c_k(\mf{g})^{-1}(1)\right|-\left|c_k(\mf{g})^{-1}(-1)\right|,
\end{equation}
the number of vertices holding opinion $+1$ on day $k$ minus the number holding $-1$; the two expressions agree because $x\in\Pi_k(\mf{g})[s]$ has $c_k(\mf{g})[x]=(-1)^{s[k-1]}$.

\begin{corollary}[The lead at the end of the expansion phase]\label{cor:lead}
Fix $\theta\in(1/2,1)$ and $T>1$, and let $K:=\lfloor\frac{1}{1-\theta}\rfloor+1=D_\theta$. For all sufficiently large $N$, all $p\in(T^{-1}N^{-\theta},TN^{-\theta})$, and all $\tau\in[T^{-1},T]$, with $V$, $c$, $G$, $\mf{g}$ as in \Cref{prop:day-one},
\[
\mb{P}\left[\Delta_K(\mf{g})\ge\frac{N}{\sqrt{pN}}\log N\right]=1-o(1).
\]
\end{corollary}
\begin{proof}
Let $T^\ast,\delta^\ast$ be the constants of \Cref{thm:faithful-trajectory}, and let $E$ be its event of probability $1-o(1)$. It suffices to show that $\Delta_K(\mf g)\ge\frac{N}{\sqrt{pN}}\log N$ on $E$ for large $N$.

Suppose first that $\frac1{1-\theta}\notin\mb Z$, so that $K=K_\theta$. On $E$, condition \ref{F:sizes} at day $K$ with the constants $(T^\ast,\delta^\ast)$ gives, for every $s\in\{0,1\}^K$,
\[
|\Pi_K(\mf g)[s]|=\wt n[s]+\tau\sqrt N\sqrt{pN}^{\,K-1}\varepsilon[s]+O\left(\sqrt N\sqrt{pN}^{\,K-1}N^{-\delta^\ast}\right).
\]
Summing with the signs $(-1)^{s[K-1]}$: the first sum $\sum_s(-1)^{s[K-1]}\wt n[s]$ vanishes, because pairing $s$ with $\ol s$ and using $\wt n[\ol s]=\wt n[s]$ (\Cref{thm:idealized-process}\ref{ideal:symmetry}) shows that it equals its own negative; the second is $\tau\sqrt N\sqrt{pN}^{K-1}S_K$ with $S_K:=\sum_s(-1)^{s[K-1]}\varepsilon[s]=2\sum_{s'\in\{0,1\}^{K-1}}\varepsilon[s'0]>0$ by \Cref{cor:sign-identities}\ref{sign:lead}. Hence, for large $N$,
\[
\Delta_K(\mf g)\ge\tau S_K\sqrt N\sqrt{pN}^{\,K-1}-O\left(\sqrt N\sqrt{pN}^{\,K-1}N^{-\delta^\ast}\right)\ge\frac{S_K}{2T}\sqrt N\sqrt{pN}^{\,K-1}.
\]
Finally $\sqrt N\sqrt{pN}^{K-1}\big/\big(\frac{N}{\sqrt{pN}}\log N\big)=\sqrt{pN}^{\,K}/(\sqrt N\log N)\ge T^{-K/2}N^{(K(1-\theta)-1)/2}/\log N\to\infty$, because $K=\lfloor\frac1{1-\theta}\rfloor+1>\frac1{1-\theta}$ gives $K(1-\theta)>1$. So $\Delta_K(\mf g)\ge\frac{N}{\sqrt{pN}}\log N$ for large $N$.

Suppose now that $\frac1{1-\theta}=k_0\in\mb Z$, so that $K=k_0+1$. On $E$, \Cref{thm:faithful-trajectory}\ref{traj:edge} gives
\[
(-1)^{s[K-1]}\left(|\Pi_K(\mf g)[s]|-N\nu[s]\right)\ge\tfrac\zeta2N\qquad\text{for every }s\in\{0,1\}^K .
\]
Sum over the $2^K$ strings $s$ and use $\sum_s(-1)^{s[K-1]}\nu[s]=0$, which is \Cref{cor:sign-identities}\ref{sign:nu} with $r=K$:
\[
\Delta_K(\mf g)=\sum_{s\in\{0,1\}^K}(-1)^{s[K-1]}|\Pi_K(\mf g)[s]|\ge N\sum_s(-1)^{s[K-1]}\nu[s]+2^K\frac{\zeta}{2}N=2^{k_0}\zeta N,
\]
which exceeds $\frac{N}{\sqrt{pN}}\log N$ for large $N$.
\end{proof}

\section{Proof of the main theorem}\label{sec:main-proof}

The proof has two phases. In the expansion phase, \Cref{cor:lead} gives a lead of at least $\frac{N}{\sqrt{pN}}\log N$ after $K=\lfloor\frac1{1-\theta}\rfloor+1$ days. In the contraction phase, one further update reduces the minority to at most $N/5$ vertices, after which each update contracts it by a factor at most $16C_{\mr{J}}^2/(pN)$. This is a deterministic consequence of two properties that $\mb{G}(V,p)$ has with high probability, minimum degree and jumbledness, using the argument of \cite{CKLT21}.

\subsection{The contraction phase}

\begin{definition}[Jumbled graphs]\label{def:jumbled}
Let $G$ be a graph on $V$, let $p\in(0,1)$, and let $\beta>0$. For $U,W\subseteq V$ let $e_G(U,W)$ be the number of pairs $(u,w)\in U\times W$ with $uw\in E(G)$; for disjoint $U,W$ this is the number of edges between $U$ and $W$, and in general $e_G(U,W)=\sum_{u\in U}\deg^G_Wu$. The graph $G$ is \emph{$(p,\beta)$-jumbled} if
\[
\left|e_G(U,W)-p|U||W|\right|\le\beta\sqrt{|U||W|}\qquad\text{for all }U,W\subseteq V .
\]
\end{definition}

\begin{lemma}[Pseudorandomness of $\mb{G}(V,p)$]\label{lem:cklt-jumbled}
There is an absolute constant $C_{\mr{J}}>0$ such that the following holds. Fix a real parameter $\theta\in(1/2,1)$ and a constant $T>1$. For all sufficiently large $N\in\mb{Z}$ and $p\in(T^{-1}N^{-\theta},TN^{-\theta})$, if $V$ is a set with $|V|=N$ and $G\sim\mb{G}(V,p)$, then with probability $1-o(1)$ every vertex of $G$ has degree at least $0.9\,pN$ and $G$ is $\big(p,C_{\mr{J}}\sqrt{pN}\big)$-jumbled.
\end{lemma}
\begin{proof}
The degree of a vertex is a $\mr{Bin}(N-1,p)$ variable with mean $p(N-1)\ge0.95pN$ for large $N$, so by Chernoff's inequality it is less than $0.9pN$ with probability at most $\exp(-cpN)$ for an absolute constant $c>0$; a union bound over the $N$ vertices gives probability at most $N\exp(-cpN)=o(1)$, because $pN\ge T^{-1}N^{1-\theta}\gg\log N$. For jumbledness we use \cite[Lemma~4.1]{CKLT21}, restating \cite[Corollary~2.3]{KS06}, in the range $\log^2N\le pN$ and $p\le0.99$: with probability $1-o(1)$, $\mb{G}(V,p)$ is $(p,\beta)$-jumbled with $\beta=O(\sqrt{pN})$. Both restrictions hold uniformly in our density window for large $N$, since $pN\ge T^{-1}N^{1-\theta}\gg\log^2N$ and $p\le TN^{-\theta}=o(1)$.
\end{proof}

The next lemma uses the argument of \cite[Lemma~4.2]{CKLT21}, with numerical constants sufficient for the minimum-degree bound above. Its update rule retains the current opinion in a tie, as in \Cref{sec:setup}.

\begin{lemma}[One day of majority dynamics on a jumbled graph]\label{lem:cklt-contraction}
Let $G$ be a $(p,\beta)$-jumbled graph on a set $V$ with $|V|=N\ge1$, let $c\colon V\to\{\pm1\}$ be any coloring, let $\mf{g}=(G,c)$, and let $j\ge1$. Write $M_i:=c_i(\mf{g})^{-1}(-1)$.
\begin{enumerate}[label=\textup{(\roman*)},ref=\textup{(\roman*)}]
\item\label{cklt:jump} If $\Delta_j(\mf{g})\ge10\beta/p$, then $|M_{j+1}|\le N/5$.
\item\label{cklt:contract} If every vertex of $G$ has degree at least $0.9pN$ and $|M_j|\le N/5$, then
\[
|M_{j+1}|\le\frac{16\beta^2}{(pN)^2}|M_j|.
\]
\end{enumerate}
\end{lemma}
\begin{proof}
Write $U:=M_{j+1}$, $x:=|U|$, and $m:=|M_j|$, so that $\Delta_j(\mf{g})=N-2m$. A vertex $v\in U$ has at least as many neighbors in $M_j$ as in its complement: either strictly more of its neighbors hold $-1$ on day $j$, or there is a tie and $v$ already holds $-1$. Thus
\[
\tfrac12\deg^G_Vv\le\deg^G_{M_j}v,
\qquad\text{and hence}\qquad
\tfrac12 e_G(U,V)\le e_G(U,M_j).
\]
Both conclusions are immediate if $x=0$, so assume $x>0$.

\ref{cklt:jump}: Jumbledness gives
\[
\tfrac12\bigl(pxN-\beta\sqrt{xN}\bigr)
\le\tfrac12 e_G(U,V)
\le e_G(U,M_j)
\le pxm+\beta\sqrt{xm}.
\]
Using $m\le N$ and $\Delta_j(\mf{g})=N-2m$, we obtain
\[
10\beta x\le px\,\Delta_j(\mf{g})
\le\beta\sqrt{xN}+2\beta\sqrt{xm}
\le3\beta\sqrt{xN}.
\]
Since $\beta>0$, squaring and canceling $x>0$ gives $100x\le9N$, and therefore $x\le9N/100\le N/5$.

\ref{cklt:contract}: The minimum-degree bound and the same majority inequality give
\[
\frac{9}{20}pNx\le e_G(U,M_j)
\le pxm+\beta\sqrt{xm}.
\]
Since $m\le N/5$,
\[
\frac{pN}{4}x
\le p\left(\frac{9N}{20}-m\right)x
\le\beta\sqrt{xm}.
\]
Both sides are nonnegative; squaring and canceling $x>0$ yields $(pN)^2x\le16\beta^2m$, proving the assertion because $pN>0$.
\end{proof}

\subsection{Proof of \Cref{thm:succinct_main_result} and \Cref{cor:random-opinions}}

\begin{proof}[Proof of \Cref{thm:succinct_main_result}]
Let $K:=\lfloor\frac1{1-\theta}\rfloor+1$, so that $2K+1=k$. Let $E_1$ be the event $\Delta_K(\mf{g})\ge\frac{N}{\sqrt{pN}}\log N$ of \Cref{cor:lead}, and $E_2$ the event of \Cref{lem:cklt-jumbled} that $G$ has minimum degree at least $0.9pN$ and is $(p,\beta)$-jumbled with $\beta:=C_{\mr{J}}\sqrt{pN}$. Both have probability $1-o(1)$, hence so does $E_1\cap E_2$, and it suffices to show that $c_k(\mf{g})\equiv1$ on this event for large $N$.

Write $M_j:=c_j(\mf{g})^{-1}(-1)$. For large $N$ we have $\log N\ge10C_{\mr{J}}$, so on $E_1\cap E_2$,
\[
\Delta_K(\mf{g})\ge\frac{N}{\sqrt{pN}}\log N
\ge\frac{10C_{\mr{J}}N}{\sqrt{pN}}
=\frac{10\beta}{p}.
\]
Part~\ref{cklt:jump} of \Cref{lem:cklt-contraction} therefore gives $|M_{K+1}|\le N/5$. Put
\[
\rho:=\frac{16\beta^2}{(pN)^2}
=\frac{16C_{\mr{J}}^2}{pN}
\le\frac{16C_{\mr{J}}^2T}{N^{1-\theta}},
\]
so $0<\rho\le1$ for large $N$. Successive applications of \ref{cklt:contract} give
\[
|M_{K+1+r}|\le\rho^r\frac N5\le\frac N5
\qquad(0\le r\le K).
\]
The second inequality keeps the hypothesis of \ref{cklt:contract} valid at every step. Since $K(1-\theta)>1$, taking $r=K$ yields
\[
\left|c_k(\mf{g})^{-1}(-1)\right|
=|M_{2K+1}|
\le\frac{N\rho^K}{5}
\le\frac{(16C_{\mr{J}}^2T)^K}{5}N^{1-K(1-\theta)}<1
\]
for large $N$. Since the left-hand side is a nonnegative integer, it is $0$: every vertex has opinion $+1$ on day $k$.
\end{proof}

\begin{proof}[Proof of \Cref{cor:random-opinions}]
Let $a:=|c^{-1}(1)|\sim\mr{Bin}(N,1/2)$ and $\tau:=(a-\lfloor N/2\rfloor)/\sqrt N$, so that $a=\lfloor N/2\rfloor+\lfloor\tau\sqrt N\rfloor$ exactly, since $\tau\sqrt N$ is an integer. Fix $T'>1$ and let $\epsilon_{T'}(N)$ denote the $o(1)$ of \Cref{thm:succinct_main_result} for the parameters $(\theta,T'+1)$; by the convention of \Cref{sec:conventions} it is a function of $N$ alone, uniform over $p\in((T'+1)^{-1}N^{-\theta},(T'+1)N^{-\theta})$ and $\tau\in[(T'+1)^{-1},T'+1]$. Suppose $T'\ge T$, so that our range of $p$ is contained in this one.

Condition on $c$. Since $G$ is independent of $c$, its conditional law is $\mb{G}(V,p)$. If $\tau\in[T'^{-1},T']$, then $(G,c)$ is in the setting of \Cref{thm:succinct_main_result} with the parameters $(\theta,T'+1)$, so $\mb{P}[c_k(\mf{g})\not\equiv1\mid c]\le\epsilon_{T'}(N)$. If $\tau\le-T'^{-1}$, consider $\mf{g}':=(G,-c)$: it has $|(-c)^{-1}(1)|=N-a=\lfloor N/2\rfloor+\lfloor\tau'\sqrt N\rfloor$ with $\tau':=(N-a-\lfloor N/2\rfloor)/\sqrt N\in[-\tau,-\tau+N^{-1/2}]\subseteq[T'^{-1},T'+1]$ when also $\tau\ge-T'$, so $\mb{P}[c_k(\mf{g}')\not\equiv1\mid c]\le\epsilon_{T'}(N)$; and majority dynamics commutes with the global flip, $c_t((G,-c))=-c_t((G,c))$ for all $t$, because the update rule of \Cref{sec:setup} is odd under $c\mapsto-c$ (the sign of the neighborhood sum flips, and a tie remains a tie). Hence
\[
\mb{P}\left[c_k(\mf{g})\text{ is not constant}\right]\le\mb{P}\left[|\tau|<T'^{-1}\right]+\mb{P}\left[|\tau|>T'\right]+\epsilon_{T'}(N).
\]
Now $\tau\sqrt N=a-\lfloor N/2\rfloor$ with $a\sim\mr{Bin}(N,1/2)$: by Chebyshev's inequality $\mb{P}[|\tau|>T']\le\mb{P}[|a-N/2|>T'\sqrt N-1]\le\frac{N/4}{(T'\sqrt N-1)^2}\le\frac{1}{2T'^2}$ for large $N$, and since $\max_i\mb{P}[a=i]=\binom{N}{\lfloor N/2\rfloor}2^{-N}\le C/\sqrt N$ for an absolute constant $C$ by Stirling's formula, $\mb{P}[|\tau|<T'^{-1}]\le(2T'^{-1}\sqrt N+1)\,C/\sqrt N\le3CT'^{-1}$ for large $N$. Therefore
\[
\limsup_{N\to\infty}\ \sup_{p}\ \mb{P}\left[c_k(\mf{g})\text{ is not constant}\right]\le3CT'^{-1}+\frac{1}{2T'^2}
\]
for every $T'\ge T$, the supremum being over $p\in(T^{-1}N^{-\theta},TN^{-\theta})$; letting $T'\to\infty$ shows that the probability is $o(1)$, uniformly in $p$. Finally, if $c_k(\mf{g})$ is constant then so is $c_t(\mf{g})$ for every $t\ge k$: every vertex then sees only neighbors of its own opinion, so the neighborhood sum has the sign of that opinion or vanishes, and in either case the opinion is kept.
\end{proof}

\subsection{The uniform density range}\label{sec:uniform-proof}

Only the main steps of the uniform-density argument are recorded here; every detail is checked in the accompanying \href{https://github.com/gopalkgoel/sparse-majority-dynamics-lean}{Lean formalization}. Put
\[
\alpha_0:=1-\theta,\qquad S:=\sqrt{pN},\qquad
K:=\left\lfloor\frac1{\alpha_0}\right\rfloor+1.
\]

\emph{Reduction to the sparse range.} Densities of order $N^{-1/2}$ and above are covered by the dense-range argument of Fountoulakis, Kang, and Makai \cite{FKM20}, whose estimates are formalized alongside the present argument and yield the following fixed-initial-coloring statement: there is $L=L(T,\eta)$ such that, for all sufficiently large $N$, every $p\ge L/\sqrt N$ (including $p=1$), and every coloring $c$ as in \Cref{thm:uniform-density}, unanimity holds after four updates with probability at least $1-\eta$. Since $\theta>1/2$ gives $K\ge3$ and unanimity persists once reached, this proves \Cref{thm:uniform-density} for $p\ge L/\sqrt N$. Set $B:=\max\{2T,L+1\}$. It remains to prove the theorem for
\[
T^{-1}N^{-\theta}\le p\le BN^{-1/2},
\]
and the rest of this subsection concerns only this sparse interval. The argument below is run with $B$ in place of $T$, so that it covers the density lower endpoint $T^{-1}N^{-\theta}\ge B^{-1}N^{-\theta}$ and the bias range $[T^{-1},T]\subseteq[B^{-1},B]$; the two thresholds for $N$ are combined by taking the larger one, giving a single $N_0$ independent of $p$ and of $c$.

\emph{The sparse interval.} The local estimates and the finite-horizon idealized evolution hold uniformly under independent lower and upper bounds on $p$. The balanced finite reference process and its exact color-flip cancellation supply the reference trajectory throughout this interval.

The point requiring care is when to leave the linear-response regime. The class-size displacement on day $d$ has scale $\sqrt N S^{d-1}$, and the normalized shift seen by the next majority decision has scale
\[
a_d:=\frac{S^d}{\sqrt N}.
\]
For a fixed power law this is $N^{(d\alpha_0-1)/2}$ up to a constant factor. At a reciprocal exponent, say $\alpha_0=1/m$, the quantity $a_m$ has constant order and the critical-day argument applies. For a density varying with $N$, however, its effective exponent $\alpha_N:=\log(pN)/\log N$ can approach $1/m$ at an arbitrary rate. Then
\[
a_m=N^{(m\alpha_N-1)/2}
\]
may tend to zero, remain bounded, or tend to infinity, even subpolynomially. Thus no predetermined day is uniformly guaranteed to be safely linear or already critical.

Instead we stop according to the response itself. On every day $k$ before the stopping point, the actual trajectory is faithful in the sense of \Cref{def:faithful}: the number of vertices in each history class $s$ equals its idealized value $\wt n[s]$ plus the first-order response $\tau\sqrt N S^{k-1}\varepsilon[s]$, and each inter-class edge count equals its idealized value corrected by the corresponding first-order relative response, in each case up to an error of at most $CN^{-\delta}$ times the scale of that response, for constants $C$ and $\delta>0$ depending only on $\theta$ and $B$. This is the induction of \Cref{thm:faithful-trajectory}, run on the present interval. With $\delta$ fixed, choose
\[
r:=\min\left\{\frac14,\ \delta,\ \frac{K\alpha_0-1}{2}\right\}>0.
\]
This choice gives $r/4<\delta$ and $(K-r)\alpha_0>1$. Stop on the first day $d$ for which $a_d\ge S^{r-1}$. At the upper end of the sparse interval the first response is still below this threshold, since
\[
\frac{a_1}{S^{r-1}}=\frac{S^{2-r}}{\sqrt N}=O(N^{-r/4}).
\]
At the lower density endpoint the threshold has been crossed by day $K-1$, because
\[
\frac{a_{K-1}}{S^{r-1}}
=\frac{S^{K-r}}{\sqrt N}
\ge T^{-(K-r)/2}N^{((K-r)\alpha_0-1)/2}\longrightarrow\infty.
\]
These two inequalities are uniform over the whole sparse interval. By minimality, the stopping response lies in the window
\[
S^{r-1}\le a_d<S^r,
\qquad d\le K-1.
\]

On this window the nonlinear Gaussian threshold estimate replaces the linear expansion and gives a lead of order $N\min\{a_d,1\}$ after the next update. The relative error is bounded, up to constants and a fixed logarithmic power, by
\[
\frac{(\log N)^\ell}{S^r}+S^rN^{-\delta}=o(1).
\]
Here the lower density bound makes the first term vanish, while $S=O(N^{1/4})$ and $r/4<\delta$ handle the second. Moreover $S\min\{a_d,1\}\ge S^r\gg\log N$, so this lead exceeds the cleanup threshold $N\log N/S$. Only the next class sizes are needed at this terminal update; there is no need to propagate a further faithful edge-count state. Since $d\le K-1$, the cleanup lead is present by day $K$. The jump and $K$ contraction updates of \Cref{lem:cklt-contraction} therefore give unanimity by day
\[
2K+1=2\left\lfloor\frac1{1-\theta}\right\rfloor+3.
\]

For random initial opinions, repeat the binomial window and color-flip argument of \Cref{cor:random-opinions} with \Cref{thm:uniform-density} in place of \Cref{thm:succinct_main_result}; enlarging the window parameter only enlarges the density interval.

\appendix
\addtocontents{toc}{\protect\setcounter{tocdepth}{1}}
\crefalias{section}{appendix}
\crefalias{subsection}{appendix}

\section{Enumeration, graphicality, and probabilistic tools}
\label{app:tools}

This appendix collects the external inputs and elementary estimates used in \Cref{app:local-transference,app:graph-enumeration-local}. Nothing here refers to the setup of \Cref{sec:setup,sec:local}; each result is either quoted from the literature or proved from scratch, and the two later appendices cite only results from here and from the main text.

\subsection{Degree-sequence laws and enumeration comparisons}
\label{app:tools-enumeration}

A \emph{bigraph} is a bipartite graph with a distinguished ordered bipartition; its degree sequence is the pair of degree sequences of the two parts. We write $G(n,m)$ for a uniformly random simple graph on $[n]$ with $m$ edges, $G(\ell,n,m)$ for a uniformly random bigraph with parts $[\ell]$ and $[n]$ and $m$ edges, and $\mb{G}(n,p)$, respectively $\mb{G}(\ell,n,p)$, for the corresponding models in which each admissible pair is an edge independently with probability $p$; thus $\mb{G}(n,p)=\mb{G}([n],p)$. We have the following definitions of degree sequences in graphs and bigraphs respectively.

\begin{definition}
\label{def:graph_deg_seq}
    Let $\mc{D}_{n,m}^{\text{graph}}$ be the set of all sequences $\mbf{d}\in\{0,\ldots,n-1\}^n$ such that $\sum_i d_i = 2m$. Let $\mb{P}_{G(n,m)}$ be the probability distribution on $\mc{D}_{n,m}^{\text{graph}}$ given by sampling the degrees of the random graph $G(n,m)$. Let $\mb{P}_{B(n,m)}$ be the probability distribution on $\mc{D}_{n,m}^{\text{graph}}$ given by sampling $d_i\sim \mr{Bin}(n-1,\tfrac{1}{2})$ and conditioning on $\sum_i d_i=2m$.
\end{definition}

\begin{definition}
\label{def:bigraph_deg_seq}
    Let $\mc{D}_{\ell,n,m}^{\text{bigraph}}$ be the set of all sequences $\mbf{d}=(\mbf{s},\mbf{t})\in\{0,\ldots,n\}^\ell\times\{0,\ldots,\ell\}^n$ such that $\sum_i s_i = \sum_j t_j = m$. Let $\mb{P}_{G(\ell,n,m)}$ be the probability distribution on $\mc{D}_{\ell,n,m}^{\text{bigraph}}$ given by sampling the degrees of the random bigraph $G(\ell,n,m)$. Let $\mb{P}_{B(\ell,n,m)}$ be the probability distribution on $\mc{D}_{\ell,n,m}^{\text{bigraph}}$ given by sampling $s_i\sim \mr{Bin}(n,\tfrac{1}{2})$ and $t_j\sim\mr{Bin}(\ell,\tfrac{1}{2})$, and conditioning on $\sum_i s_i=\sum_j t_j = m$.
\end{definition}

In both binomial laws, the choice of success probability $1/2$ is immaterial after conditioning on the displayed total sums; any fixed success probability in $(0,1)$ gives the same conditional law.
We use the following comparison results of Liebenau and Wormald \cite{LW17,LW20}.

\begin{theorem}[\cite{LW17}, Theorem 1.4]
\label{thm:graph_counting}
    Let $\mu_0>0$ be a sufficiently small constant, and let $1/2<\alpha<3/5$. Let $n$ and $m$ be integers, and let $d=2m/n$. Assume that $\mu=\tfrac{d}{n-1}$ satisfies $\mu\le\mu_0$, and for all fixed $K>0$, $\log^K n/n=o(\mu)$. Uniformly over all $\mbf{d}\in\mc{D}_{n,m}^{\text{graph}}$ with $|d_i-d|\le d^\alpha$ for all $i$, we have
    \[\mb{P}_{G(n,m)}(\mbf{d}) = \mb{P}_{B(n,m)}(\mbf{d})\exp\left(\frac{1}{4} - \frac{\gamma_2(\mbf{d})^2}{4\mu^2(1-\mu)^2}\right)\left(1+O\left(\frac{\log^2 n}{\sqrt{n}} + d^{5\alpha-3}\right)\right),\]
    where $\gamma_2(\mbf{d}) = \tfrac{1}{(n-1)^2}\sum_i (d_i-d)^2$.
\end{theorem}

\begin{theorem}[\cite{LW20}, Theorem 1.1, bipartite case]
\label{thm:bigraph_counting}
    Let $\mu_0>0$ be a sufficiently small constant, and let $1/2<\alpha<3/5$. Let $\ell,n,m$ be integers such that
    \[\mu:=\frac{m}{n\ell}<\mu_0\text{ and }(\ell+n)^{5-5\alpha}=o(\ell n m^{3-5\alpha}),\]
    and for all fixed $K>0$, $\ell\log^K n + n\log^K\ell = o(m)$. Let $s=m/\ell$ and $t=m/n$.
    
    Uniformly over all $(\mbf{s},\mbf{t})\in\mc{D}_{\ell,n,m}^{\text{bigraph}}$ with $|s_i-s|\le s^\alpha$ for all $i$ and $|t_j-t|\le t^\alpha$ for all $j$, we have
    \[\mb{P}_{G(\ell,n,m)}((\mbf{s},\mbf{t})) = \mb{P}_{B(\ell,n,m)}((\mbf{s},\mbf{t}))H(\mbf{s},\mbf{t})\left(1+O\left(\frac{\log^2 \ell}{\sqrt{\ell}} + \frac{\log^2 n}{\sqrt{n}} + \min\{s,t\}^{5\alpha-5}\frac{m^2}{\ell n}\right)\right),\]
    where
    \[H(\mbf{s},\mbf{t}) = \exp\left(-\frac{1}{2}\left(1-\frac{\sigma^2(\mbf{s})}{s(1-\mu)}\right)\left(1-\frac{\sigma^2(\mbf{t})}{t(1-\mu)}\right)\right)\]
    where $\sigma^2(\mbf{s}) = \tfrac{1}{\ell}\sum_i (s_i-s)^2$ and $\sigma^2(\mbf{t}) = \tfrac{1}{n}\sum_j (t_j-t)^2$.
\end{theorem}

\subsection{Binomial point probabilities and Gaussian approximation}
\label{app:binomial-estimates}

We have the following two-tiered estimate for the PDF of a binomial random variable.
\begin{lemma}
\label{lem:binomial_estimate}
    Fix a real parameter $\theta\in(1/2,1)$ and a real constant $T>1$. For all sufficiently large $n\in\mb{Z}$ and $p\in(T^{-1}n^{-\theta},Tn^{-\theta})$, the following is true.

    Select $n'\in\mb{Z}$ and $p'\in(0,1)$ such that $|n'-n|< T\frac{n}{\sqrt{pn}}$ and $|p'-p|< T\frac{p}{\sqrt{pn}}$. There exist real numbers $\mc{N}_1$ and $\mc{N}_2$ such that uniformly for $x\in[-\sqrt{pn}\log n,\sqrt{pn}\log n]\cap(\mb{Z}-pn)$, we have
    \[\mb{P}[\mr{Bin}(n',p')=pn+x] = (1+O(p\log^2 n))\mc{N}_1\left(\frac{p'(n'-pn)}{(1-p')pn}\right)^x\frac{(pn)!(pn)^x}{(pn+x)!}\]
    and
    \[\mb{P}[\mr{Bin}(n',p')=pn+x] = \left(1+O(\log^3 n/\sqrt{pn})\right)\mc{N}_2e^{-(x-\alpha\sqrt{pn})^2/2pn}\]
    where $\alpha = \sqrt{pn}\log\left(\frac{p'(n'-pn)}{(1-p')pn}\right)$. Factorials with noninteger arguments are interpreted through the Gamma function.
\end{lemma}
\begin{proof}
    Uniformly for $x\in[-\sqrt{pn}\log n,\sqrt{pn}\log n]\cap(\mb{Z}-pn)$, we have
    \begin{align*}
        \mb{P}[\mr{Bin}(n',p') = pn + x] &= \frac{n'!}{(pn+x)!(n'-pn-x)!}p'^{pn+x}(1-p')^{n'-pn-x} \\
        &= \frac{n'!}{(pn)!(n'-pn)!}p'^{pn}(1-p')^{n'-pn}\cdot\frac{(n'-pn)!}{(n'-pn-x)!}\frac{(pn)!}{(pn+x)!}\left(\frac{p'}{1-p'}\right)^x \\
        &= (1+O(p\log^2 n))\frac{n'!}{(pn)!(n'-pn)!}p'^{pn}(1-p')^{n'-pn}\\
        &\qquad\cdot(n'-pn)^x\frac{(pn)!}{(pn+x)!}\left(\frac{p'}{1-p'}\right)^x \\
        &= (1+O(p\log^2 n))\mc{N}_1\left(\frac{p'(n'-pn)}{(1-p')pn}\right)^x\frac{(pn)!(pn)^x}{(pn+x)!},
    \end{align*}
    for some normalization $\mc{N}_1$ independent of $x$. This proves the first equality.

    For the second equality, we use the fact that
    \[n! = \sqrt{2\pi n}(n/e)^n(1+O(1/n))\]
    as $n\to\infty$. Indeed, uniformly for $x\in[-\sqrt{pn}\log n,\sqrt{pn}\log n]\cap(\mb{Z}-pn)$, we have
    \begin{align*}
        \frac{(pn+x)!}{(pn)!(pn)^x} &= (1 + O(1/pn))e^{-x}\left(1 + \frac{x}{pn}\right)^{pn+x+\frac{1}{2}} \\
        &= (1 + O(1/pn))e^{-x}\exp\left(\left(pn+x+\frac{1}{2}\right)\left[\frac{x}{pn}-\frac{x^2}{2(pn)^2}+ O\left(\frac{\log^3 n}{(pn)^{3/2}}\right)\right]\right) \\
        &= \left(1 + O\left(\frac{\log^3 n}{\sqrt{pn}}\right)\right)\exp\left(\frac{x^2}{2pn}\right),
    \end{align*} 
    which combined with the previous equality implies the second equality.
\end{proof}

\begin{theorem}
\label{thm:general_poly_calc}
    Fix a real parameter $\theta\in(1/2,1)$, a positive integer $d\ge 1$, an integer matrix $M\in\mb{Z}^{r\times d}$ with nonzero, pairwise orthogonal rows, an $r$-dimensional vector inequality relation $\ge'\in\{\ge,>\}^r$, a monomial $\rho\in\mb{R}[x_1,\ldots,x_d]$, and a real constant $T>1$. For all sufficiently large $n\in\mb{Z}$ and $p\in(T^{-1}n^{-\theta},Tn^{-\theta})$, the following is true.

    Select $\eta\in((T^{-1}n,Tn)\cap\mb{Z})^d$ such that $\norm{M\eta}_\infty < T\frac{n}{\sqrt{pn}}$, and select $\alpha\in(-T,T)^d$. Then,
    \[\mb{E}[\rho(A-p\eta)\mbf{1}_{MA\ge'0}] = \mb{E}[\rho(X)\mbf{1}_{MX\ge -pM\eta}] + O(\sqrt{pn}^{\deg\rho-1}\log^{3+\deg\rho+d}n)\]
    where $A_t\sim\mr{Bin}\left(\eta_t,\frac{\frac{p}{1-p}e^{\alpha_t/\sqrt{p\eta_t}}}{1+\frac{p}{1-p}e^{\alpha_t/\sqrt{p\eta_t}}}\right)$ are all independently sampled, and $X_t\sim\mc{N}(\sqrt{p\eta_t}\alpha_t,p\eta_t)$ are all independently sampled.
\end{theorem}
\begin{proof}
    Call a convex region $\mc{C}$ in $\mb{R}^d$ (each boundary can be either open or closed) \textit{good} if it is bounded by pairwise orthogonal planes with integer normal vectors whose entries are bounded by $\norm{M}_\infty$, and distance to the origin bounded by $\sqrt{pn}\log n$. For such a region, define
    \[\mc{P}_1(\mc{C}) := \bigcup_{\substack{x\in \prod_{t\in[d]}(\mb{Z}-p\eta_t)\cap[-\sqrt{p\eta_t}\log\eta_t,\sqrt{p\eta_t}\log\eta_t] \\ x\in\mc{C}}}\prod_{t\in[d]}[x_t,x_t+1),\]
    and
    \[\mc{P}_2(\mc{C}) := \left\{x\in\prod_{t\in[d]}[-\sqrt{p\eta_t}\log\eta_t,\sqrt{p\eta_t}\log\eta_t] : x\in\mc{C}\right\},\]
    and
    \[\mc{P}_3(\mc{C}) := \left\{x\in\mb{R}^d : x\in\mc{C}\right\}.\]

    Note that $\mb{P}[|A_t-p\eta_t| > \sqrt{p\eta_t}\log\eta_t] = O(n^{-\omega(1)})$ by Chernoff's inequality, so \Cref{lem:binomial_estimate} gives
    \[\mb{E}[\rho(A-p\eta)\mbf{1}_{MA\ge'0}] = \frac{\displaystyle\sum_{\substack{x\in\prod_{t\in[d]}(\mb{Z}-p\eta_t)\cap[-\sqrt{p\eta_t}\log\eta_t,\sqrt{p\eta_t}\log\eta_t] \\ Mx\ge' -pM\eta}}\rho(x)f(x)}{\displaystyle\sum_{x\in\prod_{t\in[d]}(\mb{Z}-p\eta_t)\cap[-\sqrt{p\eta_t}\log\eta_t,\sqrt{p\eta_t}\log\eta_t]}f(x)} + O(\sqrt{pn}^{\deg\rho-1}\log^{3+\deg\rho}n),\]
    where
    \[f(x):=\exp\left(-\sum_{t\in[d]}\frac{(x_t-\alpha_t\sqrt{p\eta_t})^2}{2p\eta_t}\right).\]
    The first step is to turn the summations into integrals over $\mc{P}_1$.
    
    Note that for $x\in\prod_{t\in[d]}[-\sqrt{p\eta_t}\log\eta_t,\sqrt{p\eta_t}\log\eta_t]$, we have
    \[f(x+\ve) = f(x) + O(\log n/\sqrt{pn})\]
    for any $\ve\in[0,1)^d$. In particular, this implies that
    \[f(x) = \int_{[0,1)^d}d\ve\, f(x+\ve) + O(\log n/\sqrt{pn})\]
    and
    \[\rho(x)f(x) = \int_{[0,1)^d}d\ve\,\rho(x+\ve)f(x+\ve) + O(\sqrt{pn}^{\deg\rho-1}\log^{1+\deg\rho} n)\]
    for all $x\in\prod_{t\in[d]}[-\sqrt{p\eta_t}\log\eta_t,\sqrt{p\eta_t}\log\eta_t]$. Thus, letting $\mc{C}_M = \{x\in\mb{R}^d : Mx\ge' -pM\eta\}$, we have
    \begin{align*}\mb{E}[\rho(A-p\eta)\mbf{1}_{MA\ge'0}] &= \frac{\int_{\mc{P}_1(\mc{C}_M)}dx\,\rho(x)f(x) + O(\sqrt{pn}^{\deg\rho + d-1}\log^{1+\deg\rho + d} n)}{\int_{\mc{P}_1(\mb{R}^d)}dx\,f(x) + O(\sqrt{pn}^{d-1}\log^{1+d}n)}\\
    &\quad+ O(\sqrt{pn}^{\deg\rho-1}\log^{3+\deg\rho}n).\end{align*}
    There is a constant $C$ such that for any good $\mc{C}$, $\mc{P}_1\setminus\mc{P}_2$ and $\mc{P}_2\setminus\mc{P}_1$ are both contained in $B(\partial\mc{P}_2,C)$, and $\mr{vol}\,B(\partial\mc{P}_2,C)=O(\sqrt{pn}^{d-1}\log^{d-1}n)$. Therefore,
    \begin{align*}\mb{E}[\rho(A-p\eta)\mbf{1}_{MA\ge'0}] &= \frac{\int_{\mc{P}_2(\mc{C}_M)}dx\,\rho(x)f(x) + O(\sqrt{pn}^{\deg\rho + d-1}\log^{1+\deg\rho + d} n)}{\int_{\mc{P}_2(\mb{R}^d)}dx\,f(x) + O(\sqrt{pn}^{d-1}\log^{1+d}n)}\\
    &\quad+ O(\sqrt{pn}^{\deg\rho-1}\log^{3+\deg\rho}n).\end{align*}
    Note that $\mc{P}_2\subset\mc{P}_3$ and $\mc{P}_3\setminus\mc{P}_2\subset\{x\in\mb{R}^d : \norm{x}_\infty>C\sqrt{pn}\log n\}$ for some constant $C$, and
    \begin{align*}
        &\int_{\norm{x}_\infty>C\sqrt{pn}\log n}dx\,f(x) = O(n^{-\omega(1)}) \\
        &\int_{\norm{x}_\infty>C\sqrt{pn}\log n}dx\,\norm{\rho(x)}_\infty f(x) = O(n^{-\omega(1)}).
    \end{align*}
    Therefore,
    \begin{align*}
        \mb{E}[\rho(A-p\eta)\mbf{1}_{MA\ge'0}] &= \frac{\int_{\mc{P}_3(\mc{C}_M)}dx\,\rho(x)f(x) + O(\sqrt{pn}^{\deg\rho + d-1}\log^{1+\deg\rho + d} n)}{\int_{\mc{P}_3(\mb{R}^d)}dx\,f(x) + O(\sqrt{pn}^{d-1}\log^{1+d}n)}\\
        &\quad+ O(\sqrt{pn}^{\deg\rho-1}\log^{3+\deg\rho}n) \\
        &= \frac{\int_{x\in\mc{C}_M}dx\,\rho(x)f(x)}{\int_{x\in\mb{R}^d}dx\, f(x)} + O(\sqrt{pn}^{\deg\rho-1}\log^{3+\deg\rho+d}n),
    \end{align*}
    as desired.
\end{proof}

The next result is the first-order expansion, in a small tilt of the success probabilities, of the conditional expectations appearing in \Cref{thm:general_poly_calc}. It is the tool with which the idealized process of \Cref{sec:idealized} is perturbed (\Cref{app:idealized}). Unlike \Cref{thm:general_poly_calc}, it compares two binomial models with each other rather than with a Gaussian one; the two models need only share the centering $p\wt{\eta}$ and the conditioning event (their trial counts may differ from $\wt\eta$ and from each other), and no structure of $M$ is used.

\begin{theorem}[Tilted expansion]
\label{thm:tilted_expansion}
    Fix a real parameter $\theta\in(1/2,1)$, a positive integer $d\ge 1$, a matrix $M\in\mb{R}^{r\times d}$, an $r$-dimensional vector inequality relation $\ge'\in\{\ge,>\}^r$, a monomial $\rho\in\mb{R}[x_1,\ldots,x_d]$, and a real constant $T>1$. For all sufficiently large $n\in\mb{Z}$ and $p\in(T^{-1}n^{-\theta},Tn^{-\theta})$, the following is true.

    Select a reference size vector $\wt{\eta}\in((T^{-1}n,Tn)\cap\mb{Z})^d$ and trial counts $\wt{\eta}',\eta\in\mb{Z}_{\ge1}^d$ such that $\norm{\wt{\eta}'-\wt{\eta}}_\infty,\norm{\eta-\wt{\eta}}_\infty < T\frac{n}{\sqrt{pn}}$, and select $\wt{q},q\in(0,1)^d$ such that $\norm{\wt{q}-p}_\infty, \norm{q-p}_\infty< T\frac{p}{\sqrt{pn}}$ and $\norm{\beta}_\infty < \frac{T}{\sqrt{pn}\log^2n}$, where $\beta\in\mb{R}^d$ is defined coordinatewise by
    \[\beta_t:=\log\left(\frac{\frac{q_t}{1-q_t}(\eta_t-p\wt{\eta}_t)}{\frac{\wt{q}_t}{1-\wt{q}_t}(\wt{\eta}'_t-p\wt{\eta}_t)}\right).\]
    Let $A_t\sim\mr{Bin}(\eta_t,q_t)$ and $\wt{A}_t\sim\mr{Bin}(\wt{\eta}'_t,\wt{q}_t)$ all be sampled independently. Then
    \begin{align*}
        \mb{E}[\rho(A-p\wt{\eta})\mbf{1}_{MA\ge'0}] &= \mb{E}\left[\rho(\wt{A}-p\wt{\eta})\mbf{1}_{M\wt{A}\ge'0}\right] \\
        &\quad+ \sum_{t\in[d]}\beta_t\,\mb{E}\left[(\wt{A}_t-p\wt{\eta}_t)\rho(\wt{A}-p\wt{\eta})\mbf{1}_{M\wt{A}\ge'0}\right] \\
        &\quad- \sum_{t\in[d]}\beta_t\,\mb{E}[\wt{A}_t-p\wt{\eta}_t]\,\mb{E}\left[\rho(\wt{A}-p\wt{\eta})\mbf{1}_{M\wt{A}\ge' 0}\right] \\
        &\quad+ O\left(\max\left\{p,\ \norm{\beta}_\infty^2\,pn\log^2 n\right\}\sqrt{pn}^{\deg\rho}\log^{2+\deg\rho}n\right).
    \end{align*}
\end{theorem}
\begin{proof}
        Write $\wt{c}:=p\wt{\eta}$ for the common center, $D:=\deg\rho$, and
    \[\Lambda:=\prod_{t\in[d]}\left(\mb{Z}-\wt{c}_t\right)\cap\left[-\sqrt{p\wt{\eta}_t}\log\wt{\eta}_t,\ \sqrt{p\wt{\eta}_t}\log\wt{\eta}_t\right]\]
    for the lattice window around the center, and $\mc{R}:=\{x\in\Lambda:Mx\ge'-M\wt{c}\}$. Since $\wt{\eta}_t\in(T^{-1}n,Tn)$ we have $\sqrt{p\wt\eta_t}\log\wt\eta_t\asymp\sqrt{pn}\log n$, so on $\Lambda$ every coordinate satisfies $|x_t|=O(\sqrt{pn}\log n)$ and $|\rho(x)|=O(\sqrt{pn}^{D}\log^Dn)$. All implied constants below depend only on $\theta,d,M,\rho,T$.

    \emph{Step 1: truncation.} The means satisfy $|\mb{E}\wt{A}_t-\wt{c}_t|\le|\wt\eta'_t-\wt\eta_t|\wt q_t+\wt\eta_t|\wt q_t-p|\le2T\sqrt{pn}+T^2\sqrt{pn}$ and likewise $|\mb{E}A_t-\wt{c}_t|\le2T\sqrt{pn}+T^2\sqrt{pn}$, and the variances are $O(pn)$, so Chernoff's inequality gives $\mb{P}[A-\wt{c}\notin\Lambda]=O(n^{-\omega(1)})$ and $\mb{P}[\wt{A}-\wt{c}\notin\Lambda]=O(n^{-\omega(1)})$. For large $n$, the trial counts $\eta_t$ and $\wt\eta'_t$ are at most $2Tn$, so $|\rho(A-\wt c)|$ and $|\rho(\wt A-\wt c)|$ are deterministically $O(n^D)$, while multiplying by one centered coordinate gives $O(n^{D+1})$; the coefficients of $\rho$ are absorbed into the implied constants. Consequently the contributions of the complement of $\Lambda$ to all expectations in the statement are $O(n^{-\omega(1)})$, and the events $\{MA\ge'0\}$ and $\{M\wt{A}\ge'0\}$ coincide with $\{A-\wt{c}\in\mc{R}\}$ and $\{\wt{A}-\wt{c}\in\mc{R}\}$ on $\Lambda$, because $A=\wt{c}+x$ satisfies $MA\ge'0$ if and only if $Mx\ge'-M\wt{c}$, and likewise for $\wt{A}$.

    \emph{Step 2: the two point probabilities.} Apply the first formula of \Cref{lem:binomial_estimate}, for each $t$, with its $n$ equal to $\wt{\eta}_t$ (so that its center is $p\wt\eta_t=\wt c_t$ and its window is the $t$th factor of $\Lambda$), its constant $T$ replaced by $T^{3/2+\theta}$, and $(n',p')=(\wt\eta'_t,\wt q_t)$, respectively $(n',p')=(\eta_t,q_t)$; the hypotheses hold because $p\in(T^{-1-\theta}\wt\eta_t^{-\theta},T^{1+\theta}\wt\eta_t^{-\theta})$, $|q_t-p|,|\wt q_t-p|<Tp/\sqrt{pn}\le T^{3/2}p/\sqrt{p\wt\eta_t}$, and $|\eta_t-\wt\eta_t|,|\wt\eta'_t-\wt\eta_t|<Tn/\sqrt{pn}\le T^{3/2}\wt\eta_t/\sqrt{p\wt\eta_t}$. Writing
    \[\wt\kappa_t:=\frac{\wt q_t(\wt\eta'_t-\wt c_t)}{(1-\wt q_t)\wt c_t},\qquad\kappa_t:=\frac{q_t(\eta_t-\wt c_t)}{(1-q_t)\wt c_t},\qquad g_t(x):=\frac{\wt c_t!\,\wt c_t^{\,x}}{(\wt c_t+x)!},\qquad w(x):=\prod_{t\in[d]}\wt\kappa_t^{\,x_t}g_t(x_t),\]
    we obtain, uniformly for $x\in\Lambda$,
    \[\mb{P}\left[\wt{A}=\wt{c}+x\right]=\left(1+O(p\log^2n)\right)\wt{\mc{N}}\,w(x),\qquad\mb{P}\left[A=\wt{c}+x\right]=\left(1+O(p\log^2n)\right)\mc{N}\,e^{\beta\cdot x}\,w(x),\]
    for normalizations $\wt{\mc N},\mc N$ not depending on $x$, where we used that $\kappa_t/\wt\kappa_t=e^{\beta_t}$ by the definition of $\beta$, and that a product of $d$ factors $1+O(p\log^2n)$ is again $1+O(p\log^2n)$.

    \emph{Step 3: reduction to averages against $w$.} For a function $h$ on $\Lambda$ write
    \[\langle h\rangle:=\frac{\sum_{x\in\Lambda}h(x)w(x)}{\sum_{x\in\Lambda}w(x)}.\]
    By Step 1, $\wt{\mc N}\sum_{x\in\Lambda}(1+O(p\log^2n))w(x)=\mb P[\wt A-\wt c\in\Lambda]=1-O(n^{-\omega(1)})$, so for every function $h$ on $\Lambda$ with $|h|\le H$,
    \begin{equation}\label{eq:tilt-average}
    \begin{aligned}
    \mb{E}\left[h(\wt{A}-\wt{c})\mathbf{1}_{\wt A-\wt c\in\Lambda}\right]
    &=\mb P[\wt A-\wt c\in\Lambda]\,
    \frac{\sum_{x\in\Lambda}h(x)\left(1+O(p\log^2n)\right)w(x)}{\sum_{x\in\Lambda}\left(1+O(p\log^2n)\right)w(x)}\\
    &=\langle h\rangle+O\left(p\log^2n\cdot H\right).
    \end{aligned}
    \end{equation}
    The window-mass factor is $1-O(n^{-\omega(1)})$ by Step 1; its contribution $O(Hn^{-\omega(1)})$ is absorbed in the displayed error. The same holds for $A$ with $w$ replaced by $e^{\beta\cdot x}w$ in the definition of the average. Applying \eqref{eq:tilt-average} with $h=\rho\mathbf 1_{\mc R}$, $h=x_t\rho\mathbf 1_{\mc R}$, and $h=x_t$, and using Step 1,
    \begin{equation}\label{eq:tilt-conversion}
    \begin{aligned}
    \mb{E}\left[\rho(\wt{A}-\wt{c})\mathbf{1}_{M\wt{A}\ge'0}\right]&=\langle\rho\mathbf 1_{\mc R}\rangle+O\left(p\sqrt{pn}^{D}\log^{2+D}n\right),\\
    \mb{E}\left[(\wt{A}_t-\wt{c}_t)\rho(\wt{A}-\wt{c})\mathbf{1}_{M\wt{A}\ge'0}\right]&=\langle x_t\rho\mathbf 1_{\mc R}\rangle+O\left(p\sqrt{pn}^{D+1}\log^{3+D}n\right),\\
    \mb{E}\left[\wt{A}_t-\wt{c}_t\right]&=\langle x_t\rangle+O\left(p\sqrt{pn}\log^3n\right),
    \end{aligned}
    \end{equation}
    and
    \begin{equation}\label{eq:tilt-A-side}
    \mb{E}\left[\rho(A-\wt{c})\mathbf{1}_{MA\ge'0}\right]=\frac{\sum_{x\in\mc R}\rho(x)e^{\beta\cdot x}w(x)}{\sum_{x\in\Lambda}e^{\beta\cdot x}w(x)}+O\left(p\sqrt{pn}^{D}\log^{2+D}n\right).
    \end{equation}

    \emph{Step 4: expansion of the tilt.} On $\Lambda$ we have $|\beta\cdot x|\le d\norm\beta_\infty\max_t|x_t|=O(1/\log n)$, so $e^{\beta\cdot x}=1+\beta\cdot x+O((\beta\cdot x)^2)=1+\beta\cdot x+O(\delta)$ with $\delta:=\norm\beta_\infty^2pn\log^2n=O(1/\log^2n)$. Inserting this into \eqref{eq:tilt-A-side}, the numerator becomes $\langle\rho\mathbf 1_{\mc R}\rangle+\sum_t\beta_t\langle x_t\rho\mathbf 1_{\mc R}\rangle+O(\delta\sqrt{pn}^{D}\log^Dn)$ times $\sum_\Lambda w$, and the denominator becomes $1+u+O(\delta)$ times $\sum_\Lambda w$, where $u:=\sum_t\beta_t\langle x_t\rangle$. By \eqref{eq:tilt-conversion} and Step 1, $\langle x_t\rangle=\mb E[\wt A_t-\wt c_t]+O(p\sqrt{pn}\log^3n)=O(\sqrt{pn})$, so $|u|=O(\norm\beta_\infty\sqrt{pn})=O(1/\log^2n)$ and $u^2\le d^2\norm\beta_\infty^2pn\cdot O(1)\le O(\delta)$; hence $(1+u+O(\delta))^{-1}=1-u+O(\delta)$. Multiplying out, and using $|\langle x_t\rho\mathbf 1_{\mc R}\rangle|=O(\sqrt{pn}^{D+1}\log^{D+1}n)$ so that $\sum_t\beta_t\langle x_t\rho\mathbf 1_{\mc R}\rangle\cdot u=O(\norm\beta_\infty^2pn\log^{D+1}n\sqrt{pn}^{D})=O(\delta\sqrt{pn}^D\log^{D}n)$,
    \[\frac{\sum_{x\in\mc R}\rho(x)e^{\beta\cdot x}w(x)}{\sum_{x\in\Lambda}e^{\beta\cdot x}w(x)}=\langle\rho\mathbf 1_{\mc R}\rangle+\sum_{t\in[d]}\beta_t\left(\langle x_t\rho\mathbf 1_{\mc R}\rangle-\langle x_t\rangle\langle\rho\mathbf 1_{\mc R}\rangle\right)+O\left(\delta\sqrt{pn}^{D}\log^{D}n\right).\]

    \emph{Step 5: conclusion.} Substitute \eqref{eq:tilt-conversion} into the right-hand side. The zeroth-order term contributes an error $O(p\sqrt{pn}^D\log^{2+D}n)$. In the first-order terms, each $\beta_t=O(1/(\sqrt{pn}\log^2n))$ multiplies an error of size $O(p\sqrt{pn}^{D+1}\log^{3+D}n)$, giving $O(p\sqrt{pn}^{D}\log^{1+D}n)$. Together with \eqref{eq:tilt-A-side} and the bound
    \[\delta\sqrt{pn}^D\log^Dn\le\norm\beta_\infty^2\,pn\log^2n\,\sqrt{pn}^D\log^{2+D}n,\]
    all errors are $O(\max\{p,\norm\beta_\infty^2pn\log^2n\}\sqrt{pn}^{D}\log^{2+D}n)$, which is the statement.
\end{proof}

\subsection{Sufficient conditions for graphicality}
\label{app:tools-graphicality}

The following sufficient conditions, consequences of the Erdős--Gallai and Gale--Ryser theorems, are used in \Cref{app:local-transference} to show that the degree arrays arising there are graphical.

\begin{lemma}\label{lem:graphical-sufficient}
Let $d_1,\ldots,d_n$ be nonnegative integers with even sum $S$ and maximum degree $\Delta$. If
\[
S\ge \Delta(\Delta+1),
\]
then $(d_1,\ldots,d_n)$ is graphical.
\end{lemma}

\begin{proof}
Arrange the sequence in nonincreasing order and write $L_r:=\sum_{i=1}^r d_i$, so that $L_r\le r\Delta$. By the Erdős--Gallai theorem, it suffices to prove $L_r\le r(r-1)+\sum_{i=r+1}^n\min(r,d_i)$ for every $1\le r\le n$. If $r\ge\Delta$, then $\min(r,d_i)=d_i$ for $i>r$, and
\begin{align*}
r(r-1)+\sum_{i=r+1}^n\min(r,d_i)&=r(r-1)+S-L_r\ge r(r-1)+\Delta(\Delta+1)-r\Delta\\
&=r\Delta+(r-\Delta)(r-\Delta-1)\ge r\Delta\ge L_r,
\end{align*}
where the second inequality holds because $(r-\Delta)(r-\Delta-1)\ge0$ for every integer $r$. If $r<\Delta$, then $\min(r,d_i)\ge rd_i/\Delta$ for every $i$, and
\[
\sum_{i=r+1}^n\min(r,d_i)\ge\frac r\Delta(S-L_r)\ge\frac r\Delta(S-r\Delta)\ge\frac{r}{\Delta}\left(\Delta(\Delta+1)-r\Delta\right)=r\Delta-r(r-1)\ge L_r-r(r-1).\qedhere
\]
\end{proof}

\begin{lemma}\label{lem:bipartite-graphical-sufficient}
Let $a_1,\ldots,a_\ell$ and $b_1,\ldots,b_n$ be nonnegative integer sequences with equal sum $M$. Let $\Delta_L:=\max_i a_i$ and $\Delta_R:=\max_j b_j$. If
\[
M\ge \Delta_L\Delta_R,
\]
then the pair of sequences is bipartite-graphical.
\end{lemma}

\begin{proof}
Arrange $a_1,\ldots,a_\ell$ in nonincreasing order. By Gale--Ryser, it suffices to show
\[
\sum_{i=1}^r a_i\le\sum_{j=1}^n\min(r,b_j)
\]
for every $1\le r\le\ell$. If $r\ge\Delta_R$, the right side is $M$, and the inequality is immediate. If $r<\Delta_R$, then $\min(r,b_j)\ge r b_j/\Delta_R$ for every $j$, so the right side is at least $rM/\Delta_R\ge r\Delta_L\ge\sum_{i=1}^r a_i$.
\end{proof}

\subsection{Concentration and counting}
\label{app:tools-concentration}

We have the following second moment bound for the degree sequences of $\mb{G}(n,p)$ and $\mb{G}(\ell,n,p)$; it is used in the proof of \Cref{lem:local-gamma-concentration}. Recall that $\deg^G_Sv$ denotes the number of neighbors of $v$ in $S$ in the graph $G$.
\begin{lemma}
\label{lem:gamma_bound}
    Fix a real parameter $\theta\in(1/2,1)$ and constants $T>1$ and $K>0$. There exists a constant $\Cst{lem:gamma_bound}(K)=\Cst{lem:gamma_bound}(\theta,T,K)>0$ such that, writing $C:=\Cst{lem:gamma_bound}(K)$, for all sufficiently large $n\in\mb{Z}$ and $p\in(T^{-1}n^{-\theta},Tn^{-\theta})$ the following are true.
    \begin{enumerate}[label=\textup{(\roman*)},ref=\textup{(\roman*)}]
        \item\label{lem:gamma_bound:graph} Sample $G\sim \mb{G}(n,p)$ and let $(d_1,\ldots,d_n)$ be the resulting degree sequence, with mean $d$. Then,
        \[\mb{P}\left[\sum_{i=1}^n(d_i-d)^2\ge Cpn^2\text{ and }\max_{i\in[n]}|d_i-pn|\le pn\right]\le\exp(-Kn).\]
        \item\label{lem:gamma_bound:bipartite} Consider $\ell\in\mb{Z}$ with $\ell\in[T^{-1}n,Tn]$. Sample $G\sim\mb{G}(\ell,n,p)$ and let $(s_1,\ldots,s_\ell),(t_1,\ldots,t_n)$ be the resulting degree sequences, with means $s$ and $t$ respectively. Then,
        \[\mb{P}\left[\left(\sum_{i=1}^\ell(s_i-s)^2\ge Cp\ell n\text{ or }\sum_{j=1}^n(t_j-t)^2\ge Cp\ell n\right)\text{ and }\mc{K}_{\ell,n}\right]\le\exp(-Kn),\]
        where $\mc{K}_{\ell,n}$ denotes the event that $|s_i-pn|\le pn$ for all $i\in[\ell]$ and $|t_j-p\ell|\le pn$ for all $j\in[n]$.
    \end{enumerate}
\end{lemma}
\begin{proof}
    We first record a moment generating function bound for a single truncated binomial variable. Let $Y\sim\mr{Bin}(r,p)$, let $B>0$ with $B\ge rp$, write $\ol{Y}:=Y-rp$, and put $\lambda:=1/(20B)$. Bernstein's inequality (Theorem 2.8.4 of \cite{Ver18}), applied to the $r$ independent centered Bernoulli variables whose sum is $\ol{Y}$, which are bounded by $1$ and whose variances sum to $rp(1-p)\le B$, gives for $0\le u\le2B$
    \[\mb{P}\left[|\ol{Y}|>u\right]\le2\exp\left(-\frac{u^2/2}{B+u/3}\right)\le2\exp\left(-\frac{3u^2}{10B}\right).\]
    For $X\ge0$, $b>0$, and an increasing differentiable $g$ we have $\mb{E}[g(\min(X,b))]=g(0)+\int_0^bg'(u)\mb{P}[X>u]\,du$; applying this with $X=|\ol{Y}|$, $b=2B$, and $g(u)=e^{\lambda u^2}$, and using $e^{\lambda\ol{Y}^2}\mbf{1}_{|\ol{Y}|\le2B}\le e^{\lambda\min(|\ol{Y}|,2B)^2}$,
    \begin{align*}
    \mb{E}\left[e^{\lambda\ol{Y}^2}\mbf{1}_{|\ol{Y}|\le2B}\right]&\le1+\int_0^{2B}2\lambda u\,e^{\lambda u^2}\,\mb{P}\left[|\ol{Y}|>u\right]du\\
    &\le1+\int_0^\infty4\lambda u\exp\left(-\frac{u^2}{4B}\right)du=1+8\lambda B=\frac75\le e^{1/2}.
    \end{align*}
    Consequently, if $\ol{Y}_1,\ldots,\ol{Y}_q$ are independent, each of the form $Y_i-r_ip$ with $Y_i\sim\mr{Bin}(r_i,p)$ and $r_ip\le B$, then for every $t>0$, by Markov's inequality applied to $\exp(\lambda\sum_i\ol{Y}_i^2)$ on the truncation event,
    \begin{equation}\label{eq:trunc-mgf}
    \mb{P}\left[\sum_{i=1}^q\ol{Y}_i^2\ge t\text{ and }\max_i|\ol{Y}_i|\le2B\right]\le e^{-\lambda t}\prod_{i=1}^q\mb{E}\left[e^{\lambda\ol{Y}_i^2}\mbf{1}_{|\ol{Y}_i|\le2B}\right]\le\exp\left(-\frac{t}{20B}+\frac q2\right).
    \end{equation}

    We prove \cref{lem:gamma_bound:graph}. Since the empirical mean $d$ minimizes $a\mapsto\sum_i(d_i-a)^2$, we have $\sum_i(d_i-d)^2\le\sum_i(d_i-p(n-1))^2$, so it suffices to bound the probability of $\mc{E}\cap\mc{K}$, where $\mc{E}$ is the event that $\sum_{i=1}^n(d_i-p(n-1))^2\ge Cpn^2$ and $\mc{K}$ is the event that $|d_i-pn|\le pn$ for all $i\in[n]$.

    We first show that $\mc{E}$ is contained in a union of events that involve only independent variables. Let $S$ be a uniformly random subset of $[n]$, let $U=[n]\setminus S$, and let $\mb{E}_S$ denote averaging over the randomness of $S$ only. For $v\in[n]$ write
    \[\Delta_S(v):=\deg^G_Sv-p|S\setminus\{v\}|,\qquad\Delta_U(v):=\deg^G_Uv-p|U\setminus\{v\}|,\]
    so that $\Delta_S(v)+\Delta_U(v)=d_v-p(n-1)$. Both $\Delta_S(v)$ and $\Delta_U(v)$ are determined by $G$ and by $S\setminus\{v\}$, and $\mbf{1}_{v\in S}$ is independent of $S\setminus\{v\}$. Hence
    \begin{align*}
        \mb{E}_S\left[\mbf{1}_{v\in U}\Delta_S(v)^2 + \mbf{1}_{v\in S}\Delta_U(v)^2\right]
        &=\frac{1}{2}\mb{E}_S\left[\Delta_S(v)^2 + \Delta_U(v)^2\right]\\
        &\ge\frac{1}{4}\mb{E}_S\left[\left(\Delta_S(v)+\Delta_U(v)\right)^2\right]
        = \frac{1}{4}(d_v-p(n-1))^2.
    \end{align*}
    Summing over $v$, on the event $\mc{E}$ we have
    \[\mb{E}_S\left[\sum_{v\in U}\Delta_S(v)^2 + \sum_{v\in S}\Delta_U(v)^2\right]\ge Cpn^2/4,\]
    so there is some $S\subseteq[n]$ with $\sum_{v\in U}\Delta_S(v)^2\ge Cpn^2/8$ or $\sum_{v\in S}\Delta_U(v)^2\ge Cpn^2/8$. Exchanging the roles of $S$ and $U$ in the second case, and noting that $\Delta_S(v)=\deg^G_Sv-p|S|$ for $v\in U$, we conclude that $\mc{E}\subseteq\bigcup_{S\subseteq[n]}\mc{E}_S$, where
    \[\mc{E}_S:=\left\{\sum_{v\in[n]\setminus S}(\deg^G_Sv-p|S|)^2\ge Cpn^2/8\right\}.\]

    Fix $S$. Over the randomness of $G$, the variables $\deg^G_Sv$ for $v\in[n]\setminus S$ are independent $\mr{Bin}(|S|,p)$ variables, since they are determined by the disjoint edge sets $\{v\}\times S$. On $\mc{K}$ we have $0\le\deg^G_Sv\le d_v\le2pn$ and $0\le p|S|\le pn$, so $|\deg^G_Sv-p|S||\le2pn$ for every $v$. Applying \eqref{eq:trunc-mgf} with $B=pn$, $t=Cpn^2/8$, and $q=n-|S|\le n$,
    \[\mb{P}[\mc{E}_S\cap\mc{K}]\le\exp\left(-\frac{Cn}{160}+\frac{n}{2}\right).\]
    A union bound over the $2^n$ sets $S$ gives $\mb{P}[\mc{E}\cap\mc{K}]\le\exp(-(C/160-1/2-\log2)n)\le\exp(-Kn)$ as soon as $C\ge160(K+2)$.

    For \cref{lem:gamma_bound:bipartite} no bisection is needed: the variables $s_i$, $i\in[\ell]$, are independent $\mr{Bin}(n,p)$ variables, and the variables $t_j$, $j\in[n]$, are independent $\mr{Bin}(\ell,p)$ variables. As before, $\sum_i(s_i-s)^2\le\sum_i(s_i-pn)^2$ and $\sum_j(t_j-t)^2\le\sum_j(t_j-p\ell)^2$. Apply \eqref{eq:trunc-mgf} with $B=Tpn$, which is at least $np$ and at least $\ell p$, with $t=Cp\ell n\ge Cpn^2/T$, and with $q=\ell\le Tn$, respectively $q=n$; the truncation hypotheses $|s_i-pn|\le pn$ and $|t_j-p\ell|\le pn$ are stronger than the required $|\ol{Y}_i|\le2B$. Each of the two events in question therefore has probability at most $\exp(-Cn/(20T^2)+Tn/2)$ on the truncation event, and their union has probability at most $2\exp(-Cn/(20T^2)+Tn/2)\le\exp(-Kn)$ as soon as $C\ge20T^2(K+T)$. Both parts therefore hold with $\Cst{lem:gamma_bound}(K):=\max\{160(K+2),\,20T^2(K+T)\}$.
\end{proof}
We have the following basic counting lemma, also used in the proof of \Cref{lem:local-gamma-concentration}.
\begin{lemma}
\label{lem:basic_counting}
    Fix a real parameter $\theta\in(1/2,1)$ and constants $T>1$ and $C>0$. There exists a constant $\Cst{lem:basic_counting}(C)=\Cst{lem:basic_counting}(\theta,T,C)>0$ such that for all sufficiently large $n\in\mb{Z}$ and $p\in(T^{-1}n^{-\theta},Tn^{-\theta})$, the number of integer sequences
    \[\left|\left\{(y_1,\ldots,y_n)\in\left([-C\sqrt{np},C\sqrt{np}]\cap\mb{Z}\right)^n : y_1+\cdots+y_n = 0 \right\}\right|\]
    is at least $\exp(-\Cst{lem:basic_counting}(C)\,n)(\sqrt{np})^n$.
\end{lemma}
\begin{proof}
    Write $n_0:=n-\lceil n^{2/3}\rceil$ and $J:=[-C\sqrt{np},C\sqrt{np}]\cap\mb{Z}$, so that $|J|\ge C\sqrt{np}$. By Chebyshev's inequality, at least half of all sequences $(y_1,\ldots,y_{n_0})\in J^{n_0}$ have $|y_1+\cdots+y_{n_0}|\le2C\sqrt{n}\sqrt{np}$. Each such sequence has at least one continuation $(y_{n_0+1},\ldots,y_n)\in J^{n-n_0}$ making the overall sum $0$, since the remaining $\lceil n^{2/3}\rceil$ entries, each ranging over the integer interval $J$, can contribute any integer of absolute value at most $\lceil n^{2/3}\rceil\lfloor C\sqrt{np}\rfloor$, and $\lceil n^{2/3}\rceil\lfloor C\sqrt{np}\rfloor\ge 2C\sqrt n\sqrt{np}$ for large $n$. Therefore, the number of sequences in question is at least
    \[\frac{1}{2}|J|^{n_0}\ge\frac12\left(C\sqrt{np}\right)^{n-n^{2/3}-1}\ge\exp(-\Cst{lem:basic_counting}(C)\,n)(\sqrt{np})^n\]
    for a suitable constant $\Cst{lem:basic_counting}(C)>0$ depending only on $\theta$, $T$, and $C$ (the factor $(\sqrt{np})^{-n^{2/3}-1}\ge\exp(-n^{2/3}\log n)$ and the factor $\frac12C^{\,n-n^{2/3}-1}$ are both $\ge\exp(-O(n))$), as desired.
\end{proof}

\subsection{Conditioning a Gaussian vector on an open set}
\label{app:tools-gaussian}

The following elementary facts are used in \Cref{sec:universal} and \Cref{app:inverse}.

\begin{lemma}[Conditioning on an open set]\label{lem:open-conditioning}
Let $X\sim\mc{N}(\gamma,\Sigma)$ with $\Sigma\in\mb{R}^{d\times d}$ positive definite and let $\mc{O}\subseteq\mb{R}^d$ be a nonempty open convex set. Then $\mb{P}[X\in\mc{O}]>0$, $\mb{E}[X\mid X\in\mc{O}]\in\mc{O}$, and $\mr{Cov}(X\mid X\in\mc{O})$ is positive definite.
\end{lemma}
\begin{proof}
The density $\varphi$ of $X$ is positive on all of $\mb{R}^d$, and $\mc{O}$ is a nonempty open set, hence has positive Lebesgue measure; therefore $\mb{P}[X\in\mc{O}]=\int_{\mc{O}}\varphi>0$, and the law of $X$ conditioned on $\{X\in\mc{O}\}$ has the density $\varphi\one_{\mc{O}}/\mb{P}[X\in\mc{O}]$, which is positive on $\mc{O}$ and vanishes off $\mc{O}$.

Let $\mu:=\mb{E}[X\mid X\in\mc{O}]$ and suppose $\mu\notin\mc{O}$. Since $\mc{O}$ is convex and open, the separating hyperplane theorem gives $u\in\mb{R}^d\setminus\{0\}$ with $u\cdot x\le u\cdot\mu$ for all $x\in\mc{O}$, and the inequality is strict for every $x\in\mc{O}$ because $\mc{O}$ is open. Hence $u\cdot X<u\cdot\mu$ almost surely under the conditioned law, and taking expectations, $u\cdot\mu=\mb{E}[u\cdot X\mid X\in\mc{O}]<u\cdot\mu$, a contradiction. Thus $\mu\in\mc{O}$.

Finally let $u\in\mb{R}^d\setminus\{0\}$. Then $u^\top\mr{Cov}(X\mid X\in\mc{O})\,u=\mr{Var}(u\cdot X\mid X\in\mc{O})$. If this vanished, the conditioned law would be supported on the hyperplane $\{x:u\cdot x=u\cdot\mu\}$, which has Lebesgue measure zero, whereas the conditioned law has a density that is positive on the open set $\mc{O}$ and therefore gives positive mass to $\mc{O}$ minus any Lebesgue-null set. Hence the variance is positive, and $\mr{Cov}(X\mid X\in\mc{O})$ is positive definite.
\end{proof}

\section{The fiber law: transference to the row model}
\label{app:local-transference}

This appendix proves the lemmas used in the proof of \Cref{prop:local-degree-transference} and \Cref{prop:fiber-degree-array-estimates}, in the order in which they depend on one another: graphicality of the arrays in the fiber, comparison of the graphical degree law $\msf{G}_{\Pi,m}$ with the conditioned row model, second-moment regularity in both models, and finally the exact-total lower bound and binomial regularity in the row model. Throughout, the notation is that of \Cref{def:local-admissibility,def:aux-degree-laws}, and the tools of \Cref{app:tools} are used freely.

\subsection{Graphicality of the arrays in the fiber}

\begin{lemma}[Automatic graphicality]\label{lem:auto-graphicality}
Fix $\theta\in(1/2,1)$, an integer $k\ge1$, and $T>1$. For all sufficiently large $N$, the following holds. Let $V$ be a finite set with $|V|=N$, let $p\in(T^{-1}N^{-\theta},TN^{-\theta})$, and let $\Pi\in\mr{Part}_k(V)$ and $m\in\mc{M}(\Pi)$ satisfy
\[
|\Pi[s]|\ge T^{-1}N
\]
and
\[
\left|m[s,t]-p|\Pi[s]||\Pi[t]|\right|\le T\frac{N^2p}{\sqrt{pN}}
\]
for every $s,t\in\{0,1\}^k$. If $\mbf{d}\in\mc{A}(\Pi)$ satisfies $m(\Pi,\mbf{d})=m$ and $\kappa_p(\Pi,\mbf{d})$, then $\mbf{d}\in\mc{D}(\Pi)$.
\end{lemma}

\begin{proof}
By \Cref{fact:graphical-blockwise}, it suffices to show that for every internal block pair $\{s,s\}$ the sequence $(\mbf{d}[x,s])_{x\in\Pi[s]}$ is the degree sequence of some simple graph on $\Pi[s]$, and that for every bipartite block pair $\{s,t\}$ the pair of sequences $(\mbf{d}[x,t])_{x\in\Pi[s]}$, $(\mbf{d}[y,s])_{y\in\Pi[t]}$ is the pair of degree sequences of some simple bipartite graph between $\Pi[s]$ and $\Pi[t]$.

Fix $s$. The sequence $(\mbf{d}[x,s])_{x\in\Pi[s]}$ has even sum $m[s,s]$, and by the assumptions $m[s,s]\asymp pN^2$. By $\kappa_p$,
\[
\Delta_s:=\max_{x\in\Pi[s]}\mbf{d}[x,s]\le p|\Pi[s]|+(pN)^{4/7}=O(pN).
\]
Since $\Delta_s(\Delta_s+1)=O(p^2N^2)=o(pN^2)$, we have $m[s,s]\ge\Delta_s(\Delta_s+1)$ for all sufficiently large $N$, and \Cref{lem:graphical-sufficient} implies that there is a simple graph on $\Pi[s]$ with degree sequence $(\mbf{d}[x,s])_{x\in\Pi[s]}$. Now fix $s\ne t$. The two sequences $(\mbf{d}[x,t])_{x\in\Pi[s]}$ and $(\mbf{d}[y,s])_{y\in\Pi[t]}$ have the common sum $m[s,t]\asymp pN^2$, and by $\kappa_p$ their maxima $\Delta_{s,t}$ and $\Delta_{t,s}$ are $O(pN)$, so $m[s,t]\ge\Delta_{s,t}\Delta_{t,s}$ for all sufficiently large $N$, and \Cref{lem:bipartite-graphical-sufficient} implies that there is a simple bipartite graph between $\Pi[s]$ and $\Pi[t]$ with these two degree sequences. Thus $\mbf{d}\in\mc{D}(\Pi)$.
\end{proof}

\subsection{Comparison of the graphical and binomial degree laws}

\begin{lemma}[Enumeration comparison]\label{lem:local-enumeration-comparison}
Fix $\theta\in(1/2,1)$, an integer $k\ge1$, and constants $T>1$ and $\phi\in(0,1/2)$. There is a constant $\Cst{lem:local-enumeration-comparison}=\Cst{lem:local-enumeration-comparison}(\theta,k,T)>0$, and for every $C_\Gamma>0$ a constant $\Cstp{lem:local-enumeration-comparison}(C_\Gamma)=\Cstp{lem:local-enumeration-comparison}(\theta,k,T,C_\Gamma)>1$, such that for all sufficiently large $N$ the following holds. Let $V$ be a finite set with $|V|=N$, let $p\in(T^{-1}N^{-\theta},TN^{-\theta})$, suppose that $(y,\lambda)$ is $(T,\phi,p)$-locally admissible with $y=(\Pi,m,1)$, and let $\mbf{d}\in\mc{A}(\Pi)$ satisfy $m(\Pi,\mbf{d})=m$ and $\kappa_p(\Pi,\mbf{d})$.
\begin{enumerate}[label=\textup{(\roman*)},ref=\textup{(\roman*)},leftmargin=*]
\item\label{lem:local-enumeration-comparison:weak} (Weak comparison.)
\begin{align*}
\msf{G}_{\Pi,m}(\mbf{d})&\le\exp\{\Cst{lem:local-enumeration-comparison}(pN)^{2/7}\}\,\msf{R}_{\Pi,\lambda}(\mbf{d}\mid m(\Pi,\mbf{D})=m),\\
\msf{G}_{\Pi,m}(\mbf{d})&\ge\exp\{-\Cst{lem:local-enumeration-comparison}(pN)^{2/7}\}\,\msf{R}_{\Pi,\lambda}(\mbf{d}\mid m(\Pi,\mbf{D})=m).
\end{align*}
\item\label{lem:local-enumeration-comparison:strong} (Strong comparison.) If $\mbf{d}$ also satisfies $\Gamma_{C_\Gamma,p}(\Pi,m,\mbf{d})$, then
\begin{align*}
\msf{G}_{\Pi,m}(\mbf{d})&\le \Cstp{lem:local-enumeration-comparison}(C_\Gamma)\,\msf{R}_{\Pi,\lambda}(\mbf{d}\mid m(\Pi,\mbf{D})=m),\\
\msf{G}_{\Pi,m}(\mbf{d})&\ge \Cstp{lem:local-enumeration-comparison}(C_\Gamma)^{-1}\,\msf{R}_{\Pi,\lambda}(\mbf{d}\mid m(\Pi,\mbf{D})=m).
\end{align*}
\end{enumerate}
\end{lemma}

\begin{proof}
\emph{Factorization over block pairs.} Write $\{m\}:=\{\mbf{d}'\in\mc{A}(\Pi):m(\Pi,\mbf{d}')=m\}$. For $\mbf{d}\in\{m\}$ and a block pair $\{s,t\}$ of $\Pi$, let $\mbf{d}_{s,s}:=(\mbf{d}[x,s])_{x\in\Pi[s]}$ be the degree sequence of $\mbf{d}$ on the internal block pair $\{s,s\}$, which has sum $m[s,s]$, and let $\mbf{d}_{s,t}:=\big((\mbf{d}[x,t])_{x\in\Pi[s]},(\mbf{d}[y,s])_{y\in\Pi[t]}\big)$ be its pair of degree sequences on the bipartite block pair $\{s,t\}$, which have common sum $m[s,t]$. The array $\mbf{d}$ is the same as the collection of these sequences, and both laws in the statement factor over block pairs, in the notation of \Cref{def:graph_deg_seq,def:bigraph_deg_seq}:
\begin{equation}\label{eq:enum-factor}
\begin{aligned}
\msf{G}_{\Pi,m}(\mbf{d})&=\prod_{s\in\{0,1\}^k}\mb{P}_{G(n[s],m[s,s]/2)}(\mbf{d}_{s,s})\prod_{\{s,t\}:\,s\ne t}\mb{P}_{G(n[s],n[t],m[s,t])}(\mbf{d}_{s,t}),\\
\msf{R}_{\Pi,\lambda}(\mbf{d}\mid\{m\})&=\prod_{s\in\{0,1\}^k}\mb{P}_{B(n[s],m[s,s]/2)}(\mbf{d}_{s,s})\prod_{\{s,t\}:\,s\ne t}\mb{P}_{B(n[s],n[t],m[s,t])}(\mbf{d}_{s,t}).
\end{aligned}
\end{equation}
The first identity is \Cref{def:aux-degree-laws}: under $\msf{G}_{\Pi,m}$, the restrictions of $G$ to the block pairs are independent uniformly random simple graphs on $\Pi[s]$ with $m[s,s]/2$ edges, respectively bipartite graphs between $\Pi[s]$ and $\Pi[t]$ with $m[s,t]$ edges. For the second identity, $\msf{R}_{\Pi,\lambda}$ is a product law and $\{m\}$ is an intersection of one constraint per block pair, so the block-pair sequences remain independent after conditioning; and for independent binomial variables with a common success probability $\lambda$, every point probability on the event that their sum equals $M$ carries the same factor $\lambda^{M}(1-\lambda)^{\text{trials}-M}$, so the conditional law given the sum does not depend on $\lambda$ and coincides with the laws $\mb{P}_{B(\cdot)}$, which were defined with success probability $1/2$. In particular the numerical values of $\lambda[s,t]$ play no role in this proof.

\emph{Hypotheses of the enumeration theorems.} Write $\bar d_{s,t}:=m[s,t]/n[s]$ for the average of the sequence $(\mbf{d}[x,t])_{x\in\Pi[s]}$. By \cref{LA:sizes}, $n[s]\asymp N$; by \cref{LA:sizes} and \cref{LA:edge-scale}, $\bar d_{s,t}\asymp pN$ and
\begin{equation}\label{eq:enum-avg}
\left|\bar d_{s,t}-p\,n[t]\right|\le\frac{TN^2p}{\sqrt{pN}\,n[s]}\le T^2\sqrt{pN}.
\end{equation}
Hence $\kappa_p(\Pi,\mbf{d})$ and \eqref{eq:enum-avg} give, for all $s,t\in\{0,1\}^k$ and $x\in\Pi[s]$,
\begin{equation}\label{eq:enum-entry}
\left|\mbf{d}[x,t]-\bar d_{s,t}\right|\le(pN)^{4/7}+T^2\sqrt{pN}\le2(pN)^{4/7}\le\bar d_{s,t}^{\,7/12}
\end{equation}
for large $N$, which is the maximum-degree hypothesis of \Cref{thm:graph_counting,thm:bigraph_counting} with $\alpha=7/12$, applied to each block pair after identifying $\Pi[s]$ with $[n[s]]$. Their remaining hypotheses hold as well: the densities $\mu_s$ and $\mu_{s,t}$ defined below are $\asymp p\to0$; the average degrees are $\asymp pN\ge T^{-1}N^{1-\theta}$, so $\log^KN/N=o(\mu_s)$ and $N\log^KN=o(m[s,t])$ for every fixed $K$; and for $\alpha=7/12$ the condition $(\ell+n)^{5-5\alpha}=o(\ell n\,m^{3-5\alpha})$ of \Cref{thm:bigraph_counting}, with $\ell=n[s]$, $n=n[t]$, $m=m[s,t]$, reads $N^{25/12}=o(N^{2}(pN^2)^{1/12})$ up to constants, that is, $(pN)^{1/12}\to\infty$. Every $\mbf{d}$ in the statement is graphical by \Cref{lem:auto-graphicality}, whose hypotheses are supplied by \cref{LA:sizes} and \cref{LA:edge-scale}; thus each block-pair degree sequence of $\mbf{d}$ is realized by a simple graph, respectively bipartite graph, so the left-hand sides in \eqref{eq:enum-ratios} below are positive.

\emph{The block-pair ratios.} Put $\mu_s:=\bar d_{s,s}/(n[s]-1)$, $\mu_{s,t}:=m[s,t]/(n[s]n[t])$, both $\asymp p$, and
\[
\gamma_2(\mbf{d}_{s,s}):=\frac{1}{(n[s]-1)^2}\sum_{x\in\Pi[s]}\left(\mbf{d}[x,s]-\bar d_{s,s}\right)^2,\qquad
\sigma^2_{s,t}:=\frac{1}{n[s]}\sum_{x\in\Pi[s]}\left(\mbf{d}[x,t]-\bar d_{s,t}\right)^2 .
\]
The error terms of \Cref{thm:graph_counting,thm:bigraph_counting} are $o(1)$ in the present ranges: with $\alpha=7/12$ they are $O(\log^2N/\sqrt N+(pN)^{-1/12})$, respectively $O(\log^2N/\sqrt N+(pN)^{-25/12}\cdot p^2N^2)=O(\log^2N/\sqrt N+(pN)^{-1/12})$. The theorems therefore read
\begin{equation}\label{eq:enum-ratios}
\begin{aligned}
\frac{\mb{P}_{G(n[s],m[s,s]/2)}(\mbf{d}_{s,s})}{\mb{P}_{B(n[s],m[s,s]/2)}(\mbf{d}_{s,s})}&=(1+o(1))\exp\left(\frac14-\frac{\gamma_2(\mbf{d}_{s,s})^2}{4\mu_s^2(1-\mu_s)^2}\right),\\
\frac{\mb{P}_{G(n[s],n[t],m[s,t])}(\mbf{d}_{s,t})}{\mb{P}_{B(n[s],n[t],m[s,t])}(\mbf{d}_{s,t})}&=(1+o(1))H_{s,t},
\end{aligned}
\end{equation}
where
\[
H_{s,t}:=\exp\left(-\frac12\left(1-\frac{\sigma^2_{s,t}}{\bar d_{s,t}(1-\mu_{s,t})}\right)\left(1-\frac{\sigma^2_{t,s}}{\bar d_{t,s}(1-\mu_{s,t})}\right)\right).
\]
Taking logarithms in \eqref{eq:enum-factor} and substituting \eqref{eq:enum-ratios},
\begin{equation}\label{eq:enum-log}
\log\frac{\msf{G}_{\Pi,m}(\mbf{d})}{\msf{R}_{\Pi,\lambda}(\mbf{d}\mid\{m\})}
=\sum_{s\in\{0,1\}^k}\left(\frac14-\frac{\gamma_2(\mbf{d}_{s,s})^2}{4\mu_s^2(1-\mu_s)^2}\right)+\sum_{\{s,t\}:\,s\ne t}\log H_{s,t}+o(1).
\end{equation}

\emph{Proof of \ref{lem:local-enumeration-comparison:weak}.} By \eqref{eq:enum-entry},
\[
\gamma_2(\mbf{d}_{s,s})\le\frac{4n[s](pN)^{8/7}}{(n[s]-1)^2}=O\left(p(pN)^{1/7}\right),\qquad
\frac{\gamma_2(\mbf{d}_{s,s})^2}{4\mu_s^2(1-\mu_s)^2}=O\left(\frac{p^2(pN)^{2/7}}{p^2}\right)=O\left((pN)^{2/7}\right),
\]
and
\[
\sigma^2_{s,t}\le4(pN)^{8/7},\qquad
\frac{\sigma^2_{s,t}}{\bar d_{s,t}(1-\mu_{s,t})}=O\left((pN)^{1/7}\right),
\]
so that $|\log H_{s,t}|\le\frac12\left(1+O\left((pN)^{1/7}\right)\right)^2=O\left((pN)^{2/7}\right)$.
Each of the $2^k+\binom{2^k}{2}$ terms in \eqref{eq:enum-log} is therefore $O((pN)^{2/7})$ in absolute value, with implied constants depending only on $\theta,k,T$; this is \ref{lem:local-enumeration-comparison:weak}.

\emph{Proof of \ref{lem:local-enumeration-comparison:strong}.} If $\Gamma_{C_\Gamma,p}(\Pi,m,\mbf{d})$ holds, then $\sum_{x\in\Pi[s]}(\mbf{d}[x,t]-\bar d_{s,t})^2\le C_\Gamma pN^2$ for all $s,t$, so in place of the bounds above,
\begin{align*}
\gamma_2(\mbf{d}_{s,s})&\le\frac{C_\Gamma pN^2}{(n[s]-1)^2}=O(C_\Gamma p),&
\frac{\gamma_2(\mbf{d}_{s,s})^2}{4\mu_s^2(1-\mu_s)^2}&=O\left(C_\Gamma^2\right),\\
\sigma^2_{s,t}&\le\frac{C_\Gamma pN^2}{n[s]}=O(C_\Gamma pN),&
|\log H_{s,t}|&=O\left((1+C_\Gamma)^2\right),
\end{align*}
and every term in \eqref{eq:enum-log} is bounded in absolute value by a constant depending only on $\theta,k,T,C_\Gamma$; this is \ref{lem:local-enumeration-comparison:strong}.
\end{proof}

\subsection{Second-moment regularity}

\begin{lemma}[Graphical fiber second-moment regularity]\label{lem:local-gamma-concentration}
Fix $\theta\in(1/2,1)$, an integer $k\ge1$, and constants $T>1$ and $\phi\in(0,1/2)$. There exists a constant $\Cst{lem:local-gamma-concentration}=\Cst{lem:local-gamma-concentration}(\theta,k,T,\phi)\ge1$ such that for all sufficiently large $N$, all $p\in(T^{-1}N^{-\theta},TN^{-\theta})$, and every finite set $V$ with $|V|=N$, if $(y,\lambda)$ is $(T,\phi,p)$-locally admissible with $y=(\Pi,m,1)\in\mc{Y}_k(V)$, then
\[
\msf{G}_{\Pi,m}\left(\Gamma_{\Cst{lem:local-gamma-concentration},p}(\Pi,m,\mbf{D})\mid\mc{I}_\Pi,\kappa_p(\Pi,\mbf{D})\right)\ge1-e^{-N}.
\]
\end{lemma}

\begin{proof}
Let $P_{\Pi,p}$ denote the law of $\delta(\Pi,G)$ for $G\sim\mb{G}(V,p)$, with the partition $\Pi$ fixed, and abbreviate $\{m\}:=\{\mbf{D}\in\mc{A}(\Pi):m(\Pi,\mbf{D})=m\}$, $\kappa_p:=\kappa_p(\Pi,\cdot)$, and $\Gamma_C:=\Gamma_{C,p}(\Pi,m,\cdot)$ for $C>0$. By \Cref{def:aux-degree-laws}, $\msf{G}_{\Pi,m}=P_{\Pi,p}[\,\cdot\mid\{m\}]$, so for every event $E\subseteq\mc{A}(\Pi)$,
\begin{equation}\label{eq:G-as-conditional}
\msf{G}_{\Pi,m}(E)=\frac{P_{\Pi,p}\left[E\cap\{m\}\right]}{P_{\Pi,p}\left[\{m\}\right]}.
\end{equation}
Applying \eqref{eq:G-as-conditional} to numerator and denominator, the quantity to be bounded is
\begin{equation}\label{eq:gamma-goal}
\msf{G}_{\Pi,m}\left(\Gamma_{\Cst{lem:local-gamma-concentration}}^c\mid\mc{I}_\Pi\cap\kappa_p\right)
=\frac{\msf{G}_{\Pi,m}\left(\Gamma_{\Cst{lem:local-gamma-concentration}}^c\cap\mc{I}_\Pi\cap\kappa_p\right)}{\msf{G}_{\Pi,m}\left(\mc{I}_\Pi\cap\kappa_p\right)}
=\frac{P_{\Pi,p}\left[\Gamma_{\Cst{lem:local-gamma-concentration}}^c\cap\mc{I}_\Pi\cap\kappa_p\cap\{m\}\right]}{P_{\Pi,p}\left[\mc{I}_\Pi\cap\kappa_p\cap\{m\}\right]}.
\end{equation}
We will exhibit a constant $C_5$, depending only on $\theta,k,T,\phi$ and not on $\Cst{lem:local-gamma-concentration}$, such that the denominator is at least $e^{-C_5N}$ (Steps 1--3), and then choose $\Cst{lem:local-gamma-concentration}$ so that the numerator is at most $e^{-(C_5+1)N}$ (Step 4). The constants $C_0,\ldots,C_5$ and $C_\ast$ below are internal to this proof and depend only on $\theta,k,T,\phi$.

\emph{Step 1: lower bound for $P_{\Pi,p}[\{m\}]$.}
Under $P_{\Pi,p}$, the edge counts of the block pairs of $\Pi$ are independent, so $P_{\Pi,p}[\{m\}]$ is the product, over the block pairs of $\Pi$, of the probabilities that the edge count of the internal block pair $\{s,s\}$ equals $m[s,s]/2$, respectively that the edge count of the bipartite block pair $\{s,t\}$ equals $m[s,t]$. Each such edge count is a binomial variable $\mr{Bin}(R,p)$ with $R=\binom{n[s]}{2}$, respectively $R=n[s]n[t]$, so that $Rp\asymp N^2p$ by \cref{LA:sizes}, and by \cref{LA:edge-scale} the target value $j$ ($=m[s,s]/2$, respectively $m[s,t]$) satisfies $|j-Rp|=O(N^2p/\sqrt{pN})$. We use the following elementary lower bound for binomial point probabilities. For an integer $0<j<R$, put $q:=j/R$. The term $\binom Rj q^j(1-q)^{R-j}$ is the largest of the $R+1$ terms of the binomial distribution $\mr{Bin}(R,q)$, whose mode is at $\lfloor(R+1)q\rfloor=j$, and these terms sum to $1$, so $\binom Rj q^j(1-q)^{R-j}\ge(R+1)^{-1}$. Therefore
\begin{align*}
\mb P[\mr{Bin}(R,p)=j]=\binom Rj p^j(1-p)^{R-j}
&\ge\frac{1}{R+1}\left(\frac pq\right)^{j}\left(\frac{1-p}{1-q}\right)^{R-j}
=\frac{1}{R+1}\exp\left(-R\,D(q\,\|\,p)\right)\\
&\ge\frac{1}{R+1}\exp\left(-\frac{(j-Rp)^2}{Rp(1-p)}\right),
\end{align*}
where $D(q\|p)=q\log\frac qp+(1-q)\log\frac{1-q}{1-p}$, and the last step uses $\log u\le u-1$ twice: $q\log\frac qp\le\frac{q(q-p)}{p}$ and $(1-q)\log\frac{1-q}{1-p}\le\frac{(1-q)(p-q)}{1-p}$, whose sum is $\frac{(q-p)^2}{p(1-p)}$. In our situation $(j-Rp)^2/(Rp)=O\left(N^4p^2/(pN\cdot N^2p)\right)=O(N)$ and $\log(R+1)=O(\log N)$, so each of the $2^k+\binom{2^k}{2}$ probabilities is at least $\exp(-C_0N)$ for a constant $C_0=C_0(\theta,k,T)$, and hence, with $C_1:=(2^k+\binom{2^k}{2})C_0$,
\begin{equation}\label{eq:m-lower}
P_{\Pi,p}\left[\{m\}\right]\ge \exp(-C_1N).
\end{equation}

\emph{Step 2: a large set $\mc{E}_0$ of good arrays.} Let $\mc E_0$ be the set of arrays $\mbf d\in\mc{A}(\Pi)$ with $m(\Pi,\mbf d)=m$ and
\begin{equation}\label{eq:E0-window}
\left|\mbf d[x,t]-\frac{m[s,t]}{n[s]}\right|\le \frac{0.01}{T2^k}\sqrt{pN}
\qquad\text{for every }s,t\in\{0,1\}^k\text{ and }x\in\Pi[s].
\end{equation}
Every array in $\mc E_0$ satisfies $\mc{I}_\Pi$. Indeed, fix $s\in\{0,1\}^k$, $x\in\Pi[s]$, and $1\le r\le k-1$. Summing \eqref{eq:E0-window} over the $2^k$ values of $t$,
\[
\left|\sum_{t\in\{0,1\}^k}(-1)^{t[r-1]}\mbf d[x,t]-\frac{1}{n[s]}\sum_{t\in\{0,1\}^k}(-1)^{t[r-1]}m[s,t]\right|\le\frac{0.01}{T}\sqrt{pN},
\]
while \cref{LA:separation} together with $n[s]\le N$ gives
\[
\left|\frac{1}{n[s]}\sum_{t\in\{0,1\}^k}(-1)^{t[r-1]}m[s,t]\right|\ge\frac{1}{n[s]}\cdot T^{-1}\frac{N^2p}{\sqrt{pN}}\ge T^{-1}\sqrt{pN},
\]
with the sign of the sum equal to $(-1)^{s[r]}$. Hence $\sum_{t}(-1)^{t[r-1]}\mbf d[x,t]$ is nonzero with sign $(-1)^{s[r]}$, which is exactly the $r$th inequality in $\mc{I}_s(\mbf d[x,\bullet])$ (both the weak and the strict form). Every array in $\mc E_0$ also satisfies $\kappa_p$: by \cref{LA:sizes} and \cref{LA:edge-scale}, $|m[s,t]/n[s]-p\,n[t]|\le TN^2p/(\sqrt{pN}\,n[s])\le T^2\sqrt{pN}$, so \eqref{eq:E0-window} gives $|\mbf d[x,t]-p\,n[t]|\le(T^2+0.01)\sqrt{pN}\le(pN)^{4/7}$ for all sufficiently large $N$. Finally, squaring and summing \eqref{eq:E0-window} over $x\in\Pi[s]$ gives $\sum_{x\in\Pi[s]}(\mbf d[x,t]-m[s,t]/n[s])^2\le 10^{-4}pN^2$, so every array in $\mc E_0$ satisfies $\Gamma_1$, and hence $\Gamma_{C}$ for every $C\ge1$. In summary,
\begin{equation}\label{eq:E0-inclusion}
\mc{E}_0\subseteq\{m\}\cap\mc{I}_\Pi\cap\kappa_p\cap\Gamma_1\subseteq\{m\}\cap\mc{I}_\Pi\cap\kappa_p\cap\Gamma_{C}\qquad(C\ge1),
\end{equation}
and by \Cref{lem:auto-graphicality} every array in $\mc{E}_0$ is graphical.

\emph{Step 3: lower bound for the denominator of \eqref{eq:gamma-goal}.} We first lower bound $\msf{G}_{\Pi,m}(\mc{E}_0)$ by summing pointwise bounds over $\mc{E}_0$. We claim that, under the law $\msf{R}_{\Pi,\lambda}(\cdot\mid\{m\})$, each array in $\mc E_0$ has probability at least
\begin{equation}\label{eq:E0-point}
\msf{R}_{\Pi,\lambda}\left(\mbf{d}\mid\{m\}\right)\ge(pN)^{-2^kN/2}\exp(-C_2N)\qquad\text{for every }\mbf{d}\in\mc{E}_0,
\end{equation}
with $C_2$ depending only on $\theta,k,T$. To see this, fix $\mbf{d}\in\mc{E}_0$. Since $\mbf{d}\in\{m\}$, conditioning can only increase its probability, $\msf{R}_{\Pi,\lambda}(\mbf{d}\mid\{m\})\ge\msf{R}_{\Pi,\lambda}(\mbf{d})$, and by the product structure of $\msf{R}_{\Pi,\lambda}$,
\[
\msf{R}_{\Pi,\lambda}(\mbf{d})=\prod_{s\in\{0,1\}^k}\prod_{x\in\Pi[s]}\prod_{t\in\{0,1\}^k}\mb{P}\left[\mr{Bin}\left(n[t]-\one_{s=t},\lambda[s,t]\right)=\mbf{d}[x,t]\right],
\]
a product of $2^kN$ binomial point probabilities. Each of them is a point probability at a bounded number of standard deviations from the mean: by \eqref{eq:E0-window}, \cref{LA:sizes}, \cref{LA:edge-scale}, and \cref{LA:tilt},
\[
\left|\mbf{d}[x,t]-p\,n[t]\right|\le\left|\mbf{d}[x,t]-\frac{m[s,t]}{n[s]}\right|+\left|\frac{m[s,t]}{n[s]}-p\,n[t]\right|\le(0.01+T^2)\sqrt{pN},
\]
while the parameters $n':=n[t]-\one_{s=t}$ and $p':=\lambda[s,t]$ of the binomial satisfy $|n'-n[t]|\le1$ and $|p'-p|\le Tp/\sqrt{pN}$. We may therefore apply \Cref{lem:binomial_estimate} with its $n$ equal to $n[t]$ (note $p\in(T'^{-1}n[t]^{-\theta},T'n[t]^{-\theta})$ for $T':=T^{1+\theta}$, by \cref{LA:sizes}), its $(n',p')$ as above, and $\xi:=\mbf{d}[x,t]-p\,n[t]$ in the role of the variable $x$ of the lemma; by the previous display, $\xi$ lies in the window $|\xi|\le\sqrt{p\,n[t]}\log n[t]$ of the lemma. In its second formula the parameter $\alpha=\sqrt{pn}\log\left(p'(n'-pn)/((1-p')pn)\right)$ is bounded by a constant depending only on $T$, because $p'/p=1+O(1/\sqrt{pN})$ and $(n'-pn)/((1-p')n)=1+O(p+1/n)$; and the normalization satisfies $\mc{N}_2=(1+o(1))(2\pi p\,n[t])^{-1/2}$, as one sees by summing the formula over the window (which carries all but $N^{-\omega(1)}$ of the mass, by Chernoff's inequality). Hence
\begin{align*}
\mb{P}\left[\mr{Bin}\left(n[t]-\one_{s=t},\lambda[s,t]\right)=\mbf{d}[x,t]\right]
&\ge\frac{1+o(1)}{\sqrt{2\pi p\,n[t]}}\exp\left(-\frac{\left((0.01+T^2)\sqrt{pN}+|\alpha|\sqrt{p\,n[t]}\right)^2}{2p\,n[t]}\right)\\
&\ge\frac{c}{\sqrt{pN}}
\end{align*}
for a constant $c=c(\theta,k,T)>0$, where we used $T^{-1}N\le n[t]\le N$ in the last step. Taking the product over the $2^kN$ entries gives \eqref{eq:E0-point} with $C_2:=2^k\log(1/c)$.

Next we count $\mc{E}_0$. An array $\mbf{d}$ with $m(\Pi,\mbf{d})=m$ is specified by choosing, for each of the $2^{2k}$ ordered pairs $(s,t)$, the integer sequence $(\mbf{d}[x,t])_{x\in\Pi[s]}$ with sum $m[s,t]$. Fix $(s,t)$, write $w:=0.01\sqrt{pN}/(T2^k)$, and write $m[s,t]=q\,n[s]+r$ with $q\in\mb{Z}$ and $0\le r<n[s]$. Let $\mbf{b}$ be the sequence taking the value $q+1$ on $r$ fixed elements of $\Pi[s]$ and $q$ on the remaining $n[s]-r$; then $\sum_x\mbf{b}[x]=m[s,t]$ and $|\mbf{b}[x]-m[s,t]/n[s]|\le1$ for every $x$. For every integer sequence $(y_x)_{x\in\Pi[s]}$ with $\sum_xy_x=0$ and $|y_x|\le w-1$, the sequence $\mbf{b}+y$ has sum $m[s,t]$ and satisfies \eqref{eq:E0-window}; its entries lie in $[0,n[t]-\one_{s=t}]$ for all sufficiently large $N$, because they are within $w$ of $m[s,t]/n[s]=p\,n[t]+O(\sqrt{pN})$, so the resulting arrays lie in $\mc{A}(\Pi)$ and hence in $\mc{E}_0$. By \Cref{lem:basic_counting}, applied with $n=n[s]$, with its constant $T$ equal to $T^{1+\theta}$ (so that its hypothesis on $p$ holds, because $n[s]\in[T^{-1}N,N]$ by \cref{LA:sizes}), and with its parameter $C$ equal to $c_0:=0.005/(T2^k)$ (so that $c_0\sqrt{n[s]p}\le w-1$ for large $N$), the number of such sequences $y$ is at least
\[
\exp\left(-\Cof{lem:basic_counting}{\theta,T^{1+\theta},c_0}\,n[s]\right)\left(\sqrt{n[s]p}\right)^{n[s]}\ge\exp\left(-\left(\Cof{lem:basic_counting}{\theta,T^{1+\theta},c_0}+\tfrac12\log T\right)N\right)(pN)^{n[s]/2},
\]
using $T^{-1}N\le n[s]\le N$. Multiplying over the $2^{2k}$ ordered pairs $(s,t)$, and using $\sum_{s,t}n[s]=2^kN$, we obtain, with $C_3:=2^{2k}(\Cof{lem:basic_counting}{\theta,T^{1+\theta},c_0}+\tfrac12\log T)$,
\begin{equation}\label{eq:E0-count}
|\mc{E}_0|\ge(pN)^{2^kN/2}\exp(-C_3N).
\end{equation}
Now, by \eqref{eq:E0-inclusion}, every $\mbf{d}\in\mc{E}_0$ is a graphical array satisfying $m(\Pi,\mbf{d})=m$, $\kappa_p$, and $\Gamma_1$, so \Cref{lem:local-enumeration-comparison}\ref{lem:local-enumeration-comparison:strong}, applied with its parameter $C_\Gamma$ equal to $1$, gives $\msf{G}_{\Pi,m}(\mbf{d})\ge C_\ast^{-1}\msf{R}_{\Pi,\lambda}(\mbf{d}\mid\{m\})$ for every $\mbf{d}\in\mc{E}_0$, where $C_\ast:=\Cofp{lem:local-enumeration-comparison}{\theta,k,T,1}$. Summing this over $\mbf{d}\in\mc{E}_0$ and using \eqref{eq:E0-point} and \eqref{eq:E0-count},
\begin{equation}\label{eq:G-E0}
\begin{aligned}
\msf{G}_{\Pi,m}(\mc{E}_0)=\sum_{\mbf{d}\in\mc{E}_0}\msf{G}_{\Pi,m}(\mbf{d})
&\ge C_\ast^{-1}\,|\mc{E}_0|\,(pN)^{-2^kN/2}e^{-C_2N}\\
&\ge C_\ast^{-1}e^{-(C_2+C_3)N}
\ge e^{-C_4N},\qquad C_4:=C_2+C_3+\log C_\ast.
\end{aligned}
\end{equation}
Unwinding the definition of $\msf{G}_{\Pi,m}$ through \eqref{eq:G-as-conditional} and using \eqref{eq:m-lower},
\begin{equation}\label{eq:P-E0}
P_{\Pi,p}\left[\mc{E}_0\right]=P_{\Pi,p}\left[\mc{E}_0\cap\{m\}\right]=\msf{G}_{\Pi,m}(\mc{E}_0)\,P_{\Pi,p}\left[\{m\}\right]\ge e^{-C_4N}e^{-C_1N}=:e^{-C_5N}.
\end{equation}
Finally, by \eqref{eq:E0-inclusion} the denominator of \eqref{eq:gamma-goal} satisfies
\begin{equation}\label{eq:gamma-denominator}
P_{\Pi,p}\left[\mc{I}_\Pi\cap\kappa_p\cap\{m\}\right]\ge P_{\Pi,p}\left[\mc{E}_0\right]\ge e^{-C_5N}.
\end{equation}
The constant $C_5=C_1+C_4$ depends only on $\theta,k,T,\phi$; in particular it does not depend on $\Cst{lem:local-gamma-concentration}$.

\emph{Step 4: upper bound for the numerator of \eqref{eq:gamma-goal}, and choice of $\Cst{lem:local-gamma-concentration}$.} For $C>0$, on the event $\{m\}$ the assertions $\Gamma_{C}(\mbf{D})=\Gamma_{C,p}(\Pi,m,\mbf{D})$ and $\Gamma_{C,p}(\Pi,m(\Pi,\mbf{D}),\mbf{D})$ coincide, and the latter is the assertion that, for every internal block pair $\{s,s\}$ and every bipartite block pair $\{s,t\}$, the corresponding block-pair degree sequences of $G$ have sum of squared deviations from their means at most $CpN^2$. We retain the event $\kappa_p$ in the numerator because it supplies the truncation hypothesis of \Cref{lem:gamma_bound}: $\kappa_p$ asserts that $|\mbf{D}[x,t]-p\,n[t]|\le(pN)^{4/7}$ for all $x\in V$ and $t\in\{0,1\}^k$, and $(pN)^{4/7}\le p\,n[t]$ for all sufficiently large $N$ because $n[t]\ge T^{-1}N$ by \cref{LA:sizes}.
Under $P_{\Pi,p}$, the graph induced on $\Pi[s]$ is distributed as $\mb{G}(n[s],p)$ and the bipartite graph between $\Pi[s]$ and $\Pi[t]$ as $\mb{G}(n[s],n[t],p)$. Apply \Cref{lem:gamma_bound}\cref{lem:gamma_bound:graph} to each internal block pair, with $n=n[s]$, and \Cref{lem:gamma_bound}\cref{lem:gamma_bound:bipartite} to each bipartite block pair, with $n=n[t]$ and $\ell=n[s]$, in both cases with the constant $T$ of the lemma replaced by $T^{1+\theta}$ (so that its hypotheses on $p$ and on $\ell$ hold) and with $K:=T(2C_5+2)$, and put $\Cst{lem:local-gamma-concentration}:=\max\{1,\Cof{lem:gamma_bound}{\theta,T^{1+\theta},K}\}$, which depends only on $\theta,k,T,\phi$. On $\kappa_p$ the truncation hypotheses of the lemma hold, by the previous paragraph. Since $n[s],n[t]\ge T^{-1}N$, each application has failure probability at most $e^{-KN/T}=e^{-(2C_5+2)N}$, and $\Cst{lem:local-gamma-concentration}pn[s]^2\le \Cst{lem:local-gamma-concentration}pN^2$ (respectively $\Cst{lem:local-gamma-concentration}pn[s]n[t]\le \Cst{lem:local-gamma-concentration}pN^2$), so a union bound over the $2^k+\binom{2^k}{2}\le2^{2k}$ blocks and block pairs gives
\begin{equation}\label{eq:gamma-numerator}
P_{\Pi,p}\left[\Gamma_{\Cst{lem:local-gamma-concentration}}^c\cap\mc{I}_\Pi\cap\kappa_p\cap\{m\}\right]
\le P_{\Pi,p}\left[\Gamma_{\Cst{lem:local-gamma-concentration},p}(\Pi,m(\Pi,\mbf{D}),\mbf{D})^c\cap\kappa_p\right]
\le2^{2k}e^{-(2C_5+2)N}\le e^{-(C_5+1)N}
\end{equation}
for all sufficiently large $N$. Inserting \eqref{eq:gamma-denominator} and \eqref{eq:gamma-numerator} into \eqref{eq:gamma-goal},
\[
\msf{G}_{\Pi,m}\left(\Gamma_{\Cst{lem:local-gamma-concentration}}^c\mid\mc{I}_\Pi\cap\kappa_p\right)
=\frac{P_{\Pi,p}\left[\Gamma_{\Cst{lem:local-gamma-concentration}}^c\cap\mc{I}_\Pi\cap\kappa_p\cap\{m\}\right]}{P_{\Pi,p}\left[\mc{I}_\Pi\cap\kappa_p\cap\{m\}\right]}
\le\frac{e^{-(C_5+1)N}}{e^{-C_5N}}=e^{-N}.
\]
\end{proof}

\begin{lemma}[Row-model second-moment regularity after exact totals]\label{lem:local-row-gamma}
Fix $\theta\in(1/2,1)$, an integer $k\ge1$, and constants $T>1$ and $\phi\in(0,1/2)$. For all sufficiently large $N$, all $p\in(T^{-1}N^{-\theta},TN^{-\theta})$, and every finite set $V$ with $|V|=N$, if $(y,\lambda)$ is $(T,\phi,p)$-locally admissible with $y=(\Pi,m,1)\in\mc{Y}_k(V)$, then
\[
\msf{R}_{\Pi,\lambda}\left(\Gamma_{\Cof{lem:local-gamma-concentration}{\theta,k,T,\phi},p}(\Pi,m,\mbf{D})\mid \mc{I}_\Pi,\kappa_p(\Pi,\mbf{D}),m(\Pi,\mbf{D})=m\right)=1-o(1).
\]
\end{lemma}

\begin{proof}
Abbreviate $\Gamma:=\Gamma_{\Cof{lem:local-gamma-concentration}{\theta,k,T,\phi},p}(\Pi,m,\cdot)$, $\kappa_p:=\kappa_p(\Pi,\cdot)$, and $\{m\}:=\{\mbf{D}\in\mc{A}(\Pi):m(\Pi,\mbf{D})=m\}$, and write $C_0:=\Cof{lem:local-enumeration-comparison}{\theta,k,T}$ for the constant of the weak comparison \Cref{lem:local-enumeration-comparison}\ref{lem:local-enumeration-comparison:weak}. By Bayes' rule,
\[
\msf{R}_{\Pi,\lambda}\left(\Gamma^c\mid \mc{I}_\Pi,\kappa_p,\{m\}\right)
=\frac{\msf{R}_{\Pi,\lambda}\left(\Gamma^c\cap\mc{I}_\Pi\cap\kappa_p\mid\{m\}\right)}{\msf{R}_{\Pi,\lambda}\left(\mc{I}_\Pi\cap\kappa_p\mid\{m\}\right)}.
\]
Both the numerator and the denominator are sums of $\msf{R}_{\Pi,\lambda}(\mbf{d}\mid\{m\})$ over arrays $\mbf{d}\in\{m\}\cap\kappa_p$, and for every such array the weak comparison gives
\[
\exp\{-C_0(pN)^{2/7}\}\,\msf{G}_{\Pi,m}(\mbf{d})\le\msf{R}_{\Pi,\lambda}(\mbf{d}\mid\{m\})\le\exp\{C_0(pN)^{2/7}\}\,\msf{G}_{\Pi,m}(\mbf{d}).
\]
Summing the upper bound over the arrays in the numerator and the lower bound over the arrays in the denominator, and then applying Bayes' rule for $\msf{G}_{\Pi,m}$ and \Cref{lem:local-gamma-concentration},
\begin{align*}
\msf{R}_{\Pi,\lambda}\left(\Gamma^c\mid \mc{I}_\Pi,\kappa_p,\{m\}\right)
&\le\exp\{2C_0(pN)^{2/7}\}\,\frac{\msf{G}_{\Pi,m}\left(\Gamma^c\cap\mc{I}_\Pi\cap\kappa_p\right)}{\msf{G}_{\Pi,m}\left(\mc{I}_\Pi\cap\kappa_p\right)}\\
&=\exp\{2C_0(pN)^{2/7}\}\,\msf{G}_{\Pi,m}\left(\Gamma^c\mid \mc{I}_\Pi\cap\kappa_p\right)
\le\exp\{2C_0(pN)^{2/7}-N\},
\end{align*}
which is $o(1)$ because $(pN)^{2/7}=o(N)$.
\end{proof}

\subsection{Exact totals and binomial regularity in the row model}

The exact-total lower bound rests on the following local central limit theorem for sums of independent copies of a conditioned binomial vector.
\begin{theorem}
\label{thm:fourier_BE}
    Fix a real parameter $\theta\in(1/2,1)$, a positive integer $d\ge 1$, a nonnegative integer $r$, an integer matrix $M\in\mb{Z}^{r\times d}$ with nonzero, pairwise orthogonal rows, an $r$-dimensional vector inequality relation $\ge'\in\{\ge,>\}^r$, and a real constant $T>1$. For all sufficiently large $n\in\mb{Z}$ and $p\in(T^{-1}n^{-\theta},Tn^{-\theta})$, the following is true. When $r=0$, we use the convention $\norm{M\eta}_\infty=0$ and interpret $MA\ge'0$ as the whole sample space.

    Select $\eta\in((T^{-1}n,Tn)\cap\mb{Z})^d$ such that $\norm{M\eta}_\infty < T\frac{n}{\sqrt{pn}}$ and select an integer $m\in(T^{-1}n,Tn)$. Select $q\in\left(p-T\frac{p}{\sqrt{pn}}, p+T\frac{p}{\sqrt{pn}}\right)^d$ such that $m\mb{E}[A|MA\ge'0]\in\mb{Z}^d$ where we sample $A_t\sim\mr{Bin}(\eta_t,q_t)$ all independently.

    Let $\rho$ be the probability distribution on $\mb{Z}^d$ given by sampling $A_t\sim\mr{Bin}(\eta_t,q_t)$ all independently, and then conditioning on $MA\ge'0$.
    
    Sample $m$ independent copies $A^{(1)},\ldots,A^{(m)}\sim\rho$, and let $S=A^{(1)}+\cdots+A^{(m)}$. Then,
    \[\mb{P}[S = m\mb{E}_{A\sim\rho}[A]] = \Omega((n^2p)^{-d/2}).\]
\end{theorem}

\begin{remark}
The proof below in fact gives the uniform two-sided estimate
\[
\mb{P}[S=m\mb{E}_{A\sim\rho}[A]]
=\Theta\left((n^2p)^{-d/2}\right).
\]
Indeed, the low-frequency integral is the density at zero of a Gaussian whose covariance is uniformly comparable to $n^2p$ times the identity, while the high-frequency integral is $o((n^2p)^{-d/2})$ in absolute value. We state only the lower bound because this is all that is used in \Cref{lem:local-exact-total-lower-bound}; the hypotheses do not force the normalized covariance, and hence the leading Gaussian constant, to converge.
\end{remark}

\begin{proof}
    If $r=0$, the conditioning event $MA\ge'0$ is the whole sample space. The coordinates of $S$ are then independent, with $S_t\sim\mr{Bin}(m\eta_t,q_t)$ and integer means $m\eta_tq_t$. Since $m\eta_tq_t(1-q_t)\asymp n^2p\to\infty$, Stirling's formula gives
    \begin{align*}
    \mb{P}[S=m\mb{E}_{A\sim\rho}[A]]
    &=\prod_{t=1}^d\mb{P}[\mr{Bin}(m\eta_t,q_t)=m\eta_tq_t]\\
    &=\prod_{t=1}^d\frac{1+o(1)}{\sqrt{2\pi m\eta_tq_t(1-q_t)}}
    =\Theta((n^2p)^{-d/2}).
    \end{align*}
    This proves the claim when $r=0$. Henceforth assume $r\ge1$.

    Let $B$ be the random variable distributed as $A-\mb{E}_{A\sim\rho}[A]$ where $A$ is sampled from $\rho$. Sample $m$ independent copies $B^{(1)},\ldots,B^{(m)}$ and let
    \[S' = B^{(1)} + \cdots + B^{(m)}.\]
    It suffices to lower bound the probability $\mb{P}[S'=0]$. Note that $S'$ is supported on a subset of $\mb{Z}^d$, so we have
    \[\mb{P}[S'=0] = \frac{1}{(2\pi)^d}\int_{[-\pi,\pi]^d}dt\,\mb{E}[e^{it\cdot S'}].\]
    We now estimate the characteristic function $\mb{E}[e^{it\cdot S'}]$ for various regimes of $t$.

    To do this, we first need to prove regularity conditions on the covariance matrix $\mr{Cov}B$. Indeed, let $\alpha_t = \sqrt{p\eta_t}\log\left(\frac{q_t}{1-q_t}\frac{1-p}{p}\right)$, which is bounded by a constant depending only on $T$ because of the hypothesis on $q$, and sample independent $X_t\sim\mc{N}(\sqrt{p\eta_t}\alpha_t,p\eta_t)$ for all $t\in[d]$. Applying \Cref{thm:general_poly_calc} with $\rho\in\{1,x_i,x_ix_j\}$ and dividing, we see that
    \[\mr{Cov}B = \mr{Cov}[X|MX\ge -pM\eta] + O(\sqrt{pn}\log^{5+d}n);\]
    here we used that $\mb{P}[MX\ge-pM\eta]$ is bounded below by a positive constant depending only on $T$, $M$, and $d$, because each bounding hyperplane $\{x:(Mx)_r=-(pM\eta)_r\}$ lies within $|(pM\eta)_r|/\norm{\mu_r}_2\le T\sqrt{pn}$ of the origin (where $\mu_r$ denotes the $r$th row of $M$), while the mean of $X$ has norm $O(T\sqrt{pn})$ and each coordinate of $X$ has standard deviation $\sqrt{p\eta_t}\asymp\sqrt{pn}$; so the constraints sit at $O(T)$ standard deviations.
    The matrix $\mr{Cov}[X|MX\ge -pM\eta]$ is positive definite: otherwise some nontrivial linear combination of the $X_t$ would be almost surely constant after conditioning on $MX\ge -pM\eta$, which is impossible because the conditioned law has a density on the region $\{Mx\ge'-pM\eta\}$, whose interior is nonempty since the rows of $M$ are orthogonal. Note that $\frac{1}{pn}\mr{Cov}[X|MX\ge -pM\eta]=\mr{Cov}[X/\sqrt{pn}\mid M(X/\sqrt{pn})\ge-pM\eta/\sqrt{pn}]$ is a continuous function of the parameters $(\eta_t/n)_{t\in[d]}$, $(\alpha_t)_{t\in[d]}$, and $pM\eta/\sqrt{pn}$, and does not otherwise depend on $n$ or $p$. These parameters range over the compact sets $[T^{-1},T]^d$, $[-T',T']^d$, and $[-T,T]^r$ (for a constant $T'$ depending only on $T$), on which the conditioned Gaussian has full-dimensional support, so there is some constant $C$ such that the entries of $\frac{1}{pn}\mr{Cov}[X|MX\ge -pM\eta]$ and its inverse are bounded by $C$ in absolute value. This implies that the entries of $\frac{1}{pn}\mr{Cov}B$ and its inverse are bounded by $2C$ in absolute value.

    We now apply the multidimensional Berry--Esseen theorem for convex sets \cite{Ben05}: for independent centered random vectors $B^{(1)},\ldots,B^{(m)}$ in $\mb{R}^d$ with $\Sigma:=\mr{Cov}(\sum_iB^{(i)})=m\,\mr{Cov}B$ positive definite, and $Z\sim\mc{N}(0,\Sigma)$,
    \begin{align*}
    \sup_{U\subset\mb{R}^d\text{ convex}}\left|\mb{P}[S'\in U] - \mb{P}[Z\in U]\right| &\le C_{\mr{BE}}\,d^{1/4}\sum_{i=1}^m\mb{E}\left\|\Sigma^{-1/2}B^{(i)}\right\|_2^3\\
    &=C_{\mr{BE}}\,d^{1/4}m^{-1/2}\,\mb{E}\left\|(\mr{Cov}B)^{-1/2}B\right\|_2^3
    \end{align*}
    for a universal constant $C_{\mr{BE}}$. The third moment on the right is $O(1)$: conditioning on the event $MA\ge'0$, whose probability is $\Omega(1)$ by the previous paragraph and \Cref{thm:general_poly_calc} with $\rho=1$, changes the third central moments of the coordinates $A_t$ by at most a constant factor and moves their means by $O(\sqrt{pn})$, so $\mb{E}\|B\|_2^3=O((pn)^{3/2})$, while the entries of $(\mr{Cov}B)^{-1}$ are $O(1/(pn))$. Hence, for all convex sets $U\subset\mb{R}^d$,
    \[|\mb{P}[S'\in U] - \mb{P}[Z\in U]| \le \frac{C_1}{n^{1/2-0.01/d}}\]
    for a constant $C_1$ depending only on $\theta,d,M,T$, where $Z\sim\mc{N}(0,m\mr{Cov}B)$ (we use the weaker exponent $1/2-0.01/d$ for convenience below). Fix any $t\in[-\pi,\pi]^d$ with $\norm{t}_\infty\le\frac{n^{0.01/d}}{\sqrt{n^2p}}$, and write $K:=\lfloor n^{1/6}\rfloor$. For all $\ell\in\mb{Z}$, let
    \[U_\ell = \{x\in\mb{R}^d: 2\pi\ell/K\le t\cdot x < 2\pi(\ell+1)/K\},\]
    which is a convex set, and note that $e^{it\cdot x} = e^{2\pi i\ell/K} + O(1/K)$ for $x\in U_\ell$. Furthermore, we have $|t\cdot S'| < n^{0.02/d}$ and $|t\cdot Z|<n^{0.02/d}$ with probability $1-O(n^{-\omega(1)})$: each coordinate of $B$ is at most $2\sqrt{pn}\log n$ in absolute value except with probability $n^{-\omega(1)}$ (a Chernoff bound for $A_t$, the conditioning on an event of probability $\Omega(1)$, and $|\mb{E}_\rho A_t-\mb{E}A_t|=O(\sqrt{pn})$), so after truncating at this level Hoeffding's inequality shows that each coordinate of $S'$ is $O(\sqrt{n^2p}\log^2n)$ with probability $1-n^{-\omega(1)}$, whence $|t\cdot S'|\le d\norm{t}_\infty\norm{S'}_\infty=O(n^{0.01/d}\log^2n)$; the claim for $Z$ follows from Gaussian tails. Hence only the $O(n^{1/6+0.02/d})$ strips with $|\ell|\le n^{0.02/d}K$ carry non-negligible mass. Thus, we have
    \begin{align*}
        \mb{E}[e^{it\cdot S'}] &= O(n^{-\omega(1)}) + \sum_{\substack{\ell\in\mb{Z} \\ |\ell|\le n^{0.02/d}K}}(e^{2\pi i\ell/K} + O(1/K))\mb{P}[S'\in U_\ell] \\ 
        &= O(n^{-1/6}) + \sum_{\substack{\ell\in\mb{Z} \\ |\ell|\le n^{0.02/d}K}}e^{2\pi i\ell/K}\left(\mb{P}[Z\in U_\ell] + O(n^{-1/2+0.01/d})\right) \\
        &= O(n^{-1/6}) + O(n^{-1/3+0.03/d}) + \mb{E}[e^{it\cdot Z}] = O(n^{-1/6}) + \mb{E}[e^{it\cdot Z}],
    \end{align*}
    so
    \begin{align*}\mb{P}[S'=0] &= O(n^{-0.1}(n^2p)^{-d/2}) + \frac{1}{(2\pi)^d}\displaystyle\int_{\substack{t\in[-\pi,\pi]^d \\ \norm{t}_\infty\le n^{0.01/d}/\sqrt{n^2p}}}dt\,\mb{E}[e^{it\cdot Z}]\\
    &\quad+ \frac{1}{(2\pi)^d}\displaystyle\int_{\substack{t\in[-\pi,\pi]^d \\ \norm{t}_\infty> n^{0.01/d}/\sqrt{n^2p}}}dt\,\mb{E}[e^{it\cdot S'}].\end{align*}
    There is a constant $C_2$ such that
    \[|\mb{E}[e^{it\cdot Z}]|\le \exp(-C_2 n^2p\norm{t}_\infty^2),\]
    which implies that
    \[\displaystyle\int_{\substack{t\in\mb{R}^d \\ \norm{t}_\infty> n^{0.01/d}/\sqrt{n^2p}}}dt\,\mb{E}[e^{it\cdot Z}] = O(n^{-\omega(1)}).\]
    If $f_Z$ denotes the density of $Z$, then
    \begin{align*}
        \mb{P}[S'=0] &= \frac{1}{(2\pi)^d}\int_{t\in\mb{R}^d}dt\mb{E}[e^{it\cdot Z}] + \frac{1}{(2\pi)^d}\displaystyle\int_{\substack{t\in[-\pi,\pi]^d \\ \norm{t}_\infty> n^{0.01/d}/\sqrt{n^2p}}}dt\,\mb{E}[e^{it\cdot S'}] + o((n^2p)^{-d/2}) \\
        &= f_Z(0) + \frac{1}{(2\pi)^d}\displaystyle\int_{\substack{t\in[-\pi,\pi]^d \\ \norm{t}_\infty> n^{0.01/d}/\sqrt{n^2p}}}dt\,\mb{E}[e^{it\cdot S'}] + o((n^2p)^{-d/2}) \\
        &= \Omega((n^2p)^{-d/2}) + \frac{1}{(2\pi)^d}\displaystyle\int_{\substack{t\in[-\pi,\pi]^d \\ \norm{t}_\infty> n^{0.01/d}/\sqrt{n^2p}}}dt\,\mb{E}[e^{it\cdot S'}].
    \end{align*}
    In order to estimate this error term, we use the following decomposition of $\rho$.

    Let $\mu_1,\ldots,\mu_r\in\mb{Z}^d$ be the rows of $M$. Denote $y\in\prod_{t\in[d]}(\mb{Z}-p\eta_t)$ as \textit{good} if
    \[\frac{y}{\sqrt{pn}}\in \lceil T+2\rceil\mu_1+\cdots+\lceil T+2\rceil\mu_r+\left[0,\frac{1}{d\norm{M}_\infty}\right]^d.\]
    Every good $y$ satisfies $M(p\eta+y)\ge'0$, with room to spare. Indeed, write $y/\sqrt{pn}=\lceil T+2\rceil\sum_{r'}\mu_{r'}+z$ with $z\in[0,1/(d\norm{M}_\infty)]^d$; by the orthogonality of the rows, and since $\norm{\mu_r}_2^2\ge1$ for a nonzero integer vector while $|\mu_r\cdot z|\le\norm{\mu_r}_1/(d\norm{M}_\infty)\le1$,
    \[(My)_r=\sqrt{pn}\left(\lceil T+2\rceil\norm{\mu_r}_2^2+\mu_r\cdot z\right)\ge(T+1)\sqrt{pn}>T\sqrt{pn}>p\norm{M\eta}_\infty\ge|(pM\eta)_r|,\]
    so $(M(p\eta+y))_r>0$ for every $r$.

    Let $f(x)$ be defined as in the proof of \Cref{thm:general_poly_calc}. Following the proof of \Cref{thm:general_poly_calc}, uniformly for good $y$, we have
    \[\mb{P}_{A\sim\rho}[A = p\eta+y] = (1+O(\log^{3+d} n/\sqrt{pn}))\frac{f(y)}{\displaystyle\int_{\substack{x\in\mb{R}^d \\ Mx\ge' -pM\eta}}dx\, f(x)} = \Omega((pn)^{-d/2}).\]
    Armed with this result, we are now ready to provide the relevant estimate of
    \[|\mb{E}[e^{it\cdot S'}]| = |\mb{E}[e^{it\cdot(S-mp\eta)}]| = |\mb{E}_{A\sim\rho}[e^{it\cdot(A-p\eta)}]|^m.\]
    The above analysis implies that we can decompose $A-p\eta$ into
    \[A-p\eta = \mbf{1}_{\mc{E}}T_1 + \mbf{1}_{\overline{\mc{E}}}T_2\]
    where $\mc{E}$ is an event of probability $C_3$, where $T_1$ is uniformly distributed on the set of good elements, and where $T_1,\mbf{1}_{\mc{E}}$ and $T_2,\mbf{1}_{\mc{E}}$ are pairs of independent random variables; indeed, one may take $C_3:=\#\{\text{good }y\}\cdot\min_{y\text{ good}}\mb{P}_{A\sim\rho}[A=p\eta+y]$, which lies in $[c,1)$ for a constant $c=c(\theta,d,M,T)>0$ because there are $\Theta((pn)^{d/2})$ good elements, each of $\rho$-probability $\Omega((pn)^{-d/2})$ by the previous display, and because $\rho$ also charges points that are not good. Thus, we have
    \[|\mb{E}_{A\sim\rho}[e^{it\cdot(A-p\eta)}]| = |C_3\mb{E}[e^{it\cdot T_1}] + (1-C_3)\mb{E}[e^{it\cdot T_2}]|\le 1-C_3 + C_3|\mb{E}[e^{it\cdot T_1}]|.\]
    Note that the set of good elements forms a discrete rectangular prism. Let the number of lattice points in its $j$th coordinate interval be $W_j\sqrt{pn}\in\mb{Z}$ (note that $W_j = \Theta(1)$). Then, we see that
    \[|\mb{E}[e^{it\cdot T_1}]| = \prod_{j=1}^d\frac{1}{W_j\sqrt{pn}}\left|\frac{\sin\left(\frac{t_jW_j\sqrt{pn}}{2}\right)}{\sin\left(\frac{t_j}{2}\right)}\right|.\]
    At $t_j=0$ the corresponding factor is interpreted by continuity and equals $1$.
    If $\frac{5}{W_j\sqrt{pn}}\le|t_j|\le\pi$ for some $j$, this quantity is at most $4/5$ (as $|\sin(t_j/2)|\ge|t_j|/\pi$), so
    \[|\mb{E}[e^{it\cdot S'}]|\le(1-C_3/5)^m\le(1-c/5)^m = O(n^{-\omega(1)}).\]
    If $n^{0.01/d}/\sqrt{n^2p}<|t_j| < \frac{5}{W_j\sqrt{pn}}$ for some $j$, then this quantity is at most
    \[\frac{1}{W_j\sqrt{pn}}\frac{\frac{t_jW_j\sqrt{pn}}{2}\left(1-0.01t_j^2W_j^2pn\right)}{t_j/2 + O(t_j^3)}\le 1 - n^{0.01/d}/n,\]
    so
    \[|\mb{E}[e^{it\cdot S'}]|\le\left(1-C_3n^{0.01/d}/n\right)^m\le\left(1-c\,n^{0.01/d}/n\right)^m = O(n^{-\omega(1)}).\]
    Thus, we have
    \[\mb{P}[S'=0] = \Omega((n^2p)^{-d/2}) + \frac{1}{(2\pi)^d}\displaystyle\int_{\substack{t\in[-\pi,\pi]^d \\ \norm{t}_\infty> n^{0.01/d}/\sqrt{n^2p}}}dt\,\mb{E}[e^{it\cdot S'}] = \Omega((n^2p)^{-d/2}),\]
    as desired.
\end{proof}

\begin{lemma}[Exact-total lower bound]\label{lem:local-exact-total-lower-bound}
Fix $\theta\in(1/2,1)$, an integer $k\ge1$, and constants $T>1$ and $\phi\in(0,1/2)$. There exists a constant $\Cst{lem:local-exact-total-lower-bound}=\Cst{lem:local-exact-total-lower-bound}(\theta,k,T,\phi)>0$ such that, for all sufficiently large $N$, all $p\in(T^{-1}N^{-\theta},TN^{-\theta})$, and every finite set $V$ with $|V|=N$, if $(y,\lambda)$ is $(T,\phi,p)$-locally admissible with $y=(\Pi,m,1)\in\mc{Y}_k(V)$, then
\[
\msf{R}_{\Pi,\lambda}\left(m(\Pi,\mbf{D})=m\mid\mc{I}_\Pi\right)
\ge (N^2p)^{-\Cst{lem:local-exact-total-lower-bound}}.
\]
\end{lemma}

\begin{proof}
By \cref{LA:conditioning} and \Cref{fact:conditioned-rows}, under $\msf{R}_{\Pi,\lambda}(\cdot\mid\mc{I}_\Pi)$ the rows in a fixed part $\Pi[s]$ are independent copies of $A_{s,\bullet}\mid\mc{I}_s(A_{s,\bullet})$, and rows in different parts are independent. By \cref{LA:solvability},
\[
n[s]\,\mb E[A_{s,t}\mid\mc I_s(A_{s,\bullet})]=m[s,t]
\]
for every $s,t\in\{0,1\}^k$. Hence the event $m(\Pi,\mbf D)=m$ is the intersection, over $s\in\{0,1\}^k$, of the independent events
\begin{equation}\label{eq:exact-total-s}
\sum_{x\in\Pi[s]}\mbf D[x,\bullet]=m[s,\bullet],
\end{equation}
each of which asks that the sum of $n[s]$ independent copies of the conditioned row vector $A_{s,\bullet}\mid\mc I_s(A_{s,\bullet})$ equal its mean $n[s]\,\mb E[A_{s,\bullet}\mid\mc I_s(A_{s,\bullet})]=m[s,\bullet]$, an integer vector. It therefore suffices to show that each event \eqref{eq:exact-total-s} has probability at least $(N^2p)^{-C_s}$ under $\msf R_{\Pi,\lambda}(\cdot\mid\mc I_\Pi)$ with $C_s:=2^{k-1}+1$, and then to multiply over the $2^k$ values of $s$.

Fix $s\in\{0,1\}^k$ and put $d:=2^k$. A nondegenerate $d$-dimensional lattice sum with fluctuation scale $\sqrt{N^2p}$ suggests mass of order $(N^2p)^{-d/2}$ at its integral mean. \Cref{thm:fourier_BE} gives this bound for the conditioned binomial rows; we verify its hypotheses below. We apply \Cref{thm:fourier_BE} with $d=2^k$, with its $n$ equal to $N$, its $m$ equal to $n[s]$, and with $\eta_t:=n[t]-\one_{s=t}$ and $q_t:=\lambda[s,t]$ for $t\in\{0,1\}^k$, so that its $A_t\sim\mr{Bin}(\eta_t,q_t)$ is exactly $A_{s,t}$. The $k-1$ inequalities defining $\mc I_s(A_{s,\bullet})$ are of the form $\mu_r\cdot A_{s,\bullet}\ge'_r0$ for $1\le r\le k-1$, where $\mu_r\in\{\pm1\}^{\{0,1\}^k}$ has entries $\mu_r[t]=(-1)^{s[r]}(-1)^{t[r-1]}$ and $\ge'_r\in\{\ge,>\}$ is determined by the pair $s[r-1]s[r]$; let $M$ be the matrix with rows $\mu_1,\ldots,\mu_{k-1}$. Distinct rows of $M$ are orthogonal: for $r\ne r'$, flipping bit $r-1$ of $t$ is a bijection of $\{0,1\}^k$ that negates $\mu_r[t]\mu_{r'}[t]$, so $\sum_t\mu_r[t]\mu_{r'}[t]=0$. We now verify the remaining hypotheses of \Cref{thm:fourier_BE}, with its constant $T$ replaced by $2T$. By \cref{LA:sizes}, we have $\eta_t\in((2T)^{-1}N,2TN)$ and $n[s]\in((2T)^{-1}N,2TN)$, and, for each $1\le r\le k-1$,
\[
\left|(M\eta)_r\right|=\left|\sum_{t\in\{0,1\}^k}(-1)^{t[r-1]}n[t]-(-1)^{s[r-1]}\right|\le T\frac{N}{\sqrt{pN}}+1<2T\frac{N}{\sqrt{pN}},
\]
where the term $(-1)^{s[r-1]}$ comes from the indicator $\one_{s=t}$ in $\eta_t$. By \cref{LA:tilt}, $q_t=\lambda[s,t]\in(p-Tp/\sqrt{pN},p+Tp/\sqrt{pN})$. Finally, the integrality hypothesis $n[s]\,\mb E[A_{s,\bullet}\mid MA_{s,\bullet}\ge'0]\in\mb Z^{d}$ of \Cref{thm:fourier_BE} is exactly \cref{LA:solvability}: the event $MA_{s,\bullet}\ge'0$ is $\mc I_s(A_{s,\bullet})$, and $n[s]\,\mb E[A_{s,t}\mid\mc I_s(A_{s,\bullet})]=m[s,t]\in\mb Z$. Therefore \Cref{thm:fourier_BE} gives, with $d=2^k$,
\begin{align*}
\mb P\left[\sum_{x\in\Pi[s]}\mbf D[x,\bullet]=m[s,\bullet]\ \middle|\ \mc I_s(\mbf D[x,\bullet])\text{ for all }x\in\Pi[s]\right]
&=\Omega\left((N^2p)^{-2^{k-1}}\right)\\
&\ge (N^2p)^{-2^{k-1}-1}=(N^2p)^{-C_s}
\end{align*}
for all sufficiently large $N$, where the implied constant in $\Omega(\cdot)$ depends only on $\theta,k,T,\phi$ and is absorbed by the extra factor $(N^2p)^{-1}$. Multiplying over the $2^k$ values of $s$ gives the result with $\Cst{lem:local-exact-total-lower-bound}:=\sum_sC_s=2^k(2^{k-1}+1)$.
\end{proof}

\begin{lemma}[Binomial regularity under the row model]\label{lem:local-binomial-kappa}
Fix $\theta\in(1/2,1)$, an integer $k\ge1$, and constants $T>1$ and $\phi\in(0,1/2)$. For all sufficiently large $N$, all $p\in(T^{-1}N^{-\theta},TN^{-\theta})$, and every finite set $V$ with $|V|=N$, if $(y,\lambda)$ is $(T,\phi,p)$-locally admissible with $y=(\Pi,m,1)\in\mc{Y}_k(V)$, then
\[
\msf{R}_{\Pi,\lambda}\left(\kappa_p(\Pi,\mbf{D})\mid\mc{I}_\Pi\right)=1-O(N^{-\omega(1)})
\]
and
\[
\msf{R}_{\Pi,\lambda}\left(\kappa_p(\Pi,\mbf{D})\mid\mc{I}_\Pi,m(\Pi,\mbf{D})=m\right)=1-O(N^{-\omega(1)}).
\]
\end{lemma}

\begin{proof}
For fixed $x\in\Pi[s]$ and $t\in\{0,1\}^k$, the tilt bound \cref{LA:tilt} gives $|\mb{E}A_{s,t}-p|\Pi[t]||\le|\Pi[t]|\cdot Tp/\sqrt{pN}+\lambda[s,t]\le T\sqrt{pN}+2p$, so Chernoff's inequality gives
\[
\mb P\left[\left|A_{s,t}-p|\Pi[t]|\right|>(pN)^{4/7}\right]=O(N^{-\omega(1)}).
\]
By \cref{LA:conditioning} and \Cref{fact:conditioned-rows}, under $\msf{R}_{\Pi,\lambda}(\cdot\mid\mc{I}_\Pi)$ the row $\mbf{D}[x,\bullet]$ has the law of $A_{s,\bullet}$ conditioned on $\mc{I}_s(A_{s,\bullet})$. Hence, by Bayes' rule and \cref{LA:conditioning},
\begin{align*}
\msf{R}_{\Pi,\lambda}\left(\left|\mbf{D}[x,t]-p|\Pi[t]|\right|>(pN)^{4/7}\ \middle|\ \mc{I}_\Pi\right)
&=\frac{\mb P\left[\left|A_{s,t}-p|\Pi[t]|\right|>(pN)^{4/7},\ \mc I_s(A_{s,\bullet})\right]}{\mb P\left[\mc I_s(A_{s,\bullet})\right]}\\
&\le\frac{\mb P\left[\left|A_{s,t}-p|\Pi[t]|\right|>(pN)^{4/7}\right]}{\phi}
=O(N^{-\omega(1)}).
\end{align*}
A union bound over the $2^kN$ pairs $(x,t)\in V\times\{0,1\}^k$ gives
\[
\msf{R}_{\Pi,\lambda}\left(\kappa_p(\Pi,\mbf{D})^c\mid\mc{I}_\Pi\right)\le2^kN\cdot O(N^{-\omega(1)})=O(N^{-\omega(1)}),
\]
which is the first estimate. For the second, Bayes' rule and \Cref{lem:local-exact-total-lower-bound} give
\begin{align*}
\msf{R}_{\Pi,\lambda}\left(\kappa_p(\Pi,\mbf{D})^c\mid\mc{I}_\Pi,m(\Pi,\mbf{D})=m\right)
&=\frac{\msf{R}_{\Pi,\lambda}\left(\kappa_p(\Pi,\mbf{D})^c\cap\{m(\Pi,\mbf{D})=m\}\mid\mc{I}_\Pi\right)}{\msf{R}_{\Pi,\lambda}\left(m(\Pi,\mbf{D})=m\mid\mc{I}_\Pi\right)}\\
&\le\frac{\msf{R}_{\Pi,\lambda}\left(\kappa_p(\Pi,\mbf{D})^c\mid\mc{I}_\Pi\right)}{(N^2p)^{-\Cof{lem:local-exact-total-lower-bound}{\theta,k,T,\phi}}}\\
&=O\left((N^2p)^{\Cof{lem:local-exact-total-lower-bound}{\theta,k,T,\phi}}N^{-\omega(1)}\right)=O(N^{-\omega(1)}).
\end{align*}
\end{proof}

\section{The kernel: edge splitting and degree regularity}
\label{app:graph-enumeration-local}

This appendix proves the two results about uniformly random graphs with prescribed degree sequences that are used in the proof of \Cref{prop:kernel-splitting-estimates}: the edge-count concentration of \Cref{lem:edge_enum}, and the degree-into-a-subset tail bound of \Cref{thm:nice_deg}. Both rest on the imported enumeration theorems of \Cref{app:tools-enumeration}.

\subsection{Edge enumeration}

We have the following result, which follows from the graph-case single-edge expansion in Theorem 1.6 of Liebenau and Wormald \cite{LW17}, whose printed error term may require correction on its full stated degree domain, and from Theorem 1.5 of Liebenau and Wormald \cite{LW20}. We discuss this point immediately after the statement, before using the resulting estimate.

\begin{lemma}
\label{lem:edge_enum}
    Fix a real parameter $\theta\in(1/2,1)$ and a constant $T>1$. For all sufficiently large $n\in\mb{Z}$ and $p\in(T^{-1}n^{-\theta},Tn^{-\theta})$, the following are true.
    \begin{enumerate}[label=\textup{(\roman*)},ref=\textup{(\roman*)}]
        \item\label{lem:edge_enum:graph} Select $m\in\mb{Z}$ satisfying $|m-pn(n-1)/2| \le T n^2p/\sqrt{pn}$, and select $\mbf{d}\in\mc{D}_{n,m}^{\text{graph}}$ such that $|d_i-pn|\le(pn)^{4/7}$ for all $i\in[n]$. Pick a subset $U\subset[n]$ satisfying $|U|,n-|U|\ge T^{-1}n$. Sample $G$ uniformly at random from all graphs with degree sequence $\mbf{d}$, and let $X = |E(G[U])|$ (the number of edges in the induced subgraph on vertex set $U$). Then,
        \[\mb{P}\left[\left|X - \frac{\left(\sum_{x\in U}d_x\right)^2}{4m}\right| \ge n^2p\sqrt{\frac{(pn)^{1/7}}{n}}\log n\right] < \frac{1}{\log n}.\]
        \item\label{lem:edge_enum:cut} Select $m\in\mb{Z}$ satisfying $|m-pn(n-1)/2| \le T n^2p/\sqrt{pn}$, and select $\mbf{d}\in\mc{D}_{n,m}^{\text{graph}}$ such that $|d_i-pn|\le(pn)^{4/7}$ for all $i\in[n]$. Pick a subset $U\subset[n]$ satisfying $|U|,n-|U|\ge T^{-1}n$. Sample $G$ uniformly at random from all graphs with degree sequence $\mbf{d}$, and let $X = |E(G[U,[n]\setminus U])|$ (the number of edges in the induced bipartite subgraph with vertex sets $U$ and $[n]\setminus U$). Then,
        \[\mb{P}\left[\left|X - \frac{\left(\sum_{x\in U}d_x\right)\left(\sum_{y\in[n]\setminus U}d_y\right)}{2m}\right| \ge n^2p\sqrt{\frac{(pn)^{1/7}}{n}}\log n\right] < \frac{1}{\log n}.\]
        \item\label{lem:edge_enum:bipartite} Select $\ell\in\mb{Z}\cap[T^{-1}n,Tn]$. Select $m\in\mb{Z}$ satisfying $|m-p\ell n| \le T n^2p/\sqrt{pn}$, and select $(\mbf{s},\mbf{t})\in\mc{D}_{\ell,n,m}^{\text{bigraph}}$ such that $|s_i-pn|\le(pn)^{4/7}$ for all $i\in[\ell]$ and $|t_j-p\ell|\le (pn)^{4/7}$ for all $j\in[n]$. Pick subsets $U\subset[\ell]$ and $U'\subset[n]$ satisfying $|U|,\ell-|U|,|U'|,n-|U'|\ge T^{-1}n$. Sample $G$ uniformly at random from all bigraphs with degree sequence $(\mbf{s},\mbf{t})$, and let $X = |E(G[U,U'])|$ (the number of edges in the induced bipartite subgraph with vertex sets $U$ and $U'$). Then,
        \[\mb{P}\left[\left|X - \frac{\left(\sum_{x\in U}s_x\right)\left(\sum_{y\in U'}t_y\right)}{m}\right| \ge n^2p\sqrt{\frac{(pn)^{1/7}}{n}}\log n\right] < \frac{1}{\log n}.\]
    \end{enumerate}
\end{lemma}

\begin{remark}[The graph-case single-edge input]
The error term in the printed statement of Theorem 1.6 of Liebenau and Wormald \cite{LW17} appears to require correction on its full stated degree domain. We record the fuller expansion needed here. In the graph case write
\[
D:=\frac{2m}{n},\qquad \Delta_{ab}:=d_a+d_b-2D,\qquad
\sigma^2:=\frac1n\sum_i(d_i-D)^2,\qquad
\ve:=\frac{(pn)^{1/7}}{n}.
\]
The hypotheses give $D\asymp pn$, $|D-pn|=O(\sqrt{pn})$, and $|d_i-D|\le D^{7/12}$ for large $n$. The hypotheses of the expansion hold because $D/(n-1)\asymp p\to0$ and $\log^K n/n=o(p)$ for every fixed $K$. The full expansion in Lemma 7.1(b) and equation (7.4) of \cite{LW17}, with $\alpha=7/12$, is
\begin{align}
\mb{P}[ab\in E(G)]={}&\frac{d_ad_b}{D(n-1)}
\biggl(1-\frac{(d_a-D)(d_b-D)}{D(n-1-D)}\notag\\
&\quad+\frac{\Delta_{ab}(n-1)\sigma^2}{D^2n(n-1-D)}
+\frac{\Delta_{ab}}{D(n-1)}
+O\left(\frac{D^{-2/3}}{n}\right)\biggr).
\label{eq:c1-corrected-edge-expansion}
\end{align}
The last two displayed terms do not appear in the printed statement of Theorem 1.6. On our degree window they are respectively $O(D^{-2/7}/n)$ and $O(D^{-3/7}/n)$, while the remainder in the fuller expansion is $O(D^{-2/3}/n)$. All three are $o(\ve)$, while the quadratic term in \eqref{eq:c1-corrected-edge-expansion} is $O(D^{1/7}/n)=O(\ve)$. Since $D(n-1)=2m(1-1/n)$ and $1/n\le\ve$, the only consequence used below is the unchanged estimate
\begin{equation}\label{eq:c1-single-edge}
\mb{P}[ab\in E(G)]=(1+O(\ve))\frac{d_ad_b}{2m}.
\end{equation}
The analogous estimate in the bipartite case follows directly from Theorem 1.5 of Liebenau and Wormald \cite{LW20}.
\end{remark}

\begin{proof}
We prove \cref{lem:edge_enum:graph}; the proof of \cref{lem:edge_enum:cut} is identical, and the proof of \cref{lem:edge_enum:bipartite} is identical using the bipartite estimate from Theorem 1.5 of Liebenau and Wormald \cite{LW20}. Write $X:=|E(G[U])|$ and, as in the preceding remark, $\ve=(pn)^{1/7}/n$, so that
\[
X=\sum_{\{a,b\}\subset U}\mbf{1}_{ab\in E(G)}.
\]
By \eqref{eq:c1-single-edge},
    \[\mb{E}[X] = (1+O(\ve))\sum_{\{a,b\}\subset U}\frac{d_ad_b}{2m} = \frac{1}{4m}\left(\sum_{a\in U}d_a\right)^2 + O\left((np)^{8/7}\right),\]
    where we used
    \[\sum_{\{a,b\}\subset U}d_ad_b=\frac12\Bigl(\sum_{a\in U}d_a\Bigr)^2-\frac12\sum_{a\in U}d_a^2,\qquad\frac{1}{4m}\sum_{a\in U}d_a^2=O(pn),\]
    and $\frac{\ve}{2m}\sum_{\{a,b\}\subset U}d_ad_b=O(\ve n^2p)=O((pn)^{8/7})$.

    In order to bound the variance, we estimate $\mb{P}[ab\in E(G)\text{ and }cd\in E(G)]$ for two distinct pairs $ab,cd\subset U$. Suppose first that $ab$ and $cd$ share no vertex. Condition on the neighborhoods $N_G(a)$ and $N_G(b)$. Given these, the graph $G-\{a,b\}$ is uniformly distributed among the graphs on $[n]\setminus\{a,b\}$ with degree sequence $\mbf{d}'$, where $d'_w=d_w-\mbf{1}_{w\in N_G(a)}-\mbf{1}_{w\in N_G(b)}$, and it has $m'=m-d_a-d_b+\mbf{1}_{ab\in E(G)}$ edges. The same corrected expansion applies to this degree sequence, since its entries are within $(pn)^{4/7}+O(\sqrt{pn})$ of its average $2m'/(n-2)=D+O(p)$ and the density $2m'/((n-2)(n-3))\asymp p$, so
    \[\mb{P}\left[cd\in E(G)\mid N_G(a),N_G(b)\right]=(1+O(\ve))\frac{d'_cd'_d}{2m'}.\]
    The point is that $d'_c$ and $d'_d$ differ from $d_c$ and $d_d$ only when $c$ or $d$ is adjacent to $a$ or $b$. A deterministic decrease of $d_c$ by $1$ would change $d'_cd'_d/(2m')$ by a relative $1/(pn)$, which is much larger than $\ve$; but each of the events $\{ac\in E(G)\}$, $\{bc\in E(G)\}$, $\{ad\in E(G)\}$, $\{bd\in E(G)\}$ has probability $O(p)$ conditionally on $ab\in E(G)$. (For instance, conditionally on $N_G(b)\ni a$, the graph $G-b$ is uniform with a degree sequence that again satisfies the hypotheses of the corrected expansion above, which gives $\mb{P}[ac\in E(G)\mid N_G(b)]=O(p)$.) Hence $\mb{E}[d'_cd'_d\mid ab\in E(G)]=d_cd_d\left(1-O(p/(pn))\right)=d_cd_d(1-O(1/n))$, while $1/(2m')=(1+O(pn/m))/(2m)=(1+O(1/n))/(2m)$ deterministically on $\{ab\in E(G)\}$. Since $1/n\le\ve$, averaging the previous display over $N_G(a),N_G(b)$ on the event $\{ab\in E(G)\}$ gives
    \[\mb{P}[ab\in E(G)\text{ and }cd\in E(G)]=\mb{P}[ab\in E(G)]\,\mb{P}[cd\in E(G)\mid ab\in E(G)] = \left(1+O(\ve)\right)\frac{d_ad_bd_cd_d}{(2m)^2}.\]
    Suppose next that the two pairs share a vertex, say they are $ab$ and $bc$. Conditioning on $N_G(a)$ as above, the vertex $b$ has degree exactly $d_b-1$ in $G-a$ on the event $\{ab\in E(G)\}$, so the same computation gives only
    \[\mb{P}[ab\in E(G)\text{ and }bc\in E(G)]=\left(1+O\left(\frac{1}{pn}\right)\right)\frac{d_ad_b^2d_c}{(2m)^2}=O(p^2);\]
    the relative error $1/(pn)$ here is genuinely larger than $\ve$, but this case involves only $O(n^3)$ ordered pairs and will be negligible.

    Now write $\mb{E}[X^2]=\sum_{(e,f)}\mb{P}[e,f\in E(G)]$, the sum running over ordered pairs of pairs $e,f\subset U$, and split it according to whether $e=f$, $e\ne f$ share a vertex, or $e,f$ are disjoint. The first part is $\mb{E}[X]=O(n^2p)$, the second is $O(n^3)\cdot O(p^2)=O(n^3p^2)$, and by the disjoint-pair estimate the third is
    \[(1+O(\ve))\sum_{e,f\text{ disjoint}}\mb{P}[e\in E(G)]\,\mb{P}[f\in E(G)]=(1+O(\ve))\left(\mb{E}[X]^2+O(n^3p^2)\right),\]
    since $\mb{E}[X]^2=\sum_{(e,f)}\mb{P}[e\in E(G)]\mb{P}[f\in E(G)]$ differs from the sum over disjoint pairs by the $O(n^3)$ terms with $e=f$ or $e,f$ sharing a vertex, each of which is $O(p^2)$. As $\mb{E}[X]^2=O(n^4p^2)$ and $n^2p+n^3p^2\le\ve n^4p^2=n^3p^2(pn)^{1/7}$, we conclude that
    \[\mr{Var}\,X=\mb{E}[X^2]-\mb{E}[X]^2=O\left(n^3p^2(pn)^{1/7}\right),\qquad\text{so}\qquad\sqrt{\mr{Var}\,X} = O\left(n^2p\sqrt{\frac{(pn)^{1/7}}{n}}\right).\]
    Finally, $\left|X-\frac{1}{4m}\left(\sum_{a\in U}d_a\right)^2\right|\le|X-\mb{E}X|+O((pn)^{8/7})$ and $(pn)^{8/7}=o\left(n^2p\sqrt{(pn)^{1/7}/n}\right)$, so the result follows from Chebyshev's inequality.
\end{proof}

\subsection{Degree regularity}

We prove results about the degrees to vertex subsets in random graphs using the enumeration theorems of \Cref{app:tools-enumeration}. The calculations mostly follow a simplified version of those in Appendix A of \cite{CKLT21}.

Our key result is as follows.
\begin{theorem}
\label{thm:nice_deg}
    Fix a real parameter $\theta\in(1/2,1)$ and a constant $T>1$. There exists a constant $\Cst{thm:nice_deg}=\Cst{thm:nice_deg}(\theta,T)>0$ such that for all sufficiently large $n\in\mb{Z}$ and $p\in(T^{-1}n^{-\theta},Tn^{-\theta})$, the following are true.
    \begin{enumerate}[label=\textup{(\roman*)},ref=\textup{(\roman*)}]
        \item\label{thm:nice_deg:graph} Select $m\in\mb{Z}$ satisfying $|m-pn(n-1)/2|\le Tn^2p/\sqrt{pn}$, and select $\mbf{d}\in\mc{D}_{n,m}^{\text{graph}}$ such that $d_i = pn + \beta_i\sqrt{pn}$ where $|\beta_i|\le\log n$. Pick a vertex $v\in[n]$ and a subset $S\subset[n]$ with size $h=|S|$ satisfying $h,n-h\ge T^{-1}n$. Pick an integer $t\in[0,d_v]$ and write $t=ph+\tau\sqrt{ph}$. Sample $G$ uniformly at random from all graphs with degree sequence $\mbf{d}$. If $|\tau|\ge \log^{100} n$, then
        \[\mb{P}[\deg^G_Sv = t] \le \exp(-\Cst{thm:nice_deg}\tau^2).\]
        \item\label{thm:nice_deg:bipartite} Select $\ell\in\mb{Z}\cap[T^{-1}n,Tn]$. Select $m\in\mb{Z}$ satisfying $|m-p\ell n| \le T n^2p/\sqrt{pn}$, and select $(\mbf{s},\mbf{t})\in\mc{D}_{\ell,n,m}^{\text{bigraph}}$ such that $s_i = pn + \alpha_i\sqrt{pn}$ and $t_j = p\ell + \beta_j\sqrt{p\ell}$, where $|\alpha_i|,|\beta_j|\le\log n$ for all $i\in[\ell]$ and $j\in[n]$. Pick a vertex $v\in[\ell]$ and a subset $S\subset[n]$ with size $h=|S|$ satisfying $h,n-h\ge T^{-1}n$. Pick an integer $t\in[0,s_v]$ and write $t=ph+\tau\sqrt{ph}$. Sample $G$ uniformly at random from all bigraphs with degree sequence $(\mbf{s},\mbf{t})$. If $|\tau|\ge\log^{100} n$, then
        \[\mb{P}[\deg^G_Sv = t] \le \exp(-\Cst{thm:nice_deg}\tau^2).\]
    \end{enumerate}
\end{theorem}
We prove this through a series of lemmas. We begin with a pair of technical estimates.
\begin{lemma}
\label{lem:nice_deg_technical}
    Fix a real parameter $\theta\in(1/2,1)$ and a constant $T>1$. For all sufficiently large $n\in\mb{Z}$ and $p\in(T^{-1}n^{-\theta},Tn^{-\theta})$, the following are true.
    \begin{enumerate}[label=\textup{(\roman*)},ref=\textup{(\roman*)}]
        \item\label{lem:nice_deg_technical:graph} Select $m\in\mb{Z}$ satisfying $|m-pn(n-1)/2|\le Tn^2p/\sqrt{pn}$, and select $d_v\in\mb{Z}$ satisfying $|d_v-pn|\le\sqrt{pn}\log n$. Then,
        \[\frac{\binom{(n-1)(n-2)/2}{m-d_v}\binom{(n-1)(n-2)}{2m-2d_v}^{-1}}{\binom{n(n-1)/2}{m}\binom{n(n-1)}{2m}^{-1}} = \exp\left(O(\log n)\right)p^{-d_v}e^{pn}.\]
        \item\label{lem:nice_deg_technical:bipartite} Select $\ell\in\mb{Z}\cap[T^{-1}n,Tn]$. Select $m\in\mb{Z}$ satisfying $|m-p\ell n| \le T n^2p/\sqrt{pn}$, and select $s_v\in\mb{Z}$ satisfying $|s_v-pn|\le\sqrt{pn}\log n$. Then,
        \[\frac{\binom{\ell n}{m}}{\binom{(\ell-1)n}{m-s_v}} = \exp\left(O(\log n)\right)p^{-s_v}e^{pn}.\]
    \end{enumerate}
\end{lemma}
\begin{proof}
    We prove \cref{lem:nice_deg_technical:graph}; the proof of \cref{lem:nice_deg_technical:bipartite} is similar and simpler. (For \cref{lem:nice_deg_technical:bipartite}, the ratio equals $\frac{(\ell n)!}{(\ell n-n)!}\cdot\frac{(m-s_v)!}{m!}\cdot\frac{(\ell n-m-(n-s_v))!}{(\ell n-m)!}$, whose three logarithms are $n\log(\ell n)-\frac{n}{2\ell}+O(1/\ell)$, $-s_v\log m+O(p)$, and $-(n-s_v)\log(\ell n-m)+\frac{n}{2\ell}+O(p)$ by the same expansions as below; summing them and expanding $\frac{m}{p\ell n}=1+O(1/\sqrt{pn})$ exactly as in the last step below gives $-s_v\log p+pn+O(\log n)$.)

    Note that
    \begin{align*}
        \frac{\binom{(n-1)(n-2)/2}{m-d_v}}{\binom{n(n-1)/2}{m}} &= \left[\frac{\left[\frac{n(n-1)}{2}\right]!}{\left[\frac{n(n-1)}{2}-(n-1)\right]!}\right]^{-1}\left[\frac{m!}{(m-d_v)!}\right]\left[\frac{\left[\frac{n(n-1)}{2}-m\right]!}{\left[\frac{n(n-1)}{2}-m-(n-1-d_v)\right]!}\right].
    \end{align*}
    We find that
    \begin{align*}
        \log\left[\frac{\left[\frac{n(n-1)}{2}\right]!}{\left[\frac{n(n-1)}{2}-(n-1)\right]!}\right] &= \sum_{i=0}^{n-2}\log\left(\frac{n(n-1)}{2}-i\right) \\
        &= (n-1)\log\left(\frac{n(n-1)}{2}\right) + \sum_{i=0}^{n-2}\log\left(1-\frac{i}{\frac{n(n-1)}{2}}\right) \\
        &= (n-1)\log\left(\frac{n(n-1)}{2}\right) - \sum_{i=0}^{n-2}\frac{i}{\frac{n(n-1)}{2}} + O(1/n) \\
        &= (n-1)\log\left(\frac{n(n-1)}{2}\right) - 1 + O(1/n)
    \end{align*}
    and
    \begin{align*}
        \log\left[\frac{m!}{(m-d_v)!}\right]  &= d_v\log m + \sum_{i=0}^{d_v-1}\log\left(1-\frac{i}{m}\right) \\
        &= d_v\log m + O(d_v^2/m) \\
        &= d_v\log m + O(p)
    \end{align*}
    and
    \begin{align*}
        &\log\left[\frac{\left[\frac{n(n-1)}{2}-m\right]!}{\left[\frac{n(n-1)}{2}-m-(n-1-d_v)\right]!}\right] \\
        &\quad= \sum_{i=0}^{n-2-d_v}\log\left(\frac{n(n-1)}{2}-m-i\right) \\
        &\quad= (n-1-d_v)\log\left(\frac{n(n-1)}{2}-m\right) + \sum_{i=0}^{n-2-d_v}\log\left(1-\frac{i}{\frac{n(n-1)}{2}-m}\right) \\
        &\quad= (n-1-d_v)\log\left(\frac{n(n-1)}{2}-m\right) - \sum_{i=0}^{n-2-d_v}\frac{i}{\frac{n(n-1)}{2}-m} + O(1/n) \\
        &\quad= (n-1-d_v)\log\left(\frac{n(n-1)}{2}-m\right) - 1 + O(p),
    \end{align*}
    so
    \[\frac{\binom{(n-1)(n-2)/2}{m-d_v}}{\binom{n(n-1)/2}{m}} = (1+O(p))\left(\frac{n(n-1)}{2}\right)^{1-n}m^{d_v}\left(\frac{n(n-1)}{2}-m\right)^{n-1-d_v}.\]
    We similarly derive
    \[\frac{\binom{(n-1)(n-2)}{2m-2d_v}}{\binom{n(n-1)}{2m}} = (1+O(p))\left(n(n-1)\right)^{3-2n}(2m)^{2d_v}\left(n(n-1)-2m\right)^{2n-3-2d_v}.\]
    Thus,
    \[\frac{\binom{(n-1)(n-2)/2}{m-d_v}\binom{(n-1)(n-2)}{2m-2d_v}^{-1}}{\binom{n(n-1)/2}{m}\binom{n(n-1)}{2m}^{-1}} = (1+O(p))\left[\frac{n(n-1)}{2m}\right]^{d_v}\left[\frac{\frac{n(n-1)}{2}}{\frac{n(n-1)}{2}-m}\right]^{n-2-d_v}.\]
    We see that
    \begin{align*}
        \log\left(\left[\frac{n(n-1)}{2m}\right]^{d_v}\right) &= d_v\log\left(\frac{n^2}{2m}\right) + O(p) \\
        &= -d_v\log p - d_v\log\left(\frac{2m}{pn^2}\right) + O(p) \\
        &= -d_v\log p - d_v\left[\frac{2m}{pn^2}-1\right] + O(1) \\
        &= -d_v\log p - \frac{2m}{n} + pn + O(\log n),
    \end{align*}
    where in the third line we used $\frac{2m}{pn^2}-1=O(1/\sqrt{pn})$, so that $d_v\left(\frac{2m}{pn^2}-1\right)^2=O(1)$, and in the last line we wrote $d_v=pn+\beta_v\sqrt{pn}$ with $\beta_v:=(d_v-pn)/\sqrt{pn}$, so that $|\beta_v|\le\log n$ by hypothesis and $\beta_v\sqrt{pn}\left(\frac{2m}{pn^2}-1\right)=O(\log n)$; and
    \begin{align*}
        \log\left(\left[\frac{\frac{n(n-1)}{2}}{\frac{n(n-1)}{2}-m}\right]^{n-2-d_v}\right) &= (n-2-d_v)\log\left[\frac{1}{1-\frac{2m}{n(n-1)}}\right] \\
        &= \frac{2m}{n} + O(np^2).
    \end{align*}
    Adding these two directly implies the desired result.
\end{proof}

\begin{lemma}
\label{lem:nice_deg_bulk}
    Fix a real parameter $\theta\in(1/2,1)$ and a constant $T>1$. For all sufficiently large $n\in\mb{Z}$ and $p\in(T^{-1}n^{-\theta},Tn^{-\theta})$, the following are true.
    \begin{enumerate}[label=\textup{(\roman*)},ref=\textup{(\roman*)}]
        \item\label{lem:nice_deg_bulk:graph} Select $m\in\mb{Z}$ satisfying $|m-pn(n-1)/2|\le Tn^2p/\sqrt{pn}$, and select $\mbf{d}\in\mc{D}_{n,m}^{\text{graph}}$ such that $d_i = pn + \beta_i\sqrt{pn}$ where $|\beta_i|\le\log n$. Pick a vertex $v\in[n]$ and a subset $S\subset[n]$ with size $h=|S|$ satisfying $h,n-h\ge T^{-1}n$. Pick an integer $t\in[0,d_v]$. Sample $G$ uniformly at random from all graphs with degree sequence $\mbf{d}$. We have
        \[\mb{P}[\deg^G_Sv = t] = \exp\left(O(\log^4 n)\right)\frac{\binom{h-\mbf{1}_S(v)}{t}\binom{n-h-\mbf{1}_{[n]\setminus S}(v)}{d_v-t}}{\binom{n-1}{d_v}}\mb{E}_R[\exp(\Lambda_R)]\]
        where $R=R_1\cup R_2$ is a random set chosen by picking $R_1$ uniformly from $\binom{S\setminus\{v\}}{t}$ and $R_2$ uniformly from $\binom{[n]\setminus(S\cup\{v\})}{d_v-t}$ (if one of these two families is empty, then the corresponding binomial coefficient vanishes and both sides are $0$), and where
        \[\Lambda_R := \sum_{i\in R}\frac{\beta_i}{\sqrt{pn}} - \sum_{i\in [n]\setminus(\{v\}\cup R)}\frac{\sqrt{pn}}{n}\beta_i.\]
        \item\label{lem:nice_deg_bulk:bipartite} Select $\ell\in\mb{Z}\cap[T^{-1}n,Tn]$. Select $m\in\mb{Z}$ satisfying $|m-p\ell n| \le T n^2p/\sqrt{pn}$, and select $(\mbf{s},\mbf{t})\in\mc{D}_{\ell,n,m}^{\text{bigraph}}$ such that $s_i = pn + \alpha_i\sqrt{pn}$ and $t_j = p\ell + \beta_j\sqrt{p\ell}$, where $|\alpha_i|,|\beta_j|\le\log n$ for all $i\in[\ell]$ and $j\in[n]$. Pick a vertex $v\in[\ell]$ and a subset $S\subset[n]$ with size $h=|S|$ satisfying $h,n-h\ge T^{-1}n$. Pick an integer $t\in[0,s_v]$. Sample $G$ uniformly at random from all bigraphs with degree sequence $(\mbf{s},\mbf{t})$. We have
        \[\mb{P}[\deg^G_Sv = t] = \exp\left(O(\log^4 n)\right)\frac{\binom{h}{t}\binom{n-h}{s_v-t}}{\binom{n}{s_v}}\mb{E}_R[\exp(\Lambda_R)]\]
        where $R=R_1\cup R_2$ is a random set chosen by picking $R_1$ uniformly from $\binom{S}{t}$ and $R_2$ uniformly from $\binom{[n]\setminus S}{s_v-t}$, and where
        \[\Lambda_R:=\sum_{j\in R}\frac{\beta_j}{\sqrt{p\ell}} - \sum_{j\in [n]\setminus R}\frac{\sqrt{p\ell}}{\ell}\beta_j.\]
    \end{enumerate}
\end{lemma}
\begin{proof}
    We prove only \cref{lem:nice_deg_bulk:graph}; the proof of \cref{lem:nice_deg_bulk:bipartite} is similar but uses \Cref{thm:bigraph_counting} instead of \Cref{thm:graph_counting}, and \Cref{lem:nice_deg_technical}\cref{lem:nice_deg_technical:bipartite} instead of \Cref{lem:nice_deg_technical}\cref{lem:nice_deg_technical:graph}.

    Let $\mb{P}_{\mbf{d}}$ denote the uniform distribution on graphs with degree sequence $\mbf{d}$. Given $R\subset[n]\setminus\{v\}$ of size $d_v$, we begin by computing $\mb{P}_{G\sim\mb{P}_{\mbf{d}}}[N_G(v) = R]$. To this end, define
    \[\mbf{d}_R = (d_w - \mbf{1}_{w\in R})_{w\in[n]\setminus\{v\}}.\]
    Note that $2m = \sum_{i\in[n]}d_i$ and let $m_R = m-d_v$, and let $d=2m/n$ and $d_R = 2m_R/(n-1)$. Note that by Bayes' theorem,
    \[\mb{P}_{G\sim\mb{P}_\mbf{d}}[N_G(v) = R] = \mb{P}_{G\sim G(n,m)}[N_G(v)=R]\frac{\mb{P}_{G(n-1,m_R)}[\mbf{d}_R]}{\mb{P}_{G(n,m)}[\mbf{d}]} = \frac{\binom{(n-1)(n-2)/2}{m-d_v}}{\binom{n(n-1)/2}{m}}\frac{\mb{P}_{G(n-1,m_R)}[\mbf{d}_R]}{\mb{P}_{G(n,m)}[\mbf{d}]},\]
    where $\mb{P}_{G(n,m)}$ and $\mb{P}_{G(n-1,m_R)}$ are the degree-sequence laws of \Cref{def:graph_deg_seq}: the graphs with degree sequence $\mbf{d}$ and $N_G(v)=R$ correspond bijectively to the graphs on $[n]\setminus\{v\}$ with degree sequence $\mbf{d}_R$, and the latter number $\binom{(n-1)(n-2)/2}{m-d_v}\mb{P}_{G(n-1,m_R)}[\mbf{d}_R]$, while the former are $\binom{n(n-1)/2}{m}\mb{P}_{G(n,m)}[\mbf{d}]$ in total.
    We compute the remaining two probabilities by appealing to \Cref{thm:graph_counting}. Indeed, note that $\mu = \frac{d}{n-1} = p + O(p/\sqrt{pn})$, and $\gamma_2(\mbf{d}) = O(p\log^2 n)$, and similarly $\mu_R = \frac{d_R}{n-2} = p + O(p/\sqrt{pn})$ and $\gamma_2(\mbf{d}_R) = O(p\log^2 n)$. Thus, by \Cref{thm:graph_counting}, we have
    \begin{align*}
        \mb{P}_{G\sim\mb{P}_\mbf{d}}[N_G(v) = R] &= \exp\left(O(\log^4 n)\right)\frac{\binom{(n-1)(n-2)/2}{m-d_v}}{\binom{n(n-1)/2}{m}}\frac{\frac{\prod_{w\in [n]\setminus\{v\}}\binom{n-2}{d_w-\mbf{1}_{w\in R}}}{\binom{(n-1)(n-2)}{2m-2d_v}}}{\frac{\prod_{w\in[n]}\binom{n-1}{d_w}}{\binom{n(n-1)}{2m}}} \\
        &= \exp\left(O(\log^4 n)\right)\frac{\binom{(n-1)(n-2)/2}{m-d_v}\binom{(n-1)(n-2)}{2m-2d_v}^{-1}}{\binom{n(n-1)/2}{m}\binom{n(n-1)}{2m}^{-1}}\frac{\prod_{w\in [n]\setminus\{v\}}\binom{n-2}{d_w-\mbf{1}_{w\in R}}}{\prod_{w\in[n]}\binom{n-1}{d_w}}.
    \end{align*}
    Note that
    \begin{align*}
        \frac{\prod_{w\in [n]\setminus\{v\}}\binom{n-2}{d_w-\mbf{1}_{w\in R}}}{\prod_{w\in[n]}\binom{n-1}{d_w}} &= \frac{1}{\binom{n-1}{d_v}}\prod_{i\in R}\frac{d_i}{n-1}\prod_{i\in[n]\setminus(\{v\}\cup R)}\frac{n-1-d_i}{n-1} \\ 
        &= \frac{\exp(O(1))}{\binom{n-1}{d_v}}p^{d_v}(1-p)^{n-1-d_v}\prod_{i\in R}\left[1 + \frac{\beta_i}{\sqrt{pn}}\right]\prod_{i\in [n]\setminus(\{v\}\cup R)}\left[1 - \frac{\sqrt{pn}}{n}\beta_i\right] \\
        &= \frac{\exp(O(\log^2 n))}{\binom{n-1}{d_v}}p^{d_v}e^{-np}\exp\left[\sum_{i\in R}\frac{\beta_i}{\sqrt{pn}} - \sum_{i\in [n]\setminus(\{v\}\cup R)}\frac{\sqrt{pn}}{n}\beta_i\right],
    \end{align*}
    so by \Cref{lem:nice_deg_technical}\cref{lem:nice_deg_technical:graph}, we have
    \[\mb{P}_{G\sim\mb{P}_\mbf{d}}[N_G(v) = R] = \exp\left(O(\log^4 n)\right)\frac{\exp(\Lambda_R)}{\binom{n-1}{d_v}}.\]
    Now pick $R_1$ uniformly from $\binom{S\setminus\{v\}}{t}$ and $R_2$ uniformly from $\binom{[n]\setminus(\{v\}\cup S)}{d_v-t}$, and let $R=R_1\cup R_2$. Under this setup, we have
    \[\mb{P}_{G\sim\mb{P}_{\mbf{d}}}[\deg^G_Sv=t] = \binom{h-\mbf{1}_S(v)}{t}\binom{n-h-\mbf{1}_{[n]\setminus S}(v)}{d_v-t}\mb{E}_R[\mb{P}_{G\sim\mb{P}_\mbf{d}}[N_G(v) = R]],\]
    so
    \[\mb{P}[\deg^G_Sv=t] = \exp\left(O(\log^4 n)\right)\frac{\binom{h-\mbf{1}_S(v)}{t}\binom{n-h-\mbf{1}_{[n]\setminus S}(v)}{d_v-t}}{\binom{n-1}{d_v}}\mb{E}_R[\exp(\Lambda_R)],\]
    as desired.
\end{proof}

The following lemma compares $\mb{E}[e^X]$ with $e^{\mb{E}X}$ when $X$ is a sum of small weights over a uniformly random subset of prescribed size; it is what lets us replace $\mb{E}_R[\exp(\Lambda_R)]$ by $\exp(\mb{E}_R[\Lambda_R])$ up to a factor $\exp(O(\log^2n))$.
\begin{lemma}
\label{lem:nice_deg_cute}
    Fix a real parameter $\theta\in(1/2,1)$ and a constant $T>1$. For all sufficiently large $n\in\mb{Z}$ and $p\in(T^{-1}n^{-\theta},Tn^{-\theta})$, the following is true.

    Fix real numbers $a_1,\ldots,a_n\in\mb{R}$ with $|a_i|\le T\log n/\sqrt{pn}$ for all $i$, and consider an integer $s\in[T^{-1}pn,Tpn]$. Pick $\xi\in\{0,1\}^n$ uniformly at random, and then condition on $\sum_i\xi_i = s$. Let $X = \sum_i \xi_i a_i$. Then,
    \[\mb{E}\left[e^X\right] \le O(\sqrt{pn})\exp\left(\mb{E}[X]+O(\log^2 n)\right).\]
\end{lemma}
\begin{proof}
    Let $q=s/n$, and let $Y = \sum_i \zeta_i a_i$, where the $\zeta_i$ are independent $\mr{Ber}(q)$ random variables. Since $q\asymp p$,
    \[\mb{E}\left[e^Y\right] \ge \mb{P}[\mr{Bin}(n,q)=s]\mb{E}\left[e^X\right] \ge \Omega(1/\sqrt{pn})\mb{E}\left[e^X\right].\]
    Also,
    \[
    \log\mb{E}\left[e^Y\right]
    =\sum_{i=1}^n\log\left(1+q(e^{a_i}-1)\right)
    =q\sum_{i=1}^n a_i+O\left(q\sum_{i=1}^n a_i^2\right)
    =\mb{E}[X]+O(\log^2 n),
    \]
    using $|a_i|\le T\log n/\sqrt{pn}$. Combining the two displays proves the claim.
\end{proof}
We are now ready to prove our main result.
\begin{proof}[Proof of \Cref{thm:nice_deg}]
    We only prove \cref{thm:nice_deg:graph}, as \cref{thm:nice_deg:bipartite} has a similar proof but uses \Cref{lem:nice_deg_bulk}\cref{lem:nice_deg_bulk:bipartite} instead of \Cref{lem:nice_deg_bulk}\cref{lem:nice_deg_bulk:graph}.

    Define $\Lambda_R$ as in \Cref{lem:nice_deg_bulk}. We begin by computing its expectation over the randomness of $R$. We have
    \begin{align*}
        \mb{E}_R[\Lambda_R] &= \sum_{x\in S\setminus\{v\}}\frac{t}{|S\setminus\{v\}|}\frac{\beta_x}{\sqrt{pn}} + \sum_{x\in[n]\setminus(S\cup\{v\})}\frac{d_v-t}{|[n]\setminus(S\cup\{v\})|}\frac{\beta_x}{\sqrt{pn}} \\
        &\quad-\sum_{x\in S\setminus\{v\}}\left(1-\frac{t}{|S\setminus\{v\}|}\right)\frac{\sqrt{pn}}{n}\beta_x - \sum_{x\in[n]\setminus(S\cup\{v\})}\left(1 - \frac{d_v-t}{|[n]\setminus(S\cup\{v\})|}\right)\frac{\sqrt{pn}}{n}\beta_x \\
        &= \sum_{x\in S}\left[\frac{t}{h} - p\right]\frac{\beta_x}{\sqrt{pn}} + \sum_{x\in[n]\setminus S}\left[\frac{d_v-t}{n-h}-p\right]\frac{\beta_x}{\sqrt{pn}} + O(\log n) \\
        &= O(|\tau|\log n + \log^2 n).
    \end{align*}
    We first dispose of the case in which $t$ or $d_v-t$ is small. Let $\epsilon=\epsilon(T)>0$ be a sufficiently small constant, and suppose that $t\le\epsilon pn$ or $d_v-t\le\epsilon pn$. The hypergeometric factor in \Cref{lem:nice_deg_bulk} is the probability that a uniformly random $d_v$-subset of $[n]\setminus\{v\}$ meets $S\setminus\{v\}$ in exactly $t$ elements. This hypergeometric variable has mean $d_v(h-\mbf{1}_S(v))/(n-1)$, which lies in $[d_v/(2T),d_v(1-1/(2T))]$ for large $n$, so for $\epsilon$ small enough our $t$ is at distance $\Omega(pn)$ from the mean, and Hoeffding's inequality (which applies to sampling without replacement) bounds the factor by $\exp(-cpn)$ for a constant $c=c(T)>0$. Since $|\Lambda_R|\le d_v\log n/\sqrt{pn}+\sqrt{pn}\log n=O(\sqrt{pn}\log n)$ deterministically, \Cref{lem:nice_deg_bulk} gives $\mb{P}[\deg^G_Sv=t]\le\exp\left(O(\log^4 n)+O(\sqrt{pn}\log n)-cpn\right)\le\exp(-cpn/2)$, while $|t-ph|\le\max(ph,d_v)=O(pn)$ gives $\tau^2=(t-ph)^2/(ph)=O(pn)$. Hence $\mb{P}[\deg^G_Sv=t]\le\exp(-c_1\tau^2)$ in this case, for a suitable constant $c_1=c_1(T)>0$.

    We may therefore assume that $\epsilon pn\le t\le d_v-\epsilon pn$. Writing $\xi_i:=\mbf{1}_{i\in R}$, we have $\Lambda_R=\sum_{i\ne v}\xi_ia_i-\sum_{i\ne v}\frac{\sqrt{pn}}{n}\beta_i$ with $a_i:=\beta_i\left(\frac{1}{\sqrt{pn}}+\frac{\sqrt{pn}}{n}\right)=O(\log n/\sqrt{pn})$, and $R_1$ and $R_2$ are independent uniformly random subsets of $S\setminus\{v\}$ and of $[n]\setminus(S\cup\{v\})$ of sizes $t$ and $d_v-t$ respectively; both sizes are of order $pn$, hence of order $p$ times the size of the ambient set, which is the hypothesis of \Cref{lem:nice_deg_cute}. Applying that lemma separately to the two pieces, and using $\log(pn)=O(\log n)$, gives $\mb{E}_R[\exp(\Lambda_R)]\le\exp(\mb{E}_R[\Lambda_R]+O(\log^2 n))=\exp(O(|\tau|\log n+\log^2n))$. Jensen's inequality supplies the matching lower bound $\mb E_R[e^{\Lambda_R}]\ge e^{\mb E_R[\Lambda_R]}$, so in fact $\mb E_R[e^{\Lambda_R}]=\exp(O(|\tau|\log n+\log^2n))$. Together with \Cref{lem:nice_deg_bulk}, whose factor $\exp(O(\log^4 n))$ is absorbed into $\exp(O(|\tau|\log n))$ because $|\tau|\ge\log^{100}n$, we obtain
    \begin{align*}
        \mb{P}[\deg^G_Sv=t] &= \exp(O(|\tau|\log n + \log^2 n))\frac{\binom{h-\mbf{1}_S(v)}{t}\binom{n-h-\mbf{1}_{[n]\setminus S}(v)}{d_v-t}}{\binom{n-1}{d_v}} \\
        &= \exp(O(|\tau|\log n + \log^2 n))\frac{\binom{h}{t}\binom{n-h}{d_v-t}}{\binom{n}{d_v}},
    \end{align*}
    where the second line uses $\binom{h-1}{t}/\binom{h}{t}=1-t/h=1-O(p)$ and its two analogues.
    We now bound the hypergeometric term. Indeed, we have that
    \begin{align*}
        \frac{\binom{h}{t}\binom{n-h}{d_v-t}}{\binom{n}{d_v}} &= \frac{\binom{d_v}{t}\binom{n-d_v}{h-t}}{\binom{n}{h}} \\
        &= (1+o(1))\binom{d_v}{t}\frac{h^t(n-h)^{d_v-t}}{n^{d_v}} \\
        &\le (1+o(1))\mb{P}\left[\left|\mr{Bin}(d_v,\tfrac{h}{n})-d_v\tfrac{h}{n}\right|\ge\left(\frac{|\tau|}{\sqrt T}-\log n\right)\sqrt{pn}\right] \\
        &\le 2(1+o(1))\exp\left(-\frac{2(|\tau|/\sqrt T-\log n)^2}{1+\beta_v/\sqrt{pn}}\right),
    \end{align*}
    where the second line holds because the ratio of the two sides is $\exp(O(t^2/h+d_v^2/n))=\exp(O(p^2n))$, which tends to $1$ since $\theta>1/2$; the first inequality uses $t-d_vh/n=\tau\sqrt{ph}-\beta_v\sqrt{pn}\,h/n$ together with $h\ge T^{-1}n$ and $h\le n$; and the last inequality follows from Hoeffding's inequality. Thus, since $|\tau|\ge\log^{100} n$, there is some constant $c_2>0$ (depending on $T$) such that
    \[\mb{P}[\deg^G_Sv=t]\le\exp(-c_2\tau^2)\]
    in this case. Taking $\Cst{thm:nice_deg}:=\min\{c_1,c_2\}$ proves the theorem.
\end{proof}

\section{Inverse-function theorems for conditional-mean maps}
\label{app:inverse}

This appendix proves that the map sending the mean of a Gaussian vector to its conditional mean given a fixed open convex set is a diffeomorphism onto that set, and that a map which is uniformly close to such a diffeomorphism on compact sets still hits every target in a compact subset of the set, with a preimage close to the exact one. The first statement is what makes the universal recursion of \Cref{sec:universal} well defined; the second is what produces the tilts of the idealized process in \Cref{app:idealized}. The first statement is the classical fact that the mean-value map of a regular exponential family is a diffeomorphism onto the interior of the convex support; we give the short self-contained proof. The second is an application of Brouwer's fixed point theorem. The appendix uses only \Cref{app:tools-gaussian}.

\subsection{Strong bijections}

\begin{definition}[Strong bijection]\label{def:strong-bijection}
Let $d\ge1$ and let $B\subseteq\mb{R}^d$ be open. A map $f\colon\mb{R}^d\to B$ is a \emph{strong bijection} if it is a bijection and both $f$ and $f^{-1}\colon B\to\mb{R}^d$ are smooth.
\end{definition}

\begin{lemma}[Lipschitz bounds on compact sets]\label{lem:strong-bijection-lipschitz}
Let $f\colon\mb{R}^d\to B$ be a strong bijection and let $K\subseteq B$ be compact. Then there is a constant $L=L(f,K)\ge1$ such that
\[
L^{-1}|y-z|\le\left|f^{-1}(y)-f^{-1}(z)\right|\le L|y-z|\qquad\text{for all }y,z\in K .
\]
\end{lemma}
\begin{proof}
We first record that a smooth map $h\colon W\to\mb{R}^d$ on an open set $W\subseteq\mb{R}^d$ is Lipschitz on every compact $K\subseteq W$. For $K=\emptyset$ this is immediate, so assume $K\ne\emptyset$. Since $K$ is compact and $W$ is open there is $\rho>0$ such that the closed $\rho$-neighborhood $K_\rho:=\{x:\operatorname{dist}(x,K)\le\rho\}$ is contained in $W$; the derivative $Dh$ is continuous, hence bounded by some $L_1$ on the compact set $K_\rho$. If $y,z\in K$ and $|y-z|\le\rho$, the segment from $y$ to $z$ lies in $K_\rho$, and the mean value inequality gives $|h(y)-h(z)|\le L_1|y-z|$. If $|y-z|>\rho$, then $|h(y)-h(z)|\le2\sup_K|h|\le(2\sup_K|h|/\rho)\,|y-z|$. So $h$ is Lipschitz on $K$ with constant $\max\{L_1,2\sup_K|h|/\rho\}$.

Applying this to $h=f^{-1}$ on $W=B$ gives the upper bound with some constant $L'$. For the lower bound, the set $K':=f^{-1}(K)$ is compact, being the image of a compact set under the continuous map $f^{-1}$; applying the fact to $h=f$ on $W=\mb{R}^d$ gives $L''$ with $|f(a)-f(b)|\le L''|a-b|$ for $a,b\in K'$, and taking $a=f^{-1}(y)$, $b=f^{-1}(z)$ yields $|y-z|\le L''|f^{-1}(y)-f^{-1}(z)|$. Take $L:=\max\{L',L'',1\}$.
\end{proof}

\subsection{The conditional-mean map}

\begin{theorem}[The conditional-mean map is a strong bijection]\label{thm:orthant-bijection}
Let $d\ge1$, let $\Sigma\in\mb{R}^{d\times d}$ be positive definite, and let $\mc{O}\subseteq\mb{R}^d$ be a nonempty open convex set. For $\gamma\in\mb{R}^d$ let $X^\gamma\sim\mc{N}(\gamma,\Sigma)$. The map
\[
F_{\mc{O}}\colon\mb{R}^d\to\mc{O},\qquad F_{\mc{O}}(\gamma):=\mb{E}\left[X^\gamma\mid X^\gamma\in\mc{O}\right],
\]
is well defined, that is, its values lie in $\mc{O}$, and it is a strong bijection. Its Jacobian at $\gamma$ is $\mr{Cov}(X^\gamma\mid X^\gamma\in\mc{O})\,\Sigma^{-1}$, which is invertible.
\end{theorem}

\begin{proof}
That $F_{\mc{O}}$ is well defined with values in $\mc{O}$ is \Cref{lem:open-conditioning}. Define
\[
\Phi\colon\mb{R}^d\to\mb{R},\qquad\Phi(\vartheta):=\log\int_{\mc{O}}\exp\left(\vartheta\cdot x-\tfrac12x^\top\Sigma^{-1}x\right)dx .
\]
The integrand is bounded by $\exp(|\vartheta||x|-\frac12\lambda_{\min}|x|^2)$, where $\lambda_{\min}>0$ is the smallest eigenvalue of $\Sigma^{-1}$, so the integral converges for every $\vartheta$, and differentiation under the integral sign (justified by the same Gaussian domination, locally uniformly in $\vartheta$) shows that $\Phi$ is smooth with
\[
\nabla\Phi(\vartheta)=\mb{E}_\vartheta[X],\qquad\nabla^2\Phi(\vartheta)=\mr{Cov}_\vartheta(X),
\]
where $\mb{P}_\vartheta$ denotes the probability measure on $\mc{O}$ with density proportional to $\exp(\vartheta\cdot x-\frac12x^\top\Sigma^{-1}x)$. Since
\[
-\tfrac12(x-\gamma)^\top\Sigma^{-1}(x-\gamma)=x\cdot\Sigma^{-1}\gamma-\tfrac12x^\top\Sigma^{-1}x-\tfrac12\gamma^\top\Sigma^{-1}\gamma,
\]
the law of $X^\gamma$ conditioned on $\{X^\gamma\in\mc{O}\}$ is $\mb{P}_{\Sigma^{-1}\gamma}$. Hence
\begin{equation}\label{eq:F-gradient}
F_{\mc{O}}(\gamma)=\nabla\Phi(\Sigma^{-1}\gamma)\qquad\text{and}\qquad DF_{\mc{O}}(\gamma)=\nabla^2\Phi(\Sigma^{-1}\gamma)\,\Sigma^{-1}=\mr{Cov}(X^\gamma\mid X^\gamma\in\mc{O})\,\Sigma^{-1},
\end{equation}
so $F_{\mc{O}}$ is smooth, and its Jacobian is invertible because $\mr{Cov}(X^\gamma\mid X^\gamma\in\mc{O})$ is positive definite by \Cref{lem:open-conditioning}.

\emph{Injectivity.} By \eqref{eq:F-gradient} and \Cref{lem:open-conditioning}, $\nabla^2\Phi$ is positive definite everywhere, so $\Phi$ is strictly convex. If $\nabla\Phi(\vartheta)=\nabla\Phi(\vartheta')$ with $\vartheta\ne\vartheta'$, then the strictly convex function $t\mapsto\Phi(\vartheta+t(\vartheta'-\vartheta))$ would have equal derivatives at $t=0$ and $t=1$, which is impossible. Hence $\nabla\Phi$ is injective, and so is $F_{\mc{O}}=\nabla\Phi\circ\Sigma^{-1}$.

\emph{Surjectivity.} Fix $y\in\mc{O}$ and consider the smooth function $\Psi(\vartheta):=\Phi(\vartheta)-\vartheta\cdot y$. We show that $\Psi(\vartheta)\to\infty$ as $|\vartheta|\to\infty$. Since $\mc{O}$ is open there is $\eta>0$ with $B(y,2\eta)\subseteq\mc{O}$. For a unit vector $e$ and $t>0$, the ball $B(y+\frac32\eta e,\frac12\eta)$ is contained in $B(y,2\eta)\subseteq\mc{O}$ and in the half-space $\{x:(x-y)\cdot e>\eta\}$, so
\[
\Psi(te)=\log\int_{\mc{O}}\exp\left(t(x-y)\cdot e-\tfrac12x^\top\Sigma^{-1}x\right)dx
\ge\log\left(e^{t\eta}\,m_0\right)=t\eta+\log m_0,
\]
where
\[
m_0:=\mr{vol}\left(B(0,\tfrac12\eta)\right)\inf_{|x-y|\le2\eta}e^{-\frac12x^\top\Sigma^{-1}x}>0
\]
does not depend on $e$. Hence $\Psi\ge t\eta+\log m_0$ on the sphere of radius $t$, which exceeds $\Psi(0)$ for $t$ large. A continuous function on a closed ball attains its minimum; as $\Psi$ is larger on the boundary of a large ball than at its center, the minimum over the ball is attained at an interior point $\vartheta^\ast$, where $\nabla\Psi(\vartheta^\ast)=0$, that is, $\nabla\Phi(\vartheta^\ast)=y$. Thus $F_{\mc{O}}(\Sigma\vartheta^\ast)=y$.

\emph{Smoothness of the inverse.} $F_{\mc{O}}$ is a smooth bijection $\mb{R}^d\to\mc{O}$ whose Jacobian is everywhere invertible, so by the inverse function theorem it is a local diffeomorphism around every point; hence its inverse is smooth on $\mc{O}$.
\end{proof}

\begin{corollary}[Conditional means on a cone]\label{cor:cone-bijection}
Let $d\ge1$, let $0\le r\le d$, let $M\in\mb{R}^{r\times d}$ have rank $r$, and let $\Sigma\in\mb{R}^{d\times d}$ be positive definite. Let $\mc{C}_M:=\{x\in\mb{R}^d:Mx>0\}$, where $Mx>0$ means that every coordinate is positive, so that $\mc{C}_M=\mb{R}^d$ when $r=0$. Then $\mc{C}_M$ is a nonempty open convex set, and the map $F_{\mc{C}_M}$ of \Cref{thm:orthant-bijection} is a strong bijection $\mb{R}^d\to\mc{C}_M$.
\end{corollary}
\begin{proof}
$\mc{C}_M$ is the intersection of $r$ open half-spaces, hence open and convex. It is nonempty: since $M$ has rank $r$, the map $x\mapsto Mx$ is onto $\mb{R}^r$, so some $x$ has $Mx=(1,\ldots,1)$. Apply \Cref{thm:orthant-bijection} with $\mc{O}=\mc{C}_M$.
\end{proof}

\subsection{Solving perturbed equations}

\begin{theorem}[Perturbation of a strong bijection]\label{thm:perturbed-bijection}
Let $d\ge1$, let $B\subseteq\mb{R}^d$ be open, let $f\colon\mb{R}^d\to B$ be a strong bijection, let $\mc{K}\subseteq B$ be compact, and let $C_0>0$. Put $R_0:=\max_{y\in\mc{K}}|f^{-1}(y)|+1$ when $\mc{K}\ne\emptyset$, and $R_0:=1$ when $\mc{K}=\emptyset$. There exist $\varepsilon_0>0$ and $C>0$, depending only on $f$, $\mc{K}$, and $C_0$, such that the following holds. Let $0<\varepsilon<\varepsilon_0$ and let $g\colon\mb{R}^d\to\mb{R}^d$ be a continuous map with
\[
\sup_{|x|\le R_0}\left|g(x)-f(x)\right|\le C_0\,\varepsilon .
\]
Then for every $y\in\mc{K}$ there is $x\in\mb{R}^d$ with
\[
g(x)=y\qquad\text{and}\qquad\left|x-f^{-1}(y)\right|\le C\varepsilon .
\]
In particular this applies to $B=\mc{C}_M$ as in \Cref{cor:cone-bijection} and, for $T>1$, to the compact set
\[
\mc{K}_T:=\left\{y\in\mc{C}_M:\norm{y}_\infty\le T\text{ and }\min_{1\le i\le r}(My)_i\ge T^{-1}\right\}
\]
(the closed cube $[-T,T]^d$ when $r=0$), with $R_0$, $\varepsilon_0$, and $C$ depending only on $d$, $M$, $f$, $C_0$, and $T$.
\end{theorem}
\begin{proof}
If $\mc{K}=\emptyset$, the conclusion is vacuous, so take $\varepsilon_0=C=1$. Assume henceforth that $\mc{K}\ne\emptyset$. Since $\mc{K}$ is compact and $B$ is open, $r_0:=\frac12\operatorname{dist}(\mc{K},\mb{R}^d\setminus B)$ is positive (with $r_0:=1$ if $B=\mb{R}^d$), and $\mc{K}':=\{z:\operatorname{dist}(z,\mc{K})\le r_0\}$ is a compact subset of $B$. Let $L:=L(f,\mc{K}')$ be the constant of \Cref{lem:strong-bijection-lipschitz}, and put
\[
\varepsilon_0:=\min\left\{\frac{r_0}{C_0},\ \frac{1}{L\,C_0}\right\},\qquad C:=L\,C_0.
\]
Fix $\varepsilon\in(0,\varepsilon_0)$, a map $g$ as in the statement, and $y\in\mc{K}$, and let $x_0:=f^{-1}(y)$, so that $|x_0|\le R_0-1$. For $x$ in the closed ball $\ol{B}(x_0,1)$ we have $|x|\le R_0$, hence $|f(x)-g(x)|\le C_0\varepsilon<r_0$, so the point
\[
z(x):=y+f(x)-g(x)
\]
lies in $\mc{K}'\subseteq B$, and $h(x):=f^{-1}(z(x))$ is defined. The map $h\colon\ol{B}(x_0,1)\to\mb{R}^d$ is continuous, being a composition of continuous maps, and by \Cref{lem:strong-bijection-lipschitz} on $\mc{K}'$,
\[
|h(x)-x_0|=\left|f^{-1}(z(x))-f^{-1}(y)\right|\le L\,|z(x)-y|\le L\,C_0\,\varepsilon\le1 .
\]
Thus $h$ maps $\ol{B}(x_0,1)$ into itself, and by Brouwer's fixed point theorem \cite[Corollary~2.15]{Hat02} it has a fixed point $x^\ast$. Then $f(x^\ast)=z(x^\ast)=y+f(x^\ast)-g(x^\ast)$, that is, $g(x^\ast)=y$, and the displayed inequality with $x=x^\ast$ gives $|x^\ast-f^{-1}(y)|\le L\,C_0\,\varepsilon=C\varepsilon$.

For the last assertion, $\mc{K}_T$ is closed and bounded, hence compact, and it is contained in $\mc{C}_M$; the constants then depend on $\mc{K}_T$ only through $d$, $M$, and $T$.
\end{proof}

\section{The idealized process: existence, perturbation, and the edge day}
\label{app:idealized}

This appendix proves the theorems of \Cref{sec:idealized}. Its inputs are the Gaussian approximations of \Cref{app:binomial-estimates} (\Cref{thm:general_poly_calc,thm:tilted_expansion}), the inverse-function theorems of \Cref{app:inverse}, and the universal recursion of \Cref{sec:universal}. Throughout, $\theta\in(1/2,1)$ and $T>1$ are fixed, $N$ is large, $p\in(T^{-1}N^{-\theta},TN^{-\theta})$, and for a size vector $n\in\mb{Z}_{\ge1}^{\{0,1\}^k}$ and a tilt $\lambda$ we write $A_{s,\bullet}(n,\lambda)$ for the row vector of \Cref{def:local-template}. For $s\in\{0,1\}^k$ we write $M_s\in\{0,\pm1\}^{(k-1)\times2^k}$ for the matrix whose $r$th row is the coefficient vector of the form $Z_{s,r}$ of \Cref{def:gaussian-row}, and $\ge'_s\in\{\ge,>\}^{k-1}$ for the relation such that $\mc{I}_s(x)$ is the assertion $M_sx\ge'_s0$ (\Cref{def:local-admissibility} and the proof of \Cref{lem:local-exact-total-lower-bound}); the rows of $M_s$ are nonzero and pairwise orthogonal.

\subsection{Uniqueness of tilts}

\begin{lemma}[The conditioned mean map is injective]\label{lem:tilt-uniqueness}
Let $d\ge1$, let $\eta\in\mb{Z}_{\ge1}^d$, let $M\in\mb{R}^{r\times d}$, and let $\ge'\in\{\ge,>\}^r$. Suppose that the set $S:=\{a\in\prod_{t\in[d]}\{0,\ldots,\eta_t\}:Ma\ge'0\}$ is not contained in an affine hyperplane of $\mb{R}^d$. For $q\in(0,1)^d$ let $A^q$ have independent coordinates $A^q_t\sim\mr{Bin}(\eta_t,q_t)$. Then $\mb{P}[MA^q\ge'0]>0$ for every $q$, and the map
\[
(0,1)^d\to\mb{R}^d,\qquad q\mapsto\mb{E}\left[A^q\mid MA^q\ge'0\right],
\]
is injective.
\end{lemma}
\begin{proof}
Every point $a$ of the box $\prod_t\{0,\ldots,\eta_t\}$ has $\mb{P}[A^q=a]=\prod_t\binom{\eta_t}{a_t}q_t^{a_t}(1-q_t)^{\eta_t-a_t}>0$, and $S$ is nonempty (a set contained in no affine hyperplane contains $d+1$ affinely independent points), so $\mb{P}[MA^q\ge'0]=\mb{P}[A^q\in S]>0$. Write $\vartheta_t:=\log\frac{q_t}{1-q_t}$, which is a bijection $(0,1)^d\to\mb{R}^d$, and $\binom{\eta}{a}:=\prod_t\binom{\eta_t}{a_t}$. Since $q_t^{a_t}(1-q_t)^{\eta_t-a_t}=e^{\vartheta_ta_t}(1-q_t)^{\eta_t}$,
\[
\mb{P}\left[A^q=a\mid A^q\in S\right]=\binom{\eta}{a}\exp\left(\vartheta\cdot a-\psi(\vartheta)\right)\quad(a\in S),\qquad\psi(\vartheta):=\log\sum_{a\in S}\binom{\eta}{a}e^{\vartheta\cdot a}.
\]
The function $\psi$ is smooth on $\mb{R}^d$ (a logarithm of a finite positive sum of exponentials), and differentiating gives $\nabla\psi(\vartheta)=\mb{E}[A^q\mid A^q\in S]$ and $\nabla^2\psi(\vartheta)=\mr{Cov}(A^q\mid A^q\in S)$. For $u\in\mb{R}^d\setminus\{0\}$, $u^\top\nabla^2\psi(\vartheta)u=\mr{Var}(u\cdot A^q\mid A^q\in S)$, and this vanishes only if $u\cdot a$ is constant over $S$ (all points of $S$ have positive conditional probability), that is, only if $S$ lies in an affine hyperplane, which is excluded. Hence $\psi$ is strictly convex, so $\nabla\psi$ is injective (if $\nabla\psi(\vartheta)=\nabla\psi(\vartheta')$ with $\vartheta\ne\vartheta'$, the strictly convex function $t\mapsto\psi(\vartheta+t(\vartheta'-\vartheta))$ would have equal derivatives at $0$ and $1$), and so is $q\mapsto\nabla\psi(\vartheta(q))$.
\end{proof}

In the applications below, $M=M_s$ or $M=M_{sb}$ (defined in \Cref{lem:gaussian-limits} below), and $S$ contains every lattice point of a box of side $\sqrt{pN}$ around a point $p\,\eta+\sqrt{pN}\,x_0$ with $x_0$ interior to the cone, so that the hypothesis of \Cref{lem:tilt-uniqueness} holds for large $N$.

\subsection{Gaussian limits of the binomial row model}

We first record the regularity of the Gaussian functionals that appear as limits.

\begin{lemma}[Regularity of Gaussian functionals]\label{lem:gaussian-regularity}
Let $d\ge1$, $0\le r\le d$, let $M\in\mb{R}^{r\times d}$ have rank $r$, and let $\mc{P}\subseteq\mb{R}^d\times\mb{R}_{>0}^d\times\mb{R}^r$ be compact. For $(m,v,u)\in\mc{P}$ let $X\sim\mc{N}(m,\operatorname{diag}(v))$ and let $\mc{E}_u:=\{x\in\mb{R}^d:Mx\ge-u\}$, the inequality holding coordinatewise. Then there are constants $c_{\mc{P}}>0$ and $L_{\mc{P}}$ such that $\mb{P}[X\in\mc{E}_u]\ge c_{\mc{P}}$ on $\mc{P}$, and each of the functions
\[
(m,v,u)\mapsto\mb{P}\left[X\in\mc{E}_u\right],\qquad\mb{E}\left[X_t\mathbf{1}_{X\in\mc{E}_u}\right],\qquad\mb{E}\left[X_tX_{t'}\mathbf{1}_{X\in\mc{E}_u}\right]\qquad(t,t'\in[d])
\]
is Lipschitz on $\mc{P}$ with constant $L_{\mc{P}}$. Consequently the conditional expectations $\mb{E}[X_t\mid X\in\mc{E}_u]$ and $\mb{E}[X_tX_{t'}\mid X\in\mc{E}_u]$, and the conditional covariances, are Lipschitz on $\mc{P}$ as well.
\end{lemma}
\begin{proof}
Since $M$ has rank $r$, the linear map $x\mapsto Mx$ is onto $\mb{R}^r$, so $\mc{E}_u$, the preimage of the closed orthant $\{z\ge-u\}$, has nonempty interior, and $\mb{P}[X\in\mc{E}_u]>0$. The functions in the statement are continuous in $(m,v,u)$: in $(m,v)$ by dominated convergence, and in $u$ because the boundary of $\mc{E}_u$ is a finite union of hyperplanes, which is Lebesgue-null. A positive continuous function on the compact set $\mc{P}$ is bounded below, which gives $c_{\mc{P}}$.

For the Lipschitz bounds, enlarge $\mc P$ to a compact rectangular box in $\mb R^d\times\mb R_{>0}^d\times\mb R^r$. All variances on this box remain bounded away from zero, and the box contains the line segments used in the mean value inequality below. Proving the bounds on this box suffices for the original set. Write $\varphi_{m,v}$ for the density of $X$ and let $h$ be one of the functions $1$, $x_t$, $x_tx_{t'}$. The partial derivatives $\partial_{m_t}\varphi_{m,v}=\frac{x_t-m_t}{v_t}\varphi_{m,v}$ and $\partial_{v_t}\varphi_{m,v}=\big(\frac{(x_t-m_t)^2}{2v_t^2}-\frac{1}{2v_t}\big)\varphi_{m,v}$ are bounded in absolute value by $C(1+|x|^2)\varphi_{m,v}(x)$ for a constant $C$ depending on $\mc{P}$, and $(1+|x|^2)|h(x)|\varphi_{m,v}(x)$ is integrable uniformly over $\mc{P}$; hence $(m,v)\mapsto\int_{\mc{E}_u}h\,\varphi_{m,v}$ is continuously differentiable with derivatives bounded on $\mc{P}$, uniformly in $u$, and therefore Lipschitz in $(m,v)$. For the dependence on $u$, let $u,u'\in\mb{R}^r$ and write $Z:=MX$. The symmetric difference $\mc{E}_u\,\triangle\,\mc{E}_{u'}$ is contained in $\bigcup_{i\le r}\{x:-\max(u_i,u_i')\le(Mx)_i<-\min(u_i,u'_i)\}$, so
\[
\left|\int_{\mc{E}_u}h\varphi_{m,v}-\int_{\mc{E}_{u'}}h\varphi_{m,v}\right|\le\sum_{i\le r}\mb{E}\left[|h(X)|\,\mathbf{1}\{Z_i\in J_i\}\right],\qquad J_i:=\left[-\max(u_i,u'_i),-\min(u_i,u'_i)\right].
\]
Each $Z_i$ is a nondegenerate Gaussian variable whose variance is bounded below on $\mc{P}$, so its density is bounded by a constant $C'$; and given $Z_i=z$, the vector $X$ is Gaussian with mean affine in $z$ and bounded covariance, so $\mb{E}[|h(X)|\mid Z_i=z]\le C''(1+z^2)$. Integrating over $z\in J_i$ gives $\mb{E}[|h(X)|\mathbf 1\{Z_i\in J_i\}]\le C'C''\sup_{z\in J_i}(1+z^2)\,|u_i-u'_i|$, which is at most a constant times $|u-u'|$ on the compact set $\mc{P}$. The conditional quantities are quotients of Lipschitz functions by the function $\mb{P}[X\in\mc{E}_u]$, which is Lipschitz and bounded below by $c_{\mc{P}}$; hence they are Lipschitz.
\end{proof}

The next lemma is the bridge between the binomial row models of \Cref{sec:local} and the Gaussian row models of \Cref{sec:universal}. For $s\in\{0,1\}^k$ and $b\in\{0,1\}$ let $M_{sb}\in\mb{R}^{k\times2^k}$ be the matrix whose rows are the rows of $M_s$ followed by $(-1)^b$ times the coefficient vector of $Z_{s,k}$, and $\ge'_{sb}\in\{\ge,>\}^k$ the relation such that $\mc{J}_{sb}(x)$ is the assertion $M_{sb}x\ge'_{sb}0$; as for $M_s$, the rows of $M_{sb}$ are nonzero and pairwise orthogonal. For $u\in\mb{R}$ let $\mc{J}^{u}_{sb}(x)$ be the assertion ``$\mc{I}_s(x)$ and $(-1)^b(Z_{s,k}(x)+u)>0$'', the majority decision with its threshold shifted by $u$; thus $\mc{J}^0_{sb}(x)$ is $\mc{J}_{sb}(x)$ with a strict last inequality, which is immaterial for Gaussian vectors. The shift appears because the binomial row model centered at $p\eta$ has $Z_{s,k}(p\eta)\ne0$ when the sizes $\eta$ are not balanced between the two opinions.

\begin{lemma}[Gaussian limits of the binomial row model]\label{lem:gaussian-limits}
Fix $\theta\in(1/2,1)$, $k\ge1$, $s\in\{0,1\}^k$, $T>1$, $\ell\ge1$, and $R>0$. There are constants $\ell'=\ell'(k,\ell)$, $c_0=c_0(k,T,R)>0$, and $C=C(\theta,k,T,\ell,R)$ such that for all sufficiently large $N$, all $p\in(T^{-1}N^{-\theta},TN^{-\theta})$, every $\xi\in(0,T]$, and every $\eta\in\mb{Z}^{\{0,1\}^k}$ with
\[
\left|\eta[t]-N\nu[t]\right|\le\frac{N}{\sqrt{pN}}\log^\ell N\ \text{ for all }t,\qquad\norm{M_s\eta}_\infty\le\xi\,\frac{N}{\sqrt{pN}},\qquad\left|Z_{s,k}(\eta)\right|\le T\frac{N}{\sqrt{pN}},
\]
the following holds with $\epsilon:=\frac{\log^{\ell'}N}{\sqrt{pN}}+\xi$ and $u_\eta:=p\,Z_{s,k}(\eta)/\sqrt{pN}\in[-T,T]$. For $\sigma\in\mb{R}^{\{0,1\}^k}$ with $\norm{\sigma}_\infty\le R$, let $A^\sigma$ have independent coordinates $A^\sigma_t\sim\mr{Bin}\big(\eta[t]-\one_{s=t},\Lambda_{N,p}(\sigma[t],\nu[t])\big)$ and let $W^\sigma$ have independent coordinates $W^\sigma_t\sim\mc{N}(\sigma[t],\nu[t])$. Then all of the following quantities on the binomial side are smooth functions of $\sigma$, and for all $t,t'\in\{0,1\}^k$ and $b\in\{0,1\}$:
\begin{enumerate}[label=\textup{(\roman*)},ref=\textup{(\roman*)}]
\item\label{gl:prob} $\mb{P}[\mc{I}_s(W^\sigma)]\ge c_0$ and $\mb{P}[\mc{J}^{u_\eta}_{sb}(W^\sigma)]\ge c_0$; $\ \big|\mb{P}[\mc{I}_s(A^\sigma)]-\mb{P}[\mc{I}_s(W^\sigma)]\big|\le C\epsilon$; and $\big|\mb{P}[\mc{J}_{sb}(A^\sigma)\mid\mc{I}_s(A^\sigma)]-\mb{P}[\mc{J}^{u_\eta}_{sb}(W^\sigma)\mid\mc{I}_s(W^\sigma)]\big|\le C\epsilon$;
\item\label{gl:mean} $\displaystyle\left|\frac{\mb{E}[A^\sigma_t\mid\mc{I}_s(A^\sigma)]-p\,\eta[t]}{\sqrt{pN}}-\mb{E}[W^\sigma_t\mid\mc{I}_s(W^\sigma)]\right|\le C\epsilon$, and the same with $\mc{I}_s(A^\sigma)$ replaced by $\mc{J}_{sb}(A^\sigma)$ and $\mc{I}_s(W^\sigma)$ by $\mc{J}^{u_\eta}_{sb}(W^\sigma)$;
\item\label{gl:cov} $\displaystyle\left|\frac{\mr{Cov}(A^\sigma_t,A^\sigma_{t'}\mid\mc{I}_s(A^\sigma))}{pN}-\mr{Cov}(W^\sigma_t,W^\sigma_{t'}\mid\mc{I}_s(W^\sigma))\right|\le C\epsilon$.
\end{enumerate}
If moreover $|Z_{s,k}(\eta)|\le\xi\frac{N}{\sqrt{pN}}$, then $|u_\eta|\le\xi\le\epsilon$, and \ref{gl:prob} and \ref{gl:mean} hold with $\mc{J}^{u_\eta}_{sb}(W^\sigma)$ replaced by $\mc{J}_{sb}(W^\sigma)$. In particular $F_{s,\nu}(\sigma)=\mb{E}[W^\sigma\mid\mc{I}_s(W^\sigma)]$ is the vector on the Gaussian side of \ref{gl:mean}.
\end{lemma}
\begin{proof}
Write $d:=2^k$, $\eta'_t:=\eta[t]-\one_{s=t}$, and let $N$ be large enough that $\eta'_t\in(\frac12N\nu[t],2N\nu[t])$ for all $t$. Put
\[
\alpha_t:=\frac{\sigma[t]}{\nu[t]}\sqrt{\frac{\eta'_t}{N}},\qquad\text{so that}\qquad\Lambda_{N,p}(\sigma[t],\nu[t])=\frac{\frac{p}{1-p}e^{\alpha_t/\sqrt{p\eta'_t}}}{1+\frac{p}{1-p}e^{\alpha_t/\sqrt{p\eta'_t}}},
\]
and $|\alpha_t|\le R\sqrt{2/\nu[t]}$. Let $T_1$ be a constant, depending on $\theta,k,T,R$, exceeding $2T$, $2/\min_t\nu[t]$, and $\max_t R\sqrt{2/\nu[t]}$. We apply \Cref{thm:general_poly_calc} with $d$, $n=N$, the constant $T_1$, $\eta=\eta'$, $\alpha$ as above, and $\rho\in\{1,x_t,x_tx_{t'}\}$, once with $(M,\ge')=(M_s,\ge'_s)$ and once with $(M,\ge')=(M_{sb},\ge'_{sb})$; its hypotheses hold because $\eta'_t\in(T_1^{-1}N,T_1N)$, $|\alpha_t|<T_1$, $p\in(T_1^{-1}N^{-\theta},T_1N^{-\theta})$, and $\norm{M_s\eta'}_\infty\le\norm{M_s\eta}_\infty+1\le TN/\sqrt{pN}+1<T_1N/\sqrt{pN}$, while $\norm{M_{sb}\eta'}_\infty=\max\{\norm{M_s\eta'}_\infty,|Z_{s,k}(\eta')|\}\le\max\{\norm{M_s\eta}_\infty,|Z_{s,k}(\eta)|\}+1<T_1N/\sqrt{pN}$. Its conclusion is, for each admissible $\rho$ and $M$,
\begin{equation}\label{eq:gl-gpc}
\mb{E}\left[\rho(A^\sigma-p\eta')\mathbf{1}_{MA^\sigma\ge'0}\right]=\mb{E}\left[\rho(X)\mathbf{1}_{MX\ge-pM\eta'}\right]+O\left(\sqrt{pN}^{\deg\rho-1}\log^{5+d}N\right),
\end{equation}
where $X$ has independent coordinates $X_t\sim\mc{N}(\sqrt{p\eta'_t}\,\alpha_t,\,p\eta'_t)$, and the events $\{MA^\sigma\ge'0\}$ are $\mc{I}_s(A^\sigma)$ and $\mc{J}_{sb}(A^\sigma)$ respectively.

\emph{Normalization.} Put $\hat X:=X/\sqrt{pN}$, which has independent coordinates $\hat X_t\sim\mc{N}(\hat m_t,\hat v_t)$ with
\[
\hat m_t=\frac{\sigma[t]\,\eta'_t}{N\nu[t]},\qquad\hat v_t=\frac{\eta'_t}{N},\qquad|\hat m_t-\sigma[t]|\le\frac{R}{\nu[t]}\left(\frac{\log^\ell N}{\sqrt{pN}}+\frac1N\right),\qquad|\hat v_t-\nu[t]|\le\frac{\log^\ell N}{\sqrt{pN}}+\frac1N,
\]
and the event $\{MX\ge-pM\eta'\}$ is $\{M\hat X\ge-u\}$ with $u:=pM\eta'/\sqrt{pN}$. For $M=M_s$, $\norm{u}_\infty\le\xi+p/\sqrt{pN}\le2\epsilon$. For $M=M_{sb}$, the first $k-1$ coordinates of $u$ are those of the previous case and the last is $(-1)^bpZ_{s,k}(\eta')/\sqrt{pN}=(-1)^b(u_\eta+O(p/\sqrt{pN}))$, so that $\{M_{sb}\hat X\ge-u\}$ is, up to a shift of the first $k-1$ thresholds by $O(\epsilon)$ and of the last by $O(p/\sqrt{pN})\le O(\epsilon)$, the event $\mc{J}^{u_\eta}_{sb}(\hat X)$. Dividing \eqref{eq:gl-gpc} by $\sqrt{pN}^{\deg\rho}$,
\begin{equation}\label{eq:gl-normalized}
\mb{E}\left[\rho\left(\frac{A^\sigma-p\eta'}{\sqrt{pN}}\right)\mathbf{1}_{MA^\sigma\ge'0}\right]=\mb{E}\left[\rho(\hat X)\mathbf{1}_{M\hat X\ge-u}\right]+O\left(\frac{\log^{5+d}N}{\sqrt{pN}}\right).
\end{equation}

\emph{Comparison of the Gaussian functionals.} Apply \Cref{lem:gaussian-regularity} with $M\in\{M_s,M_{sb}\}$ and the compact parameter set $\mc{P}$ of all $(m,v,u)$ with $\norm{m}_\infty\le2R/\min_t\nu[t]$, $\frac12\nu[t]\le v_t\le2\nu[t]$, and $\norm{u}_\infty\le2T+1$. For large $N$ the triple $(\hat m,\hat v,u)$ lies in $\mc{P}$, and so does the triple $(\sigma,\nu,0)$ for $M=M_s$, respectively $(\sigma,\nu,(0,\ldots,0,(-1)^bu_\eta))$ for $M=M_{sb}$; the two triples are at distance $O(\log^\ell N/\sqrt{pN}+\xi+p/\sqrt{pN})=O(\epsilon)$ from each other. Hence the right-hand side of \eqref{eq:gl-normalized} differs by $O(\epsilon)$ from $\mb{E}[\rho(W^\sigma)\mathbf{1}_{\mc{I}_s(W^\sigma)}]$ for $M=M_s$, respectively from $\mb{E}[\rho(W^\sigma)\mathbf{1}_{\mc{J}^{u_\eta}_{sb}(W^\sigma)}]$ for $M=M_{sb}$, where we used that $\{M_sW^\sigma\ge0\}$ and $\{M_{sb}W^\sigma\ge-(0,\ldots,0,(-1)^bu_\eta)\}$ coincide with $\mc{I}_s(W^\sigma)$ and $\mc{J}^{u_\eta}_{sb}(W^\sigma)$ up to null sets. Taking $\ell':=5+d+\ell$, we conclude that for $\rho\in\{1,x_t,x_tx_{t'}\}$,
\begin{equation}\label{eq:gl-final}
\mb{E}\left[\rho\left(\frac{A^\sigma-p\eta'}{\sqrt{pN}}\right)\mathbf{1}_{\mc{I}_s(A^\sigma)}\right]=\mb{E}\left[\rho(W^\sigma)\mathbf{1}_{\mc{I}_s(W^\sigma)}\right]+O(\epsilon),
\end{equation}
and the same with $\mc{J}_{sb}(A^\sigma)$ and $\mc{J}^{u_\eta}_{sb}(W^\sigma)$ in place of $\mc{I}_s(A^\sigma)$ and $\mc{I}_s(W^\sigma)$.

\emph{Uniform denominator bounds.} The lower probability bound in \Cref{lem:gaussian-regularity} applies also at the actual parameters $(\hat m,\hat v,u)\in\mc{P}$. Choose $c_0>0$ as the minimum of these compact-family lower bounds over the finitely many matrices $M_s$ and $M_{sb}$ with $s\in\{0,1\}^k$ and $b\in\{0,1\}$. Equation \eqref{eq:gl-normalized} with $\rho=1$ gives, for each of these matrices,
\[
\mb{P}[MA^\sigma\ge'0]\ge c_0-O\left(\frac{\log^{5+d}N}{\sqrt{pN}}\right)\ge\frac{c_0}{2}
\]
uniformly for sufficiently large $N$. The error tends to zero uniformly in the density window, regardless of $\xi\in(0,T]$; the later error $O(\epsilon)$ need not tend to zero. Thus both the binomial history and child-event probabilities are bounded away from zero before we divide the raw-moment comparisons.

\emph{Conclusion.} The lower bounds $\mb{P}[\mc{I}_s(W^\sigma)]\ge c_0$ and $\mb{P}[\mc{J}^{u_\eta}_{sb}(W^\sigma)]\ge c_0$ follow from the choice of $c_0$ and \Cref{lem:gaussian-regularity} at the parameters $(\sigma,\nu,0)$ and $(\sigma,\nu,(0,\ldots,0,(-1)^bu_\eta))$. The two approximation bounds of \ref{gl:prob} are \eqref{eq:gl-final} with $\rho=1$, the second after dividing the $\mc{J}$ version by the $\mc{I}$ version and using the Gaussian lower bound $c_0$ and the binomial lower bound $c_0/2$ just proved. For \ref{gl:mean}, \eqref{eq:gl-final} with $\rho=x_t$ and $\rho=1$ gives
\[
\frac{\mb{E}[A^\sigma_t\mid\mc{I}_s(A^\sigma)]-p\eta'_t}{\sqrt{pN}}=\frac{\mb{E}[W^\sigma_t\mathbf{1}_{\mc{I}_s(W^\sigma)}]+O(\epsilon)}{\mb{P}[\mc{I}_s(W^\sigma)]+O(\epsilon)}=\mb{E}[W^\sigma_t\mid\mc{I}_s(W^\sigma)]+O(\epsilon),
\]
and $p\eta'_t$ differs from $p\eta[t]$ by at most $p\le\sqrt{pN}\,\epsilon$; the $\mc{J}$ version is identical. For \ref{gl:cov}, use $\rho=x_tx_{t'}$, $x_t$, $x_{t'}$, and $1$ in the same way. Finally, if $|Z_{s,k}(\eta)|\le\xi N/\sqrt{pN}$ then $|u_\eta|\le\xi\le\epsilon$, and by \Cref{lem:gaussian-regularity} the Gaussian quantities with the event $\mc{J}^{u_\eta}_{sb}(W^\sigma)$ differ from those with $\mc{J}_{sb}(W^\sigma)$ by $O(\epsilon)$. Smoothness in $\sigma$: each binomial quantity is a rational function, with positive denominator, of the finitely many values $\Lambda_{N,p}(\sigma[t],\nu[t])^{a}(1-\Lambda_{N,p}(\sigma[t],\nu[t]))^{\eta'_t-a}$, which are smooth in $\sigma$.
\end{proof}

\subsection{The linear response of the binomial row model}

\begin{lemma}[Linear response to a small tilt]\label{lem:tilt-vs-linear-map}
Fix $\theta\in(1/2,1)$, an integer $k$ with $1\le k<\frac1{1-\theta}$, $s\in\{0,1\}^k$, $T>1$, and $R>0$. There are constants $\kappa=\kappa(\theta,k)>0$ and $C=C(\theta,k,T,R)$ such that for all sufficiently large $N$, all $p\in(T^{-1}N^{-\theta},TN^{-\theta})$, all $\tau\in[T^{-1},T]$, and every $n\in\mb{Z}^{\{0,1\}^k}$ with $|n[t]-\wt{n}[t]|\le T\sqrt N\sqrt{pN}^{\,k-1}$ for all $t$, the following holds, where $(\wt n,\wt m,\wt\lambda)$ is the idealized process of \Cref{thm:idealized-process} with $D=D_\theta$ and $\wt A_{s,\bullet}=A_{s,\bullet}(\wt n_k,\wt\lambda_k)$. Write $\beta_0:=\sqrt{pN}^{\,k-1}/\sqrt N$. For $\sigma\in\mb{R}^{\{0,1\}^k}$ with $\norm\sigma_\infty\le R$ let $q^\sigma\in(0,1)^{\{0,1\}^k}$ be defined by
\[
\log\frac{q^\sigma[t]\left(n[t]-\one_{s=t}-p\,\wt n[t]\right)}{\left(1-q^\sigma[t]\right)p\,\wt n[t]}
=\wt\vartheta[s,t]+\tau\beta_0\,\sigma[t],\qquad\wt\vartheta[s,t]:=\log\frac{\wt\lambda[s,t]\left(\wt n[t]-\one_{s=t}-p\,\wt n[t]\right)}{\left(1-\wt\lambda[s,t]\right)p\,\wt n[t]},
\]
and let $A^\sigma$ have independent coordinates $A^\sigma_t\sim\mr{Bin}(n[t]-\one_{s=t},q^\sigma[t])$. Then $q^\sigma$ and the quantities below are smooth in $\sigma$, $|q^\sigma[t]-p|\le C\,p/\sqrt{pN}$, and for all $t\in\{0,1\}^k$ and $b\in\{0,1\}$:
\begin{enumerate}[label=\textup{(\roman*)},ref=\textup{(\roman*)}]
\item\label{lr:mean} $\displaystyle\left|\frac{\mb{E}[A^\sigma_t\mid\mc{I}_s(A^\sigma)]-\wt m[s,t]/\wt n[s]}{\tau\beta_0\,pN}-(\Sigma_s\sigma)[t]\right|\le CN^{-\kappa}$, where $\Sigma_s$ is the conditional covariance of \Cref{lem:cond-cov-pd};
\item\label{lr:split} with $W_{s,\bullet}$ the day-$k$ Gaussian row model,
\begin{multline*}
\left|\frac{\mb{P}[\mc{J}_{sb}(A^\sigma)\mid\mc{I}_s(A^\sigma)]-\mb{P}[\mc{J}_{sb}(\wt A_{s,\bullet})\mid\mc{I}_s(\wt A_{s,\bullet})]}{\tau\beta_0\sqrt{pN}}\right.\\
\left.-\frac{\nu[sb]}{\nu[s]}\sum_{t'}\sigma[t']\left(\mb{E}[W_{s,t'}\mid\mc{J}_{sb}(W_{s,\bullet})]-\mb{E}[W_{s,t'}\mid\mc{I}_s(W_{s,\bullet})]\right)\right|\le CN^{-\kappa};
\end{multline*}
\item\label{lr:condmean} $\displaystyle\left|\mb{E}[A^\sigma_t\mid\mc{J}_{sb}(A^\sigma)]-\mb{E}[\wt A_{s,t}\mid\mc{J}_{sb}(\wt A_{s,\bullet})]\right|\le C\tau\beta_0\,pN$;
\item\label{lr:nondeg} $\mb{P}[\mc{I}_s(A^\sigma)]\ge\phi^\ast_k/4$ and $\phi^\ast_k/2\le\mb{P}[\mc{J}_{sb}(A^\sigma)\mid\mc{I}_s(A^\sigma)]\le1-\phi^\ast_k/2$, with $\phi^\ast_k$ the constant of \Cref{lem:universal-nondegeneracy}.
\end{enumerate}
\end{lemma}
\begin{proof}
Write $d:=2^k$ and $c_t:=p\,\wt n[t]$. Since $k(1-\theta)<1$, the number $c_1:=\frac12(1-k(1-\theta))$ is positive and
\begin{equation}\label{eq:beta0-small}
\beta_0\sqrt{pN}=\frac{\sqrt{pN}^{\,k}}{\sqrt N}\le T^{k/2}N^{-c_1},
\end{equation}
so $\beta_0\sqrt{pN}\log^{O(1)}N\to0$. We use \Cref{thm:idealized-process}\ref{ideal:asymptotics} throughout; in particular $\wt n[t]\asymp N$, $\wt m[s,t]/\wt n[s]-c_t=c_t\mu[s,t]/\sqrt{pN}+O(\log^\ell N)=O(\sqrt{pN})$, and $|\wt\lambda[s,t]-p|=O(p/\sqrt{pN})$ (because $\Lambda_{N,p}(\gamma,\nu)-p=O(p/\sqrt{pN})$ and $\log^\ell N/N\le p/\sqrt{pN}$ for large $N$).

\emph{The tilt.} By definition, $\operatorname{logit}q^\sigma[t]=\operatorname{logit}\wt\lambda[s,t]+\log\frac{\wt n[t]-\one_{s=t}-c_t}{n[t]-\one_{s=t}-c_t}+\tau\beta_0\sigma[t]$, and the middle term is $O(|n[t]-\wt n[t]|/N)=O(\beta_0)$ by the hypothesis on $n$. Hence $|\operatorname{logit}q^\sigma[t]-\operatorname{logit}\wt\lambda[s,t]|=O(\beta_0)$, so $|q^\sigma[t]-\wt\lambda[s,t]|=O(p\beta_0)$ and $|q^\sigma[t]-p|=O(p/\sqrt{pN})$; smoothness in $\sigma$ is clear. Now apply \Cref{thm:tilted_expansion} with $d$, $n=N$, a sufficiently large constant $T_1=T_1(\theta,k,T,R)\ge T$ in place of $T$, reference sizes $\wt\eta_t:=\wt n[t]$, trial counts $\wt\eta'_t:=\wt n[t]-\one_{s=t}$ and $\eta_t:=n[t]-\one_{s=t}$, success probabilities $\wt q:=\wt\lambda[s,\bullet]$ and $q:=q^\sigma$, and $\rho\in\{1,x_t\}$, once with $(M_s,\ge'_s)$ and once with $(M_{sb},\ge'_{sb})$. Its hypotheses hold for large $N$: $p\in(T_1^{-1}N^{-\theta},T_1N^{-\theta})$; $\wt\eta_t\asymp N$; $|\wt\eta'_t-\wt\eta_t|\le1$; $|\eta_t-\wt\eta_t|\le T\sqrt N\sqrt{pN}^{k-1}+1\le T^{1+k/2}N/\sqrt{pN}$ by \eqref{eq:beta0-small}; $|\wt q_t-p|,|q_t-p|=O(p/\sqrt{pN})$; and its tilt vector is exactly $\beta_t=\tau\beta_0\sigma[t]$, by the definition of $q^\sigma$, so $\norm\beta_\infty\le TR\beta_0<T_1/(\sqrt{pN}\log^2N)$ by \eqref{eq:beta0-small}. Writing $E_1:=\max\{p,\norm\beta_\infty^2pN\log^2N\}$, the conclusions are, for the event $\mc{I}_s$ (and identically for $\mc{J}_{sb}$),
\begin{align}
\mb{P}[\mc{I}_s(A^\sigma)]&=\wt P+\sum_{t'}\beta_{t'}J_{t'}+O(E_1\log^2N),\label{eq:lr-prob}\\
\mb{E}\left[(A^\sigma_t-c_t)\mathbf{1}_{\mc{I}_s(A^\sigma)}\right]&=\wt E_t\wt P+\sum_{t'}\beta_{t'}K_{t't}+O(E_1\sqrt{pN}\log^3N),\label{eq:lr-mean}
\end{align}
where $\wt P:=\mb{P}[\mc{I}_s(\wt A)]$, $\wt E_t:=\mb{E}[\wt A_t-c_t\mid\mc{I}_s(\wt A)]$, $J_{t'}:=\mr{Cov}(\wt A_{t'},\mathbf{1}_{\mc{I}_s(\wt A)})$, and $K_{t't}:=\mb{E}[(\wt A_{t'}-c_{t'})(\wt A_t-c_t)\mathbf{1}_{\mc{I}_s(\wt A)}]-\mb{E}[\wt A_{t'}-c_{t'}]\,\wt E_t\wt P$; here $\wt A:=\wt A_{s,\bullet}$. By \Cref{thm:idealized-process}\ref{ideal:solvable}, $\wt E_t=\wt m[s,t]/\wt n[s]-c_t=O(\sqrt{pN})$. By \Cref{lem:gaussian-limits} applied with $\eta=\wt n_k$ (so $\norm{M_s\eta}_\infty=0$ by \Cref{thm:idealized-process}\ref{ideal:symmetry} and \Cref{cor:sign-identities}\ref{sign:nu}, and we may take $\xi=\sqrt{pN}/N$) and with $\sigma$ the vector $\wt\sigma$ defined by $\Lambda_{N,p}(\wt\sigma[t],\nu[t])=\wt\lambda[s,t]$, which satisfies $\norm{\wt\sigma-\gamma[s,\bullet]}_\infty=O(\log^\ell N/\sqrt{pN})$ by \ref{ideal:asymptotics} and is therefore bounded, we have $\wt P\ge\phi^\ast_k/2$, $\mb{P}[\mc{J}_{sb}(\wt A)\mid\mc{I}_s(\wt A)]=\nu[sb]/\nu[s]+O(\epsilon_0)$, and
\begin{equation}\label{eq:lr-cov-limit}
\frac{K_{t't}-J_{t'}\wt E_t}{\wt P}=\mr{Cov}\left(\wt A_{t'},\wt A_t\mid\mc{I}_s(\wt A)\right)=pN\,\Sigma_s[t',t]+O(pN\epsilon_0),\qquad\epsilon_0:=\frac{\log^{\ell'}N}{\sqrt{pN}}+\frac{\sqrt{pN}}{N},
\end{equation}
where the first equality is a direct expansion of the definitions. The second is \Cref{lem:gaussian-limits}\ref{gl:cov} together with the Lipschitz dependence of the Gaussian conditional covariance on the mean (\Cref{lem:gaussian-regularity}), which converts the covariance of $W^{\wt\sigma}$ given $\mc{I}_s$ into $\Sigma_s$ at the cost of $O(\log^\ell N/\sqrt{pN})$. Also $|J_{t'}|\le\sqrt{\mr{Var}(\wt A_{t'})}$, which is $O(\sqrt{pN})$, and $|K_{t't}|=O(pN)$.

\emph{Proof of \ref{lr:mean}.} Dividing \eqref{eq:lr-mean} by \eqref{eq:lr-prob}, with $u:=\sum_{t'}\beta_{t'}J_{t'}/\wt P=O(\norm\beta_\infty\sqrt{pN})=O(\beta_0\sqrt{pN})$,
\begin{align*}
\mb{E}[A^\sigma_t\mid\mc{I}_s(A^\sigma)]-c_t&=\frac{\wt E_t\wt P+\sum_{t'}\beta_{t'}K_{t't}+O(E_1\sqrt{pN}\log^3N)}{\wt P\left(1+u+O(E_1\log^2N)\right)}\\
&=\wt E_t+\sum_{t'}\beta_{t'}\frac{K_{t't}-J_{t'}\wt E_t}{\wt P}+O\left(E_1\sqrt{pN}\log^3N+\beta_0^2\,pN\sqrt{pN}\right),
\end{align*}
where the last error collects the second-order terms $u^2\wt E_t$ and $u\sum\beta K/\wt P$, both $O(\beta_0^2pN\sqrt{pN})$, and $\wt E_t\cdot O(E_1\log^2N)=O(E_1\sqrt{pN}\log^2N)$. Substituting $\wt E_t=\wt m[s,t]/\wt n[s]-c_t$ and \eqref{eq:lr-cov-limit}, and dividing by $\tau\beta_0pN$, the main term becomes $(\Sigma_s\sigma)[t]+O(\epsilon_0)$ and the error terms become
\[
O\left(\frac{p\sqrt{pN}\log^3N}{\beta_0pN}+\frac{\beta_0^2pN\log^2N\sqrt{pN}\log^3N}{\beta_0pN}+\beta_0\sqrt{pN}\right)
=O\left(\frac{\sqrt p\log^3N}{\sqrt{pN}^{\,k-1}}+\beta_0\sqrt{pN}\log^5N\right),
\]
using $\norm\beta_\infty\le TR\beta_0$ and $\beta_0=\sqrt{pN}^{k-1}/\sqrt N$. The first term is at most $\sqrt p\log^3N$ and the second is $O(N^{-c_1}\log^5N)$ by \eqref{eq:beta0-small}; together with $\epsilon_0=O(\log^{\ell'}N/\sqrt{pN})$, all are $O(N^{-\kappa})$ for $\kappa:=\frac12\min\{c_1,\theta/2,(1-\theta)/2\}$ and large $N$.

\emph{Proof of \ref{lr:split}.} Write $\wt P_J:=\mb{P}[\mc{J}_{sb}(\wt A)]$ and $J^{(b)}_{t'}:=\mr{Cov}(\wt A_{t'},\mathbf{1}_{\mc{J}_{sb}(\wt A)})$. Dividing the $\mc{J}_{sb}$ version of \eqref{eq:lr-prob} by the $\mc{I}_s$ version,
\[
\mb{P}[\mc{J}_{sb}(A^\sigma)\mid\mc{I}_s(A^\sigma)]=\frac{\wt P_J}{\wt P}+\sum_{t'}\beta_{t'}\left(\frac{J^{(b)}_{t'}}{\wt P}-\frac{\wt P_J\,J_{t'}}{\wt P^2}\right)+O\left(E_1\log^2N+\beta_0^2pN\right),
\]
and a direct expansion gives $\frac{J^{(b)}_{t'}}{\wt P}-\frac{\wt P_JJ_{t'}}{\wt P^2}=\frac{\wt P_J}{\wt P}\left(\mb{E}[\wt A_{t'}\mid\mc{J}_{sb}(\wt A)]-\mb{E}[\wt A_{t'}\mid\mc{I}_s(\wt A)]\right)$. By \Cref{lem:gaussian-limits}\ref{gl:prob} and \ref{gl:mean} at $\wt\sigma$, together with the Lipschitz dependence on the mean of \Cref{lem:gaussian-regularity}, this equals $\sqrt{pN}\,\frac{\nu[sb]}{\nu[s]}\left(\mb{E}[W_{s,t'}\mid\mc{J}_{sb}(W_{s,\bullet})]-\mb{E}[W_{s,t'}\mid\mc{I}_s(W_{s,\bullet})]\right)+O(\sqrt{pN}\epsilon_0)$. Dividing by $\tau\beta_0\sqrt{pN}$ and using $\beta_{t'}=\tau\beta_0\sigma[t']$ gives the main term of \ref{lr:split}; the errors are $O(\epsilon_0)$ from the main term and
\[
O\left(\frac{p\log^2N}{\beta_0\sqrt{pN}}+\frac{\beta_0^2pN\log^4N}{\beta_0\sqrt{pN}}+\frac{\beta_0^2pN}{\beta_0\sqrt{pN}}\right)=O\left(\frac{\sqrt p\log^2N}{\sqrt{pN}^{\,k-1}}+\beta_0\sqrt{pN}\log^4N\right)=O(N^{-\kappa}).
\]

\emph{Proof of \ref{lr:condmean}.} Dividing the $\mc{J}_{sb}$ versions of \eqref{eq:lr-mean} and \eqref{eq:lr-prob}, exactly as in the proof of \ref{lr:mean} but with $\mc{J}_{sb}$ in place of $\mc{I}_s$ and with the lower bound $\wt P_J\ge\wt P\,\phi^\ast_k/4$, the difference $\mb{E}[A^\sigma_t\mid\mc{J}_{sb}(A^\sigma)]-\mb{E}[\wt A_t\mid\mc{J}_{sb}(\wt A)]$ equals $\sum_{t'}\beta_{t'}\mr{Cov}(\wt A_{t'},\wt A_t\mid\mc{J}_{sb}(\wt A))$ plus the error $O(E_1\sqrt{pN}\log^3N+\beta_0^2pN\sqrt{pN})$; the sum is $O(\norm\beta_\infty pN)=O(\tau\beta_0pN)$ and the error is smaller, by the computation in the proof of \ref{lr:mean}.

\emph{Proof of \ref{lr:nondeg}.} By \eqref{eq:lr-prob}, $\mb{P}[\mc{I}_s(A^\sigma)]=\wt P+O(\beta_0\sqrt{pN}+E_1\log^2N)$, and $\wt P\ge\phi^\ast_k/2$, so the first bound holds for large $N$ by \eqref{eq:beta0-small}. By \ref{lr:split}, $\mb{P}[\mc{J}_{sb}(A^\sigma)\mid\mc{I}_s(A^\sigma)]=\mb{P}[\mc{J}_{sb}(\wt A)\mid\mc{I}_s(\wt A)]+O(\beta_0\sqrt{pN})=\nu[sb]/\nu[s]+O(\epsilon_0+\beta_0\sqrt{pN})$, and $\nu[sb]/\nu[s]=\mb{P}[\mc{J}_{sb}(W_{s,\bullet})\mid\mc{I}_s(W_{s,\bullet})]\in[\phi^\ast_k,1-\phi^\ast_k]$ by \eqref{eq:nu-recursion} and \Cref{lem:universal-nondegeneracy}.
\end{proof}

\subsection{Proof of \Cref{thm:idealized-process}}

\begin{proof}[Proof of \Cref{thm:idealized-process}]
Throughout, constants implicit in $O(\cdot)$ depend only on $\theta$, $D$, and $T$, and ``for large $N$'' means for $N$ exceeding a threshold depending on the same parameters, uniformly in $p\in(T^{-1}N^{-\theta},TN^{-\theta})$. We write $\wt A:=\wt A_{s,\bullet}$ when $s$ is fixed.

\emph{The reflection.} Let $R$ be the reflection $(Rx)[t]:=x[\ol t]$ of the proof of \Cref{lem:universal-symmetry}. For every real vector $x$, every $s\in\{0,1\}^k$, and every $b$,
\begin{equation}\label{eq:exact-reflection}
\mc{I}_s(Rx)\iff\mc{I}_{\ol s}(x)\qquad\text{and}\qquad\mc{J}_{sb}(Rx)\iff\mc{J}_{\ol s\,\ol b}(x),
\end{equation}
including ties. Indeed, with $Y_r:=\sum_u(-1)^{u[r-1]}x[u]$, the $r$th inequality of $\mc{I}_s(Rx)$ reads
\[
-Y_r\ge_{s[r-1]s[r]}0,
\]
and by inspection of the table defining $\ge_{ab}$, the assertion $-Y\ge_{ab}0$ is equivalent to $Y\ge_{\ol a\,\ol b}0$ for each of the four pairs $ab$: for instance, $-Y\ge0$ holds if and only if $Y\le0$, which is $Y\ge_{11}0$. The same applies to the $k$th inequality of $\mc{J}_{sb}$.

\emph{Day one.} Define $\wt n_1$ and $\wt m_1$ by \ref{ideal:base}. Then \ref{ideal:symmetry} holds at $k=1$, and so does \ref{ideal:asymptotics} with any $\ell\ge1$: $|\wt n[s]-N/2|\le1$, and $|\wt m[s,t]-p\wt n[s]\wt n[t]|=p\wt n[s]\one_{s=t}\le p\wt n[s]\wt n[t]\frac{\log N}{pN}$ because $\wt n[t]\ge N/3$; recall $\nu[s]=1/2$ and $\mu\equiv0$ at day $1$.

\emph{Induction.} Suppose that for some $1\le k\le D-1$ the vectors $\wt n_j,\wt m_j$ ($j\le k$) and $\wt\lambda_j$ ($j\le k-1$) have been constructed so that \ref{ideal:base}--\ref{ideal:symmetry} hold as far as they concern these vectors, and \ref{ideal:asymptotics} holds for $j\le k$ with an exponent $\ell_k$ depending only on $k$. We construct $\wt\lambda_k$, $\wt n_{k+1}$, and $\wt m_{k+1}$, and prove \ref{ideal:solvable}--\ref{ideal:asymptotics} for them with an exponent $\ell_{k+1}$ depending only on $k$; the theorem then follows with $\ell:=\ell_D$, since $\log^{\ell_j}N\le\log^{\ell_D}N$.

\emph{Step 1: the sizes are balanced.} By \ref{ideal:symmetry} at day $k$, pairing $t$ with $\ol t$ gives $Z_{s,r}(\wt n_k)=0$ for all $s$ and $1\le r\le k$; that is, $M_s\wt n_k=0$ and $Z_{s,k}(\wt n_k)=0$.

\emph{Step 2: existence of $\wt\lambda_k$.} Fix $s\in\{0,1\}^k$. Let $v_s:=(\nu[t]\mu[s,t])_{t\in\{0,1\}^k}$, which lies in $\mc{C}_s$ with $Z_{s,r}(v_s)\ge\varsigma^\ast_k$ for $r\le k-1$ by \Cref{lem:universal-nondegeneracy}, and let $T'$ be a constant, depending only on $k$, with $T'\ge2\max_{s,t}|\nu[t]\mu[s,t]|+1$ and $T'\ge2/\varsigma^\ast_k$. Consider the compact set $\mc{K}_{T'}\subseteq\mc{C}_s$ of \Cref{thm:perturbed-bijection} for $M=M_s$, the strong bijection $f:=F_{s,\nu}\colon\mb{R}^{\{0,1\}^k}\to\mc{C}_s$ of \Cref{prop:cond-mean-bijection}, and the radius $R_0:=\max_{y\in\mc{K}_{T'}}|F_{s,\nu}^{-1}(y)|+1$, a constant depending only on $k$. Apply \Cref{lem:gaussian-limits} with $\eta:=\wt n_k$, exponent $\ell_k$, radius $R_0$, and $\xi:=\log^{\ell'}N/\sqrt{pN}$, where $\ell'=\ell'(k,\ell_k)$ is the exponent of that lemma; its hypotheses hold by \ref{ideal:asymptotics} at day $k$ and Step 1, and $\epsilon=2\log^{\ell'}N/\sqrt{pN}$. For $\sigma\in\mb{R}^{\{0,1\}^k}$ let $A^\sigma$ be as in that lemma, that is, $A^\sigma=A_{s,\bullet}(\wt n_k,\lambda^\sigma)$ with $\lambda^\sigma[s,t]:=\Lambda_{N,p}(\sigma[t],\nu[t])$, and define
\[
g(\sigma):=\left(\frac{\mb{E}[A^\sigma_t\mid\mc{I}_s(A^\sigma)]-p\,\wt n[t]}{\sqrt{pN}}\right)_{t\in\{0,1\}^k}.
\]
The map $g$ is defined and smooth on all of $\mb{R}^{\{0,1\}^k}$: $\mb{P}[\mc{I}_s(A^\sigma)]>0$ for every $\sigma$ because $A^\sigma$ has full support on its box and, for large $N$, the box contains points satisfying $\mc{I}_s$ (Step 4 below exhibits many), and the conditional expectation is a rational function with positive denominator of the smooth functions $\lambda^\sigma[s,t]$. By \Cref{lem:gaussian-limits}\ref{gl:mean}, $\sup_{|\sigma|\le R_0}|g(\sigma)-F_{s,\nu}(\sigma)|\le C_0\epsilon$ for a constant $C_0$ depending only on $\theta,k,T$. Let
\[
y^{(s)}:=\left(\frac{\wt m[s,t]/\wt n[s]-p\,\wt n[t]}{\sqrt{pN}}\right)_{t\in\{0,1\}^k}.
\]
By \ref{ideal:asymptotics} at day $k$, $\wt m[s,t]/\wt n[s]=p\wt n[t]\big(1+\frac{\mu[s,t]}{\sqrt{pN}}\big)+O(\log^{\ell_k}N)$ and $\wt n[t]/N=\nu[t]+O(\log^{\ell_k}N/\sqrt{pN})$, whence
\begin{equation}\label{eq:target-close}
\left\|y^{(s)}-v_s\right\|_\infty=O\left(\frac{\log^{\ell_k}N}{\sqrt{pN}}\right),
\end{equation}
so $y^{(s)}\in\mc{K}_{T'}$ for large $N$ by the choice of $T'$. \Cref{thm:perturbed-bijection} (with $\varepsilon:=\epsilon$, which is below its $\varepsilon_0$ for large $N$) yields $\sigma_s\in\mb{R}^{\{0,1\}^k}$ with $g(\sigma_s)=y^{(s)}$ and $|\sigma_s-F^{-1}_{s,\nu}(y^{(s)})|\le C\epsilon$. Since $F^{-1}_{s,\nu}(v_s)=\gamma[s,\bullet]$ by \Cref{def:universal-nu-mu}, \Cref{lem:strong-bijection-lipschitz} on $\mc{K}_{T'}$ and \eqref{eq:target-close} give
\begin{equation}\label{eq:sigma-close}
\left\|\sigma_s-\gamma[s,\bullet]\right\|_\infty=O\left(\frac{\log^{\ell''}N}{\sqrt{pN}}\right),\qquad\ell'':=\max\{\ell_k,\ell'\}.
\end{equation}
Set $\wt\lambda[s,t]:=\lambda^{\sigma_s}[s,t]=\Lambda_{N,p}(\sigma_s[t],\nu[t])$. Then $\wt A_{s,\bullet}=A^{\sigma_s}$, and $g(\sigma_s)=y^{(s)}$ says exactly that $\wt n[s]\,\mb{E}[\wt A_{s,t}\mid\mc{I}_s(\wt A)]=\wt m[s,t]$ for all $t$. Together with $\wt n[s]\ge1$, $\wt m[s,t]>0$ (from \ref{ideal:asymptotics}) and $\mb{P}[\mc{I}_s(\wt A)]>0$, this is \ref{ideal:solvable} except for uniqueness.

\emph{Step 3: the tilt estimate.} Since $\partial_g\Lambda_{N,p}(g,v)=\Lambda_{N,p}(g,v)(1-\Lambda_{N,p}(g,v))/(v\sqrt{pN})=O(p/\sqrt{pN})$ for bounded $g/v$, \eqref{eq:sigma-close} gives $|\wt\lambda[s,t]-\Lambda_{N,p}(\gamma[s,t],\nu[t])|=O\big(\frac{p}{\sqrt{pN}}\cdot\frac{\log^{\ell''}N}{\sqrt{pN}}\big)=O(\log^{\ell''}N/N)$, which is the last bound of \ref{ideal:asymptotics} at day $k$ once $\ell_{k+1}\ge\ell''+1$.

\emph{Step 4: uniqueness.} Apply \Cref{lem:tilt-uniqueness} with $d=2^k$, $\eta_t:=\wt n[t]-\one_{s=t}$, $M=M_s$, and $\ge'=\ge'_s$. Let $\rho_0:=\varsigma^\ast_k/2^{k+1}$ and let $B$ be the box of all $a\in\mb{Z}^{\{0,1\}^k}$ with $|a[t]-p\wt n[t]-\sqrt{pN}v_s[t]|\le\sqrt{pN}\rho_0$ for all $t$. For $a\in B$, Step 1 gives $M_sa=\sqrt{pN}\,M_s(v_s+w)$ with $\norm{w}_\infty\le\rho_0$, so $Z_{s,r}(a)\ge\sqrt{pN}(\varsigma^\ast_k-2^k\rho_0)>0$ for all $r\le k-1$; hence every $a\in B$ satisfies $\mc{I}_s(a)$ strictly, and $B\subseteq S$ for large $N$ (when $B$ lies inside $\prod_t\{0,\ldots,\eta_t\}$). A box of integer points of side at least $2$ contains $d+1$ affinely independent points, so $S$ is not contained in an affine hyperplane, and \Cref{lem:tilt-uniqueness} shows that $\wt\lambda[s,\bullet]$ is the unique tilt with the solvability property for the row $s$. This completes \ref{ideal:solvable}.

\emph{Step 5: symmetry of the tilt.} Define $\lambda'[s,t]:=\wt\lambda[\ol s,\ol t]$. The row vector $A'_{s,\bullet}:=A_{s,\bullet}(\wt n_k,\lambda')$ has independent coordinates $A'_{s,t}\sim\mr{Bin}(\wt n[t]-\one_{s=t},\wt\lambda[\ol s,\ol t])=\mr{Bin}(\wt n[\ol t]-\one_{\ol s=\ol t},\wt\lambda[\ol s,\ol t])$, using $\wt n[t]=\wt n[\ol t]$; this is the law of $\wt A_{\ol s,\ol t}$, so $A'_{s,\bullet}\overset{d}{=}R\wt A_{\ol s,\bullet}$. By \eqref{eq:exact-reflection} and \ref{ideal:solvable} for the row $\ol s$,
\[
\mb{E}\left[A'_{s,t}\mid\mc{I}_s(A'_{s,\bullet})\right]=\mb{E}\left[\wt A_{\ol s,\ol t}\mid\mc{I}_{\ol s}(\wt A_{\ol s,\bullet})\right]=\frac{\wt m[\ol s,\ol t]}{\wt n[\ol s]}=\frac{\wt m[s,t]}{\wt n[s]},
\]
so $\lambda'[s,\bullet]$ has the solvability property for the row $s$, and by uniqueness $\lambda'=\wt\lambda$, that is, $\wt\lambda[\ol s,\ol t]=\wt\lambda[s,t]$.

\emph{Step 6: day $k+1$ and its symmetry.} The local template $(\wt n^{\mr{loc}},\wt\ell^{\mr{loc}},\wt m^{\mr{loc}})$ of $(\wt n_k,\wt m_k,\wt\lambda_k)$ is defined, since the hypotheses of \Cref{def:local-template} are part of \ref{ideal:solvable}; define $\wt n_{k+1}$ and $\wt m_{k+1}$ by \ref{ideal:evolution}. By the symmetry of $\wt n_k$ and $\wt\lambda_k$, $\wt A_{\ol s,\bullet}\overset{d}{=}R\wt A_{s,\bullet}$, and by \eqref{eq:exact-reflection} the events $\mc{I}_{\ol s}(\wt A_{\ol s,\bullet})$ and $\mc{J}_{\ol s\,\ol b}(\wt A_{\ol s,\bullet})$ correspond to $\mc{I}_s(\wt A_{s,\bullet})$ and $\mc{J}_{sb}(\wt A_{s,\bullet})$ under this identification. Hence $\wt n^{\mr{loc}}[\ol s\,\ol b]=\wt n^{\mr{loc}}[sb]$, $\wt\ell^{\mr{loc}}[\ol s\,\ol b,\ol t]=\wt\ell^{\mr{loc}}[sb,t]$, and therefore $\wt m^{\mr{loc}}[\ol s\,\ol b,\ol t\,\ol c]=\wt m^{\mr{loc}}[sb,tc]$, which is \ref{ideal:symmetry} at day $k+1$.

\emph{Step 7: asymptotics at day $k+1$.} Fix $s,t$ and $b,c$. Let $W_{s,\bullet}$ be the day-$k$ Gaussian row model, whose means are $\gamma[s,\bullet]$. By \Cref{lem:gaussian-limits}\ref{gl:prob} and \ref{gl:mean} at $\sigma=\sigma_s$ (with $Z_{s,k}(\wt n_k)=0$ by Step 1, so the events $\mc{J}_{sb}$ are unshifted), followed by the Lipschitz dependence on the mean of \Cref{lem:gaussian-regularity} and \eqref{eq:sigma-close},
\begin{align}
\mb{P}\left[\mc{J}_{sb}(\wt A)\mid\mc{I}_s(\wt A)\right]&=\mb{P}\left[\mc{J}_{sb}(W_{s,\bullet})\mid\mc{I}_s(W_{s,\bullet})\right]+O(\epsilon'')=\frac{\nu[sb]}{\nu[s]}+O(\epsilon''),\label{eq:split-asym}\\
\mb{E}\left[\wt A_{s,t}\mid\mc{J}_{sb}(\wt A)\right]&=p\,\wt n[t]+\sqrt{pN}\left(\mb{E}\left[W_{s,t}\mid\mc{J}_{sb}(W_{s,\bullet})\right]+O(\epsilon'')\right),\label{eq:condmean-asym}
\end{align}
with $\epsilon'':=\log^{\ell''}N/\sqrt{pN}$, where \eqref{eq:split-asym} uses \eqref{eq:nu-recursion}. By \Cref{def:local-template} and \eqref{eq:split-asym}, $\wt n^{\mr{loc}}[sb]=\wt n[s]\,\mb{P}[\mc{J}_{sb}(\wt A)\mid\mc{I}_s(\wt A)]=N\nu[sb]+O(N\epsilon'')$, using $\wt n[s]=N\nu[s]+O(N\epsilon'')$; the floor changes this by at most $1$, which proves the size bound of \ref{ideal:asymptotics} at day $k+1$ with any $\ell_{k+1}\ge\ell''+1$. For the edge counts, write $e_{sb,t}:=\nu[t]^{-1}\mb{E}[W_{s,t}\mid\mc{J}_{sb}(W_{s,\bullet})]$, so that $\mu[sb,tc]=e_{sb,t}+e_{tc,s}-\mu[s,t]$ by \eqref{eq:mu-recursion}. Since $\wt\ell^{\mr{loc}}[sb,t]=\wt n^{\mr{loc}}[sb]\,\mb{E}[\wt A_{s,t}\mid\mc{J}_{sb}(\wt A)]$ and $\sqrt{pN}\,\nu[t]/(p\wt n[t])=(1+O(\epsilon''))/\sqrt{pN}$, \eqref{eq:condmean-asym} gives
\[
\wt\ell^{\mr{loc}}[sb,t]=\wt n^{\mr{loc}}[sb]\,p\,\wt n[t]\left(1+\frac{e_{sb,t}}{\sqrt{pN}}+O\left(\frac{\epsilon''}{\sqrt{pN}}\right)\right),
\]
and likewise for $\wt\ell^{\mr{loc}}[tc,s]$. Divide the product by $\wt m[s,t]=p\wt n[s]\wt n[t](1+\mu[s,t]/\sqrt{pN}+O(\log^{\ell_k}N/pN))$ and expand; the second-order terms are $O(1/pN)$, and we obtain
\begin{align*}
\wt m[sb,tc]&=p\,\wt n^{\mr{loc}}[sb]\,\wt n^{\mr{loc}}[tc]\left(1+\frac{e_{sb,t}+e_{tc,s}-\mu[s,t]}{\sqrt{pN}}+O\left(\frac{\log^{\ell''}N}{pN}\right)\right)\\
&=p\,\wt n^{\mr{loc}}[sb]\,\wt n^{\mr{loc}}[tc]\left(1+\frac{\mu[sb,tc]}{\sqrt{pN}}+O\left(\frac{\log^{\ell''}N}{pN}\right)\right).
\end{align*}
Replacing $\wt n^{\mr{loc}}$ by $\wt n_{k+1}$ changes the right-hand side by a relative $O(1/N)$, which is absorbed. This proves the edge-count bound of \ref{ideal:asymptotics} at day $k+1$ with $\ell_{k+1}:=\ell''+1$, and completes the induction. Finally, since \ref{ideal:base}, \ref{ideal:evolution}, and the uniqueness in \ref{ideal:solvable} determine the family recursively, it depends only on $(N,p)$.
\end{proof}

\subsection{Proof of \Cref{thm:perturbed-evolution}}

\begin{proof}[Proof of \Cref{thm:perturbed-evolution}]
Throughout, constants implicit in $O(\cdot)$ and thresholds for ``large $N$'' depend only on $\theta,k,T,\delta$; they are uniform in $p$, $\tau$, and $y$. Let $\ell=\ell(\theta,D_\theta,T)$ be the exponent of \Cref{thm:idealized-process}, put
\[
\beta_0:=\frac{\sqrt{pN}^{\,k-1}}{\sqrt N},\qquad c_1:=\tfrac12\left(1-k(1-\theta)\right)>0,\qquad\epsilon'':=\frac{\log^{\ell}N}{\sqrt{pN}},
\]
and recall from \eqref{eq:beta0-small} that $\beta_0\sqrt{pN}\le T^{k/2}N^{-c_1}$; in particular $\beta_0\le1/\sqrt{pN}$ for large $N$, and $\epsilon''\le N^{-(1-\theta)/4}$ for large $N$. Let $\kappa=\kappa(\theta,k)$ be the exponent of \Cref{lem:tilt-vs-linear-map}, and define
\[
\delta_1:=\tfrac12\min\left\{\kappa,\ \delta,\ c_1,\ \tfrac{1-\theta}{4}\right\},
\]
so that $N^{-\delta}$, $\epsilon''$, and $\beta_0$ are all $O(N^{-2\delta_1})$, and $1/\sqrt{pN}$, $\beta_0\sqrt{pN}$, $\epsilon''$, $N^{-\delta}$ are all at most $N^{-\delta_1}$ for large $N$.
Fix $y=(\Pi,m,1)\in\mc{G}_k^{(p)}(T,\delta,\tau)$ and write $n[s]:=|\Pi[s]|$. Conditions \ref{F:sizes} and \ref{F:edges} read
\begin{align}
n[t]&=\wt n[t]+\beta_0N\left(\tau\varepsilon[t]+O(N^{-\delta})\right),\label{eq:pe-sizes}\\
m[s,t]&=\wt m[s,t]\left(1+\tau\beta_0\left(\frac{\varepsilon[s]}{\nu[s]}+\frac{\varepsilon[t]}{\nu[t]}\right)\right)+O\left(\beta_0N^{-\delta}N^2p\right),\label{eq:pe-edges}
\end{align}
and \Cref{thm:idealized-process}\ref{ideal:asymptotics} gives $\wt n[t]=N\nu[t](1+O(\epsilon''))$ and $\wt m[s,t]=p\wt n[s]\wt n[t](1+\mu[s,t]/\sqrt{pN}+O(\epsilon''/\sqrt{pN}))$. In particular $|n[t]-\wt n[t]|\le T_2\beta_0N$ with $T_2:=T(\max_t|\varepsilon[t]|+1)$, so \Cref{lem:tilt-vs-linear-map} applies to $n$ with $T_2$ in place of $T$; we use it with the radius $R_0$ specified in Step 1, and write $A^\sigma$, $q^\sigma$ for the objects of that lemma for the row $s$, and $\wt A:=\wt A_{s,\bullet}$.

\emph{Step 1: the tilt.} Fix $s$. Let $f(\sigma):=\Sigma_s\sigma$, a linear bijection of $\mb{R}^{\{0,1\}^k}$ by \Cref{lem:cond-cov-pd}, hence a strong bijection onto $B=\mb{R}^{\{0,1\}^k}$; let $\mc{K}:=[-R_\varepsilon,R_\varepsilon]^{\{0,1\}^k}$ with $R_\varepsilon:=\max_t|\varepsilon[t]|+1$, and let $R_0:=\max_{y'\in\mc{K}}|\Sigma_s^{-1}y'|+1$, a constant depending only on $k$. Define
\[
g(\sigma):=\left(\frac{\mb{E}[A^\sigma_t\mid\mc{I}_s(A^\sigma)]-\wt m[s,t]/\wt n[s]}{\tau\beta_0\,pN}\right)_{t\in\{0,1\}^k},\qquad
y^{(s)}:=\left(\frac{m[s,t]/n[s]-\wt m[s,t]/\wt n[s]}{\tau\beta_0\,pN}\right)_{t\in\{0,1\}^k}.
\]
The map $g$ is continuous on $\mb{R}^{\{0,1\}^k}$ (it is smooth, by \Cref{lem:tilt-vs-linear-map}, whose argument for smoothness does not use the bound on $\sigma$), and \Cref{lem:tilt-vs-linear-map}\ref{lr:mean} gives $\sup_{|\sigma|\le R_0}|g(\sigma)-f(\sigma)|\le C_0N^{-\kappa}$. We compute $y^{(s)}$. By \eqref{eq:pe-sizes} and $N/\wt n[s]=\nu[s]^{-1}(1+O(\epsilon''))$, $n[s]=\wt n[s]\big(1+\tau\beta_0\varepsilon[s]/\nu[s]+O(\beta_0N^{-\delta}+\beta_0\epsilon'')\big)$, and dividing \eqref{eq:pe-edges} by this,
\[
\frac{m[s,t]}{n[s]}=\frac{\wt m[s,t]}{\wt n[s]}\left(1+\tau\beta_0\frac{\varepsilon[t]}{\nu[t]}+O\left(\beta_0N^{-\delta}+\beta_0\epsilon''+\beta_0^2\right)\right)+O\left(\beta_0N^{-\delta}Np\right).
\]
Since $\wt m[s,t]/\wt n[s]=pN\nu[t](1+O(\epsilon''))$, this gives
\[
\frac{m[s,t]}{n[s]}-\frac{\wt m[s,t]}{\wt n[s]}=\tau\beta_0pN\varepsilon[t]+O\left(\beta_0pN\left(N^{-\delta}+\epsilon''+\beta_0\right)\right),
\]
that is,
\begin{equation}\label{eq:pe-target}
y^{(s)}=\varepsilon+O\left(N^{-\delta}+\epsilon''+\beta_0\right)=\varepsilon+O\left(N^{-2\delta_1}\right)
\end{equation}
by the choice of $\delta_1$,
so $y^{(s)}\in\mc{K}$ for large $N$. \Cref{thm:perturbed-bijection} with $\varepsilon:=N^{-\kappa}$ yields $\sigma_s$ with $g(\sigma_s)=y^{(s)}$ and $|\sigma_s-\Sigma_s^{-1}y^{(s)}|\le CN^{-\kappa}$; since $\Sigma_s$ is symmetric, \eqref{eq:beta-defining} says $\Sigma_s\beta[s,\bullet]=\varepsilon$, so by \eqref{eq:pe-target}
\begin{equation}\label{eq:pe-sigma}
\left\|\sigma_s-\beta[s,\bullet]\right\|_\infty=O\left(N^{-2\delta_1}\right),
\end{equation}
and in particular $\norm{\sigma_s}_\infty\le R_0$ for large $N$. Set $\lambda_y[s,t]:=q^{\sigma_s}[t]$. Then $A_{s,\bullet}(n,\lambda_y)=A^{\sigma_s}$, and $g(\sigma_s)=y^{(s)}$ says exactly that $n[s]\,\mb{E}[A_{s,t}\mid\mc{I}_s(A_{s,\bullet})]=m[s,t]$ for all $t$; moreover $\mb{P}[\mc{I}_s(A_{s,\bullet})]>0$ by \Cref{lem:tilt-vs-linear-map}\ref{lr:nondeg}. By the definition of $q^\sigma$ in \Cref{lem:tilt-vs-linear-map}, $\vartheta_y[s,t]=\wt\vartheta[s,t]+\tau\beta_0\sigma_s[t]$, so \eqref{eq:pe-sigma} gives
\[
\left|\vartheta_y[s,t]-\wt\vartheta[s,t]-\tau\beta_0\beta[s,t]\right|=\tau\beta_0\left|\sigma_s[t]-\beta[s,t]\right|=O\left(\beta_0N^{-2\delta_1}\right),
\]
which is \ref{pe:tilt} once $T_1$ exceeds the implied constant.

\emph{Step 2: uniqueness.} As in Step 4 of the proof of \Cref{thm:idealized-process}, apply \Cref{lem:tilt-uniqueness} with $\eta_t:=n[t]-\one_{s=t}$: for $a$ in the box of integer points with $|a[t]-pn[t]-\sqrt{pN}v_s[t]|\le\sqrt{pN}\rho_0$, where $v_s:=(\nu[t]\mu[s,t])_t$ and $\rho_0:=\varsigma^\ast_k/2^{k+1}$, we have $M_sa=pM_sn+\sqrt{pN}M_s(v_s+w)$ with $\norm w_\infty\le\rho_0$; by \eqref{eq:pe-sizes}, the pairing symmetry of \Cref{thm:idealized-process}\ref{ideal:symmetry}, and \Cref{cor:sign-identities}\ref{sign:history}, $\norm{M_sn}_\infty=O(\beta_0N^{1-\delta})$, so $p\norm{M_sn}_\infty=O(\sqrt{pN}\cdot\beta_0\sqrt{pN}N^{-\delta})=o(\sqrt{pN})$, and every such $a$ satisfies $\mc{I}_s(a)$ strictly for large $N$. Hence the support is not contained in an affine hyperplane, and $\lambda_y[s,\bullet]$ is the unique tilt of the row $s$ with the solvability property.

\emph{Step 3: admissibility.} We verify \cref{LA:reg}--\cref{LA:split} for $(y,\lambda_y)$ with $\phi_1:=\phi^\ast_k/4$ and a constant $T_1$ depending only on $\theta,k,T,\delta$. \cref{LA:reg} is part of \Cref{def:faithful}. \cref{LA:sizes}: $n[s]=N\nu[s](1+O(\epsilon''+\beta_0))\ge N/T_1$ once $T_1\ge2/\min_s\nu[s]$. For $1\le r\le k-1$, by \eqref{eq:pe-sizes}, the pairing symmetry of $\wt n$, and \Cref{cor:sign-identities}\ref{sign:history},
\[
\left|\sum_t(-1)^{t[r-1]}n[t]\right|=\left|\tau\beta_0N\sum_t(-1)^{t[r-1]}\varepsilon[t]\right|+O\left(\beta_0N^{1-\delta}\right)=O\left(\beta_0N^{1-\delta}\right),
\]
which is at most $T_1N/\sqrt{pN}$ because $\beta_0\sqrt{pN}\le T^{k/2}$. \cref{LA:edge-scale} and \cref{LA:positive}: multiplying the expansions of $n[s]$ and $n[t]$ from \eqref{eq:pe-sizes},
\[
pn[s]n[t]=p\wt n[s]\wt n[t]\left(1+\tau\beta_0\left(\frac{\varepsilon[s]}{\nu[s]}+\frac{\varepsilon[t]}{\nu[t]}\right)+O\left(\beta_0N^{-\delta}+\beta_0\epsilon''+\beta_0^2\right)\right),
\]
and comparing with \eqref{eq:pe-edges} and the expansion of $\wt m[s,t]$ gives
\[
m[s,t]-pn[s]n[t]=p\wt n[s]\wt n[t]\cdot O\left(\frac{1}{\sqrt{pN}}+\frac{\log^{\ell}N}{pN}+\beta_0N^{-\delta}+\beta_0\epsilon''+\beta_0^2\right)=O\left(\frac{N^2p}{\sqrt{pN}}\right),
\]
where the first two terms are the $\mu[s,t]/\sqrt{pN}$ of $\wt m[s,t]$ and its error, the third is the error term of \eqref{eq:pe-edges} divided by $p\wt n[s]\wt n[t]\asymp pN^2$, and the last two come from the expansion of $pn[s]n[t]$; the last three are $O(\beta_0)=O(1/\sqrt{pN})$ because $\beta_0\sqrt{pN}=\sqrt{pN}^{\,k}/\sqrt N\le T^{k/2}N^{(k(1-\theta)-1)/2}\le T^{k/2}$. In particular $m[s,t]=p\wt n[s]\wt n[t](1+o(1))>0$. \cref{LA:separation}: fix $s$ and $1\le r\le k-1$. In all four cases of the relation $\ge_{s[r-1]s[r]}$, the condition is implied by $(-1)^{s[r]}\sum_t(-1)^{t[r-1]}m[s,t]\ge2T_1^{-1}N^2p/\sqrt{pN}$. By \eqref{eq:pe-edges} and the expansion of $\wt m$,
\begin{align*}
\sum_t(-1)^{t[r-1]}m[s,t]&=p\wt n[s]\sum_t(-1)^{t[r-1]}\wt n[t]\left(1+\frac{\mu[s,t]}{\sqrt{pN}}\right)\left(1+\tau\beta_0\left(\frac{\varepsilon[s]}{\nu[s]}+\frac{\varepsilon[t]}{\nu[t]}\right)\right)\\
&\qquad+O\left(\frac{N^2p\,\epsilon''}{\sqrt{pN}}+\beta_0N^{-\delta}N^2p\right)\\
&=p\wt n[s]\left[\frac{N}{\sqrt{pN}}(-1)^{s[r]}Z_{s,r}(v_s)+\tau\beta_0N\sum_t(-1)^{t[r-1]}\varepsilon[t]\right]\\
&\qquad+O\left(\frac{N^2p}{\sqrt{pN}}\left(\epsilon''+\beta_0\log^\ell N+\beta_0N^{-\delta}\sqrt{pN}\right)\right),
\end{align*}
all three error terms being $o(N^2p/\sqrt{pN})$,
where we used $\sum_t(-1)^{t[r-1]}\wt n[t]=0$ (pairing), $\wt n[t]=N\nu[t]+O(N\epsilon'')$, $\wt n[t]/\nu[t]=N+O(N\epsilon'')$, and $\beta_0\le1/\sqrt{pN}$. The second bracketed term vanishes by \Cref{cor:sign-identities}\ref{sign:history}, and $Z_{s,r}(v_s)\ge\varsigma^\ast_k$ by \Cref{lem:universal-nondegeneracy}; hence $(-1)^{s[r]}\sum_t(-1)^{t[r-1]}m[s,t]\ge\frac12\nu[s]\varsigma^\ast_k\,N^2p/\sqrt{pN}$ for large $N$, which is \cref{LA:separation} once $T_1\ge4/(\min_s\nu[s]\,\varsigma^\ast_k)$. \cref{LA:tilt} is the bound $|q^{\sigma_s}[t]-p|\le Cp/\sqrt{pN}$ of \Cref{lem:tilt-vs-linear-map}. \cref{LA:conditioning} and \cref{LA:split} are \Cref{lem:tilt-vs-linear-map}\ref{lr:nondeg} with $\phi_1=\phi^\ast_k/4$. \cref{LA:solvability} was established in Step 1. This proves \ref{pe:admissible}.

\emph{Step 4: the template sizes.} Fix $s,b$, and let $(\wt n^{\mr{loc}},\wt\ell^{\mr{loc}},\wt m^{\mr{loc}})$ be the local template of $(\wt n_k,\wt m_k,\wt\lambda_k)$, so that $\wt n^{\mr{loc}}[sb]=\wt n[s]\,\mb{P}[\mc{J}_{sb}(\wt A)\mid\mc{I}_s(\wt A)]$, $\wt n[sb]=\lfloor\wt n^{\mr{loc}}[sb]\rfloor$, and $\wt m[sb,tc]=\wt m^{\mr{loc}}[sb,tc]$ by \Cref{thm:idealized-process}\ref{ideal:evolution}. By \Cref{lem:tilt-vs-linear-map}\ref{lr:split} at $\sigma_s$, then \eqref{eq:pe-sigma}, then \eqref{eq:epsilon-recursion},
\begin{align*}
\mb{P}[\mc{J}_{sb}(A_{s,\bullet})\mid\mc{I}_s(A_{s,\bullet})]&=\frac{\wt n^{\mr{loc}}[sb]}{\wt n[s]}+\tau\beta_0\sqrt{pN}\,\frac{\nu[sb]}{\nu[s]}\sum_{t'}\beta[s,t']\Big(\mb{E}[W_{s,t'}\mid\mc{J}_{sb}]-\mb{E}[W_{s,t'}\mid\mc{I}_s]\Big)\\
&\qquad+O\left(\beta_0\sqrt{pN}N^{-2\delta_1}\right)\\
&=\frac{\wt n^{\mr{loc}}[sb]}{\wt n[s]}+\tau\beta_0\sqrt{pN}\,\frac{\varepsilon[sb]}{\nu[s]}+O\left(\beta_0\sqrt{pN}N^{-2\delta_1}\right),
\end{align*}
where $\mc{J}_{sb}$ and $\mc{I}_s$ abbreviate $\mc{J}_{sb}(W_{s,\bullet})$ and $\mc{I}_s(W_{s,\bullet})$. Multiplying by $n[s]=\wt n[s]+\beta_0N(\tau\varepsilon[s]+O(N^{-\delta}))$ and using $\wt n^{\mr{loc}}[sb]/\wt n[s]=\nu[sb]/\nu[s]+O(\epsilon'')$, which follows from \Cref{thm:idealized-process}\ref{ideal:asymptotics} at days $k$ and $k+1$ since $\wt n^{\mr{loc}}[sb]\in[\wt n[sb],\wt n[sb]+1)$, together with $n[s]=N\nu[s](1+O(\epsilon''+\beta_0))$,
\begin{multline*}
n^{\mr{loc}}[sb]=\wt n^{\mr{loc}}[sb]+\tau\beta_0N\sqrt{pN}\,\varepsilon[sb]+\tau\beta_0N\varepsilon[s]\frac{\nu[sb]}{\nu[s]}\\
+O\left(\beta_0N\sqrt{pN}\left(N^{-2\delta_1}+\epsilon''+\beta_0\right)+\beta_0N\left(N^{-\delta}+\epsilon''\right)\right).
\end{multline*}
Now $\tau\beta_0N\sqrt{pN}\varepsilon[sb]=\tau\sqrt N\sqrt{pN}^{\,k}\varepsilon[sb]$, $0\le\wt n^{\mr{loc}}[sb]-\wt n[sb]<1$, and every remaining term is $O(\beta_0N\sqrt{pN}\cdot N^{-\delta_1})$: the third term because $1/\sqrt{pN}\le N^{-\delta_1}$, and the error terms because $N^{-2\delta_1}$, $\epsilon''$, $\beta_0$, and $N^{-\delta}$ are all at most $N^{-\delta_1}$ for large $N$, by the choice of $\delta_1$. This is the size bound of \ref{pe:template}.

\emph{Step 5: the template edge counts.} Fix $s,t,b,c$. By \Cref{def:local-template}, $\ell^{\mr{loc}}[sb,t]=n^{\mr{loc}}[sb]\,\mb{E}[A_{s,t}\mid\mc{J}_{sb}(A_{s,\bullet})]$ and $\wt\ell^{\mr{loc}}[sb,t]=\wt n^{\mr{loc}}[sb]\,\mb{E}[\wt A_{s,t}\mid\mc{J}_{sb}(\wt A)]$, where $\mb{E}[\wt A_{s,t}\mid\mc{J}_{sb}(\wt A)]=p\wt n[t]+O(\sqrt{pN})\asymp pN$ by \Cref{lem:gaussian-limits}\ref{gl:mean}. By \Cref{lem:tilt-vs-linear-map}\ref{lr:condmean}, $\mb{E}[A_{s,t}\mid\mc{J}_{sb}(A_{s,\bullet})]=\mb{E}[\wt A_{s,t}\mid\mc{J}_{sb}(\wt A)](1+O(\beta_0))$, and therefore
\[
m^{\mr{loc}}[sb,tc]=\frac{\ell^{\mr{loc}}[sb,t]\,\ell^{\mr{loc}}[tc,s]}{m[s,t]}=\wt m^{\mr{loc}}[sb,tc]\cdot\frac{n^{\mr{loc}}[sb]}{\wt n^{\mr{loc}}[sb]}\cdot\frac{n^{\mr{loc}}[tc]}{\wt n^{\mr{loc}}[tc]}\cdot\frac{\wt m[s,t]}{m[s,t]}\cdot\left(1+O(\beta_0)\right).
\]
By Step 4 and $\wt n^{\mr{loc}}[sb]=N\nu[sb](1+O(\epsilon''))$,
\[
\frac{n^{\mr{loc}}[sb]}{\wt n^{\mr{loc}}[sb]}=1+\tau\beta_0\sqrt{pN}\,\frac{\varepsilon[sb]}{\nu[sb]}+O\left(\beta_0\sqrt{pN}\left(N^{-\delta_1}+\epsilon''\right)\right),
\]
and likewise for $tc$; by \eqref{eq:pe-edges}, $\wt m[s,t]/m[s,t]=1+O(\beta_0)$. Multiplying out, with $\wt m^{\mr{loc}}[sb,tc]=\wt m[sb,tc]\asymp N^2p$ and second-order terms $O(\beta_0^2pN)$,
\begin{multline*}
m^{\mr{loc}}[sb,tc]=\wt m[sb,tc]\left(1+\tau\beta_0\sqrt{pN}\left(\frac{\varepsilon[sb]}{\nu[sb]}+\frac{\varepsilon[tc]}{\nu[tc]}\right)\right)\\
+O\left(N^2p\left(\beta_0\sqrt{pN}\left(N^{-\delta_1}+\epsilon''\right)+\beta_0+\beta_0^2pN\right)\right).
\end{multline*}
Since $\beta_0\sqrt{pN}=\sqrt{pN}^k/\sqrt N$, the main term is the one in \ref{pe:template}. In the error, $\beta_0=\beta_0\sqrt{pN}\cdot(1/\sqrt{pN})$ and $\beta_0^2pN=\beta_0\sqrt{pN}\cdot\beta_0\sqrt{pN}$ are both at most $\beta_0\sqrt{pN}\,N^{-\delta_1}$ for large $N$, and $\epsilon''\le N^{-\delta_1}$, by the choice of $\delta_1$; so the error is $O(\beta_0\sqrt{pN}N^{-\delta_1}N^2p)$. This is the edge-count bound of \ref{pe:template} and completes the proof, with $T_1$ the maximum of $T$ and the finitely many constants required above.
\end{proof}

\subsection{Proof of \Cref{thm:edge-day}}

\begin{proof}[Proof of \Cref{thm:edge-day}]
Throughout, constants implicit in $O(\cdot)$ and thresholds for ``large $N$'' depend only on $\theta,T,\delta$, uniformly in $p$, $\tau$, and $y$. Write $k:=k_0$, let $\ell$ be the exponent of \Cref{thm:idealized-process} with $D=D_\theta$, and put
\[
\beta_0:=\frac{\sqrt{pN}^{\,k-1}}{\sqrt N},\qquad\epsilon'':=\frac{\log^\ell N}{\sqrt{pN}},\qquad\varpi:=pN^{\theta}\in(T^{-1},T).
\]
Since $k(1-\theta)=1$, we have $pN=\varpi N^{1-\theta}$ and hence the exact identity
\begin{equation}\label{eq:edge-beta0}
\beta_0\sqrt{pN}=\frac{(pN)^{k/2}}{\sqrt N}=\varpi^{k/2}\in\left(T^{-k/2},T^{k/2}\right);
\end{equation}
Put $\delta':=\min\{\delta,(1-\theta)/4\}$ and $\delta_1:=\delta'/2$. Fix $y=(\Pi,m,1)\in\mc{G}^{(p)}_{k}(T,\delta,\tau)$, write $n[s]:=|\Pi[s]|$, and note that \eqref{eq:pe-sizes} and \eqref{eq:pe-edges} hold verbatim (they are \ref{F:sizes} and \ref{F:edges}), as do $\wt n[t]=N\nu[t](1+O(\epsilon''))$ and $\wt m[s,t]=p\wt n[s]\wt n[t](1+\mu[s,t]/\sqrt{pN}+O(\epsilon''/\sqrt{pN}))$.

\emph{Step 1: the actual sizes as an input to \Cref{lem:gaussian-limits}.} By \eqref{eq:pe-sizes}, $|n[t]-N\nu[t]|\le\frac{N}{\sqrt{pN}}\log^\ell N+O(\beta_0N)\le\frac{N}{\sqrt{pN}}\log^{\ell+1}N$ for large $N$, by \eqref{eq:edge-beta0}. By \eqref{eq:pe-sizes}, the pairing symmetry of \Cref{thm:idealized-process}\ref{ideal:symmetry}, and \Cref{cor:sign-identities}\ref{sign:history},
\[
M_sn=\tau\beta_0N\,M_s\varepsilon+O\left(\beta_0N^{1-\delta}\right)=O\left(\beta_0N^{1-\delta}\right)=O\left(N^{-\delta}\frac{N}{\sqrt{pN}}\right),
\]
so the hypothesis $\norm{M_sn}_\infty\le\xi N/\sqrt{pN}$ of \Cref{lem:gaussian-limits} holds with $\xi:=C_\xi N^{-\delta}$ for a constant $C_\xi$. On the other hand, by \Cref{cor:sign-identities}\ref{sign:lead},
\begin{equation}\label{eq:edge-shift}
u_n:=\frac{p\,Z_{s,k}(n)}{\sqrt{pN}}=\frac{p}{\sqrt{pN}}\left(\tau\beta_0N\,S_k+O(\beta_0N^{1-\delta})\right)=\tau\,\varpi^{k/2}S_k+O(N^{-\delta}),
\end{equation}
where $S_k:=\sum_{t}(-1)^{t[k-1]}\varepsilon[t]=2\sum_{t'\in\{0,1\}^{k-1}}\varepsilon[t'0]>0$;
so that for large $N$, $u_n\in[u_{\min},u_{\max}]$ with $u_{\min}:=\frac12T^{-1-k/2}S_k>0$ and $u_{\max}:=2T^{1+k/2}S_k$, constants depending only on $\theta,T$. In particular $|Z_{s,k}(n)|\le T_3N/\sqrt{pN}$ for a constant $T_3$. We apply \Cref{lem:gaussian-limits} with $\eta:=n$, the constant $\max\{T,T_3\}$, the exponent $\ell+1$, the radius $R_0$ of Step 2 of the proof of \Cref{thm:idealized-process}, and this $\xi$; its $\epsilon$ satisfies $\epsilon=O(\log^{\ell'}N/\sqrt{pN}+N^{-\delta})=O(N^{-\delta'})$. For $\sigma\in\mb{R}^{\{0,1\}^k}$ let $A^\sigma=A_{s,\bullet}(n,\lambda^\sigma)$ with $\lambda^\sigma[s,t]:=\Lambda_{N,p}(\sigma[t],\nu[t])$ be the row vector of that lemma.

\emph{Step 2: the tilt.} Fix $s$ and define
\[
g(\sigma):=\left(\frac{\mb{E}[A^\sigma_t\mid\mc{I}_s(A^\sigma)]-p\,n[t]}{\sqrt{pN}}\right)_{t},\qquad
y^{(s)}:=\left(\frac{m[s,t]/n[s]-p\,n[t]}{\sqrt{pN}}\right)_{t},\qquad v_s:=\left(\nu[t]\mu[s,t]\right)_t .
\]
As in Step 2 of the proof of \Cref{thm:idealized-process}, $g$ is smooth on $\mb{R}^{\{0,1\}^k}$, and \Cref{lem:gaussian-limits}\ref{gl:mean} gives $\sup_{|\sigma|\le R_0}|g(\sigma)-F_{s,\nu}(\sigma)|=O(N^{-\delta'})$. We compute $y^{(s)}$. Exactly as in Step 1 of the proof of \Cref{thm:perturbed-evolution} (that computation did not use $k<\frac1{1-\theta}$),
\[
\frac{m[s,t]}{n[s]}=\frac{\wt m[s,t]}{\wt n[s]}\left(1+\tau\beta_0\frac{\varepsilon[t]}{\nu[t]}+O\left(\beta_0N^{-\delta}+\beta_0\epsilon''+\beta_0^2\right)\right)+O\left(\beta_0N^{-\delta}Np\right),
\]
while $p\,n[t]=p\wt n[t]+p\beta_0N(\tau\varepsilon[t]+O(N^{-\delta}))$ by \eqref{eq:pe-sizes}. Subtracting, and using $\wt m[s,t]/\wt n[s]=p\wt n[t](1+\mu[s,t]/\sqrt{pN}+O(\epsilon''/\sqrt{pN}))$ and $\wt n[t]/\nu[t]=N(1+O(\epsilon''))$, the two terms involving $\varepsilon[t]$ cancel to first order:
\begin{align*}
\frac{m[s,t]}{n[s]}-p\,n[t]&=p\wt n[t]\frac{\mu[s,t]}{\sqrt{pN}}+p\tau\beta_0\varepsilon[t]\left(\frac{\wt n[t]}{\nu[t]}-N\right)+O\left(pN\beta_0\left(N^{-\delta}+\epsilon''+\beta_0\right)+\frac{pN\epsilon''}{\sqrt{pN}}\right)\\
&=p\wt n[t]\frac{\mu[s,t]}{\sqrt{pN}}+O\left(pN\beta_0\left(N^{-\delta}+\epsilon''+\beta_0\right)+\frac{pN\epsilon''}{\sqrt{pN}}\right).
\end{align*}
Dividing by $\sqrt{pN}$ and using \eqref{eq:edge-beta0}, which gives $\beta_0\sqrt{pN}=O(1)$ and $\beta_0=O(1/\sqrt{pN})$,
\begin{equation}\label{eq:edge-target}
y^{(s)}=v_s+O\left(N^{-\delta}+\epsilon''+\beta_0\right)=v_s+O\left(N^{-\delta'}\right).
\end{equation}
With $T'$ and $\mc{K}_{T'}$ as in Step 2 of the proof of \Cref{thm:idealized-process}, $y^{(s)}\in\mc{K}_{T'}$ for large $N$, and \Cref{thm:perturbed-bijection} with $\varepsilon:=N^{-\delta'}$ yields $\sigma_s$ with $g(\sigma_s)=y^{(s)}$ and, by \Cref{lem:strong-bijection-lipschitz} and \eqref{eq:edge-target},
\begin{equation}\label{eq:edge-sigma}
\left\|\sigma_s-\gamma[s,\bullet]\right\|_\infty=O\left(N^{-\delta'}\right).
\end{equation}
Set $\lambda_y[s,t]:=\Lambda_{N,p}(\sigma_s[t],\nu[t])$, so that $A_{s,\bullet}(n,\lambda_y)=A^{\sigma_s}$, $\mb{P}[\mc{I}_s(A_{s,\bullet})]>0$, and $g(\sigma_s)=y^{(s)}$ is the solvability identity $n[s]\,\mb{E}[A_{s,t}\mid\mc{I}_s(A_{s,\bullet})]=m[s,t]$. Since $\partial_g\Lambda_{N,p}=O(p/\sqrt{pN})$ on bounded sets, \eqref{eq:edge-sigma} gives $|\lambda_y[s,t]-\Lambda_{N,p}(\gamma[s,t],\nu[t])|=O(N^{-\delta'}p/\sqrt{pN})\le N^{-\delta_1}p/\sqrt{pN}$ for large $N$, which is \ref{ed:tilt}. Uniqueness follows from \Cref{lem:tilt-uniqueness} exactly as in Step 2 of the proof of \Cref{thm:perturbed-evolution}: $p\norm{M_sn}_\infty=O(pN^{-\delta}N/\sqrt{pN})=O(N^{-\delta}\sqrt{pN})=o(\sqrt{pN})$, so the box of lattice points around $pn+\sqrt{pN}v_s$ of side $2\sqrt{pN}\varsigma^\ast_k/2^{k+1}$ lies in the support.

\emph{Step 3: admissibility.} Conditions \cref{LA:reg}, \cref{LA:sizes}, \cref{LA:edge-scale}, \cref{LA:positive}, \cref{LA:separation}, and \cref{LA:solvability} are verified exactly as in Step 3 of the proof of \Cref{thm:perturbed-evolution}, whose computations used only \eqref{eq:pe-sizes}, \eqref{eq:pe-edges}, the asymptotics of $\wt n$ and $\wt m$, \Cref{cor:sign-identities}\ref{sign:history} for $r\le k-1$, and the bounds $\beta_0\le T^{k/2}/\sqrt{pN}$ and $\beta_0N^{-\delta}\sqrt{pN}\le T^{k/2}N^{-\delta}$, all of which hold here by \eqref{eq:edge-beta0}. \cref{LA:tilt}: $|\lambda_y[s,t]-p|\le|\Lambda_{N,p}(\gamma[s,t],\nu[t])-p|+O(p/\sqrt{pN})=O(p/\sqrt{pN})$ by \Cref{def:logit-tilt}. \cref{LA:conditioning}: by \Cref{lem:gaussian-limits}\ref{gl:prob} at $\sigma_s$, \Cref{lem:gaussian-regularity}, and \eqref{eq:edge-sigma}, $\mb{P}[\mc{I}_s(A_{s,\bullet})]=\mb{P}[\mc{I}_s(W_{s,\bullet})]+O(N^{-\delta'})\ge\phi^\ast_k/2$ for large $N$, where $W_{s,\bullet}$ is the day-$k$ Gaussian row model. \cref{LA:split}: by \Cref{lem:gaussian-limits}\ref{gl:prob} at $\sigma_s$ with the shifted events, $\mb{P}[\mc{J}_{sb}(A_{s,\bullet})\mid\mc{I}_s(A_{s,\bullet})]=\mb{P}[\mc{J}^{u_n}_{sb}(W^{\sigma_s})\mid\mc{I}_s(W^{\sigma_s})]+O(N^{-\delta'})$, and $\mb{P}[\mc{J}^{u_n}_{sb}(W^{\sigma_s})]\ge c_0$ for both $b$, where $c_0$ depends only on $k,T,R_0$; since the two events partition $\mc{I}_s(W^{\sigma_s})$ up to a null set, the conditional probability lies in $[c_0,1-c_0]$, and \cref{LA:split} holds with $\phi_1:=\min\{\phi^\ast_k,c_0\}/4$ for large $N$. This proves \ref{ed:admissible} with $T_1$ the maximum of $T$ and the constants required.

\emph{Step 4: the shifted decision.} Let $W:=W_{s,\bullet}$ be the day-$k$ Gaussian row model and, for $u\in\mb{R}$, let
\[
\pi_b(u):=\mb{P}\left[\mc{J}^{u}_{sb}(W)\mid\mc{I}_s(W)\right],\qquad\text{so that}\qquad\pi_b(0)=\frac{\nu[sb]}{\nu[s]}
\]
by \eqref{eq:nu-recursion}. For $u>0$, $\mc{J}^u_{s0}(W)\supseteq\mc{J}^0_{s0}(W)$ and $\pi_0(u)-\pi_0(0)=\mb{P}[-u<Z_{s,k}(W)\le0\mid\mc{I}_s(W)]$, while $\pi_1(u)-\pi_1(0)=-(\pi_0(u)-\pi_0(0))$ because $\pi_0(u)+\pi_1(u)=1$. The coefficient vectors of $Z_{s,1},\ldots,Z_{s,k}$ are linearly independent, so $x\mapsto(Z_{s,1}(x),\ldots,Z_{s,k}(x))$ maps $\mb{R}^{\{0,1\}^k}$ onto $\mb{R}^k$, and the conditional law of $Z_{s,k}(W)$ given $W\in\mc{C}_s$ has a density that is positive on all of $\mb{R}$ (the set $\mc{C}_s\cap\{Z_{s,k}=z\}$ is a nonempty relatively open subset of the hyperplane for every $z$, and the density of $W$ is positive). Hence $\pi_0(u)-\pi_0(0)>0$ for every $u>0$, and since it is continuous in $u$, taking the minimum over $u\in[u_{\min},u_{\max}]$ and the finitely many $s\in\{0,1\}^k$ gives a constant $\zeta'>0$ depending only on $\theta,T$. Thus
\begin{equation}\label{eq:edge-zeta}
(-1)^b\left(\pi_b(u)-\frac{\nu[sb]}{\nu[s]}\right)\ge\zeta'\qquad\text{for all }u\in[u_{\min},u_{\max}]\text{ and }b\in\{0,1\}.
\end{equation}

\emph{Step 5: the template sizes.} By \Cref{def:local-template}, \Cref{lem:gaussian-limits}\ref{gl:prob} at $\sigma_s$ with the shifted events, \Cref{lem:gaussian-regularity} (Lipschitz dependence on the mean, at the fixed shift $u_n$), and \eqref{eq:edge-sigma},
\[
n^{\mr{loc}}[sb]=n[s]\,\mb{P}\left[\mc{J}_{sb}(A_{s,\bullet})\mid\mc{I}_s(A_{s,\bullet})\right]=n[s]\left(\pi_b(u_n)+O(N^{-\delta'})\right).
\]
Therefore, by \eqref{eq:edge-zeta}, $u_n\in[u_{\min},u_{\max}]$, and $|n[s]-N\nu[s]|=O(N\epsilon''+\beta_0N)=o(N)$,
\begin{align*}
(-1)^b\left(n^{\mr{loc}}[sb]-N\nu[sb]\right)&=n[s]\,(-1)^b\left(\pi_b(u_n)-\frac{\nu[sb]}{\nu[s]}\right)+\frac{\nu[sb]}{\nu[s]}(-1)^b\left(n[s]-N\nu[s]\right)+O\left(N^{1-\delta'}\right)\\
&\ge n[s]\zeta'-o(N)\ge\tfrac12\nu[s]\zeta'N-o(N),
\end{align*}
which is at least $\zeta N$ for large $N$ with $\zeta:=\frac14\min_s\nu[s]\,\zeta'$, a constant depending only on $\theta$ and $T$. Thus the fixed displacement of the conditional split probability contributes the order-$N$ first term, while the discrepancy of the parent size and the Gaussian-comparison error are both $o(N)$. This is \ref{ed:sizes}.
\end{proof}

\bibliographystyle{amsplain0}
\bibliography{main}

\end{document}